\documentclass[oneside,english,american]{amsart}
\usepackage[T1]{fontenc}
\usepackage[latin9]{inputenc}
\usepackage{color}
\usepackage{babel}
\usepackage{float}
\usepackage{units}
\usepackage{textcomp}
\usepackage{url}
\usepackage{amsbsy}
\usepackage{amstext}
\usepackage{amsthm}
\usepackage{amssymb}
\usepackage{mathdots}
\usepackage{graphicx}
\usepackage{xargs}[2008/03/08]
\usepackage[unicode=true,pdfusetitle,
 bookmarks=false,
 breaklinks=true,pdfborder={0 0 1},backref=false,colorlinks=true]
 {hyperref}

\makeatletter

\newcommand{\noun}[1]{\textsc{#1}}
\providecommand{\tabularnewline}{\\}

\numberwithin{equation}{section}
\numberwithin{figure}{section}
\newcommand\thmsname{\protect\theoremname}
\newcommand\nm@thmtype{theorem}
\theoremstyle{plain}

\newenvironment{namedthm}[1][Undefined Theorem Name]{
  \ifx{#1}{Undefined Theorem Name}\renewcommand\nm@thmtype{theorem*}
  \else\renewcommand\thmsname{#1}\renewcommand\nm@thmtype{namedtheorem}
  \fi
  \begin{\nm@thmtype}}
  {\end{\nm@thmtype}}
\theoremstyle{plain}
\newtheorem{thm}{\protect\theoremname}[section]
\theoremstyle{definition}
\newtheorem{defn}[thm]{\protect\definitionname}
\theoremstyle{remark}
\newtheorem{rem}[thm]{\protect\remarkname}
\theoremstyle{definition}
\newtheorem*{defn*}{\protect\definitionname}
\theoremstyle{definition}
\newtheorem*{example*}{\protect\examplename}
\theoremstyle{plain}
\newtheorem{prop}[thm]{\protect\propositionname}
\theoremstyle{plain}
\newtheorem{lem}[thm]{\protect\lemmaname}
\theoremstyle{plain}
\newtheorem{cor}[thm]{\protect\corollaryname}
\newenvironment{lyxlist}[1]
	{\begin{list}{}
		{\settowidth{\labelwidth}{#1}
		 \setlength{\leftmargin}{\labelwidth}
		 \addtolength{\leftmargin}{\labelsep}
		 }}
	{\end{list}}
\theoremstyle{plain}
\newtheorem{conjecture}[thm]{\protect\conjecturename}

\usepackage[notextcomp]{kpfonts}
\usepackage{esint}

\@ifundefined{showcaptionsetup}{}{%
 \PassOptionsToPackage{caption=false}{subfig}}
\usepackage{subfig}
\makeatother

\addto\captionsamerican{\renewcommand{\conjecturename}{Conjecture}}
\addto\captionsamerican{\renewcommand{\corollaryname}{Corollary}}
\addto\captionsamerican{\renewcommand{\definitionname}{Definition}}
\addto\captionsamerican{\renewcommand{\examplename}{Example}}
\addto\captionsamerican{\renewcommand{\lemmaname}{Lemma}}
\addto\captionsamerican{\renewcommand{\propositionname}{Proposition}}
\addto\captionsamerican{\renewcommand{\remarkname}{Remark}}
\addto\captionsamerican{\renewcommand{\theoremname}{Theorem}}
\addto\captionsenglish{\renewcommand{\conjecturename}{Conjecture}}
\addto\captionsenglish{\renewcommand{\corollaryname}{Corollary}}
\addto\captionsenglish{\renewcommand{\definitionname}{Definition}}
\addto\captionsenglish{\renewcommand{\examplename}{Example}}
\addto\captionsenglish{\renewcommand{\lemmaname}{Lemma}}
\addto\captionsenglish{\renewcommand{\propositionname}{Proposition}}
\addto\captionsenglish{\renewcommand{\remarkname}{Remark}}
\addto\captionsenglish{\renewcommand{\theoremname}{Theorem}}
\providecommand{\conjecturename}{Conjecture}
\providecommand{\corollaryname}{Corollary}
\providecommand{\definitionname}{Definition}
\providecommand{\examplename}{Example}
\providecommand{\lemmaname}{Lemma}
\providecommand{\propositionname}{Proposition}
\providecommand{\remarkname}{Remark}
\providecommand{\theoremname}{Theorem}

\begin{document}

%textes

\global\long\def\tx#1{\mathrm{#1}}%
\global\long\def\dd#1{\tx d#1}%
\global\long\def\tt#1{\mathtt{#1}}%
\global\long\def\ww#1{\mathbb{#1}}%
\global\long\def\DD#1{\tx D#1}%
\global\long\def\nf#1#2{\nicefrac{#1}{#2}}%
\global\long\def\group#1{{#1}}%

\newcommand{\bigslant}[2]{{\raisebox{.3em}{$#1$}/\raisebox{-.3em}{$#2$}}}

%algébre

\global\long\def\quot#1#2{\bigslant{#1}{#2}}%

\global\long\def\rr{\mathbb{R}}%
\global\long\def\rbar{\overline{\mathbb{R}}}%
\global\long\def\cc{\mathbb{C}}%
\global\long\def\cbar{\overline{\cc}}%
\global\long\def\disc{\mathbb{D}}%
\global\long\def\dbar{\overline{\disc}}%
\global\long\def\zz{\mathbb{Z}}%
\global\long\def\zp{\mathbb{Z}_{\geq0}}%
\newcommandx\zsk[1][usedefault, addprefix=\global, 1=k]{\nicefrac{\zz}{#1\zz}}%
\global\long\def\nn{\mathbb{N}}%
\global\long\def\nbar{\overline{\nn}}%
\global\long\def\qq{\mathbb{Q}}%
\global\long\def\qbar{\overline{\qq}}%

\global\long\def\rat#1{\cc\left(#1\right) }%
\global\long\def\pol#1{\cc\left[#1\right] }%
\global\long\def\id{\tx{Id}}%

\global\long\def\GL#1#2{\tx{GL}_{#1}\left(#2\right)}%

\global\long\def\spec#1{\tx{Spec}\left(#1\right)}%

%feuilletages

\global\long\def\fol#1{\mathcal{F}_{#1}}%
\global\long\def\pp#1{\frac{\partial}{\partial#1}}%
\global\long\def\ppp#1#2{\frac{\partial#1}{\partial#2}}%
\newcommandx\sat[2][usedefault, addprefix=\global, 1=\fol{}]{\tx{Sat}_{#1}\left(#2\right)}%
\global\long\def\lif#1{\mathcal{L}_{#1}}%
\global\long\def\holo#1{\mathfrak{h}_{#1}}%
\global\long\def\sing#1{\tx{Sing}\left(#1\right)}%
\global\long\def\flow#1#2#3{\Phi_{#1}^{#2}#3}%
\global\long\def\ddd#1#2{\frac{\dd{#1}}{\dd{#2}}}%
\newcommandx\vfs[1][usedefault, addprefix=\global, 1={\cc^{2},0}]{\mathfrak{X}\left(#1\right)}%
\newcommandx\per[2][usedefault, addprefix=\global, 1=, 2=]{\mathfrak{T}_{#2}^{#1}}%
\newcommandx\persec[2][usedefault, addprefix=\global, 1=, 2=]{\mathfrak{S}_{#2}^{#1}}%

%nombres

\global\long\def\ii{\tx i}%
\global\long\def\ee{\tx e}%

\global\long\def\re#1{\Re\left(#1\right)}%
\global\long\def\im#1{\Im\left(#1\right)}%

\global\long\def\sgn#1{\mbox{sign}\left(#1\right)}%
\global\long\def\floor#1{\left\lfloor #1\right\rfloor }%
\global\long\def\ceiling#1{\left\lceil #1\right\rceil }%

%series

\global\long\def\germ#1{\cc\left\{  #1\right\}  }%
\global\long\def\frml#1{\cc\left[\left[#1\right]\right] }%
\global\long\def\mero#1{\cc\left(\left\{  #1\right\}  \right)}%

%topologie

\newcommandx\norm[2][usedefault, addprefix=\global, 1=]{\left\Vert #1\right\Vert _{#2}}%
\global\long\def\adh#1{\tx{cl}\left(#1\right)}%
\global\long\def\ccnx#1{\tx{cc}\left(#1\right)}%
\newcommandx\neigh[2][usedefault, addprefix=\global, 1=, 2=0]{\left(\cc^{#1},#2\right)}%

%géométrie

\newcommandx\proj[2][usedefault, addprefix=\global, 1=1, 2=\cc]{\mathbb{P}_{#1}\left(#2\right)}%
\global\long\def\sone{\mathbb{S}^{1}}%

%fonctions

\global\long\def\inj{\hookrightarrow}%
\global\long\def\surj{\twoheadrightarrow}%
\global\long\def\longto{\longrightarrow}%
\global\long\def\longmaps{\longmapsto}%
\global\long\def\cst{\tx{cst}}%
\global\long\def\Beta#1#2{\tx B\left(#1,#2\right)}%
\global\long\def\Gama#1{\tx{\Gamma}\left(#1\right)}%
\global\long\def\oo#1{\tx o\left(#1\right)}%
\global\long\def\OO#1{\tx O\left(#1\right)}%

\newcommandx\diff[1][usedefault, addprefix=\global, 1={\cc^{2},0}]{\tx{Diff}\left(#1\right)}%
\newcommandx\holf[1][usedefault, addprefix=\global, 1={\cc^{2},0}]{\tx{Holo}\left(#1\right)}%
\newcommandx\holb[1][usedefault, addprefix=\global, 1={\cc^{2},0}]{\tx{Holo_{c}}\left(#1\right)}%
\newcommandx\fdiff[2][usedefault, addprefix=\global, 1=2, 2=0]{\widehat{\tx{Diff}}\left(\cc^{#1},#2\right)}%
\newcommandx\homeo[1][usedefault, addprefix=\global, 1={\cc^{2},0}]{\tx{Homeo}\left(#1\right)}%

%specifique

\newcommandx\vspec[3][usedefault, addprefix=\global, 1=~]{\tx{#2}_{#1}^{#3}}%
\newcommandx\fvnf[1][usedefault, addprefix=\global, 1=~]{\vspec[#1]{\widehat{Z}}{}}%
\newcommandx\fonf[1][usedefault, addprefix=\global, 1=~]{\vspec[#1]{\widehat{X}}{}}%
\newcommandx\vnf[1][usedefault, addprefix=\global, 1=~]{\vspec[#1]{\mathcal{Z}}{}}%
\newcommandx\onf[1][usedefault, addprefix=\global, 1=~]{\vspec[#1]{\mathcal{X}}{}}%

\newcommandx\mods[3][usedefault, addprefix=\global, 1=loc, 2=k, 3=]{\tx{Mod}_{\tx{#1}}^{\tx{#3}}\left(#2\right)}%

\newcommandx\sobj[3][usedefault, addprefix=\global, 1=\varepsilon]{#2_{#1}^{#3}}%
\newcommandx\scal[3][usedefault, addprefix=\global, 1=\varepsilon]{\sobj[#1]{\mathcal{#2}}{#3}}%
\newcommandx\stx[3][usedefault, addprefix=\global, 1=\varepsilon]{\sobj[#1]{\tx{#2}}{#3}}%
\newcommandx\sfrak[3][usedefault, addprefix=\global, 1=\varepsilon]{\sobj[#1]{\mathfrak{#2}}{\tx{#3}}}%

\newcommandx\sectobase[3][usedefault, addprefix=\global, 1=, 2=j]{#1_{#3}^{#2}}%
\global\long\def\gat#1#2#3{\sectobase[#1][#3\tx g]{#2}}%
\global\long\def\sad#1#2#3{\sectobase[#1][#3\tx s]{#2}}%
\global\long\def\nod#1#2#3{\sectobase[#1][#3\tx n]{#2}}%

\newcommandx\sect[2][usedefault, addprefix=\global, 1=, 2=j]{\sectobase[\mathcal{V}][#2]{#1}}%
\newcommandx\conv[1][usedefault, addprefix=\global, 1=k]{\tx{Convergent}_{#1}}%
\newcommandx\nfsec[2][usedefault, addprefix=\global, 1=k, 2=P^{\tau}y]{\tx{Section}_{#1}\left\{  #2\right\}  }%
\global\long\def\unisp{\germ{\varepsilon}^{\times}}%
\newcommandx\hirze[1][usedefault, addprefix=\global, 1=r]{\tx{Sphere}\left(#1\right)}%
\newcommandx\bber[1][usedefault, addprefix=\global, 1=d]{\tx{Ber}\left(#1\right)}%
\newcommandx\ber[3][usedefault, addprefix=\global, 1=d]{\bber[#1,\begin{array}{c}
 #2\\
 #3 
\end{array}]}%

\title[Normal forms for convergent saddle-node unfoldings]{Analytic normal forms and inverse problems for unfoldings of 2-dimensional
saddle-nodes with analytic center manifold}
\author{C.~Rousseau$^{\star}$ \& L.~Teyssier$^{\dagger}$\texttwosuperior{}}
\date{January 2018}
\thanks{The first author was supported by NSERC in Canada. The second author
was supported by the French National Research Agency grant ANR-13-JS01-0002-01.
He would also like to express his gratitude for the welcome and support
he received during his stay at the CNRS unit UMI-3457 located in the
CRM~/~Université de Montréal.}
\subjclass[2000]{(2010 MSC) 34A26, 34C20, 34C23, 34M35, 37G10, 37L10. }
\keywords{Normal forms, holomorphic vector fields, saddle-node bifurcation,
unfolding of singularities, modulus space, analytic conjugacy.}
\email{\texttt{rousseac@dms.umontreal.ca} \textendash \&\textendash{} \texttt{teyssier@math.unistra.fr}}
\begin{abstract}
We give normal forms for generic $k$-dimensional parametric families
$\left(Z_{\varepsilon}\right)_{\varepsilon}$ of germs of holomorphic
vector fields near $0\in\cc^{2}$ unfolding a saddle-node singularity
$Z_{0}$, under the condition that there exists a family of invariant
analytic curves unfolding the weak separatrix of $Z_{0}$. These normal
forms provide a moduli space for these parametric families. In our
former 2008 paper, a modulus of a family was given as the unfolding
of the Martinet-Ramis modulus, but the realization part was missing.
We solve the realization problem in that partial case and show the
equivalence between the two presentations of the moduli space. Finally,
we completely characterize the families which have a modulus depending
analytically on the parameter. We provide an application of the result
in the field of non-linear, parameterized differential Galois theory.
\end{abstract}

\maketitle
{\scriptsize{}\hfill{}}%
\begin{tabular}{c}
{\scriptsize{}$\,^{\star}$Département de Mathématiques et de Statistique}\tabularnewline
{\scriptsize{}Université de Montréal, Canada}\tabularnewline
\texttt{\scriptsize{}\url{http://www.dms.umontreal.ca/~rousseac}}\tabularnewline
\end{tabular}{\scriptsize{}\hfill{}}%
\begin{tabular}{c}
{\scriptsize{}$\,^{\dagger}$Unité Mixte Internationale 3457}\tabularnewline
{\scriptsize{}CNRS \& CRM, Université de Montréal, Canada}\tabularnewline
\texttt{\scriptsize{}\url{http://math.u-strasbg.fr/~teyssier}}\tabularnewline
\end{tabular}{\scriptsize{}\hfill{}}{\scriptsize\par}

\bigskip{}

\section{Introduction}

Heuristically, moduli spaces of holomorphic dynamical systems not
only encode but also describe qualitatively the dynamics itself, and
to some extent allow a better understanding of remarkable dynamical
phenomena. This paper is part of a large program aimed at studying
the conjugacy classes of dynamical systems in the neighborhood of
stationary points (up to local changes of analytic coordinates). Stationary
points and their invariant manifolds organize the global dynamics
while degenerate stationary points organize the bifurcation diagrams
in families of dynamical systems. Stationary points of discrete dynamical
systems correspond to fixed-points of the iterated map(s), while for
continuous dynamical systems they correspond to singularities in the
underlying differential equation(s).

A natural tool for studying conjugacy classes is the use of normal
forms. For hyperbolic stationary points (generic situation), the system
is locally conjugate to its linear part so that the quotient space
of (local) hyperbolic systems is given by the space of linear dynamical
systems. However, for most non-hyperbolic stationary points the normalizing
changes of coordinates (sending \emph{formally} the system to a normal
form) is given by a divergent power series. Divergence is very instructive:
it tells us that the dynamics of the original system and that of the
normal form are qualitatively different. In that respect, a subclass
of singularities that has been thoroughly studied in the beginning
of the 80's is that of 1-resonant singularities: these include parabolic
fixed-points of germs of 1-dimensional diffeomorphisms, resonant-saddle
singularities and saddle-node singularities of 2-dimensional vector
fields, as well as non-resonant irregular singular points of linear
differential systems. These various resonant dynamical systems share
a lot of common properties, among which is the finite-determinacy
of their formal normal forms (\emph{e.g.} polynomial expressions in
the case of vector fields). Another property they share is that they
can be understood as the \emph{coalescence} of special ``geometric
objects'', either of stationary points or of a singular point with
a limit cycle in the case of the Hopf bifurcation at a weak focus.

\subsection{Scope of the paper}

The present work is the follow-up of~\cite{RouTey} in which we described
a set of functional moduli for unfoldings of codimension $k$ saddle-node
vector fields $Z=\left(Z_{\varepsilon}\right)_{\varepsilon}$ depending
on a finite-dimensional parameter $\varepsilon\in\neigh[k]$. Here
we focus mainly on the inverse problem and on the question of finding
(almost unique) normal forms, as we explain below.

The most basic example of such an unfolding is given by the codimension
$1$ unfolding (expressed in the canonical basis of $\cc^{2}$)
\begin{align}
Z_{\varepsilon}\left(x,y\right) & :=\left[\begin{array}{c}
x^{2}+\varepsilon\\
y
\end{array}\right]~~~~,~\varepsilon\in\cc.\label{eq:sn_model_vf}
\end{align}
Real slices of the phase-portraits are shown in Figure~\ref{fig:example}.
The merging (bifurcation) occurs at $\varepsilon=0$: for $\varepsilon\neq0$
the system has two stationary points located at $\left(\pm\sqrt{-\varepsilon},0\right)$
which collide as $\varepsilon$ reaches $0$.

\begin{figure}[H]
\hfill{}\subfloat[\foreignlanguage{english}{$\varepsilon<0$}]{\includegraphics[width=3cm]{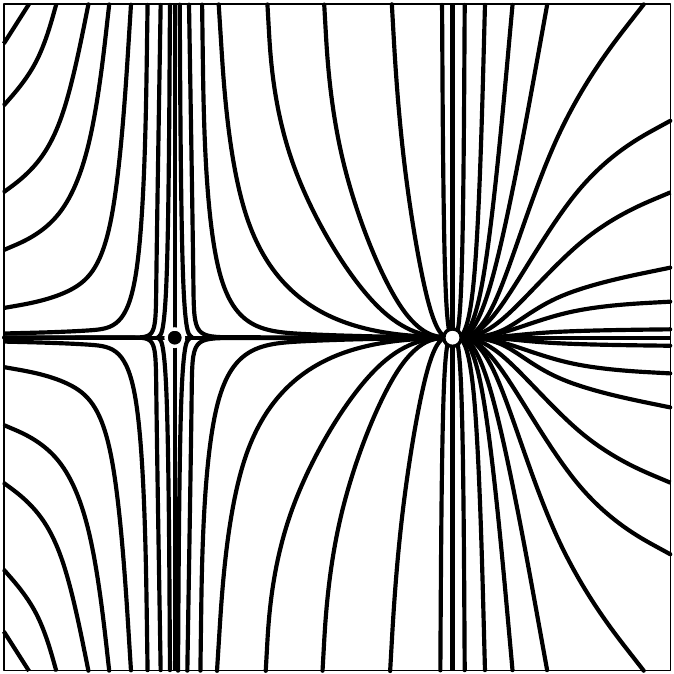}

}\hfill{}\subfloat[\foreignlanguage{english}{$\varepsilon=0$}]{\includegraphics[width=3cm]{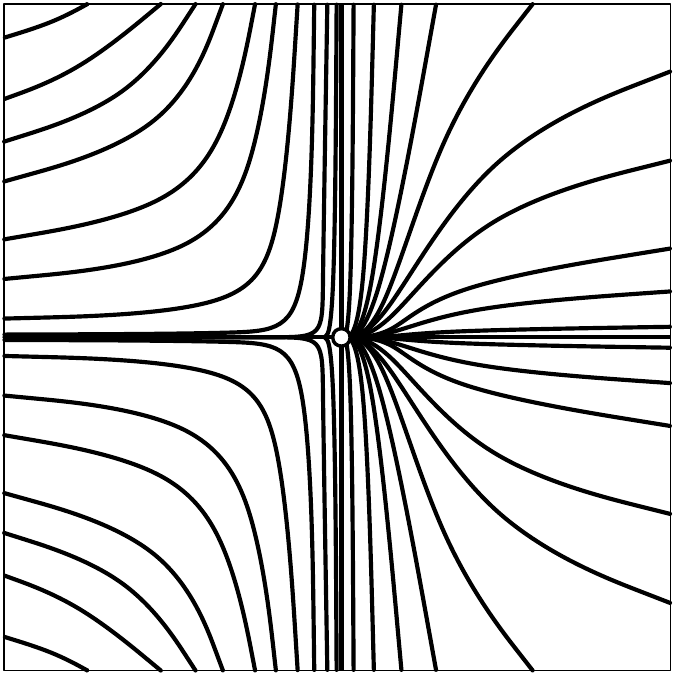}

}\hfill{}\subfloat[\foreignlanguage{english}{$\varepsilon>0$}]{\includegraphics[width=3cm]{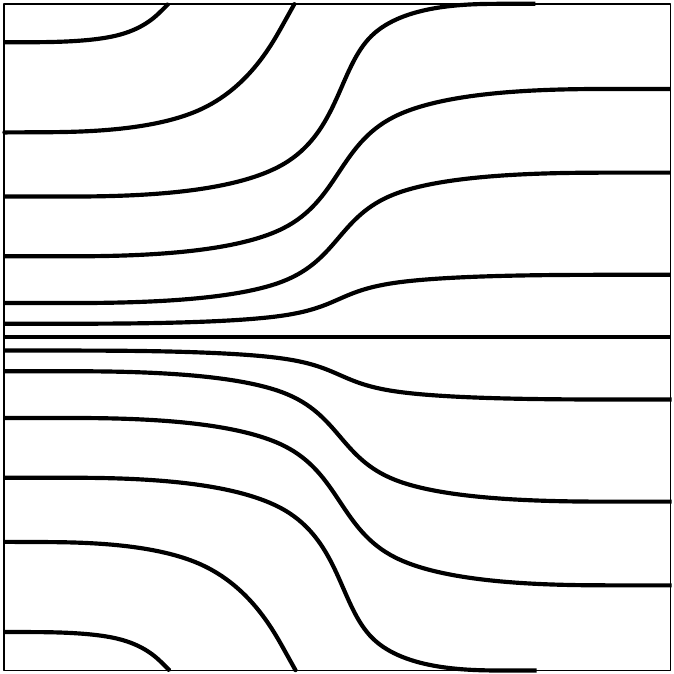}

}\hfill{}

\caption{\label{fig:example}Typical members of the simplest saddle-node bifurcation.}
\end{figure}

\subsection{Modulus of classification}

Each merging stationary point organizes the dynamics in its own neighborhood
in a rigid way. The local models of these rigid dynamics seldom agree
on overlapping areas and in general cannot be glued together. If this
incompatibility persists as the confluence happens, then we have divergence
of the normalizing series at the limit. In the case of $1$- or $2$-dimensional
resonant systems the normalizing series is $k$-summable. The divergence
is then quantified by the \emph{Stokes phenomenon}: there exists a
formal normalizing transformation, and a covering of a punctured neighborhood
of the singularity by $2k$ sectors over which there exist unique
sectorial normalizing transformations that are Gevrey-asymptotic to
the formal normalization. Comparing the normalizing transformations
on intersections of consecutive sectors provides a modulus of analytic
classification. This modulus takes the form of Stokes matrices for
irregular singularities of linear differential systems and functional
moduli for singularities of nonlinear dynamical systems (see for instance~\cite{IlYako}).

\bigskip{}

The classification of resonant systems may seem rather mysterious.
But if we remember that we are studying the merging of ``simple''
singularities, then it becomes natural to unfold the situation and
study the ``multiple'' singularity as a limiting case. Indeed, analyzing
unfoldings sheds a new light on the ``complicated'' dynamics of
the limiting systems. The idea was suggested by several mathematicians,
including V.~\noun{Arnold}, A.~\noun{Bolibruch} and \noun{J.~Martinet~\cite{MarBif}.}
It was put in practice for unfoldings of saddle-node singularities
by A.~\noun{Glutsyuk}~\cite{GluBif} on regions in parameter space
over which the confluent singularities are all hyperbolic. The system
can be linearized in the neighborhood of each singularity, and the
mismatch in the normalizing changes of coordinates tends to the components
of the saddle-node's Martinet-Ramis modulus~\cite{MaRa-SN} when
the singularities merge. But the tools were still missing for a full
classification of unfoldings of multiple singularities, in particular
on a full neighborhood in parameter space of the bifurcation value. 

The thesis of P.~\noun{Lavaurs~\cite{Lavau}} on parabolic points
of diffeomorphisms opened the way for such classifications, for he
studied the complementary regions in parameter space. The first classification
of generic unfoldings of codimension 1 fixed-point of diffeomorphisms
regarded the parabolic point~\cite{MaRouRou}, and then the resonant-saddle
and saddle-node singularities of differential equations~\cite{RouSN,RouResVF}.
The first classification of generic unfoldings of codimension $k$
saddle-nodes was done by the authors~\cite{RouTey} using the visionary
ideas of A.~\noun{Douady} and P.~\noun{Sentenac~}\cite{DES,BraDia}
that R.~\noun{Oudkerk} had used on some regions in parameter space
in his thesis~\cite{Oud}. Then followed classifications of generic
unfoldings of codimension $k$ parabolic points~\cite{RouParak}
and of non-resonant irregular singular points of Poincaré rank $k$
differential systems~\cite{HuLaRou}.

\bigskip{}

In the spirit of this general context we obtained in~\cite{RouTey}
a (family of) functional data 
\begin{eqnarray*}
\mathfrak{m}_{\varepsilon} & = & \left(\sad f{\varepsilon}{},\sad{\psi}{\varepsilon}{},\nod{\psi}{\varepsilon}{}\right)_{\varepsilon}.
\end{eqnarray*}
For $\varepsilon=0$ this data coincides with the saddle-node's modulus~\cite{MaRa-SN,MeshVoro,Tey-SN}.
Although the original work of J.~\noun{Martinet} and J.-P.~\noun{Ramis}
already covered parametric cases, it was then assumed that the (formal)
type of the singularity remained constant. On the contrary we were
interested in bifurcations, which are deformations where the additional
parameters change the type (or number) of singularities. Our main
contribution was to reconcile Glutsyuk's and Lavaurs's viewpoint and
devise a uniform framework valid for a complete neighborhood of the
bifurcation value of the parameter. That being said, the very nature
of our geometric construction prevented the modulus to be continuous
on the whole parameter space. This space needs to be split into a
finite number of \emph{cells} whose closures cover a neighborhood
of the bifurcation value, on which the modulus is analytic on $\varepsilon$
with continuous extension to the closure. 

\subsection{The inverse (or realization) problem}

At the time of the first works on the question, identifying the moduli
space was still out of reach. Performing this identification is called
the inverse problem. It was first solved for codimension $1$ parabolic
fixed-points and resonant-saddle singularities~\cite{RouColRea,RouRea},
as well as for the irregular singularities of linear differential
systems with Poincaré rank $1$~\cite{RouLamLin2}. For codimension
$k$ the realization problem was first solved for unfoldings of non-resonant
irregular singular points of Poincaré rank $k$~\cite{HurtuRou}.
But the realization question is still open for unfoldings of codimension
$k$ parabolic points. 

\bigskip{}

Let us formulate the inverse problem in the case at hands.
\begin{namedthm}[Inverse problem]
Among all elements of the vector space $\mathcal{M}$ to which $\mathfrak{m}=\left(\mathfrak{m}_{\varepsilon}\right)_{\varepsilon}$
belongs, to identify those coming as moduli of a saddle-node bifurcation.
\end{namedthm}
The present paper answers completely this challenge in the case of
bifurcations with a persistent analytic center manifold. The common
feature to that case and the one studied in~\cite{HurtuRou} is that
solving the inverse problem ultimately provides unique normal forms
(privileged representative in each analytic class). 

Having persistent analytic center manifold can be read in the modulus
as the condition $\nod{\psi}{}{}=\id$. Although any element of the
specialization of $\mathcal{M}$ at $\varepsilon=0$ can be realized
as the modulus of a saddle-node vector field~\cite{MaRa-SN,Tey-SN},
this property does not hold anymore for bifurcations: the typical
element of\emph{ $\mathcal{M}\cap\left\{ \nod{\psi}{}{}=\id\right\} $
}can never be realized as a modulus of saddle-node bifurcation. Let
us explain how this is so. It is rather easy to get convinced that
there is no obstruction to realize any given deformation $\left(\mathfrak{m}_{\varepsilon}\right)_{\varepsilon\in\adh{\mathcal{E}}}$
of a saddle-node's modulus $\mathfrak{m}_{0}$ over any given cell
$\mathcal{E}$ in parameter space. By this we mean that for each fixed
$\varepsilon\in\adh{\mathcal{E}}$ it is possible to find a holomorphic
vector field $Z_{\varepsilon}$ on a neighborhood $\mathcal{U}$ of
$\left(0,0\right)$ such that comparisons between its sectorial normalizing
maps coincide with $\mathfrak{m}_{\varepsilon}$. Furthermore the
dependence $\varepsilon\mapsto Z_{\varepsilon}$ has the expected
regularity on the cell's closure, and the neighborhood $\mathcal{U}$
is independent on $\varepsilon$. The sole obstacle lies therefore
in gluing these cellular realizations together over cellular intersections
in order to obtain a genuine analytic parametric family $Z$ whose
modulus agrees with $\mathfrak{m}$. Favorable situations can be characterized
by a strong criterion imposed on $\mathfrak{m}$, called \emph{compatibility
condition}. A necessary and sufficient condition is that two realizations
over different cells in parameter space be conjugate over the intersection
of the two cells, thus allowing correction to a uniform family. One
difficulty is to express this condition on the abstractly encoded
dynamics $\mathfrak{m}$ (that is, before performing the cellular
realization). The compatibility condition takes the simple form that
the abstract holonomy pseudogroups generated by $\mathfrak{m}$ be
conjugate, a condition which can easily be expressed in terms of the
modulus. The general case of a bifurcation without analytic center
manifold remains open, and we hope to address it in the near future.

\subsection{Summary of the paper's content}

Here we review the content of the present work. For precise statements
of our main results, as for more detailed proof techniques, we refer
to Section~\ref{sec:Statements}. Recall that one can associate two
dynamical data to a vector field $X=A\pp x+B\pp y$: 
\begin{itemize}
\item the trajectories of $X$ parametrized by the complex time in the associated
flow-system
\begin{align*}
\begin{cases}
\dot{x} & =A\left(x,y\right)\\
\dot{y} & =B\left(x,y\right)
\end{cases} & ,
\end{align*}
\item the underlying foliation $\fol X$ whose leaves coincide with orbits
of $X$, obtained by forgetting about a particular parametrization
of the trajectories. The foliation really is attached to the underlying
non-autonomous differential equation 
\begin{align*}
A\left(x,y\right)y' & =B\left(x,y\right)
\end{align*}
rather than to the vector field itself. 
\end{itemize}
The action of (analytic or formal) changes of variables $\Psi$ on
vector fields $X$ by conjugacy is obtained as the pullback
\begin{eqnarray*}
\Psi^{*}X & := & \tx D\Psi^{-1}\left(X\circ\Psi\right).
\end{eqnarray*}
The vector fields $X$ and $\Psi^{*}X$ are then (analytically or
formally) \textbf{conjugate}. When two foliations $\fol X$ and $\fol{\widetilde{X}}$
are conjugate (when $X$ is conjugate to a scaling of $\tilde{X}$
by a non-vanishing function) it is common to say that $X$ and $\widetilde{X}$
are \textbf{orbitally equivalent}. While for unfoldings we also allow
parameter changes, we restrict our study to parameter~/~coordinates
changes of the form
\begin{align*}
\Psi~:~\left(\varepsilon,x,y\right) & \longmaps\left(\phi\left(\varepsilon\right),~\Psi_{\varepsilon}\left(x,y\right)\right).
\end{align*}

\bigskip{}

In this paper we focus on families $Z=\left(Z_{\varepsilon}\right)_{\varepsilon\in\neigh[k]}$
unfolding a codimension $k$ saddle-node singularity for $\varepsilon=0$
and the study of their conjugacy class (\emph{resp}. orbital equivalence
class) under local analytic changes of variables and parameter (\emph{resp}.
and scaling by non-vanishing functions). Such families can always
be brought by a formal change of variables and parameter into the
formal normal form\footnote{As is customary we write vector fields in the form of derivations,
by identifying the canonical basis of $\cc^{2}$ with $\left(\pp x,\pp y\right)$. } 
\[
u_{\varepsilon}\left(x\right)\left(P_{\varepsilon}(x)\frac{\partial}{\partial x}+y\left(1+\mu_{\varepsilon}x^{k}\right)\frac{\partial}{\partial y}\right),
\]
where 
\begin{align*}
P_{\varepsilon}(x)=x^{k+1}+\varepsilon_{k-1}x^{k-1}+\cdots+\varepsilon_{1}x+\varepsilon_{0} & ~~~,~k\in\nn\\
u_{\varepsilon}\left(x\right)=u_{0,\varepsilon}+u_{1,\varepsilon}x+\cdots+u_{k,\varepsilon}x^{k} & ~~~,~u_{0,\varepsilon}\neq0
\end{align*}
and $\varepsilon\mapsto\left(\mu_{\varepsilon},u_{0,\varepsilon},\ldots,u_{k,\varepsilon}\right)$
is holomorphic near $0$. A proof of this widely accepted result seems
to be missing in the literature, hence we provide one. 

The first step in our previous work~\cite{RouTey} consisted in preparing
the unfolding $\left(Z_{\varepsilon}\right)_{\varepsilon}$ by bringing
it in a form where the polynomial $P_{\varepsilon}$ determines the
$\pp x$-component. Formal and analytic equivalences between such
forms must consequently preserve the coefficients of $P_{\varepsilon}$,
which then become privileged \emph{canonical parameters}. This process
eliminates the difficulty of dealing with changes of parameters and
allows to work for fixed values of $\varepsilon$. Then we established
a complete classification. The modulus was composed of two parts:
the \emph{formal part} given by the formal normal form above, and
the \emph{analytic part} given by an unfolding of the saddle-node's
functional modulus. The formal~/~analytic part of the modulus itself
consists in the Martinet-Ramis orbital part (characterizing the vector
field up to orbital equivalence) and an additional part classifying
the time. For example $\mu_{\varepsilon}$ is the formal orbital class
while $u_{\varepsilon}$ is the formal temporal class.

\bigskip{}

We completely solve the realization problem for orbital equivalence
(\emph{i.e}. for foliations) when each $Z_{\varepsilon}$ admit a
single analytic invariant manifold passing through every singularity.
But we do more: we provide almost unique ``normal forms'' (the only
degree of freedom being linear transformations in $y$), which are
polynomial in $x$ when $\mu_{0}\notin\mathbb{R}_{\leq0}$. In that
generic situation, an unfolding is orbitally equivalent to an unfolding
over $\proj\times\neigh$ of the form 
\[
P_{\varepsilon}(x)\frac{\partial}{\partial x}+y\left(1+\mu_{\varepsilon}x^{k}+\sum_{j=1}^{k}x^{j}R_{j,\varepsilon}(y)\right)\frac{\partial}{\partial y},
\]
where the $R_{j}$ are analytic in both the geometric variable $y$
and the parameter $\varepsilon$. In this generic case the construction
is a direct generalization of that of F.~\noun{Loray}'s~\cite[Theorems 2 and 4]{Loray}
for $\varepsilon=0$ and $k=1$, and only involves tools borrowed
from complex geometry. In the non-generic case (when $\mu_{0}\leq0$)
we also provide almost unique ``normal forms'', which are in some
sense global in $x$: in this case the foliation is defined on a fiber
bundle of negative degree $-\tau(k+1)<\mu_{0}$ over $\proj$ and
is induced by vector fields of the form 
\begin{equation}
X_{\varepsilon}\left(x,y\right):=P_{\varepsilon}(x)\frac{\partial}{\partial x}+y\left(1+\mu_{\varepsilon}x^{k}+\sum_{j=1}^{k}x^{j}R_{j,\varepsilon}\left(P_{\varepsilon}^{\tau}(x)y\right)\right)\frac{\partial}{\partial y}.\label{eq:intro-orbital_normal_form}
\end{equation}
This result offers a new presentation of the moduli space which has
the advantage over that of~\cite{RouTey} to be made up of functions
analytic in the parameter (it does not require the splitting of the
parameter space into cells). 

As far as normal forms are concerned, we provide some also for the
family of vector fields. This requires normalizing the ``temporal
part''. The method used is an unfolding of the construction of R\noun{.~Schäfke}
and L.~\noun{Teyssier}~\cite{SchaTey} performed for $\varepsilon=0$.
As a by-product we provide an explicit section of the cokernel of
the derivation $X_{\varepsilon}$ (\emph{i.e.} a linear complement
of the image of $X_{\varepsilon}$ acting as a Lie derivative on the
space of analytic germs).

\bigskip{}

An important observation is that the normalization we just described
does not involve classification moduli in any way (nor does it rely
on the analytical classification for that matter), at least in the
generic case $\mu_{0}\notin\ww R_{\leq0}$. Therefore it does not
answer the inverse (or realization) problem which is posed in terms
of classification moduli. This leads us to discuss the compatibility
condition. 

As we mentioned earlier we can realize any unfolding $\mathfrak{m}=\left(\mathfrak{m}_{\varepsilon}\right)_{\varepsilon\in\adh{\mathcal{E}}}$
of a saddle-node's modulus $\mathfrak{m}_{0}$ over a cell $\mathcal{E}$
in parameter space, but we have such control of the construction that
we can guarantee this realization is an unfolding in normal form~\eqref{eq:intro-orbital_normal_form},
save for the fact that the functions $\varepsilon\mapsto R_{j,\varepsilon}$
are merely analytic on $\mathcal{E}$ with continuous extension to
the closure. It is possible to express the holonomy group of $X_{\varepsilon}$
with respect to the analytic center manifold (the geometrical dynamics)
as a representation of an abstract group of words formed with elements
of the modulus $\mathfrak{m}$ (acting in orbits space). The compatibility
condition simply states that the holonomy pseudogroups over the intersection
of two neighboring cells are conjugate by a tangent-to-identity mapping.
If the condition is satisfied then two cellular realizations are conjugate
for values of the parameter in the cells' intersection. Usually when
such a situation occurs, we need to apply a conjugacy to the vector
fields so that they match in the new coordinates. Here no need for
it. Indeed, since the realizations over the different cells are in
normal form, they necessarily are conjugate by a linear map. The additional
hypothesis in the compatibility condition that the conjugating map
is tangent-to-identity allows to conclude that the cellular unfoldings
actually agree and therefore define a genuine unfolding analytic in
$\varepsilon\in\neigh[k]$.

\bigskip{}

Our analysis presents in an effective way the relationship between
Rousseau-Teyssier classification moduli and the coefficients of the
normal forms, so that numerical, and in some cases symbolic, computations
can be performed. Also, we have refined our understanding of the modulus
compared to the presentation in \cite{RouTey}. The number of cells
is now the optimal number $C_{k}=\frac{1}{k+1}\binom{2k}{k}$ (the
$k^{\tx{th}}$ Catalan's number) given by the Douady-Sentenac classification~\cite{DES,DiaCata}.
Moreover we have reduced the degrees of freedom: instead of having
the modulus given up to conjugacy by linear functions depending both
on $\varepsilon$ and the cell, now the modulus is given up to conjugacy
by linear functions depending only on $\varepsilon$ in an analytic
way. This new equivalence relation in the presentation of the modulus
was essential in getting the realizations over the different cells
to match when the compatibility condition is satisfied.

Last but not least we were able to completely characterize the moduli
that depend analytically on the parameter. These only occur when $k=1$
and their normal forms are given by Bernoulli unfoldings 
\begin{equation}
P_{\varepsilon}\left(x\right)\frac{\partial}{\partial x}+y\left(1+\mu_{\varepsilon}x^{k}+xr_{\varepsilon}\left(x\right)\left(P_{\varepsilon}\left(x\right)^{\tau}y\right)^{d}\right)\frac{\partial}{\partial y}\label{eq:bernoulli_vf}
\end{equation}
with $d\in\nn$ and $d\mu\in\zz$ (in particular $\mu$ must be a
rational constant, which is seldom the case). This proves that the
compatibility condition is not trivially satisfied by every element
of $\mathcal{M}\cap\left\{ \nod{\psi}{}{}=\id\right\} $. On the contrary,
the typical situation is that of moduli which are analytic and bounded
only on single cells. This reminds us of the setting of Borel-summable
divergent power series, in particular in the case $k=1$ where the
cells are actual sectors and it can be proved that the moduli are
sectorial sums of $\frac{1}{2}$-summable power series (as in~\cite{RouColRea}).
When $k>1$ the lack of a theory of summation in more than one variable
prevents us from reaching similar conclusions, although the moduli
are natural candidates for such sums and a general summation theory
should probably contain the case we studied here. We reserve such
considerations for future works, perhaps using the theory of polynomial
summability recently introduced by J.~\noun{Mozo} and R.~\noun{Schäfke}
(this work~\cite{MozoScha} is still in preparation, we refer to
the prior monomial case performed in~\cite{CanaMozScha}).

\subsection{Applications}

Our main results can be used to solve problems outside the scope of
finding normal forms or addressing the \emph{local} inverse problem.
Let us mention two applications, the second of which we develop in
Section~\ref{subsec:Applications}.

\bigskip{}

The first (and most straightforward) one concerns the global inverse
problem, also known as non-linear Riemann-Hilbert problem, posed by
Y.~\noun{Ilyashenko} and S.~\noun{Yakovenko}~in~\cite[Chapter IV]{IlYako}.
Being given a (germ of a) complex surface $\mathcal{M}$ seen as the
total space of a fiber bundle over a divisor $\proj\subset\mathcal{M}$,
the problem is to characterize the holonomy representations of complex
foliations on $\mathcal{M}$ tangent to (and regular outside) the
divisor and transverse to the fibers, except over $k+2$ singularities
(which are all assumed non-degenerate) where the fibers are invariant
by the foliation. Using a sibling of Loray's technique, they solve
it for fiber bundles of degree $0$ and $-1$, although they only
provide details for the former case. Our results open the way to generalizations
in several directions:
\begin{itemize}
\item allowing saddle-node(s) with central manifold along the divisor and
adding to the holonomy representation the components of the modulus
of the saddle-nodes, similarly to the generalized linear Riemann-Hilbert
problem when irregular singularities are allowed; 
\item allowing foliations depending analytically on the parameter; 
\item considering realizations on fiber bundles of negative degree: we obtain
here realizations on bundles with degree given by an arbitrary non-positive
multiple of $k+1$ (see Conjecture~\ref{conj:alternate_normal_forms}
for a brief discussion of possible improvement to any non-positive
degree);
\item allowing resonant nodes: in our paper all nodes were linearizable
because their Camacho-Sad index was greater than $1$. But nodes with
smaller Camacho-Sad index pose no additional problem. 
\end{itemize}
We propose to address this matter in the near future.

\bigskip{}

The other application regards differential Galois theory: heuristically,
classification invariants carry Galoisian information (pertaining
to the integrability in Liouvillian closed-form). For instance, in
the limiting case of a saddle-node singularity it is well-known that
Martinet-Ramis moduli play the same role for non-linear equations
as Stokes matrices do for linear systems near an irregular singularity.
A Galoisian formulation of this fact in terms of Malgrange groupoid~\cite{MalgShort,Malg}
can be found in the work of G.~\noun{Casale}~\cite{Casa}. When
the differential equation depends on a parameter $\varepsilon$, the
recent thesis of D.~\noun{Davy} describes a form of ``semi-continuity''
for specializations of its parametrized Malgrange groupoid $\mathfrak{M}$.
He proves that the size of the groupoid $\mathfrak{M}_{\varepsilon}$
is constant if $\varepsilon$ is generic, more precisely if the parameter
does not belong to a (maybe empty) countable union $\Omega$ of hypersurfaces,
while for $\varepsilon\in\Omega$ the groupoid $\mathfrak{M}_{\varepsilon}$
can only get smaller. The present study illustrates and refines this
phenomenon. 

Consider the extreme case $P_{\varepsilon}\left(x\right)\pp x+y\left(1+\mu_{\varepsilon}x^{k}+\varepsilon R_{\varepsilon}\left(x,y\right)\right)\pp y$
for $R$ arbitrary: the vector field $X_{0}$ is surely ``not less
integrable'' (it is the formal normal form) than for $\varepsilon\neq0$.
This is actually the only possible kind of degeneracy near the saddle-node
bifurcation, for we will establish that $\Omega\cap\neigh[k]\subset\left\{ 0\right\} $.
We obtain the latter property by unfolding a result by M.~\noun{Berthier}
and F.~\noun{Touzet}~\cite{BerTouze}, characterizing vector fields
admitting local non-trivial Liouvillian first integrals near an elementary
singularity. We deduce that normal forms of integrable unfoldings
are necessarily a Bernoulli unfolding~\eqref{eq:bernoulli_vf}. Both
proofs are very different in nature, and we obtain a particularly
short one by framing the problem for normal forms, revealing the usefulness
of their simple expression and of the explicit section of their cokernel.

\section{\label{sec:Statements}Statement of the main results}

In all what follows $\varepsilon$ is the parameter, belonging to
some $\neigh[k]$ for $k\in\nn$, and we study (holomorphic germs
of a) parametric family of (germs at $0\in\cc^{2}$ of) vector fields
$Z=\left(Z_{\varepsilon}\right)_{\varepsilon\in\neigh[k]}$ for which
a saddle-node bifurcation occurs at $\varepsilon=0$. That is to say,
when $\varepsilon=0$ the vector field $Z_{0}$ is of \textbf{saddle-node}
type near the origin of $\cc^{2}$:
\begin{itemize}
\item $0$ is an isolated singularity of $Z_{0}$,
\item the differential at $0$ of the vector field has exactly one non-zero
eigenvalue (the singularity is elementary degenerate). 
\end{itemize}
The family $Z=\left(Z_{\varepsilon}\right)_{\varepsilon}$ is called
a holomorphic germ of an \textbf{unfolding} of $Z_{0}$. We  study
in details only\textbf{ ``}generic'' unfoldings, those which possess
the ``right number'' of parameters to encode the bifurcation structure.
Roughly speaking we require that for an open and dense set of parameters
the vector field $Z_{\varepsilon}$ have $k+1$ distinct non-degenerate
singular points. The latter merge into a saddle-node singularity of
multiplicity $k+1$ (codimension $k$) as $\varepsilon\to0$. Let
us make these statements precise.
\begin{defn}
\label{def:generic_unfolding}An unfolding $Z$ of a codimension $k\in\nn$
saddle-node $Z_{0}$ is \textbf{generic} if there exists a biholomorphic
change of coordinates and parameter such that, in the new coordinates
$(x,y)$ and new parameter $\varepsilon$, the singular points of
each $Z_{\varepsilon}$ are given by $P_{\varepsilon}(x)=y=0$, where
\begin{align*}
P_{\varepsilon}(x) & :=x^{k+1}+\varepsilon_{k-1}x^{k-1}+\cdots+\varepsilon_{1}x+\varepsilon_{0}.
\end{align*}
\end{defn}

\begin{rem}
Generic families are essentially \emph{universal}. In particular,
the bifurcation diagram of singular points is the elementary catastrophe
of codimension $k$ (in the complex domain).
\end{rem}

The analytic unstable manifold of $Z_{0}$, tangent at $0$ to the
eigenspace associated to the non-zero eigenvalue of its differential,
is called the \textbf{strong separatrix}. The other eigenspace corresponds
to a ``formal separatrix'' $\left\{ y=\widehat{s}_{0}\left(x\right)\right\} $
called the \textbf{weak separatrix} (generically divergent~\cite{Remond},
always summable in the sense of Borel~\cite{HuKiMa}). We say that
a saddle-node is convergent or divergent according to the nature of
its weak separatrix. 
\begin{defn}
We say that the generic unfolding $Z$ is \textbf{purely convergent}
when there exists a holomorphic function 
\begin{align*}
s~:~\neigh[k][0]\times\neigh & \longto\cc\\
\left(\varepsilon,x\right) & \longmapsto s_{\varepsilon}\left(x\right)
\end{align*}
 such that:

\begin{itemize}
\item each graph $\mathcal{S}_{\varepsilon}$ of $s_{\varepsilon}$ is tangent
to $Z_{\varepsilon}$ and contains $\sing{Z_{\varepsilon}}$ (the
singular set of $Z_{\varepsilon}$, consisting in all zeros of $Z_{\varepsilon}$),
\item $\mathcal{S}_{0}$ is the weak separatrix of $Z_{0}$ (in particular
the latter is convergent).
\end{itemize}
We call $\conv$ the set of all such unfoldings. 

\end{defn}

\begin{rem}
~

\begin{enumerate}
\item By applying beforehand the change of variables $\left(\varepsilon,x,y\right)\mapsto\left(\varepsilon,x,y+s_{\varepsilon}\left(x\right)\right)$
to the unfolding we can always assume that $\left\{ y=0\right\} $
is invariant by $Z_{\varepsilon}$ for all $\varepsilon\in\neigh[k]$.
\item There exist unfoldings $Z$ of a convergent saddle-node $Z_{0}$ such
that, for all $\varepsilon$ close enough to $0$, no analytic invariant
curve $\mathcal{S}_{\varepsilon}$ exist. We use the term ``purely
convergent'' to insist that in the present case \emph{every} vector
field $Z_{\varepsilon}$ for $\varepsilon\in\neigh[k]$ must admit
an analytic invariant curve. 
\end{enumerate}
\end{rem}

\subsection{Normalization of purely convergent unfoldings}

For $\mathbf{z}$ a finite-dimensional complex multivariable we write
$\germ{\mathbf{z}}$ the algebra of convergent power series in $\mathbf{z}$,
naturally identified with the space of germs of a holomorphic function
at $\mathbf{0}$. We extend this notation in the obvious manner so
that $\germ{\varepsilon,x}$ is the space of convergent power series
in the $k+1$ complex variables $\varepsilon_{0},\ldots,\varepsilon_{k-1}$
and $x$.

\subsubsection{Formal classification}

We first give an unfolded version of the well-known Bruno-Dulac-Poincaré
normal forms~\cite{Bruno,Dulac,Dulac2}. Here we do not assume that
$Z$ be purely convergent.
\begin{namedthm}[Formal Normalization Theorem]
Let $k\in\nn$ be given. For $\varepsilon=\left(\varepsilon_{0},\ldots,\varepsilon_{k-1}\right)\in\cc^{k}$
define the polynomial
\begin{align*}
P_{\varepsilon}\left(x\right) & :=x^{k+1}+\sum_{j=0}^{k-1}\varepsilon_{j}x^{j}.
\end{align*}
Take a generic unfolding $Z$ of a saddle-node of codimension $k$.
There exists $\left(\mu,u\right)\in\germ{\varepsilon}\times\germ{\varepsilon,x}$
with $\left(\varepsilon,x\right)\mapsto u_{\varepsilon}\left(x\right)$
polynomial in $x$ of degree at most $k$ and satisfying $u_{0}\left(0\right)\neq0$,
such that $Z$ is formally conjugate to the \textbf{formal normal
form }
\begin{align}
\fvnf & :=u~\fonf~~,\label{eq:formal_normal_form}
\end{align}
where 
\begin{align}
\fonf[\varepsilon]\left(x,y\right) & :=P_{\varepsilon}\left(x\right)\pp x+y\left(1+\mu_{\varepsilon}x^{k}\right)\pp y\label{eq:formal_orbital_normal_form}
\end{align}
defines the\textbf{ formal orbital normal form}. Notice that these
vector fields are polynomial in $\left(x,y\right)$ and holomorphic
in $\varepsilon\in\neigh[k]$. 
\end{namedthm}
In general the parameter of the normal form $\fvnf$ differs from
the original parameter of $Z$. However the formal change of parameter
$\varepsilon\mapsto\phi\left(\varepsilon\right)$ happens to be actually
analytic (as proved in~\cite[Theorem 3.5]{RouSN} and recalled in
Theorem~\ref{thm:preparation}). Moreover such normal forms are essentially
unique, in the sense that among all formal conjugacies only some linear
changes of variables and parameter preserve the whole family. For
example transforming $x$ into $\alpha x$ for some nonzero $\alpha\in\cc$
in $P_{\varepsilon}\pp x$ yields the vector field 
\begin{align*}
\frac{1}{\alpha}P_{\varepsilon}\left(\alpha x\right)\pp x & =\alpha^{k}P_{\widetilde{\varepsilon}}\left(x\right)\pp x
\end{align*}
where $\widetilde{\varepsilon}:=\left(\varepsilon_{j}\alpha^{1-j}\right)_{j<k}$.
Therefore by taking $\alpha^{k}=1$ the linear change $\left(\widetilde{\varepsilon},x\right)\mapsto\left(\varepsilon,\alpha x\right)$
transforms $\left(\fonf[\varepsilon]\right)_{\varepsilon}$ into $\left(\fonf[\widetilde{\varepsilon}]\right)_{\widetilde{\varepsilon}}$.
It turns out this is the only degree of freedom for formal changes
of parameters (see Section~\ref{sec:Formal}), which makes the parameter
of the normal form special. 
\begin{defn}
\label{def:canonical_parameter}The parameter of the normal form $\fvnf$
(\emph{modulo} the action of $\zsk$ on $\left(\varepsilon,x\right)$)
is called the \textbf{canonical parameter} of the original unfolding
$Z$. In all the following a representative $\varepsilon$ of the
canonical parameter is always implicitly fixed and forbidden to change.
\end{defn}

As a consequence, two formal normal forms with formal invariants $\left(\mu,u\right)$
and $\left(\widetilde{\mu},\widetilde{u}\right)$ as above are (for
fixed canonical parameter $\varepsilon$):
\begin{enumerate}
\item orbitally formally equivalent if, and only if, they have same \textbf{formal
orbital invariant} $\mu=\widetilde{\mu}$ ;
\item formally conjugate if, and only if, they have same \textbf{formal
invariant} $\left(\mu,u\right)=\left(\widetilde{\mu},\widetilde{u}\right)$.
\end{enumerate}

\subsubsection{Analytical normalization}
\begin{defn}
\label{def:section}For $k\in\nn$ a positive integer let us introduce
the functional space in the complex multivariable $\left(\varepsilon,x,v\right)\in\cc^{k+2}$:
\begin{align*}
\nfsec[k][v] & :=\left\{ f_{\varepsilon}\left(x,v\right)=v\sum_{j=1}^{k}f_{\varepsilon,j}\left(v\right)x^{j}~:~f_{\varepsilon,j}\left(v\right)\in\germ{\varepsilon,v}\right\} .
\end{align*}
We let $v$ figure explicitly in the notation $\nfsec[][v]$ since
this variable (and this variable only) will be subject to further
specification.
\end{defn}

\begin{namedthm}[Normalization Theorem]
For a given $k\in\nn$ we fix a formal orbital invariant $\mu\in\germ{\varepsilon}$
and choose $\tau\in\zp$ such that $\mu_{0}+\left(k+1\right)\tau\notin\rr_{\leq0}$.
For every $Z\in\conv[k]$ with formal invariant $\left(\mu,u\right)$,
there exist $Q,~R\in\nfsec[k]$ such that $Z$ is analytically conjugate
to
\begin{align}
\mathcal{Z} & :=\frac{u}{1+uQ}\mathcal{X}\label{eq:normal_form}
\end{align}
where
\begin{align}
\mathcal{X} & :=\fonf+Ry\pp y.\label{eq:orbital_normal_form}
\end{align}
\end{namedthm}
\begin{rem}
In case $\tau=0$ (which can be enforced whenever the generic condition
$\mu_{0}\notin\rr_{\leq0}$ holds) normal forms induce foliations
with holomorphic extension to $\proj\times\neigh$. This is no longer
true if $\tau>0$ and if $R$ is not polynomial in the $y$-variable. 
\end{rem}

Specializing the theorem to $\varepsilon=0$, we recover the earlier
results~\cite{SchaTey,Loray}. Let us briefly present the unfolded
geometric construction of F.~\noun{Loray} (performed at an orbital
level in~\cite{Loray} when $k=1$) to get the gist of the argument.
We define a holomorphic family of abstract foliated complex surfaces
$\left(\mathcal{M},\fol{}\right)=\left(\mathcal{M}_{\varepsilon},\mathcal{F}_{\varepsilon}\right)_{\varepsilon\in\neigh[k]}$
given by two charts. The first one is a domain $\mathcal{U}^{0}:=\left\{ 0\leq\left|x\right|<\rho^{0}\right\} \times\neigh$
together with some arbitrary convergent unfolding $Z$, provided the
following non-restrictive properties (see~\cite{RouTey}) are fulfilled
for all $\varepsilon\in\neigh[k]$:
\begin{itemize}
\item $Z_{\varepsilon}$ is holomorphic on the domain and has at most $k+1$
singular points in $\mathcal{U}^{0}$ (counted with multiplicity in
case of saddle-nodes) each one located within $\mathcal{U}^{0}\cap\left\{ 0\leq\left|x\right|<\nf 1{\rho^{\infty}}\right\} $
for some $\rho^{\infty}>1/\rho_{0}$,
\item $Z_{\varepsilon}$ is transverse to the lines $\left\{ x=c\right\} $
whenever $P_{\varepsilon}\left(c\right)\neq0$, 
\item $Z_{\varepsilon}$ leaves $\left\{ y=0\right\} $ invariant.
\end{itemize}
The other chart is a domain $\mathcal{U}^{\infty}:=\left\{ \nf 1{\rho^{\infty}}<\left|x\right|\leq\infty\right\} \times\neigh$
equipped with a foliation $\fol{\varepsilon}^{\infty}$ 
\begin{itemize}
\item having a single, reduced singularity at $\left(\infty,0\right)$, 
\item otherwise transverse to the lines $\left\{ x=\cst\right\} $, 
\item leaving $\left\{ y=0\right\} $ invariant.
\end{itemize}
Biholomorphic fibered transitions maps fixing $\left\{ y=0\right\} $
exists on the annulus $\mathcal{U}^{0}\cap\mathcal{U}^{\infty}$ precisely
when $Z_{\varepsilon}$ and $\fol{\varepsilon}^{\infty}$ have (up
to local conjugacy) mutually inverse holonomy maps above, say, the
invariant circle $\frac{\rho_{0}\rho_{\infty}+1}{2\rho_{\infty}}\sone\times\left\{ 0\right\} $.
The resulting complex surface $\mathcal{M}_{\varepsilon}$ is naturally
a holomorphic fibration by discs over the divisor $\lif{}\simeq\proj$.
In other words  $\mathcal{M}_{\varepsilon}$ is a germ of a Hirzebruch
surface, classified at an analytic level~\cite{Koda,Ued,GrauFisch}
by the self-intersection $-\widehat{\tau}\in\zz_{\leq0}$ of $\lif{}$
in $\mathcal{M}_{\varepsilon}$. From the compactness of $\lif{}$
stems the polynomial-in-$x$ nature of the foliation $\fol{\varepsilon}$.
Other considerations then allow to recognize that $\fol{}$ is (globally
conjugate to a family of foliations) in normal form~\eqref{eq:normal_form}.

\bigskip{}

Let us explain where $\fol{\varepsilon}^{\infty}$ comes from, and
at the same time how the Hirzebruch class $\widehat{\tau}=\left(k+1\right)\tau$
is involved. When the construction of $\left(\mathcal{M},\fol{}\right)$
is possible, the global holomorphic foliation $\fol{\varepsilon}$
leaves the compact divisor $\lif{}$ invariant and Camacho-Sad index
formula~\cite{CamaSad} applies. The sum of indices of $Z_{\varepsilon}$
at its $k+1$ singularities, with respect to $\lif{}$, is $\mu_{\varepsilon}$
so $\fol{\varepsilon}^{\infty}$ must have index $-\left(\mu_{\varepsilon}+\widehat{\tau}\right)$.
By assumption the singularity at $\left(\infty,0\right)$ can therefore
never be a (saddle-)node. Invoking the realization result of~\cite[Section 4.4]{SchaTey}
(more precisely in the chart near $\left(\infty,0\right)$) it is
always possible to find a foliation $\fol{\varepsilon}^{\infty}$
with the desired properties. On the contrary when $\mu_{\varepsilon}+\widehat{\tau}\leq0$
then no such $\fol{\varepsilon}^{\infty}$ may exist at all except
in very special cases (detailed in~\cite[Theorem 2]{Loray}) since,
for instance, the holonomy along $\lif{}$ of a node is always linearizable
while the weak holonomy of $Z_{\varepsilon}$ has no reason to be
linearizable. We discuss this problem in more details while dealing
with the non-linear Riemann-Hilbert problem below.

Therefore one can always take $\tau:=0$ except when $\mu_{0}\leq0$,
which accounts for the ``twist'' $P_{\varepsilon}\left(x\right)^{\tau}y\sim_{x\to\infty}x^{\widehat{\tau}}y$
in normal forms~\eqref{eq:normal_form}. 

\subsubsection{Normal forms uniqueness}

To fully describe the quotient space (moduli space) of $\conv$ by
analytical conjugacy~/~orbital equivalence, the Normalization Theorem
must be complemented with a description of equivalence classes within
the family of normal forms~\eqref{eq:normal_form}, leading us to
discuss its uniqueness clause. 
\begin{defn}
\label{def:normal_invariants}~

\begin{enumerate}
\item For $Z\in\conv$ we denote 
\begin{align}
\mathfrak{n}\left(Z\right) & :=\left(\mu,u,R,Q\right)\label{eq:normal_invariants}\\
\mathfrak{o}\left(Z\right) & :=\left(\mu,R\right)\label{eq:normal_orbital_invariants}
\end{align}
respectively the \textbf{normal invariant} of $Z$ and its \textbf{normal
orbital invariant}, where the functional tuples on the right-hand
side are given by the Normalization Theorem. 
\item For $c\in\unisp$ and $f\in\germ{\varepsilon,x,y}$ define 
\begin{align*}
c^{*}f & :=\left(\varepsilon,x,y\right)\longmapsto f_{\varepsilon}\left(x,~c_{\varepsilon}y\right).
\end{align*}
We extend component-wise this action of $\unisp$ to tuple of functions
such as $\mathfrak{n}$ and $\mathfrak{o}$ above.
\end{enumerate}
\end{defn}

\begin{namedthm}[Uniqueness Theorem]
~
\begin{enumerate}
\item Two normal forms~\eqref{eq:normal_form} associated to the same fixed
$\tau$ and moduli~\eqref{eq:normal_invariants} $\mathfrak{n}$
and $\widetilde{\mathfrak{n}}$ are analytically conjugate (by a change
of coordinates fixing the parameter) if, and only if, there exists
$c\in\unisp$ such that $c^{*}\mathfrak{n}=\widetilde{\mathfrak{n}}$.
For any conjugacy $\Psi~:~\left(\varepsilon,x,y\right)\mapsto\left(\varepsilon,\Psi_{\varepsilon}\left(x,y\right)\right)$
there exists a unique $t\in\germ{\varepsilon}$ such that 
\begin{align*}
\Psi & =c^{*}\flow{\mathcal{Z}}t{},
\end{align*}
where $\flow{\mathcal{Z}}t{}$ is the local flow of $\mathcal{Z}$
at time $t\in\cc$. Moreover it is fibered in the $x$-variable if,
and only if, $t=0$. In that case $\Psi$ is linear:
\begin{align*}
\Psi=c^{*}\id~:~\left(\varepsilon,x,y\right) & \longmapsto\left(\varepsilon,~x,~c_{\varepsilon}y\right).
\end{align*}
\item Let $\mathfrak{o}$ and $\widetilde{\mathfrak{o}}$ be the corresponding
orbital invariants. The normal forms are analytically orbitally equivalent
(by a change of coordinates fixing the parameter) if, and only if,
there exists $c\in\unisp$ such that $c^{*}\mathfrak{o}=\widetilde{\mathfrak{o}}$.
For any orbital equivalence $\Psi$ there exists a unique $F\in\germ{\varepsilon,x,y}$
such that 
\begin{align*}
\Psi & =c^{*}\flow{\mathcal{Z}}F{}.
\end{align*}
Moreover $\Psi$ is fibered in the $x$-variable if, and only if,
$F=0$. In that case $\Psi$ is linear.
\end{enumerate}
\end{namedthm}
\begin{rem}
In particular normal forms~\eqref{eq:normal_form} are unique when
only tangent-to-identity in the $y$-variable, fibered in the $x$-variable
conjugacies are allowed. 
\end{rem}

Again the proof is largely based on the strategy of F.~\noun{Loray}
introduced in~\cite{Loray}, although the actual implementation in
the parametric case calls for subtle adaptations. The idea is to extend
any local and fibered conjugacy between normal forms to a global conjugacy
on a ``big'' neighborhood of $\lif{}$, from which it easily follows
that only linear maps can do that.

\subsection{Inverse problem}

For given $k\in\nn$ we can split the parameter space $\neigh[k]$
into $C_{k}=\frac{1}{k+1}\binom{2k}{k}$ open cells $\mathcal{E}_{\ell}$
such that 
\begin{align*}
\bigcup_{\ell}\mathcal{E}_{\ell} & =\neigh[k]\backslash\Delta_{k},
\end{align*}
where $\Delta_{k}$ is the set of parameters $\varepsilon$ for which
$P_{\varepsilon}$ has at least a multiple root ($\Delta_{k}$ is
the discriminant curve). We recall that we can associate~\cite{RouTey}
an orbital modulus to a purely convergent unfolding $Z$
\begin{align*}
\mathfrak{m}\left(Z\right) & :=\left(\mathfrak{m}_{\ell}\right)_{1\leq\ell\leq C_{k}}\\
\mathfrak{m}_{\ell} & :=\left(\sad{\phi}{\ell}{j,}\right)_{j\in\zsk}
\end{align*}
where for each $j\in\zsk$ and each $\ell$ the map
\begin{align*}
\left(\varepsilon,h\right)\in\mathcal{E}_{\ell}\times\neigh & \longmaps\sad{\phi}{\ell,\varepsilon}{j,}\left(h\right)
\end{align*}
is holomorphic, vanishes along $\left\{ h=0\right\} $ and admits
a continuous extension to $\adh{\mathcal{E}_{\ell}}\times\neigh$.
\begin{rem}
~
\begin{enumerate}
\item The upper index ``s'' is purely notational and refers to the fact
that the function $\sad{\phi}{\ell,\varepsilon}{j,}$ comes from the
$j^{\tx{th}}$ ``s''addle intersection, where the dynamics behaves
very much like a saddle point. 
\item The diffeomorphisms $\sad{\psi}{\ell,\varepsilon}{j,}$, which unfold
the components $\sad{\psi}0{j,}$ of the (classical) Martinet-Ramis
modulus, are given by $\sad{\psi}{\ell,\varepsilon}{j,}(h)=h\exp\left(\frac{2\ii\pi\mu_{\varepsilon}}{k}+\sad{\phi}{\ell,\varepsilon}{j,}(h)\right)$.
\end{enumerate}
\end{rem}

Let us write $\mathcal{H}_{\ell}\left\{ h\right\} $ the vector space
of all such functions, so that
\begin{align*}
\mathfrak{m}\left(Z\right) & \in\prod_{\ell}\mathcal{H}_{\ell}\left\{ h\right\} ^{k}.
\end{align*}
The data $\mathfrak{m}\left(Z\right)$ is a complete orbital invariant
for the local analytic classification of purely convergent unfoldings.

\subsubsection{Orbital realization}

The definition of the \emph{compatibility condition} involves notions
going beyond the scope of the present summarized statements. We refer
to Section~\ref{subsec:Compatibility-condition} for a precise definition.
Instead let us use the following terminology.
\begin{defn*}
We say that $\left(\mu,\mathfrak{m}\right)\in\germ{\varepsilon}\times\prod_{\ell}\mathcal{H}_{\ell}\left\{ h\right\} ^{k}$
is\textbf{ realizable} if there exists a generic convergent unfolding
$Z$ with formal orbital class $\mu$ and orbital modulus $\mathfrak{m}=\mathfrak{m}\left(Z\right)$.
\end{defn*}
For the sake of completeness, let us state the following fundamental
result even though all material was not properly introduced.
\begin{namedthm}[Orbital Realization Theorem]
Let $\mu\in\germ{\varepsilon}$ be given. A functional data $\mathfrak{m}\in\prod_{\ell}\mathcal{H}_{\ell}\left\{ h\right\} ^{k}$
yields a realizable $\left(\mu,\mathfrak{m}\right)$ if, and only
if, $\left(\mu,\mathfrak{m}\right)$ satisfies the compatibility condition
(presented in Definition~\ref{def:compatibility_condition}).
\end{namedthm}
Although it is not directly used in the present paper, considerations
akin to those from~\cite{SchaTey} show that the map sending a normal
form to its orbital modulus
\begin{align*}
\mathfrak{o}=\left(\mu,R\right) & \longmaps\left(\mu,\mathfrak{m}\right)
\end{align*}
is upper-triangular, in the sense that the $n^{\tx{th}}$-jet of $\mathfrak{m}_{\ell}$
with respect to $h$ is completely determined by $\mu$ and the $n^{\tx{th}}$-jet
of $R$ with respect to $y$. In that sense passing from modulus to
normal form is a (non-effective) computable process. In the case $k=1$
we show how to compute the diagonal entries. More details on this
topic can be found in Section~\ref{subsec:first-order}. 

\subsubsection{Moduli which are analytic with respect to the parameter}

Our final main result proves that the compatibility condition defines
a proper subset of the vector space $\germ{\varepsilon}\times\prod_{\ell}\mathcal{H}_{\ell}\left\{ h\right\} ^{k}$.
\begin{namedthm}[Parametrically Analytic Orbital Moduli Theorem]
Let $\mu\in\germ{\varepsilon}$ and $\mathfrak{m}=\left(\mathfrak{m}_{\ell}\right)_{\ell}\in\prod_{\ell}\mathcal{H}_{\ell}\left\{ h\right\} ^{k}$
be given. Assume $\mathfrak{m}$ is holomorphic, in the sense that
$\mathfrak{m}_{\ell}=M|_{\mathcal{E}_{\ell}\times\neigh}$ for some
$M\in h\germ{\varepsilon,h}^{k}$. The following conditions are equivalent:

\begin{enumerate}
\item $\left(\mu,\mathfrak{m}\right)$ satisfies the compatibility condition,
\item either $\mathfrak{m}=0$, or $k=1$ and there exists $d\in\nn$, $\alpha\in\germ{\varepsilon}\backslash\left\{ 0\right\} $
such that 

\begin{itemize}
\item $d\mu\in\zz$ (in particular $\mu$ is a rational constant),
\item $M\left(h\right)=-\frac{1}{d}\log\left(1-\alpha h^{d}\right)$.
\end{itemize}
\end{enumerate}
If one of the conditions is satisfied and $\mathfrak{m}\neq0$, then
an orbital normal form realizing $\left(\mathfrak{m}_{\ell}\right)_{\ell}$
is:
\begin{align*}
\onf[\varepsilon] & =\fonf[\varepsilon]+r_{\varepsilon}xP_{\varepsilon}\left(x\right)^{\tau d}y^{d+1}\pp y
\end{align*}
for some $r\in\germ{\varepsilon}\backslash\left\{ 0\right\} $. We
then speak of \textbf{Bernoulli unfoldings} because the underlying
non-autonomous differential equation induced by the flow is Bernoulli.

\end{namedthm}
\begin{rem}
\label{rem:Bernoulli_normalization}By letting $\germ{\varepsilon}^{\times}$
act linearly through $\left(\varepsilon,x,y\right)\mapsto\left(\varepsilon,x,c_{\varepsilon}y\right)$
we may normalize further $r$ to some $\varepsilon^{\kappa}$ for
$\kappa\in\zp$. See also Section~\ref{subsec:first-order}.
\end{rem}

\subsection{\label{subsec:Applications}Application: non-linear differential
Galois theory}

M.~\noun{Berthier} and F.~\noun{Touzet}~\cite{BerTouze} proved
that the Martinet-Ramis modulus of a convergent saddle-node vector
field admitting non-trivial Liouvillian first integrals~\cite{SingVDP}
must be a ramified homography $h\mapsto\alpha h\left(1+\beta h^{d}\right)^{-\nf 1d}$,
from which they deduce that the vector field is Bernoulli. (It is
indeed straightforward to compute the modulus of a Bernoulli vector
field by solving explicitly the underlying differential equation.)
Roughly speaking this situation corresponds to differential equations
admitting ``closed-form'' solutions obtained by iteratively taking
quadrature (or exponential thereof) of elements of (algebraic extensions
of) the base-field (here, meromorphic functions on a polydisk containing
$P_{\varepsilon}^{-1}\left(0\right)\cap\left\{ y=0\right\} $).
\begin{namedthm}[Integrability Theorem]
 Let $\left(Z_{\varepsilon}\right)_{\varepsilon}$ be a generic,
purely convergent saddle-node unfolding and denote by $\mathcal{L}$
the germ of set consisting in those $\varepsilon\in\neigh[k]$ for
which $Z_{\varepsilon}$ admits a Liouvillian first integral. The
following statements are equivalent.
\begin{enumerate}
\item The locus of integrability $\mathcal{L}$ is full: $\mathcal{L}=\neigh[k]$.
\item Its (analytic) Zariski closure is full: $\overline{\mathcal{L}}^{\tt{Zar}}=\neigh[k]$.
\item Its orbital normal form $\onf$ is a Bernoulli unfolding.
\end{enumerate}
\end{namedthm}
\begin{rem}
~
\begin{enumerate}
\item The case $\overline{\mathcal{L}}^{\tt{Zar}}\neq\neigh[k]$ corresponds
to $\mathcal{L}$ being a germ at $0$ of a proper analytic subvariety.
Then $\mathcal{L}$ is the locus of parameters for which the normal
form is Bernoulli. For instance in case of the normal form given by
$R_{\varepsilon}\left(t\right):=t^{d}+L\left(\varepsilon\right)t^{d+1}$
we have $\mathcal{L}=L^{-1}\left(0\right)$, as we discuss after the
proof of the theorem.
\item In the case $k=1$ the second condition is equivalent to any of the
following three conditions: the germ $\mathcal{L}$ accumulates on
$0$, $\mathcal{L}$ is infinite, $\mathcal{L}\neq\left\{ 0\right\} $.
\end{enumerate}
\end{rem}

\begin{proof}
The property of having a non-trivial Liouvillian first integral is
both orbital and invariant by change of analytic coordinates, so we
do not lose generality by taking $Z=\onf$ in normal form~\eqref{eq:orbital_normal_form}.
Integrability is equivalent to the existence of a Godbillon-Vey sequence~\cite{GodVey}
of length at most 2, that is to the existence of two non-zero meromorphic
$1$-forms $\omega$ and $\eta$ for which 
\begin{align*}
\dd{\eta} & =0\\
\dd{\omega} & =\delta\omega\wedge\eta~~~,~\delta\in\left\{ 0,1\right\} \\
\omega\left(\onf\right) & =0.
\end{align*}
(The multivalued map $H:=\int\exp\left(\delta\int\eta\right)\omega$
is indeed a Liouvillian first integral of $\onf$, obtained by quadrature
of closed $1$-forms.) This in turn is (almost) equivalent to solving
for meromorphic, transverse $Y\neq0$ in the Lie-bracket equation
\begin{align}
\left[\onf,Y\right] & =\delta Y~~~,~\delta\in\left\{ 0,1\right\} ,\label{eq:godbillon-vey_lie}
\end{align}
since the dual basis $\left(\eta,\omega\right)$ of $\left(\onf,Y\right)$
is a Godbillon-Vey sequence and \emph{vice versa}. There is a subtlety
here, because $\onf$ may fail to meet this condition while there
could exist an integrating factor $V$ for which $V\onf$ does. We
deliberately ignore this eventuality, because the case $V\neq1$ can
be deduced from the particular case $V=1$ by a direct (albeit cumbersome)
adaptation. For the same reason we only deal with the case $\mu_{0}\notin\rr_{\leq0}$. 

The implications (3)$\Rightarrow$(1)$\Rightarrow$(2) are clear enough
(we particularly refer to Lemma~\ref{lem:bernoulli_unfold_with_bernoulli_modulus}
for the first one). Let us prove (2)$\Rightarrow$(3). The strategy
is the following: we first show that the vector field is Bernoulli
for each $\varepsilon\in\mathcal{L}$, then we invoke the analyticity
of the normal form and the fact that $\overline{\mathcal{L}}^{\tt{Zar}}$
is full to cover a whole neighborhood of $0$ in parameter space.
Hence, let us fix $\varepsilon\in\mathcal{L}$ and drop the index
$\varepsilon$ altogether. According to the above discussion one can
find $\delta\in\left\{ 0,1\right\} $ and a vector field 
\begin{align*}
Y & =A\left(x,y\right)\pp y+B\left(x,y\right)\onf
\end{align*}
solving~\eqref{eq:godbillon-vey_lie} for two functions $A\neq0$
and $B$ meromorphic on a polydisk containing $P^{-1}\left(0\right)\cap\left\{ y=0\right\} $.
From~\eqref{eq:godbillon-vey_lie} we deduce the relations
\begin{align}
\begin{cases}
\onf\cdot B & =\delta B,\\
\onf\cdot A & =\left(\delta+1+\mu x^{k}+\ppp{yR}y\right)A.
\end{cases}\label{eq:Godbillon-Vey_system}
\end{align}
The second equation tells us that $\left\{ A=0\right\} \cup\left\{ A=\infty\right\} $
is a union of separatrices of $\onf$, therefore of the form 
\begin{align*}
A\left(x,y\right) & =y^{d+1}U\left(x,y\right)\prod_{P\left(z\right)=0}\left(x-z\right)^{\ell\left(z\right)},
\end{align*}
for some choice of $d,~\ell\left(z\right)\in\zz$ and for some holomorphic
and never vanishing function $U$. Let us prove that $R=r\left(x\right)y^{d}$,
from which follows either $d\in\nn$ or $R=0$. 

The last equation of~\eqref{eq:Godbillon-Vey_system} becomes
\begin{align*}
\onf\cdot\log U & =\delta-d\left(1+\mu x^{k}\right)-\sum_{P\left(z\right)=0}\ell\left(z\right)\frac{P\left(x\right)}{x-z}+\left(y\ppp Ry-dR\right),
\end{align*}
because $\onf\cdot\log y=1+\mu x^{k}+R$. Evaluating this identity
at any one of the $k+1$ points $\left(x,y\right)=\left(x,0\right)$
such that $P\left(x\right)=0$ yields $0=\delta-d\left(1+\mu x^{k}\right)-\ell\left(x\right)P'\left(x\right)$,
since on the one hand $R$ and $y\ppp Ry$ vanish when $y=0$ while
on the other hand $\ell\left(z\right)\frac{P\left(x\right)}{x-z}$
evaluates to $0$ if $z\neq x$ and to $\ell\left(x\right)P'\left(x\right)$
otherwise. As a consequence we have equality of the polynomials $\sum_{P\left(z\right)=0}\ell\left(z\right)\frac{P\left(x\right)}{x-z}=\delta-d\left(1+\mu x^{k}\right)$
of degree $k$. Therefore 
\begin{align*}
\onf\cdot\log U & =y\ppp Ry-dR~.
\end{align*}
In the course of Section~\ref{sec:Temporal} we show that $\tx{im}\left(\onf\cdot\right)\cap\nfsec[][y]=\left\{ 0\right\} $
(see Remark~\ref{rem:section_fixed_parameter}). Hence, the fact
that $y\ppp Ry-dR\in xy\pol x_{<k}\left\{ y\right\} =\nfsec[][y]$
lies in the image of $\onf\cdot$ can only mean $y\ppp Ry-dR=0$.
From this we deduce at once that
\begin{align*}
R\left(x,y\right) & =xr\left(x\right)y^{d}~~~~~,~r\in\pol x_{<k}.
\end{align*}

The condition that, for a specific $\varepsilon$, the vector field
$\onf[\varepsilon]$ be Bernoulli corresponds to the vanishing of
all coefficients in $R_{\varepsilon}$ of $y^{n}$ but for $n=d$.
Since $\left(\varepsilon,y\right)\mapsto R_{\varepsilon}\left(y\right)$
is analytic with respect to $\varepsilon$ and $\overline{\mathcal{L}}^{_{\tt{Zar}}}=\neigh[k]$,
if $d$ is independent on $\varepsilon$ then the identity principle
implies $R_{\varepsilon}\left(x,y\right)=xr_{\varepsilon}\left(x\right)y^{d}$
for all $\left(\varepsilon,x,y\right)\in\neigh[k+2]$. The fact that
$d$ is indeed independent on $\varepsilon$ stems from Baire's category
theorem. 
\end{proof}
\begin{rem}
The proof relies in an essential way on the analyticity of the orbital
normal form $\onf$ with respect to $\varepsilon$ near $0$. Compare
with the method of proof of~\cite{BerTouze}: for $\varepsilon=0$
the argument is based on the fact that the existence of a Godvillon-Vey
sequence forces the Martinet-Ramis modulus to be a ramified homography.
This argument works as well for $\varepsilon\neq0$, but we could
not have argued on from there since the modulus is in general not
analytic at $\varepsilon=0$: although being a ramified homography
is an analytic condition, an accumulation of zeros of this relation
as $\varepsilon\to0$ could arise without holding for all $\varepsilon$
(for $k=1$, say). This situation cannot occur, and our shorter argument
does not involve the unfolded modulus of classification.
\end{rem}

The Galoisian characterization of the existence of Godbillon-Vey sequences
of length at most $2$ is performed in~\cite{Casa}, and for fixed
$\varepsilon$ its length equals the (transverse) rank $\tx{rk}\left(\mathfrak{M}_{\varepsilon}\right)$
of the Galois-Malgrange groupoid $\mathfrak{M}_{\varepsilon}$. This
rank takes values in $\left\{ 0,1,2,\infty\right\} $, integrability
corresponding to finite values. For the normal forms~\eqref{eq:orbital_normal_form}
with $R_{\varepsilon}\left(x,t\right)=\sum_{n>0}r_{\varepsilon,n}\left(x\right)t^{n}$,
we have
\begin{align*}
\tx{rk}\left(\mathfrak{M}_{\varepsilon}\right)= & \begin{cases}
1+\#\left\{ n~:~r_{\varepsilon,n}\neq0\right\}  & \tx{if~it~is~\leq2,}\\
\infty & \tx{otherwise}.
\end{cases}
\end{align*}
Therefore $\varepsilon\mapsto\tx{rk}\left(\mathfrak{M}_{\varepsilon}\right)$
is lower semi-continuous: accidental values of the rank can only correspond
to more integrable systems. 
\begin{example*}
Taking into account Remark~\ref{rem:Bernoulli_normalization}, in
the case $k=1$ and $R\neq0$ the vector field $\onf[0]$ is ``more
integrable'' (transverse rank $1$) than the generic $\onf[\varepsilon]$
(transverse rank $2$) if and only if the exponent $\kappa$ is positive. 
\end{example*}
This is a special instance of a general result on parametrized Galois-Malgrange
groupoids obtained very recently by G.~\noun{Casale} and D.~\noun{Davy}~\cite{CasaDav}.
They show that for rather general deformations of foliations $\left(\mathcal{F}_{\varepsilon}\right)_{\varepsilon}$,
the rank $\tx{rk}\left(\mathfrak{M}_{\varepsilon}\right)$ of the
specialization $\mathfrak{M}_{\varepsilon}$ of the Galois-Malgrange
groupoid of the family is lower semi-continuous in $\varepsilon$.
Moreover, the locus of discontinuity is contained in a countable union
of proper analytic subvarieties. We showed that in the case of purely
convergent saddle-node bifurcations, the locus of discontinuity is
at most a proper analytic subvariety.

\subsection{Structure of the paper}
\begin{itemize}
\item We begin with fixing notations and providing precise definitions in
Section~\ref{sec:Notations}. Readers familiar with complex foliations
may skip this section.
\item The Formal Normalization Theorem is proved in Section~\ref{sec:Formal}. 
\end{itemize}
We first present the generic case (for which one can take $\tau=0$),
since it is easier to highlight the ideas than in the case $\tau>0$.
\begin{itemize}
\item The orbital part of the Normalization and Uniqueness Theorems are
established in Section~\ref{sec:Geometric} when $\tau=0$. 
\item The temporal part of the Normalization and Uniqueness Theorems are
established in Section~\ref{sec:Temporal} when $\tau=0$. 
\item In Section~\ref{sec:Analytic} one finds the definition of compatibility
condition, and the proof of the Orbital Realization Theorem in the
generic case $\tau=0$.
\item In Section~\ref{sec:tau} we prove the Orbital Realization Theorem
in the case $\tau>0$. This provides \emph{a posteriori} a proof of
the orbital part of the Normalization and Uniqueness Theorems when
$\tau>0$.
\item In Section~\ref{sec:Bernoulli} we discuss the Bernoulli unfoldings
and prove the Parametrically Analytic Orbital Moduli Theorem.
\item Finally, in Section~\ref{sec:Computations}, we conclude with a few
words on computations.
\end{itemize}

\section{\label{sec:Notations}Preliminaries}

\subsection{Notations}

\subsubsection{\label{subsec:General-notations}General notations}
\begin{itemize}
\item We let the set $\nn:=\left\{ 1,2,\ldots\right\} $ stand for all positive
integers, whereas the set of non-negative integers will be written
$\zp=\left\{ 0,1,\ldots\right\} $.
\item For $n\in\nn$ we let $\neigh[n][0]$ stand for any small enough domain
in $\cc^{n}$ containing $0$.
\item The domain $\ww D:=\left\{ z\in\cc~:~\left|z\right|<1\right\} $ is
the standard open unit disk.
\item The closure of a subset $A$ of a topological space is written $\adh A$.
\item $k\in\nn$ is fixed, $\varepsilon=\left(\varepsilon_{0},\ldots,\varepsilon_{k-1}\right)\in\neigh[k]$
is the parameter and
\begin{align}
P_{\varepsilon}\left(x\right) & =x^{k+1}+\sum_{j=0}^{k-1}\varepsilon_{j}x^{j}.\label{eq:definition_P}
\end{align}
\item The parameter space $\neigh[k]$ is covered by the closure of $C_{k}=\frac{1}{k+1}\binom{2k}{k}$
open and contractible cells $\mathcal{E}_{\ell}$.
\item The period operator $\per=\left(\per[j]\right)_{j\in\zsk}$ is built
near Definition~\ref{def:period_operator}.
\item The very nature of constructions involves using more sub- and super-scripts
than one is generally comfortable with. To alleviate this downside
we stick to a single convention: \textbf{subscripts are }\textbf{\emph{always}}\textbf{
parameter-related}, while superscripts are in general related to the
geometric variables $\left(x,y\right)$ or to indices in power series
expansions. Example: we write $\sad V{\ell,\varepsilon}{j,}$ for
the ``s''addle part of the $j^{\tx{th}}$ sector in $x$-variable,
relatively to the parameter $\varepsilon$ being taken in the $\ell^{\tx{th}}$
parametric cell. In the course of the text we try to drop indices
whenever possible.
\item \textbf{The dependency on the parameter $\varepsilon$ is implicit
in most instances}. For example, $\mu\in\germ{\varepsilon}$ stands
for the formal orbital modulus while $\mu_{\varepsilon}$ stands for
the value of $\mu$ at the particular value of the parameter $\varepsilon$.
Yet in many places where $\varepsilon$ is fixed we do use $\mu$
instead of $\mu_{\varepsilon}$ in order to help reducing the notational
footprint. This also applies for other parametric objects.
\end{itemize}

\subsubsection{\label{subsec:Functional-spaces}Functional spaces}

In the following $\mathcal{R}$ is a commutative ring with a multiplicative
action by complex numbers.
\begin{itemize}
\item $\mathcal{R}^{\times}$ is the multiplicative group of its invertible
elements.
\item $\mathcal{R}\left[\mathbf{z}\right]$ is the commutative ring of polynomials
in the complex finite-dimensional (multi)variable $\mathbf{z}=\left(z_{1},\ldots,z_{n}\right)$
with coefficients in $\mathcal{R}$. 
\item After choosing a binary relation $\prec$ among $\left\{ =,<,\leq,\ldots\right\} $
we let $\mathcal{R}\left[\mathbf{z}\right]_{\prec d}$ be the subset
of $\mathcal{R}\left[\mathbf{z}\right]$ consisting of polynomials
$P$ such that $\deg P\prec d$. 
\item The projective limit $\mathcal{R}\left[\left[\mathbf{z}\right]\right]:=\lim_{d\to\infty}\mathcal{R}\left[\mathbf{z}\right]_{\leq d}$
is the ring of formal power series in $\mathbf{z}$ with coefficients
in $\mathcal{R}$.
\item $\germ{\mathbf{z}}$ is the algebra of convergent formal power series
in the complex multivariable $\mathbf{z}\in\cc^{n}$, naturally identified
to the set of germs of a holomorphic function near $0\in\cc^{n}$.
\end{itemize}
\begin{rem}
We will mostly use the spaces:

\begin{itemize}
\item $\frml{\varepsilon}$, $\frml{\varepsilon,x}$ and $\frml{\varepsilon,x,y}$ 
\item $\germ{\varepsilon}$, $\germ{\varepsilon}^{\times}$, $\germ{\varepsilon,x}$
and $\germ{\varepsilon,x,y}$
\item $\germ{\varepsilon}\left[x\right]$, $\germ{\varepsilon}\left[x\right]_{\leq k}^{\times}$
and
\begin{align*}
\nfsec[k][v] & :=xv\germ{\varepsilon,v}\left[x\right]_{<k}.
\end{align*}
\end{itemize}
\end{rem}

Let $\mathcal{D}\subset\cc^{n}$ be a domain containing $0$ equipped
with the affine coordinates $\mathbf{z}=\left(z_{1},\ldots,z_{n}\right)$. 
\begin{itemize}
\item $\holf[\mathcal{D}]$ is the algebra of complex-valued functions holomorphic
on $\mathcal{D}$.
\item $\holb[\mathcal{D}]$ is the Banach subalgebra of $\holf[\mathcal{D}]$
of all holomorphic functions $f~:~\mathcal{D}\to\cc$, with bounded
continuous extension to $\adh{\mathcal{D}}$, equipped with the norm
\begin{align*}
\norm[f]{\mathcal{D}} & :=\sup_{\mathbf{z}\in\mathcal{D}}\left|f\left(\mathbf{z}\right)\right|.
\end{align*}
\item $\holb[\mathcal{D}]'$ is the Banach space of all holomorphic functions
$f~:~\mathcal{D}\to\cc$ vanishing on $\left\{ z_{n}=0\right\} $
with the norm
\begin{align*}
\norm[f]{\mathcal{D}}' & :=\sup_{\mathbf{z}\in\mathcal{D}}\left|\frac{f\left(\mathbf{z}\right)}{z_{n}}\right|.
\end{align*}
Notice that $\norm[f]{\mathcal{D}}'\leq\norm[\ppp f{z_{n}}]{\mathcal{D}}$
whenever $\ppp f{z_{n}}\in\holb[\mathcal{D}]$ and $\mathcal{D}$
is convex in the variable $z_{n}$.
\item We let
\begin{align*}
\mathcal{H}_{\ell}\left\{ \mathbf{z}\right\}  & :=\bigcup_{\mathcal{D}=\neigh[n]}\holb[\group{\mathcal{E}_{\ell}\times\mathcal{D}}]',
\end{align*}
where $\mathcal{E}_{\ell}$ is a parametric cell.
\end{itemize}

\subsubsection{\label{subsec:Vector-fields}Vector fields and Lie derivative}

Let $Z=\sum_{j=1}^{n}A_{j}\pp{z_{j}}$ be a germ of a holomorphic
vector field at the origin of $\cc^{n}$ (or formal vector field at
this point).
\begin{itemize}
\item If $f$ is a formal power series or a holomorphic function in $\mathbf{z}=\left(z_{1},\ldots,z_{n}\right)\in\neigh[n]$,
we denote by $Z\cdot f$ the directional Lie derivative of $f$ along
$Z$
\begin{align*}
Z\cdot f & :=\sum_{j=1}^{n}A_{j}\ppp f{z_{j}}=\DD f\left(Z\right).
\end{align*}
The operator is extended component-wise on vectors of power series
or functions.
\item We define recursively for $n\in\zp$ the $n^{\tx{th}}$ iterate of
the Lie derivative, the operator written $Z\cdot^{n}$, by 
\begin{align*}
Z\cdot^{0} & :=\id\\
Z\cdot^{n+1} & :=Z\cdot\left(Z\cdot^{n}\right).
\end{align*}
\item The flow of $Z$ at time $t$ starting from $\mathbf{z}$ is the formal
$n$-tuple of power series $\flow Zt{\left(\mathbf{z}\right)}$ solving
the flow-system
\begin{align*}
\ppp{\flow Zt{\left(\mathbf{z}\right)}}t & =Z\circ\flow Zt{\left(\mathbf{z}\right)},
\end{align*}
which is a convergent power series in $\left(t,\mathbf{z}\right)$
if, and only if, $Z$ is holomorphic. At some point we invoke the
classical formal identity of Lie
\begin{align}
f\circ\flow Zt{} & =\sum_{n\geq0}\frac{t^{n}}{n!}Z\cdot^{n}f.\label{eq:Lie_identity}
\end{align}
\item Two vector fields $Z$ and $\widetilde{Z}$ are formally~/~locally
conjugate when there exists a $n$-tuple of formal~/~convergent
power series $\Psi$ with invertible derivative at $0$ such that
\begin{align*}
\widetilde{Z}\cdot\Psi & =Z\circ\Psi.
\end{align*}
 In that case we write $\widetilde{Z}=\Psi^{*}Z$.
\item Two vector fields $Z$ and $\widetilde{Z}$ are formally~/~locally
orbitally equivalent when there exists a formal power series~/~holomorphic
function $U$ with $U\left(0\right)\neq0$ such that $UZ$ and $\widetilde{Z}$
are conjugate (in the same convergence class).
\end{itemize}

\subsection{Conjugacy and orbital equivalence}
\begin{defn}
\label{def:unfolding_conjugacy} Two unfoldings $Z=\left(Z_{\varepsilon}\right)_{\varepsilon}$
and $\widetilde{Z}=\left(\widetilde{Z}_{\widetilde{\varepsilon}}\right)_{\widetilde{\varepsilon}}$
are locally \textbf{conjugate} (\emph{resp.} \textbf{orbitally} \textbf{equivalent})
if there exists a holomorphic mapping 
\begin{align*}
\Psi~:~\left(\varepsilon,x,y\right) & \longmapsto\left(\phi\left(\varepsilon\right),~\Psi_{\varepsilon}\left(x,y\right)\right)
\end{align*}
such that:

\begin{enumerate}
\item $\varepsilon\in\neigh[k]\mapsto\widetilde{\varepsilon}=\phi\left(\varepsilon\right)$
has invertible derivative at $0$,
\item for each $\varepsilon\in\neigh[k][0]$ the component $\Psi_{\varepsilon}$
is a local conjugacy (\emph{resp.} orbital equivalence) between $Z_{\varepsilon}$
and $\widetilde{Z}_{\phi\left(\varepsilon\right)}$.
\end{enumerate}
If the above conditions are fulfilled we write 
\begin{align*}
\Psi^{*}Z & =\widetilde{Z}.
\end{align*}
We extend in the obvious way the definition for formal conjugacy~/~orbital
equivalence.

\end{defn}

\begin{rem}
The very first step of any construction performed here consists in
recalling the preparation of the generic unfolding $Z$ (Theorem~\ref{thm:preparation}).
For unfoldings in prepared form~\eqref{eq:prepared_form} the parameter
$\varepsilon$ becomes a formal invariant. Hence we only use conjugacies
fixing $\varepsilon$, that is $\Psi~:~\left(\varepsilon,x,y\right)\mapsto\left(\varepsilon,\Psi_{\varepsilon}\left(x,y\right)\right)$.
In that setting one can always deduce $\Psi$ knowing $\Psi_{\varepsilon}$,
therefore when we use the notation $\Psi$ we generally refer to the
map $\left(\varepsilon,x,y\right)\mapsto\Psi_{\varepsilon}\left(x,y\right)$,
except when the context is ambiguous.
\end{rem}

\begin{defn}
Consider a formal transform $\Psi~:~\left(\varepsilon,x,y\right)\mapsto\left(\varepsilon,\Psi_{\varepsilon}\left(x,y\right)\right)$.
We say that $\Psi$ is \textbf{fibered} when $\Psi_{\varepsilon}\left(x,y\right)=\left(x,\psi_{\varepsilon}\left(x,y\right)\right)$.
\end{defn}

~
\begin{defn}
$\Psi$ is a \textbf{symmetry} (\emph{resp.} \textbf{orbital symmetry})
of $Z$ when $\Psi$ is a self-conjugacy (\emph{resp.} orbital self-equivalence)
of $Z$.
\end{defn}

\begin{rem}
\label{rem:symmetry}Hence, to determine the orbital symmetries of
$Z$ it suffices to determine the changes $\Psi$ such that $\Psi^{*}Z=UZ$
for some $U$ with $U_{0}\left(0,0\right)\neq0$.
\end{rem}

\section{\label{sec:Formal}Formal normalization}

The formal normalization is based on three ingredients, each one corresponding
to a step of the construction:
\begin{itemize}
\item a preparation \emph{à la} Dulac of unfoldings: for $\varepsilon=0$
one recovers Dulac prepared form~\cite{Dulac2,Dulac};
\item the existence of a formal ``family of weak separatrices'' which
we can straighten to $\left\{ y=0\right\} $;
\item a variation on Lie's identity~\eqref{eq:Lie_identity} already used
in~\cite{RouTey,Tey-SN} to perform the analytic classification of
saddle-nodes vector fields and their unfoldings. The formula reduces
the problem of finding changes of variables to solving an uncoupled
system of cohomological equations.
\end{itemize}

\subsection{\label{subsec:Preparation}Preparation}

Take $\theta\in\zsk$ and set $\alpha:=\exp\nf{2\ii\pi\theta}k$.
For $\varepsilon:=\left(\varepsilon_{j}\right)_{j<k}\in\neigh[k]$
we define
\begin{align*}
\theta^{*}\varepsilon & :=\left(\varepsilon_{j}\alpha^{j-1}\right)_{j<k}.
\end{align*}

\begin{thm}
\cite[Proposition 3.1 and Theorem 3.5]{RouTey}\label{thm:preparation}
Any generic unfolding is analytically conjugate to an unfolding of
the form
\begin{align}
Z & =UX\label{eq:prepared_form}\\
X & =\fonf+A\pp y\label{eq:prepared_form_orbital}\\
A\left(x,y\right) & =P\left(x\right)a\left(x\right)+yR\left(x,y\right)\nonumber 
\end{align}
where $\fonf$ and $P$ are defined in~\eqref{eq:formal_orbital_normal_form}
and~\eqref{eq:definition_P}, while $a\in\germ{\varepsilon,x}$,
$R\in y\germ{\varepsilon,x,y}$ and $U\in\germ{\varepsilon,x,y}$
with $U_{0}\left(0,0\right)\neq0$. In the particular case of an analytic
weak separatrix one can take $a:=0$. 

Besides if two such prepared forms $\left(Z_{\varepsilon}\right)_{\varepsilon}$
and $\left(\widetilde{Z}_{\widetilde{\varepsilon}}\right)_{\widetilde{\varepsilon}}$
are formally orbitally equivalent then there exists $\theta\in\zsk$
such that $\widetilde{\varepsilon}=\theta^{*}\varepsilon$: the parameter
is unique \emph{modulo} this action and is called \textbf{canonical}.
\end{thm}

\begin{rem}
Although the original result is stated in~\cite{RouTey} at an analytic
level, the proof that $\varepsilon$ becomes an invariant \emph{modulo}
the action of $\zsk$ stems from a formal computation and is therefore
valid for formal orbital equivalence too. The idea of the proof is
that the parameter completely determines the data of local eigenratios
and \emph{vice versa}, which are well-known orbital invariants.
\end{rem}

From now on we only deal with unfoldings in prepared form~\ref{eq:prepared_form}
and only consider transforms fixing the canonical parameter $\varepsilon$. 

\subsection{Straightening weak separatrices}
\begin{prop}
\label{pro:weak_separatrices}\emph{\cite[Proposition 2]{KlimNonLin}}
For any unfolding $X$ in prepared form~\eqref{eq:prepared_form_orbital}
there exists a formal power series 
\begin{align*}
\widehat{s} & \in\frml{\varepsilon,x}P
\end{align*}
solving the parametric family of differential equations
\begin{align}
P_{\varepsilon}\left(x\right)\ddd{\widehat{s}_{\varepsilon}}x\left(x\right) & =\widehat{s}_{\varepsilon}\left(1+\mu_{\varepsilon}x^{k}\right)+A_{\varepsilon}\left(x,\widehat{s}_{\varepsilon}\left(x\right)\right).\label{eq:weak_separatrix_diff_eq}
\end{align}
Performing the transform $\left(\varepsilon,x,y\right)\mapsto\left(\varepsilon,x,y+\widehat{s}_{\varepsilon}\left(x\right)\right)$
sends $Z$ to a prepared unfolding 
\begin{align}
Z & =\widehat{U}\left(\widehat{X}+\widehat{A}y\pp y\right)\label{eq:formally_straightened_form}
\end{align}
for some \emph{formal} power series $\widehat{U}$ and $\widehat{A}$
in $\frml{\varepsilon,x,y}$ with
\begin{align*}
\widehat{U}\left(\varepsilon,x,0\right)=:\widehat{U}_{\varepsilon}\left(x,0\right) & =u_{\varepsilon}\left(x\right)+\OO{P_{\varepsilon}\left(x\right)}\\
\widehat{A}\left(\varepsilon,x,0\right)=:\widehat{A}_{\varepsilon}\left(x,0\right) & =\OO{P_{\varepsilon}\left(x\right)},
\end{align*}
with $u\in\frml{\varepsilon}\left[x\right]$ a polynomial (in $x$)
of degree at most $k$ such that $u_{0}(0)\neq0$. In the particular
case when $Z$ is purely convergent the latter power series (and the
coefficients of $u$) are convergent.
\end{prop}

The proof of~\cite{KlimNonLin} is done for $k=1$ but the general
case is similar. It is based first on the following classical lemma,
the proof of which is included for the sake of completeness.
\begin{lem}
\label{lem:weak_separatrix_sn} Let $\widetilde{g}\in\frml{x,y}$
and $\overline{g}\in x^{p}\frml x$ be given with $p\in\zp$, such
that either \emph{$p>0$ or $\widetilde{g}(x,y)=\OO x$. }Let $h\in\frml x$
be such that $h\left(0\right)\neq0$ and define $g(x,y):=\overline{g}(x)+y^{2}\widetilde{g}(x,y)$. 

The differential equation 
\begin{equation}
x^{k+1}f'\left(x\right)+h\left(x\right)f\left(x\right)+g\left(x,f\left(x\right)\right)=0\label{eq:weak_separatrix_sn}
\end{equation}
has a unique formal solution $f$, which moreover belongs to $x^{p}\frml x$.
\end{lem}

\begin{rem}
Note that equation~\eqref{eq:weak_separatrix_sn} is nothing else
than the differential equation determining the center manifold of
the saddle-node vector field 
\[
x^{k+1}\pp x-\left(yh\left(x\right)+g\left(x,y\right)\right)\frac{\partial}{\partial y}
\]
when $g$ and $h$ are holomorphic germs.
\end{rem}

\begin{proof}
Letting $C:=h\left(0\right)\neq0$ and $g(x,0)=\overline{g}(x)=:\sum_{m=p}^{\infty}b^{m}x^{m}$;
then substituting $f(x)=:\sum_{m=p}^{\infty}a^{m}x^{m}$ into~\eqref{eq:weak_separatrix_sn}
and grouping terms of same degree $m\geq p$, we get 
\begin{align*}
Ca^{m}+b^{m}+F^{m}(a^{p},\dots,a^{m-1}) & =0
\end{align*}
for some polynomial $F^{m}$ depending on the $m$-jet of $g$ and
$h$. Hence, we can solve uniquely for each $a^{m}$.
\end{proof}
We then derive Proposition~\ref{pro:weak_separatrices} from the
following technical lemma which we will also use later on.
\begin{lem}
\emph{\label{lem:weak_separatrix_param}(See \cite{KlimNonLin} for
the case $k=1$)} Let $\widetilde{g}_{\varepsilon}\in\frml{\varepsilon,x,y}$
and $\overline{g}_{\varepsilon}\in\frml{\varepsilon,x}P$ be given,
\emph{i.e.} $\overline{g}_{\varepsilon}\left(x\right)=P_{\varepsilon}\left(x\right)\sum_{|\mathbf{m}|\geq0}\overline{g}^{\mathbf{m}}(x)\varepsilon^{\mathbf{m}}$
where each $\overline{g}^{\mathbf{m}}(x)$ is itself a formal power
series in $x$, and let $h_{\varepsilon}\in\frml{\varepsilon,x}$
be such that $h_{0}\left(0\right)\neq0$. Define $g_{\varepsilon}(x,y):=\overline{g}_{\varepsilon}(x)+y^{2}\widetilde{g}_{\varepsilon}(x,y)$. 

The family of differential equations 
\begin{equation}
P_{\varepsilon}(x)f_{\varepsilon}'(x)+h_{\varepsilon}(x)f_{\varepsilon}(x)+g_{\varepsilon}(x,f_{\varepsilon}(x))=0\label{eq:weak_separatrix_param}
\end{equation}
has a unique formal solution $f$, which moreover belongs to $\frml{\varepsilon,x}P$.
\end{lem}

\begin{proof}
Let $g_{\varepsilon}(x,y)=:P_{\varepsilon}(x)\sum_{|\mathbf{m}|\geq0}b^{\mathbf{m},0}(x)\varepsilon^{\mathbf{m}}+\sum_{|\mathbf{m}|\geq0}\left(\sum_{n\geq2}b^{\mathbf{m},n}(x)y^{n}\right)\varepsilon^{\mathbf{m}}$.
Substituting $f_{\varepsilon}\left(x\right)=P_{\varepsilon}\left(x\right)\sum_{|\mathbf{m}|\geq0}a^{\mathbf{m}}(x)\varepsilon^{\mathbf{m}}$
into~\eqref{eq:weak_separatrix_param} and setting $\varepsilon:=0$
we first get 
\[
x^{k+1}\ddd{a^{\mathbf{0}}}x\left(x\right)+\left(h^{\boldsymbol{0}}\left(x\right)+kx^{k}\right)a^{\mathbf{0}}\left(x\right)+x^{-(k+1)}g^{\mathbf{0}}\left(x,x^{k+1}a^{\mathbf{0}}\left(x\right)\right)=0,
\]
which admits a formal solution in $x^{k+1}\frml x$ by direct application
of Lemma~\ref{lem:weak_separatrix_sn} in the case $p=0$. Likewise,
by grouping terms with same $\varepsilon^{\mathbf{m}}$ for $\left|\mathbf{m}\right|\geq1$
we obtain 
\begin{equation}
x^{k+1}\ddd{a^{\boldsymbol{\boldsymbol{m}}}}x(x)+\left(h^{\mathbf{0}}\left(x\right)+kx^{k}+\ell^{\mathbf{m}}(x)\right)a^{\mathbf{m}}(x)+\left(b^{\mathbf{m},0}(x)+F^{\mathbf{m}}\left(x\right)\right)=0,\label{eq:formal_sepx_technical}
\end{equation}
where 
\[
\ell^{\mathbf{m}}(x)=\sum_{n\geq2}nx^{(n-1)(k+1)}a^{\boldsymbol{0}}(x)^{n-1}b^{\mathbf{0},n}(x)=\OO{x^{k+1}},
\]
and $F^{\mathbf{m}}\in\frml x$ is some formal power series depending
polynomially on $\left(a^{\mathbf{n}}\left(x\right)\right)_{\left|\mathbf{n}\right|<\left|\mathbf{m}\right|}$
and on the $\left|\mathbf{m}\right|$-jet of $g$ and $h$. By induction
on $\left|\mathbf{m}\right|$, we recursively find formal solutions
$a^{\mathbf{m}}\in\frml x$, for equation~\eqref{eq:formal_sepx_technical}
has the same type as~\eqref{eq:weak_separatrix_sn} with $\widetilde{g}:=0$,
and hence, has a formal solution given by Lemma~\ref{lem:weak_separatrix_sn}.
Uniqueness is straightforward.
\end{proof}

\subsection{Normalization and cohomological equations}

The tool for proving the Formal Normalization Theorem is the following.
\begin{prop}
\label{prop:cohomological_conjugacy}\cite{Tey-ExSN,Tey-EqH} Let
$W$ and $Y$ be commuting, formal (\emph{resp.} holomorphic) planar
vector fields. Let $F\in\frml{x,y}$ (\emph{resp.} a germ of a holomorphic
function) be given with $F\left(0,0\right)=0$. Then $\Psi:=\flow YF{}$
is a formal (\emph{resp.} analytic) change of variables near $\left(0,0\right)$
and
\begin{align*}
\Psi^{*}W & =W-\frac{W\cdot F}{1+Y\cdot F}Y.
\end{align*}
\end{prop}

This tool is used in the following manner.
\begin{itemize}
\item First if we could find formal solution $T$ of the (parametric families
of) cohomological equations
\begin{align}
X\cdot T & =\frac{1}{U}-\frac{1}{u}\label{eq:cohomological_temporal_normalization}
\end{align}
for a convenient choice of $u\in\germ{\varepsilon,x}^{\times}$, then
$uX$ would be formally conjugate to $Z$ by the tangential change
of variables $\mathcal{T}$ given by 
\begin{align}
\mathcal{T} & :=\flow{uX}T{}.\label{eq:formal_tangential_normalization}
\end{align}
This is the content of the proposition for $Y:=W:=uX$ and $F:=T$.
\item From Proposition~\ref{pro:weak_separatrices} we built the formal,
fibered transform $\mathcal{S}$ given by
\begin{align*}
\mathcal{S}~:~\left(x,y\right) & \longmapsto\left(x,y-\widehat{s}\left(x\right)\right)
\end{align*}
such that $\mathcal{S}^{*}\left(\fonf+\widehat{A}y\pp y\right)=X$.
\item Finally, since $y\pp y$ commutes with the normal form $\fonf$, if
we could solve formally the cohomological equation
\begin{align}
\left(\fonf+\widehat{A}y\pp y\right)\cdot O & =-\widehat{A}\label{eq:cohomological_orbital_normalization}
\end{align}
then $\fonf$ would be formally conjugate to $\fonf+\widehat{A}y\pp y$
by the fibered, transverse change of variables $\mathcal{O}$ given
by 
\begin{align}
\mathcal{O}:=\flow{y\pp y}O{}~:~\left(x,y\right) & \longmapsto\left(x,y\exp O\left(x,y\right)\right).\label{eq:formal_orbital_normalization}
\end{align}
\end{itemize}
We explain below how those formal power series are built and to which
extent they are unique. We consequently obtain a formal conjugacy
$\mathcal{O}\circ\mathcal{S}\circ\mathcal{T}$ between $\fvnf$ and
$Z$ (notice that $u$ is left invariant by the fibered $\mathcal{O}\circ\mathcal{S}$,
so that it also conjugates $\fvnf$ to $uX$). 
\begin{lem}
\label{lem:formal_cohomological} Let $X$ be in the form~\eqref{eq:prepared_form_orbital}
for $A\in\frml{\varepsilon,x,y}$, and take $G\in\frml{\varepsilon,x,y}$.
There exists a formal solution $F\in\frml{\varepsilon,x,y}$ of the
cohomological equation 
\begin{align}
X\cdot F & =G\label{eq:cohomog_eq}
\end{align}
if, and only if, $G\left(x,0\right)$ belongs to the ideal generated
by $P$. In that case $F$ is unique up to the free choice of $F\left(0,0\right)\in\frml{\varepsilon}$. 
\end{lem}

\begin{proof}
Let
\begin{align*}
F\left(x,y\right)=:\sum_{n\geq0}F^{n}\left(x\right)y^{n} & \mbox{ ~~and~~ }G\left(x,y\right)=:\sum_{n\geq0}G^{n}\left(x\right)y^{n}.
\end{align*}
We proceed by induction on $n\geq0$ by identifying coefficients of
powers of $y$ in~\eqref{eq:cohomog_eq}. For each $n\in\zp$ we
must therefore solve
\begin{align}
P\ppp{F^{n}}x+n\left(1+\mu x^{k}\right)F^{n} & =G^{n}+\oo n,\label{eq:formal_cohomog_recurrence}
\end{align}
where $\oo n$ stands for terms containing $F^{m}$ for $m<n$ only,
and are thus already known. 

\begin{itemize}
\item The case $n=0$ outlines the formal obstruction (notice that the choice
of $F^{0}\left(0\right)$ is free).
\item For $n>0$ no additional obstruction appears and $F^{n}$ is uniquely
determined. Then Lemma~\ref{lem:weak_separatrix_param} provides
the unique formal solution of the family of differential equations~\eqref{eq:formal_cohomog_recurrence}. 
\end{itemize}
\end{proof}
We finally derive the Formal Normalization Theorem by writing 
\begin{align*}
U\left(x,0\right) & =u\left(x\right)+\OO{P\left(x\right)}~~~~,~~u\in\germ{\varepsilon}\left[x\right]_{\leq k}^{\times},
\end{align*}
and finding a (unique with $T\left(0,0\right)=0$) formal solution
$T$ of~\eqref{eq:cohomological_temporal_normalization} by Lemma~\ref{lem:formal_cohomological}.
As for the power series $O$, a (unique with $O\left(0,0\right)=0$)
formal solution of~\eqref{eq:cohomological_orbital_normalization}
exists by Proposition~\ref{pro:weak_separatrices}, and Lemma~\ref{lem:formal_cohomological},
for $X$ given in~\eqref{eq:prepared_form_orbital}.
\begin{defn}
\label{def:formal_normalizing_map}Let $Z$ be an unfolding in prepared
form~\eqref{eq:prepared_form}. We write $\mathcal{N}:=\mathcal{O}\circ\mathcal{S}\circ\mathcal{T}$
the canonical \textbf{formal normalization} of $Z$ satisfying $\mathcal{N}^{*}\fvnf=Z$
where $\mathcal{O}$, $\mathcal{S}$ and $\mathcal{T}$ are built
above.
\end{defn}

\subsection{Uniqueness}

Addressing the uniqueness clause in the Formal Normalization Theorem
boils down to studying the case of the normal forms, because of the
canonical choice of normalization maps $\mathcal{N}$ done in Definition~\ref{def:formal_normalizing_map}.
\begin{lem}
\label{lem:formal_model_symmetries}Let $\Psi$ be a formal orbital
symmetry of the formal normal form $\fvnf$ (fixing the canonical
parameter).

\begin{enumerate}
\item There exist unique $F\in\frml{\varepsilon,x,y}$ and $c\in\frml{\varepsilon}^{\times}$
such that 
\begin{align*}
\Psi & =\left(c^{*}\id\right)\circ\flow{{\fvnf}}F{}
\end{align*}
where $c^{*}\id$ is the linear mapping $\left(x,y\right)\mapsto\left(x,cy\right)$.
(The converse statement clearly holds.)
\item $\Psi$ is a symmetry of $\fvnf$ if, and only if, $F\in\frml{\varepsilon}$.
\item $\Psi$ is fibered if, and only if, $F=0$. 
\end{enumerate}
\end{lem}

\begin{proof}
~

\begin{enumerate}
\item By Remark~\ref{rem:symmetry} we want to determine $V\in\frml{\varepsilon,x,y}^{\times}$
such that $\Psi^{*}\fvnf=V\fvnf$. Because $\varepsilon$ is a formal
invariant governing the eigenvalues of (the differential of) the vector
fields at the singularities, $\Psi$ cannot change the eigenvalues,
so that $V\left(x,y\right)=1+\OO{P\left(x\right)}+\OO y$. According
to Lemma~\ref{lem:formal_cohomological} there exists a (unique)
formal solution $F$ with $F\left(0,0\right)=0$ to the cohomological
equation
\begin{align*}
\fonf\cdot F & =\frac{1}{uV}-\frac{1}{u}.
\end{align*}
Therefore $\widehat{\Psi}:=\Psi\circ\left(\flow{\widehat{Z}}F{}\right)^{\circ-1}$
induces a symmetry of $\fvnf$. Write $\widehat{\Psi}~:~\left(\varepsilon,x,y\right)\mapsto\left(\varepsilon,\phi_{\varepsilon}\left(x,y\right),\psi_{\varepsilon}\left(x,y\right)\right)$.
By considering the $\pp x$-component of $\fvnf$ one obtains the
relation
\begin{align*}
\left(uP\right)\circ\phi & =\fvnf\cdot\phi.
\end{align*}
Setting $y:=0$ yields 
\begin{align*}
\left(u_{\varepsilon}P_{\varepsilon}\right)\circ\phi_{\varepsilon}\left(x,0\right) & =u_{\varepsilon}\left(x\right)P_{\varepsilon}\left(x\right)\ppp{\phi_{\varepsilon}}x\left(x,0\right)
\end{align*}
so that
\begin{align*}
\phi_{\varepsilon}\left(x,0\right) & =\flow{u_{\varepsilon}P_{\varepsilon}\pp x}{t_{\varepsilon}}{\left(x\right)}=\flow{{\fvnf}}{t_{\varepsilon}}{\left(x,0\right)}
\end{align*}
for some $t\in\frml{\varepsilon}$. Hence we may assume without loss
of generality that $F_{\varepsilon}\left(0,0\right)=t_{\varepsilon}$
and $\phi_{\varepsilon}\left(x,0\right)=x$. Writing $\phi_{\varepsilon}\left(x,y\right)=x+\sum_{n\geq\nu}\phi_{\varepsilon}^{n}\left(x\right)y^{n}$
with $\nu>0$ we obtain for the term of $y$-degree $\nu$
\begin{align*}
P'\phi^{\nu} & =P\ppp{\phi^{\nu}}x+\nu\left(1+\mu x^{k}\right)\phi^{\nu}
\end{align*}
whose unique formal solution is $\phi^{\nu}=0$, since it is the equation
of the weak separatrix of $P\frac{\partial}{\partial x}+y\left(\nu(1+\mu x^{k})+P'\right)\frac{\partial}{\partial y}$.
As a matter of consequence $\phi_{\varepsilon}\left(x,y\right)=x$
and $\widehat{\Psi}$ is fibered. Lastly, by considering the $\pp y$-component
of $\fvnf$ one obtains the relation
\begin{align*}
\left(1+\mu x^{k}\right)\psi & =\fonf\cdot\psi.
\end{align*}
Setting $y:=0$ yields 
\begin{align*}
\psi_{\varepsilon}\left(x,0\right) & =0
\end{align*}
 so that $\psi_{\varepsilon}\left(x,y\right)=y\exp N_{\varepsilon}\left(x,y\right)$
for some $N\in\frml{\varepsilon,x,y}$. The corresponding cohomological
equation reads
\begin{align*}
0 & =\fonf\cdot N
\end{align*}
and only admits $N\in\frml{\varepsilon}$ as formal solution (uniqueness
clause of Lemma~\ref{lem:formal_cohomological}). We then set $c:=\exp N$.
\item and (3) are clear from the previous arguments.
\end{enumerate}
\end{proof}
We derive the following precise statement. Item~(2) plays an essential
role in proving the (analytic) Uniqueness Theorem.
\begin{cor}
\label{cor:formal_symmetries}Consider two unfoldings $Z$ and $\widetilde{Z}$
in prepared form~\eqref{eq:prepared_form}. 

\begin{enumerate}
\item Let $\Psi$ be a formal conjugacy between $Z$ and $\widetilde{Z}$
(fixing the canonical parameter), namely $\Psi^{*}Z=\widetilde{Z}$.
Let $\mathcal{N}=\mathcal{O}\circ\mathcal{T}\circ\mathcal{S}$ and
$\widetilde{\mathcal{N}}=\widetilde{\mathcal{O}}\circ\widetilde{\mathcal{T}}\circ\widetilde{\mathcal{S}}$
be the respective canonical tangent-to-identity formal normalizations
as in Definition~\ref{def:formal_normalizing_map}. 

\begin{enumerate}
\item There exists unique $c\in\frml{\varepsilon}^{\times}$ and $t\in\frml{\varepsilon}$
such that 
\begin{align*}
\Psi & =\mathcal{N}^{\circ-1}\circ\left(c^{*}\id\right)\circ\widetilde{\mathcal{N}}\circ\flow{\widetilde{Z}}t{}.
\end{align*}
(The converse statement clearly holds.)
\item If $\Psi$ is analytic then so are $t$ and $c$. (The converse statement
does not generally hold.)
\end{enumerate}
\item If $Z$ and $\widetilde{Z}$ are analytically orbitally equivalent
(by an orbital equivalence fixing the canonical parameter) then there
exists a fibered analytic orbital equivalence between them 
\begin{align*}
 & \mathcal{S}^{\circ-1}\circ\mathcal{O}^{\circ-1}\circ\left(c^{*}\id\right)\circ\widetilde{\mathcal{O}}\circ\widetilde{\mathcal{S}}
\end{align*}
for some $c\in\germ{\varepsilon}$. 
\end{enumerate}
\end{cor}

\begin{rem}
The partial conclusion ``there exists a fibered orbital equivalence''
in Claim~(2) was proved in~\cite[Lemma 3.4]{RouTey} by unfolding
the homotopy technique of~\cite[Lemma 2.2.2]{MaRa-SN}. We give here
an alternate proof. In the other part of the conclusion, pay attention
that $\mathcal{O}\circ\mathcal{S}$ and $\widetilde{\mathcal{O}}\circ\widetilde{\mathcal{S}}$
are only formal power series, but the composition is a convergent
power series. 
\end{rem}

\begin{proof}
~

\begin{enumerate}
\item ~

\begin{enumerate}
\item follows from Lemma~\ref{lem:formal_model_symmetries}: the formal
map $\mathcal{N}\circ\Psi\circ\widetilde{\mathcal{N}}^{\circ-1}$
is a symmetry of the normal form $\fvnf$, and $\widetilde{\mathcal{N}}$
is a formal conjugacy between $\widetilde{Z}$ and $\fvnf$, hence
conjugating their flow (as formal power series):
\begin{align*}
\flow{{\fvnf}}t{}\circ\widetilde{\mathcal{N}} & =\widetilde{\mathcal{N}}\circ\flow{\widetilde{Z}}t{}.
\end{align*}
\item Here we assume that $\Psi$ is analytic. Following~(a) we have
\begin{align*}
\Psi & =\mathcal{T}^{\circ-1}\circ\left(\mathcal{S}^{\circ-1}\circ\mathcal{O}^{\circ-1}\circ\left(c^{*}\id\right)\circ\widetilde{\mathcal{O}}\circ\widetilde{\mathcal{S}}\circ\flow{u\widetilde{X}}t{}\right)\circ\widetilde{\mathcal{T}}.
\end{align*}
Using both facts that $\widehat{\Psi}:=\mathcal{S}^{\circ-1}\circ\mathcal{O}^{\circ-1}\circ\left(c^{*}\id\right)\circ\widetilde{\mathcal{O}}\circ\widetilde{\mathcal{S}}$
is fibered, and that 
\begin{align*}
\flow{u\widetilde{X}}t{\left(x,y\right)} & =\left(\flow{uP\pp x}t{\left(x\right)},~y\left(\exp t+y\phi\left(x,y,t\right)\right)\right),
\end{align*}
we first derive 
\begin{align*}
\Psi & =\mathcal{T}^{\circ-1}\circ\left(\flow{uP\pp x}t{},~\psi\right)\circ\widetilde{\mathcal{T}}.
\end{align*}
Because $T\left(0,0\right)=\widetilde{T}\left(0,0\right)=0$, we have
\begin{align*}
\Psi_{\varepsilon}\left(0,0\right) & =\left(t_{\varepsilon}\varepsilon_{0},~\cdots\right)
\end{align*}
from which we deduce the convergence of $t$. We also have the identity
\begin{align*}
\ppp{\psi}y\left(0,0\right) & =c\exp t,
\end{align*}
from which the convergence of $c$ follows also.
\end{enumerate}
\item It is sufficient to assume that $Z:=X$ is analytically conjugate
by some $\Psi$ (fixing the canonical parameter) to $\widetilde{Z}:=\widetilde{U}\widetilde{X}$
for some $\widetilde{U}\in\germ{\varepsilon,x,y}^{\times}$. In that
setting we have $u=1$ and $\mathcal{T}=\id$, so that according to~(1)
\begin{align*}
\Psi & =\mathcal{S}^{\circ-1}\circ\mathcal{O}^{\circ-1}\circ\left(c^{*}\id\right)\circ\widetilde{\mathcal{O}}\circ\widetilde{\mathcal{S}}\circ\widetilde{\mathcal{T}}\circ\flow{\widetilde{Z}}t{},
\end{align*}
where $t\in\germ{\varepsilon}$. As a matter of consequence the mapping
$\flow{\widetilde{Z}}t{}$ is analytic, and so is $\widehat{\Psi}:=\Psi\circ\left(\flow{\widetilde{X}}t{}\right)^{\circ-1}$.
Because $\widehat{\Psi}\circ\widetilde{\mathcal{T}}^{\circ-1}$ is
fibered, the $x$-component of $\widehat{\Psi}$ (which is analytic)
is equal to the $x$-component of $\widetilde{\mathcal{T}}$. The
former is of the form $\left(x,y\right)\mapsto A\left(x,y,\widetilde{T}\left(x,y\right)\right)$
for some holomorphic function $A\in\germ{\varepsilon,x,y,t}$ with
$\ppp At\neq0$, and where $\widetilde{T}$ is the solution of~\eqref{eq:cohomological_temporal_normalization}
for $U:=\widetilde{U}$. Thus $\widetilde{T}$ is a convergent power
series, and so is $\widehat{\Psi}\circ\widetilde{\mathcal{T}}^{\circ-1}$.
\end{enumerate}
\end{proof}

\section{\label{sec:Geometric}Geometric orbital normalization}

Here we prove the orbital part of the Normalization and Uniqueness
Theorems for $\tau=0$. Sections~\ref{subsec:Parametric-realization}\textendash \ref{subsec:Normal-form-recognition}
are devoted to the construction of the normal form conjugacy, while
its uniqueness is thoroughly studied in Section~\ref{subsec:Uniqueness_orbital}.
Before going into the details we start with a brief description of
the general strategy. Let us call $\disc$ the unit disk.

\bigskip{}

For fixed 
\begin{align*}
0 & <\frac{1}{\rho^{\infty}}<\rho^{0}
\end{align*}
we introduce two analytic charts:
\begin{itemize}
\item the original coordinates 
\begin{align*}
\left(x,y\right) & \in\mathcal{U}^{0}:=\rho^{0}\ww D\times\neigh,
\end{align*}
\item the coordinates at infinity
\begin{align*}
\left(u,v\right) & \in\mathcal{U}^{\infty}:=\rho^{\infty}\ww D\times\neigh
\end{align*}
with (involutive) standard transition map on $\cc^{\times}\times\neigh$
\begin{align*}
\left(u,v\right) & \longmapsto\left(\frac{1}{u},~v\right)=\left(x,y\right).
\end{align*}
\end{itemize}
For convenience we write $O^{0}$ and $O^{\infty}$ the respective
expression of a holomorphic object $O$ in the charts $\mathcal{U}^{0}$
and $\mathcal{U}^{\infty}$ respectively.

\bigskip{}

Start from an arbitrary $X^{0}\in\conv$ in prepared form~\eqref{eq:prepared_form_orbital},
with $A\in y^{2}\germ{\varepsilon,x,y}$ holomorphic and bounded on
$\mathcal{U}^{0}$, and such that $\mu_{0}\notin\rr_{\leq0}$. It
is always possible to make this assumption thanks to Theorem~\ref{thm:preparation},
since $a=0$ in that case. Notice in particular that $A$ is bounded
since we can always take a smaller $\rho^{0}$ and decrease similarly
the size of the neighborhood of $\left\{ y=0\right\} $: hence $\mathcal{U}^{0}$
can be taken inside a compact set on which $A$ is defined.

In the following we assume that $\varepsilon$ is so small that the
$k+1$ singularities of $X_{\varepsilon}^{0}$ lie in $\left\{ 0\leq\left|x\right|<\nf 1{\rho^{\infty}}\right\} \times\left\{ 0\right\} $.
The following steps constitute what we refer to as the \textbf{unfolded
Loray construction}.
\begin{lyxlist}{00.00.0000}
\item [{\textbf{Gluing}}] We find a vector field family $X^{\infty}$ on
$\mathcal{U}^{\infty}$ whose holonomy over $\rho^{\infty}\sone\times\left\{ 0\right\} $
is the inverse of $\holo{\varepsilon}$, the corresponding ``weak''
holonomy of $X^{0}$ (Section~\ref{subsec:Parametric-realization}).
Therefore foliations induced by each vector field can be glued one
to the other over the annulus $\mathcal{U}^{0}\cap\mathcal{U}^{\infty}$
by an identification of the form
\begin{align*}
\left(u,v\right) & =\left(\frac{1}{x},~y\exp\phi\left(x,y\right)\right)
\end{align*}
(Section~\ref{subsec:Gluing}). This operation results in a family
of foliated abstract complex surfaces $\left(\mathcal{M},\fol{}\right)$.
\item [{\textbf{Normalizing}}] We construct a fibered biholomorphic equivalence
between $\mathcal{M}$ and a standard neighborhood of $\left\{ y=0\right\} \simeq\proj$,
that is a complex surface with charts $\mathcal{U}^{0}$, $\mathcal{U}^{\infty}$
and transition map \emph{exactly} $\left(u,v\right)=\left(\frac{1}{x},y\right)$
(Section~\ref{subsec:Normalizing_Savelev}). Because $\proj\times\left\{ 0\right\} $
is compact the expression of the new $X^{0}$ is polynomial in $x$
with controlled degree, thus in orbital normal form~\eqref{eq:orbital_normal_form}
as expected by the Normalization Theorem (Section~\ref{subsec:Normal-form-recognition}).
\item [{\textbf{Uniqueness}}] From the special form of the normalized vector
field, it can be seen that the closure of the saturation of any small
neighborhood of $\left(0,0\right)$ contains a whole $\proj\times r\ww D$.
Therefore any local conjugacy between normal forms (which we choose
fibered thanks to Corollary~\ref{cor:formal_symmetries}~(2)) can
be analytically continued by a construction \emph{à la} Mattei-Moussu
on $\proj\times r\ww D$. But this manifold has very few fibered automorphisms,
allowing to conclude (Section~\ref{subsec:Uniqueness_orbital}). 
\end{lyxlist}
In the unfolded Loray construction, only what happens in the first
chart $\left(x,y\right)$ is of a different nature than when $\varepsilon=0$.
As seen from the other chart $\left(u,v\right)$, the only important
ingredient for the construction is the ``weak'' holonomy $\holo{\varepsilon}$
of the unfolding (see Section~\ref{subsec:Weak-holonomy} below).
Hence the original arguments do not need to be unfolded near $\left(\infty,0\right)$,
although we must take care that everything remains holomorphic in
the parameter. The first two steps of the unfolded Loray construction
require external results that need to be parametrically controlled:
\begin{enumerate}
\item the realization of the weak holonomy $\holo{}$ by a foliation near
$\left(\infty,0\right)$, obtained by the construction of~\cite{SchaTey};
\item the normalization of the transition map between the charts $\left(x,y\right)$
and $\left(u,v\right)$ on the annulus
\begin{align*}
\mathcal{A} & :=\left\{ \nf 1{\rho^{0}}<\left|u\right|<\rho^{\infty}\right\} \times\neigh,
\end{align*}
as done in~\cite{Save}.
\end{enumerate}
Both proofs are similar in spirit and only rely on complex (holomorphic)
analysis and (what amounts to) a fixed-point method. Parametric holomorphy
follows from the explicit integral formulas. Because normalizing transition
maps is relatively easy, we prove a parametric version of Savelev's
theorem in Section~\ref{subsec:Normalizing_Savelev}. It contains
the main steps and ideas upon which are based the respective proofs
of the Normalization Theorem for vector fields (Section~\ref{sec:Temporal})
and of the Realization Theorem (Section~\ref{sec:Analytic}). The
latter is nothing but an unfolded version of the main result of~\cite{SchaTey},
retrospectively making the present article more self-contained.

\subsection{\label{subsec:Weak-holonomy}Weak holonomy}

We name
\begin{align*}
\Pi~:~\left(x,y\right) & \longmapsto x
\end{align*}
the projection on the invariant line $\left\{ y=0\right\} $ and let
\begin{align*}
\Sigma & \subset\Pi^{-1}\left(x_{*}\right)
\end{align*}
be a germ of a transverse disc endowed with the coordinate $y\in\neigh$.
Starting from $y\in\Sigma$ there exists a unique path 
\begin{align*}
\gamma_{y}~:~\left[0,1\right] & \longto\mathcal{U}^{0}\\
\gamma_{y}\left(0\right) & =\left(x_{*},y\right)
\end{align*}
tangent to $X_{\varepsilon}^{0}$ such that 
\begin{align*}
\gamma:=\Pi\circ\gamma_{y}~=~s & \longmapsto x_{*}\exp2\ii\pi s.
\end{align*}
We define $\holo{\varepsilon}\left(y\right)$ as the $y$-component
of the final value $\gamma_{y}\left(1\right)$. The \textbf{weak holonomy
mapping} $\holo{\varepsilon}$ as described is a germ of a biholomorphism
near the fixed-point $0$ whose linear part is governed by the formal
orbital invariant $\mu$ in the following way:
\begin{align*}
\holo{\varepsilon}\left(y\right) & =y\exp2\ii\pi\mu_{\varepsilon}+\oo y\in\diff[\Sigma,0].
\end{align*}
The analytic dependence of trajectories of $X_{\varepsilon}^{0}$
on the parameter $\varepsilon$ ensures that $\holo{}\in\germ{\varepsilon,y}$. 

\subsection{\label{subsec:Parametric-realization}Parametric holonomy realization
at $\left(\infty,0\right)$}
\begin{thm}
\cite[Main Theorem and Section 4.4]{SchaTey}\label{thm:realization_holonomy}
Let $\left(\Delta_{\eta}\right)_{\eta\in\neigh[n]}$ be an analytic
family of elements of $\diff[\cc,0]$, that is $\left(\eta,v\right)\mapsto\Delta_{\eta}\left(v\right)\in\germ{\eta,v}$
and $\Delta_{0}'\left(0\right)\neq0$. Let $\lambda\in\germ{\eta}$
be given such that $\Delta_{\eta}'\left(0\right)=\exp\left(-2\ii\pi\lambda_{\eta}\right)$
and $\lambda_{0}\notin\rr_{\leq0}$. There exists an analytic family
of vector fields $\left(X_{\eta}^{\infty}\right)_{\eta\in\neigh[n]}$
of the form
\begin{align}
X_{\eta}^{\infty}\left(u,v\right) & =-u\pp u+v\left(\lambda_{\eta}+u\left(1+f_{\eta}\left(v\right)\right)+g_{\eta}\left(v\right)\right)\pp v~~~~,~f,~g\in v\germ{\eta,v},\label{eq:realization_holonomy_normal_form}
\end{align}
holomorphic on the domain $\mathcal{U}^{\infty}$ and satisfying for
all $\eta\in\neigh[n]$:

\begin{enumerate}
\item $\left(0,0\right)$ is the only singularity of $X_{\eta}^{\infty}$
in $\mathcal{U}^{\infty}$,
\item the holonomy of $X_{\eta}^{\infty}$ above the circle $w_{*}\sone\times\left\{ 0\right\} $,
computed on a germ of transverse disc $\left\{ u=u_{*}\right\} $
with respect to the projection $\left(u,v\right)\mapsto u$, is exactly
$\Delta_{\eta}$. 
\end{enumerate}
\end{thm}

\begin{rem}
The result of~\cite{SchaTey} asserts the existence of a vector field
of the form~\eqref{eq:realization_holonomy_normal_form} with $f:=0$
whose holonomy on $\Sigma$ is conjugate to $\Delta$ by some analytic
family $\Psi$. The conjugacy $\left(u,v\right)\mapsto\left(u,\Psi\left(v\right)\right)$
transforms the vector field into the form~\eqref{eq:realization_holonomy_normal_form}
for different $f,~g$ but its holonomy is exactly $\Delta$ on $\Sigma$.
\end{rem}

In the generic case $\lambda_{0}\notin\rr$ the theorem is (almost)
trivial. All holonomy maps 
\begin{align*}
\Delta_{\eta}~:~v\longmapsto & v\exp\left(-2\ii\pi\lambda_{\eta}+\delta_{\eta}\left(v\right)\right)~~~,~\delta_{\eta}\left(0\right)=0,
\end{align*}
are hyperbolic and locally analytically linearizable for that matter
(Koenig's theorem), the unique tangent-to-identity linearization being
given by $\Psi_{\eta}~:~v\mapsto v\exp\psi_{\eta}\left(v\right)$,
where
\begin{align*}
\psi_{\eta} & :=\sum_{n=0}^{\infty}\delta_{\eta}\circ\Delta_{\eta}^{\circ n}.
\end{align*}
Local uniform convergence ensures that $\psi$ is analytic in both
$t$ and $\eta$. The fibered mapping $\left(u,v\right)\mapsto\left(u,\Psi_{\eta}\left(v\right)\right)$
transforms the linear vector field $-u\pp u+\lambda v\pp v$ into
a vector field $X_{\eta}^{\infty}$ fulfilling the conclusions (1)-(2)
of the theorem (but not of the form~\eqref{eq:realization_holonomy_normal_form}).
However if $\lambda_{0}\in\rr$ this construction fails: the linearization
domain may shrink to a point (if $\Delta_{0}$ is not analytically
linearizable). The form~\eqref{eq:realization_holonomy_normal_form}
has the advantage of being valid for all cases, analytically in the
parameter. Notice that the presence of the term $-uv\pp v$ in~\eqref{eq:realization_holonomy_normal_form}
discards any linear realization even when $\lambda_{0}\notin\rr$.

\bigskip{}

We define $\eta:=\varepsilon\in\neigh[k]$,
\begin{align*}
\lambda_{\varepsilon} & :=\mu_{\varepsilon}\notin\rr_{\leq0},\\
\Delta_{\varepsilon}~:~v & \longmapsto\holo{\varepsilon}^{\circ-1}\left(v\right),
\end{align*}
and apply Theorem~\ref{thm:realization_holonomy}, to obtain an analytic
family $X^{\infty}$ in the chart $\left(u,v\right)$. In order to
stitch the induced foliation with that of $X_{\varepsilon}^{0}$ we
need to prepare it by changing slightly the coordinates on $\mathcal{U}^{\infty}$.
Let $\widetilde{X}_{\varepsilon}^{0}$ be the vector field corresponding
to $X_{\varepsilon}^{0}$ in the coordinates $\left(u,v\right)=\left(\frac{1}{x},y\right)$,
that is
\begin{align*}
\widetilde{X}_{\varepsilon}^{0}\left(u,v\right) & =-u^{2}P_{\varepsilon}\left(\frac{1}{u}\right)\pp u+v\left(\mu_{\varepsilon}u^{-k}+1+\OO v\right)\pp v\\
 & =uP_{\varepsilon}\left(\frac{1}{u}\right)\times\left(-u\pp u+v\left(\lambda_{\varepsilon}+h_{\varepsilon}\left(u\right)+\OO v\right)\pp v\right),
\end{align*}
where
\begin{align}
h_{\varepsilon}~:~u & \longmapsto\frac{u^{k}+\mu_{\varepsilon}}{u^{k+1}P_{\varepsilon}\left(\frac{1}{u}\right)}-\mu_{\varepsilon}~~~~\in\holf[\left(\cc^{k},0\right)\times\rho^{\infty}\ww D]\label{eq:divisor_orbital_change}
\end{align}
vanishes at $0$. Notice indeed that the polynomial $u^{k+1}P_{\varepsilon}\left(\frac{1}{u}\right)\in\pol u_{\leq k+1}$
has its roots outside the closed disc $\adh{\rho^{\infty}\ww D}$,
whereas it takes the value $1$ at $0$. Remark also that the quantity
$uP_{\varepsilon}\left(\frac{1}{u}\right)$ needs to be factored out
in order to recognize a vector field alike to $X_{\varepsilon}^{\infty}$
near $\left(\infty,0\right)$. This function is non-vanishing on the
annulus $\mathcal{A}$. Let $\widetilde{X}_{\varepsilon}^{\infty}$
be the vector field corresponding to $X_{\varepsilon}^{\infty}$ through
the inverse transform 
\begin{align}
\left(u,v\right) & \longmapsto\left(u,~v\exp\int_{u_{*}}^{u}\left(h_{\varepsilon}\left(z\right)-z\right)\frac{\dd z}{z}\right).\label{eq:divisor_orbital_straightenning}
\end{align}
By construction we have
\begin{align*}
\widetilde{X}_{\varepsilon}^{\infty}\left(u,v\right) & =-u\pp u+v\left(\lambda_{\varepsilon}+h_{\varepsilon}\left(u\right)+\OO v\right)\pp v,
\end{align*}
which glues with $\widetilde{X}_{\varepsilon}^{0}$ through $(u,v)=\left(\frac{1}{x},y\right)$
as presented in the next paragraph.

\subsection{\label{subsec:Gluing}Gluing}

Both transformed vector fields $\widetilde{X}^{0}$ and $\widetilde{X}^{\infty}$
built in the previous section have same holonomy $\Delta_{\varepsilon}$
on $\Sigma$. We glue the (foliations defined by the) vector fields
$\widetilde{X}_{\varepsilon}^{0}$ and $\widetilde{X}_{\varepsilon}^{\infty}$
over the fibered annulus $\mathcal{A}$ through a fibered map $\Phi_{\varepsilon}$
fixing $\Sigma$ and (classically) obtained by foliated path-lifting,
as we explain now. For $\left(u,v\right)\in\mathcal{A}$ we join $u_{*}$
to $u$ in $\mathcal{A}\cap\left\{ v=0\right\} $ by some path $\gamma$
and define
\begin{align*}
\Phi_{\varepsilon}\left(u,v\right) & :=\left(u,\holo{\varepsilon,\gamma}^{\infty}\circ\left(\holo{\varepsilon,\gamma}^{0}\right)^{\circ-1}\left(v\right)\right),
\end{align*}
where $\holo{\varepsilon,\gamma}^{0}$ (\emph{resp.} $\holo{\varepsilon,\gamma}^{\infty}$)
is the holonomy map obtained by lifting the path $\gamma$ through
$\Pi$ in the foliation induced by $\widetilde{X}_{\varepsilon}^{0}$
(\emph{resp.} $\widetilde{X}_{\varepsilon}^{\infty}$). The map $\Phi_{\varepsilon}$
is well-defined because when $\gamma$ is a loop both mappings $\holo{\varepsilon,\gamma}^{\infty}$
and $\holo{\varepsilon,\gamma}^{0}$ coincide with the same corresponding
iterate of $\Delta_{\varepsilon}$. Clearly $\Phi_{\varepsilon}$
depends analytically on $\varepsilon\in\neigh[k]$ and is a germ of
a fibered biholomorphism near $\mathcal{A}\cap\left\{ v=0\right\} $
satisfying
\begin{align*}
\Phi^{*}\widetilde{X}^{0} & =uP\left(\frac{1}{u}\right)\widetilde{X}^{\infty},\\
\Phi\left(u,v\right) & =\left(u,~v\exp\phi\left(u,v\right)\right),\\
\phi\left(u,0\right) & =\phi\left(u_{*},v\right)=0.
\end{align*}

\subsection{\label{subsec:Normalizing_Savelev}Normalizing}

So far the construction yields an analytic family of complex foliated
surfaces, written $\left(\mathcal{M},\fol{}\right)$, defined by the
charts $\left(\mathcal{U}^{0},\fol{}^{0}\right)$ and $\left(\mathcal{U}^{\infty},\fol{}^{\infty}\right)$
with transition map 
\begin{align}
\left(u,v\right) & \longmapsto\left(\frac{1}{u},~v\exp\phi\left(u,v\right)\right)=\left(x,y\right).\label{eq:transition_mapping}
\end{align}

\begin{rem}
The foliation $\fol{\varepsilon}$ is transverse to the fibers of
the global fibration by discs $\Pi~:~\mathcal{M}_{\varepsilon}\to\mathcal{L}$
given in the first chart by $\left(x,y\right)\mapsto x$, except along
the $k+2$ invariant discs $\left\{ P_{\varepsilon}\left(x\right)=0\right\} $
and $\left\{ x=\infty\right\} $.
\end{rem}

Each manifold $\mathcal{M}_{\varepsilon}$ is a neighborhood of the
invariant divisor $\lif{}\simeq\proj$, corresponding to $\left\{ y=0\right\} $
and $\left\{ v=0\right\} $ in the respective chart, while the natural
embedding $\proj[1]\inj\mathcal{M}$ has self-intersection $0$ according
to Camacho-Sad index formula~\cite{CamaSad} (the singularities near
$\left(0,0\right)$ contribute for a sum of Camacho-Sad indices equal
to $\mu_{\varepsilon}$ while the singularity at $\left(\infty,0\right)$
does for $-\lambda_{\varepsilon}=-\mu_{\varepsilon}$). 
\begin{defn}
\label{def:standard_sphere}~For $r>0$ we define the \textbf{standard
neighborhood of radius $r$ of the Riemann sphere} 
\begin{align*}
\hirze & :=\proj\times r\ww D,
\end{align*}
the complex surface equipped with the (global) affine coordinates
\begin{align*}
\left(u,v\right) & \in\cc\times r\ww D
\end{align*}
and transition map on $\cc^{\times}\times r\ww D$ given by $\left(u,v\right)=\left(\frac{1}{x},y\right)$,
\emph{i.e.} by~\eqref{eq:transition_mapping} with $\phi:=0$. The
other chart of $\hirze$ is the domain $\left(x,y\right)\in\cc\times r\ww D$.
When speaking of a \textbf{standard neighborhood of the sphere} we
actually refer to $\hirze$ for some $r>0$ small enough.
\end{defn}

\begin{thm}
\label{thm:fibred_gluing_parametric} Let $\mathcal{M}$ be an analytic
family of complex surfaces with transition maps~\eqref{eq:transition_mapping}.
There exists a standard neighborhood $\mathcal{V}=\hirze$ of $\mathcal{L}$,
for some $r>0$, and an analytic family of fibered holomorphic injective
mappings 
\begin{align*}
\Psi~:~\mathcal{V} & \longto\mathcal{M}
\end{align*}
agreeing with the identity on $\mathcal{L}$. 
\end{thm}

The rest of the subsection is devoted to the proof of this theorem.
We refer to Section~\eqref{subsec:Functional-spaces} for the definitions
of the functional spaces in use. We are looking for $\Psi$, or rather
its expression in the charts $\mathcal{U}^{0}$ and $\mathcal{U}^{\infty}$,
in the form
\begin{align*}
\Psi^{0}\left(x,y\right) & =\left(x,~y\exp\psi^{0}\left(x,y\right)\right)\\
\Psi^{\infty}\left(u,v\right) & =\left(u,~v\exp\psi^{\infty}\left(u,v\right)\right).
\end{align*}
The normalization equation becomes a non-linear additive Cousin problem
on $\mathcal{A}$:
\begin{align*}
\psi^{0}\left(\frac{1}{u},v\right)-\psi^{\infty}\left(u,v\right) & =\phi\circ\Psi^{\infty}\left(u,v\right).
\end{align*}

Starting from $\psi_{0}^{0}:=0$ and $\psi_{0}^{\infty}:=0$ we build
iteratively two bounded sequences of holomorphic functions 
\begin{align*}
\psi_{n}^{\sharp}\in\holb[\group{{{\neigh[k]}}\times\rho^{\sharp}\ww D\times r\ww D}] & ,~\sharp\in\left\{ 0,\infty\right\} 
\end{align*}
solution of the linearized additive Cousin problem (or discrete cohomological
equation)
\begin{align}
\psi_{n+1}^{0}\left(\frac{1}{u},v\right)-\psi_{n+1}^{\infty}\left(u,v\right) & =\phi\left(u,v\exp\psi_{n}^{\infty}\left(u,v\right)\right).\label{eq:savelev_cousin_recurrence}
\end{align}
The Cousin problem has explicit solutions given by a Cauchy-Heine
transform. From these solutions we obtain \emph{a priori} bounds on
the norm of $\psi_{n}^{\sharp}$, allowing to fix the radius $r>0$
beforehand. We let
\begin{align*}
\mathcal{U}_{r}^{0} & :=\left\{ \left(x,y\right)~:~\left|x\right|<\rho^{0},~\left|y\right|<r\right\} ,\\
\mathcal{U}_{r}^{\infty} & :=\left\{ \left(u,v\right)~:~\left|u\right|<\rho^{\infty},~\left|v\right|<r\right\} ,
\end{align*}
 be an atlas for $\hirze$. We postpone the proof of the next main
lemma to the end of the section.

\begin{figure}
\hfill{}\subfloat[For $F^{0}$ in $\left(x,y\right)$ coordinates.]{\includegraphics[width=5cm]{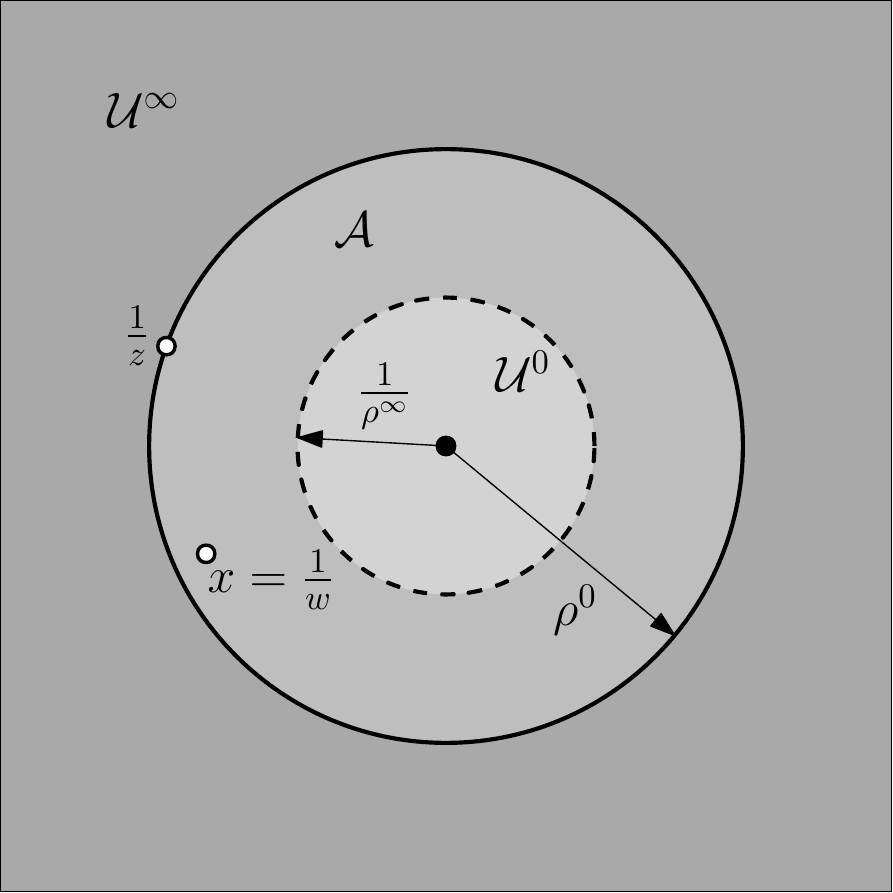}

}\hfill{}\subfloat[For $F^{\infty}$ in $\left(u,v\right)$ coordinates.]{\includegraphics[width=5cm]{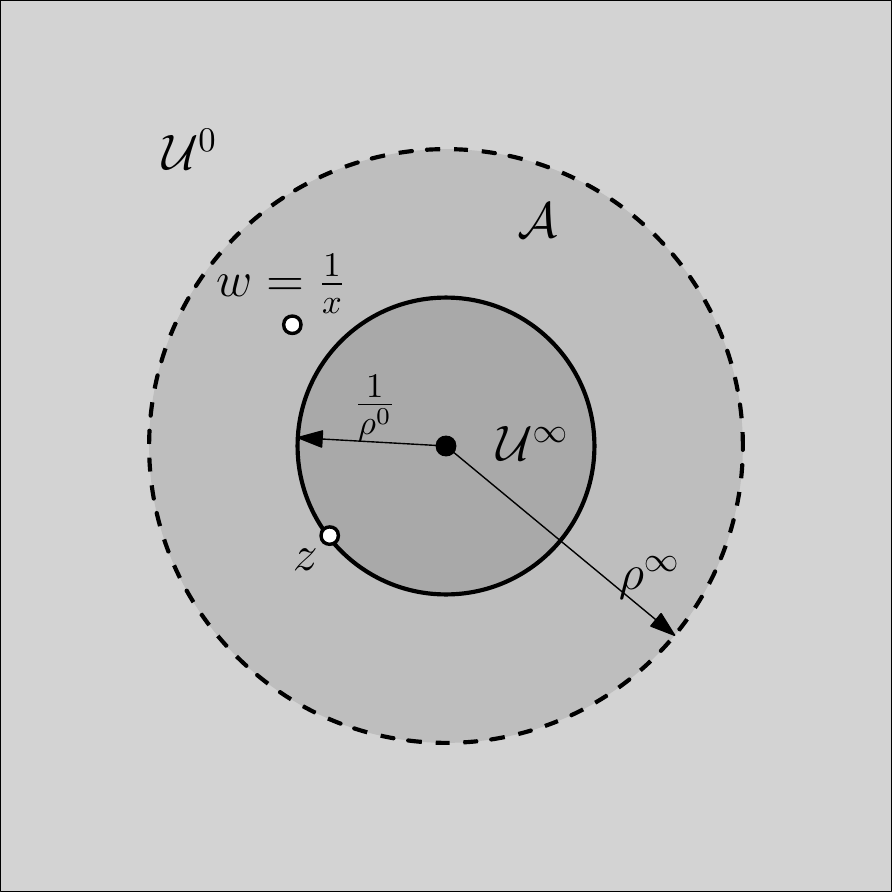}

}\hfill{}

\caption{\label{fig:Cauchy-Heine_Savelev}Integration contours in the respective
charts.}
\end{figure}
\begin{lem}
\label{lem:Savelev-Cousin}Assume that $\phi\in\holb[\mathcal{A}_{\eta}]'$
for some domain $\mathcal{A}_{\eta}:=\left\{ \frac{1}{\rho^{0}}<\left|u\right|<\rho^{\infty}\right\} \times\eta\ww D$.
Let $\psi\in\holb[\mathcal{U}_{r}^{\infty}]$ be such that the image
of $\mathcal{A}_{r}$ by $\left(u,v\right)\mapsto\left(u,v\exp\psi\left(u,v\right)\right)$
is included in $\mathcal{A}_{\eta}$. Define 
\begin{align}
\begin{cases}
F^{\infty}\left(u,v\right) & :=\frac{1}{2\ii\pi}\oint_{\rho^{\infty}\sone}\phi\left(z,v\exp\psi\left(z,v\right)\right)\frac{\dd z}{z-u},\\
F^{0}\left(x,y\right) & :=\frac{x}{2\ii\pi}\oint_{\frac{1}{\rho^{0}}\sone}\phi\left(z,y\exp\psi\left(z,y\right)\right)\frac{\dd z}{xz-1}.
\end{cases}\label{eq:Savelev_Cauchy-Heine}
\end{align}
Then the following properties hold.

\begin{enumerate}
\item $F^{0}\in\holb[\mathcal{U}_{r}^{0}]$ and $F^{\infty}\in\holb[\mathcal{U}_{r}^{\infty}]$.
Moreover for $\sharp\in\left\{ 0,\infty\right\} $
\begin{align}
\norm[F^{\sharp}]{\mathcal{U}_{r}^{\sharp}} & \leq rK\norm[\phi]{\mathcal{A}_{\eta}}'\exp\norm[\psi]{\mathcal{U}_{r}^{\infty}}\label{eq:estim_F}
\end{align}
where 
\begin{align*}
K:= & \left(1+\frac{2\rho^{0}}{\rho^{\infty}\rho^{0}-1}\right).
\end{align*}
\item For all $\left(u,v\right)\in\mathcal{A}_{r}$ we have
\begin{align}
F^{0}\left(\frac{1}{u},v\right)-F^{\infty}\left(u,v\right) & =\phi\left(u,v\exp\psi\left(u,v\right)\right).\label{eq:savelev_cousin}
\end{align}
\item These are the only holomorphic solutions of the previous equation
which are bounded, up to the addition of a function $v\mapsto f^{\infty}\left(v\right)$
with $f^{\infty}\in\holb[r\ww D]$.
\end{enumerate}
\end{lem}

\begin{rem}
The integral formula~\eqref{eq:Savelev_Cauchy-Heine} shows right
away that $F^{\sharp}$ depends holomorphically on any extraneous
parameter on which $\phi$ were to depend holomorphically. 
\end{rem}

It is straightforward to check that fixing some
\begin{align*}
0<r & \leq\eta\exp\left(-\eta K\norm[\phi]{{{\neigh[k]}}\times\eta\ww D}'\right)
\end{align*}
inductively produces well-defined sequences $\left(\psi_{n}^{\sharp}\right)_{n\in\nn}$
of $\holb[\mathcal{U}_{r}^{\sharp}]$, for we have the implication 

\[
\norm[\psi_{\varepsilon,n}^{\infty}]{\mathcal{U}_{r}^{\infty}}<\eta K\norm[\phi_{\varepsilon}]{\eta\ww D}'\Longrightarrow\left|v\exp\psi_{\varepsilon,n}^{\infty}\left(u,v\right)\right|<r\left|\exp\psi_{\varepsilon,n}^{\infty}\left(u,v\right)\right|<r\exp\left(\eta K\norm[\phi_{\varepsilon}]{\eta\ww D}'\right)\leq\eta
\]
for all $\left(u,v\right)\in\mathcal{A}_{r}$. Using \eqref{eq:estim_F}
with $\psi:=\psi_{\varepsilon,n}^{\infty}$ finally yields

\[
\norm[\psi_{\varepsilon,n+1}^{\infty}]{\mathcal{U}_{r}^{\infty}}<\eta K\norm[\phi_{\varepsilon}]{\eta\ww D}'.
\]

We establish now that both sequences converge in $\holb[\mathcal{U}_{r}^{\sharp}]$.
The hypothesis $\phi\left(u,0\right)=0$ guarantees that $\psi_{n+1}^{\sharp}\left(u,v\right)=\psi_{n}^{\sharp}\left(u,v\right)+\OO{v^{n+1}}$,
hence the bounded sequence $\left(\psi_{n}^{\sharp}\right)_{n\in\nn}$
converges for the projective topology on $\frml{\varepsilon,u}\left[\left[v\right]\right]$
(for the Krull distance actually). Therefore the sequences converge
towards holomorphic and bounded functions 
\begin{align*}
\psi^{\sharp} & :=\lim_{n\to\infty}\psi_{n}^{\sharp}\in\holb[\group{{{\neigh[k]}}\times\rho^{\sharp}\ww D\times r\ww D}]
\end{align*}
according to the next lemma.
\begin{lem}
\label{lem:bounded_krull_converge}\cite[Lemma 2.10]{SchaTey} Let
$\mathcal{D}$ be a domain in $\ww C^{m}$ and consider a bounded
sequence $\left(f_{p}\right)_{p\in\ww N}$ of $\holb[\mathcal{D}]$
satisfying the additional property that there exists some point $\mathbf{z}_{0}\in\mathcal{D}$
such that the corresponding sequence of Taylor series $\left(T_{p}\right)_{p\in\ww N}$
at $\mathbf{z}_{0}$ is convergent in $\frml{\mathbf{z}-\mathbf{z}_{0}}$
(for the projective topology). Then $\left(f_{p}\right)_{p}$ converges
uniformly on compact sets of $\mathcal{D}$ towards some $f_{\infty}\in\holb[\mathcal{D}]$.
\end{lem}

\begin{rem}
The limiting functions $\psi^{\sharp}$ are \emph{not} obtained through
the use of a fixed-point theorem, although they are a fixed-point
of~\eqref{eq:savelev_cousin_recurrence}. The method used here, based
on Lemma~\ref{lem:bounded_krull_converge}, does not use the fact
that $\holb[\mathcal{D}]$ is a Banach space, only that it is a Montel
space (any bounded subset is sequentially compact). Also the estimate~\eqref{eq:estim_F}
obtained in Lemma~\ref{lem:Savelev-Cousin}~(1) is easier to derive
than a sharper estimate aimed at establishing that $\psi_{n}^{\sharp}\mapsto\psi_{n+1}^{\sharp}$
is a contraction mapping. 
\end{rem}

\subsubsection{Proof of Lemma~\ref{lem:Savelev-Cousin}~(2)}

This is nothing but Cauchy formula.

\subsubsection{Proof of Lemma~\ref{lem:Savelev-Cousin}~(1)}

Clearly the function $F^{\sharp}$ is holomorphic on the domain $\mathcal{U}_{r}^{\sharp}$.
Notice also that modifying slightly the integration path does not
change the value of the function, so that $F^{\sharp}$ is bounded
on $\mathcal{U}_{r}^{\sharp}$. Let us evaluate its norm. 

Set $\Psi~:~\left(u,v\right)\mapsto\left(u,v\exp\psi\left(u,v\right)\right)$
and define $\rho>0$ by $2\rho:=\rho^{\infty}+\frac{1}{\rho^{0}}$.
We prove the estimate on $\norm[F^{\infty}]{\mathcal{U}_{r}^{\infty}}$
and $\norm[F^{0}]{\mathcal{U}_{r}^{0}}$ in two steps: first we bound
$\left|F^{\infty}\left(u,v\right)\right|$ when $\left|u\right|\leq\rho$
(\emph{resp.} $\norm[F^{0}]{\mathcal{U}_{r}^{0}}$ when $\left|x\right|\leq\frac{1}{\rho}$),
then when $\rho\leq\left|u\right|<\rho^{\infty}$ (\emph{resp.} $\frac{1}{\rho}\leq\left|x\right|<\rho^{0}$). 
\begin{itemize}
\item For $\left|u\right|\leq\rho$ and $\left|v\right|<r$ one has 
\begin{align*}
\left|F^{\infty}\left(u,v\right)\right| & \leq\norm[\phi\circ\Psi]{\mathcal{A}_{r}}\times\frac{1}{2\pi}\oint_{\rho^{\infty}\sone}\left|\frac{\dd z}{z-u}\right|.
\end{align*}
On the one hand 
\begin{align*}
\frac{1}{2\pi}\oint_{\rho^{\infty}\sone}\left|\frac{\dd z}{z-u}\right| & \leq\frac{1}{\rho^{\infty}-\rho}=\frac{2\rho^{0}}{\rho^{\infty}\rho^{0}-1}<K,
\end{align*}
while on the other hand, for all $\left(u,v\right)\in\mathcal{A}_{r}$,
\begin{align*}
\left|\phi\left(u,v\exp\psi\left(u,v\right)\right)\right| & \leq\left|v\right|\norm[\phi]{\mathcal{A}_{\eta}}'\exp\norm[\psi]{\mathcal{U}_{r}^{\infty}}.
\end{align*}
Taking both bounds together completes the first step of the proof.
\item This gives a corresponding bound for $F^{0}$ when $\left|x\right|\leq\frac{1}{\rho}$
since
\begin{align*}
\frac{\left|x\right|}{2\pi}\oint_{\frac{1}{\rho^{0}}\sone}\left|\frac{\dd z}{xz-1}\right| & \leq\frac{1}{\rho-\frac{1}{\rho^{0}}}=\frac{2\rho^{0}}{\rho^{\infty}\rho^{0}-1}.
\end{align*}
Taking~\eqref{eq:savelev_cousin} into account, one therefore deduces
for $\frac{1}{\rho^{0}}<\left|u\right|\leq\rho$ the estimate
\begin{align*}
\left|F^{\infty}\left(u,v\right)\right| & \leq\left|\phi\left(u,v\exp\psi\left(u,v\right)\right)\right|+\left|F^{0}\left(\frac{1}{u},v\right)\right|\\
 & \leq\left|v\right|\norm[\phi]{\mathcal{A}_{\eta}}'\exp\norm[\psi]{\mathcal{U}_{r}^{\infty}}\left(1+\frac{2\rho^{0}}{\rho^{\infty}\rho^{0}-1}\right)
\end{align*}
as expected. 
\item The bound for $F^{0}$ when $\frac{1}{\rho}\leq x<\rho^{0}$ is obtained
similarly.
\end{itemize}

\subsubsection{Proof of Lemma~\ref{lem:Savelev-Cousin}~(3)}

Assume that $\left(\widetilde{F}^{0},\widetilde{F}^{\infty}\right)$
is another pair of solution. Then for all $\left(\frac{1}{x},y\right)=\left(u,v\right)\in\mathcal{A}_{r}$
we have
\begin{align*}
f^{0}\left(x,y\right):=F^{0}\left(x,y\right)-\widetilde{F}^{0}\left(x,y\right) & =F^{\infty}\left(u,v\right)-\widetilde{F}^{\infty}\left(u,v\right)=:f^{\infty}\left(u,v\right),
\end{align*}
defining a bounded and holomorphic function $f$ on $\hirze$. The
next lemma ends the proof.
\begin{lem}
\label{lem:bounded_function_hirzebruch}If $f\in\holb[{\hirze}]$
then $\ppp{f^{\infty}}u=0$. In other words there is a natural isometry
of Banach spaces
\begin{align*}
\holb[{\hirze}] & \simeq\holb[r\ww D].
\end{align*}
\end{lem}

\begin{proof}
In the chart $\mathcal{U}_{r}^{\infty}$ expand $f^{\infty}$ into
a power series $f^{\infty}\left(u,v\right)=\sum_{n\geq0}f_{n}\left(u\right)v^{n}$
convergent on $\cc\times r\ww D$. By assumption $f$ is bounded so
that from Cauchy's estimate we get
\begin{align*}
\left|f_{n}\left(u\right)\right| & \leq\norm[f]{{\hirze}}r^{-n}
\end{align*}
for all $u\in\cc$. Liouville theorem tells us that each $f_{n}$
is constant. 
\end{proof}

\subsection{\label{subsec:Normal-form-recognition}Normal form recognition (proof
of orbital Normalization Theorem)}

The aim of this subsection is to shortly prove that the vector field
$\Psi^{*}X_{\varepsilon}^{0}$ resulting from Theorem~\ref{thm:fibred_gluing_parametric}
is in normal form~\eqref{eq:orbital_normal_form}. Because Savelev's
normalizing fibered mapping $\Psi$ agrees with the identity on $\mathcal{L}$,
each $\fol{\varepsilon}$ is induced in the chart $\mathcal{U}_{r}^{0}$
by a holomorphic vector field of the form
\begin{align*}
\mathcal{X}_{\varepsilon}^{0}\left(x,y\right) & :=\Psi^{*}X_{\varepsilon}^{0}=P_{\varepsilon}\left(x\right)\pp x+y\left(1+\mu_{\varepsilon}x^{k}+A_{\varepsilon}\left(x,y\right)\right)\pp y,\\
A & \in\holf[\left(\cc^{k},0\right)\times\rho^{0}\ww D\times r\ww D]\\
A\left(x,0\right) & =0.
\end{align*}
We must prove the following result.
\begin{lem}
There exists a sequence of polynomials $a_{n}\in\germ{\varepsilon}\left[x\right]_{\leq k}$
such that 
\begin{align*}
A\left(x,y\right) & =\sum_{n=1}^{\infty}a_{n}\left(x\right)y^{n}
\end{align*}
on $\mathcal{U}_{r}^{0}$.
\end{lem}

\begin{proof}
The expansion for $A$ is valid for $\left(x,y\right)\in\mathcal{U}_{r}^{0}$
and $a_{n}$ holomorphic in $x$. In the other chart $\left(u,v\right)=\left(\frac{1}{x},y\right)$
the vector field $\mathcal{X}_{\varepsilon}^{0}$ is orbitally equivalent
(conjugate after division by $uP_{\varepsilon}\left(\frac{1}{u}\right)$)
to
\begin{align*}
\mathcal{X}_{\varepsilon}^{\infty}\left(u,v\right) & :=-u\pp u+v\left(\lambda_{\varepsilon}+h_{\varepsilon}\left(u\right)+\frac{1}{u^{k+1}P_{\varepsilon}\left(\frac{1}{u}\right)}u^{k}A_{\varepsilon}\left(\frac{1}{u},v\right)\right)\pp v
\end{align*}
where $h$ is given by~\eqref{eq:divisor_orbital_change}. This particular
vector field must coincide with the holomorphic vector field defining
$\fol{\varepsilon}$ in the chart $\mathcal{U}_{r}^{\infty}$ after
application of~\eqref{eq:divisor_orbital_straightenning}, because
every transform used from the start is fibered so that the factor
$uP_{\varepsilon}\left(\frac{1}{u}\right)$ over $\mathcal{A}_{r}$
remains the same and no other function can be factored out. Therefore
$u^{k}A_{\varepsilon}\left(\frac{1}{u},v\right)$ is holomorphic near
$\left(0,0\right)$, and the conclusion follows.
\end{proof}

\subsection{\label{subsec:Uniqueness_orbital}Proof of orbital Uniqueness Theorem~(2)}

Assume that there exists an orbital equivalence between two normal
forms $\mathcal{X}$ and $\widetilde{\mathcal{X}}$. Those vector
fields are in prepared form~\eqref{eq:prepared_form_orbital} thus
they satisfy the hypothesis of the results presented in Section~\ref{sec:Formal},
and in particular there exists a fibered analytical conjugacy $\Psi$
near $\left(0,0\right)$ between $\mathcal{X}$ and $\widetilde{\mathcal{X}}$,
according to Corollary~\ref{cor:formal_symmetries}~(2). 

Let $\left(\hirze,~\fol{}\right)$ and $\left(\hirze,~\widetilde{\fol{}}\right)$
be the family of foliated standard neighborhoods of the sphere induced
respectively by $\mathcal{X}$ and $\widetilde{\mathcal{X}}$. The
fibered mappings $\Psi$ are holomorphic and injective on a domain
$\mathcal{D}\subset\mathcal{U}^{0}\subset\hirze$ containing $\left(0,0\right)$.
By a foliated path-lifting technique (as before) $\Psi$ can be analytically
continued on the domain
\begin{align*}
\mathcal{U}_{\varepsilon} & :=\sat[\fol{\varepsilon}]{\mathcal{D}}\subset\hirze.
\end{align*}
Using the special form of the normal form $\mathcal{X}_{\varepsilon}$
we derive the following lemma in Section~\ref{subsec:saturated_contains_Hirzebruch}.
\begin{lem}
\label{lem:saturated_contains_Hirzebruch}There exists $r'>0$ such
that $\hirze[r']\backslash\left\{ x=\infty\right\} \subset\mathcal{U}_{\varepsilon}$
for all $\varepsilon\in\neigh[k]$.
\end{lem}

This lemma implies that $\Psi_{\varepsilon}$ extends to a fibered,
injective and holomorphic mapping $\hirze[r']\backslash\left\{ x=\infty\right\} \to\hirze$.
The fact that $\Psi_{\varepsilon}$ extends analytically to $\left\{ x=\infty\right\} $
uses a variation on the Mattei-Moussu construction. The proof is standard,
but we include it for the sake of completeness.
\begin{lem}
\label{lem:mattei-moussu}\cite[Theorem 2]{MaMou} We consider two
germs of a holomorphic vector field $X$ and $\widetilde{X}$, both
with a singularity at the origin of same eigenratio $\lambda\notin\rr_{\geq0}$
and in the form
\begin{align}
x\frac{\partial}{\partial x}+\lambda y(1+\OO x)\frac{\partial}{\partial y} & .\label{eq:mamou_form}
\end{align}
Fix a germ of a transverse disc $\Sigma:=\left\{ x=x_{*},~y\in\neigh\right\} $,
for $x_{*}$ small enough, and assume that there exists an injective
and holomorphic mapping $\psi~:~\Sigma\to\left\{ x=x_{*}\right\} $
conjugating the respective holonomies induced by $X$ and $\widetilde{X}$,
computed through the fibration $\left(x,y\right)\mapsto x$ by turning
around $\left\{ x=0\right\} $. Then there exists a holomorphic and
injective, fibered mapping $\Psi$ conjugating $X$ and $\widetilde{X}$
on a connected neighborhood of $\left(0,0\right)$ containing $\Sigma$.
We can even require that $\Psi$ coincides with $\psi$ on $\Sigma$.
\end{lem}

\begin{proof}
Assume first that $\lambda<0$. We can consider that the holonomies
$\Gamma$ and $\widetilde{\Gamma}$ are defined on $\Sigma:=\left\{ x=x_{*}\right\} \times r'\disc$
and set $\Psi(x_{*},y):=(x_{*},\psi\left(y\right))$ on $\Sigma$.
We then extend $\Psi$ over the circle $\left\{ |x|=\left|x_{*}\right|\right\} $
as a map of the form $\Psi(x,y)=(x,\psi(x,y))$, with $\psi(x_{*},y)=\psi(y)$:
the extension is done by the path-lifting technique detailed in Section~\ref{subsec:Gluing}.
$\Psi$ is of course well-defined because $\psi$ conjugates the holonomies.
To extend $\Psi$ to $\rho\disc\times r'\disc$, we use the path-lifting
along rays $\left\{ \arg|x|=\cst\right\} $. Starting at $(x_{0},y)$
we lift the ray through $x_{0}$ up to $|x|=\rho$ in the leaf of
$X$. We apply $\Psi$ to the resulting point and then lift the ray
back in the leaf of $\widetilde{X}$. The corresponding point is called
$\Psi(x_{0},y)$. We must show that 
\begin{align*}
\{x_{0}\}\times C_{1}r'\disc & \subset\Psi(\{x_{0}\}\times r'\disc)\subset\{x_{0}\}\times C_{2}r'\disc
\end{align*}
for some positive constants $C_{1},~C_{2}$ independent of $x_{0}$.
For this purpose we can suppose that the $\OO x$ part in~\eqref{eq:mamou_form}
is bounded by $\nf 12$ (this is the case if $\rho$ is sufficiently
small). Then 
\[
|\lambda|~|y|\left(1-\frac{1}{2}\left|x_{0}\right|\exp t\right)<\ddd{\left|y\right|}t<|\lambda|~|y|\left(1+\frac{1}{2}\left|x_{0}\right|\exp t\right),
\]
yielding by Gronwall inequality that 
\[
|y(0)|\exp\left(\lambda t-\int_{0}^{t}\frac{1}{2}x_{0}\exp t~\dd t\right)\leq|y(t)|\leq|y(0)|\exp\left(\lambda t+\int_{0}^{t}\frac{1}{2}|x_{0}|\exp t~\dd t\right).
\]
The conclusion follows since $\exp\left(\int_{0}^{t}\frac{1}{2}|x_{0}|\exp t~\dd t\right)\in\left]\exp\frac{-|x_{0}|}{2},1\right[$
is bounded and bounded away from $0$ for $t<0$. 

The previous argument remains valid when $\lambda$ is not real. It
suffices to replace $|\lambda|$ by $\left|\re{\lambda}\right|$.
\end{proof}
\begin{rem}
The proof clearly shows that $\Psi$ depends analytically on $\varepsilon$
were $X$ and $\widetilde{X}$ to depend analytically on $\varepsilon$.
\end{rem}

The following lemma proved in Section~\ref{subsec:diffeo_Hirzebruch}
allows to complete the proof of the Uniqueness Theorem~(2) by observing
that injective holomorphic mappings on some standard neighborhood
of the sphere are of a rather special kind.
\begin{lem}
\label{lem:diffeo_hirzebruch}Take some analytic family of maps $\Psi~:~\hirze[r']\to\hirze$
satisfying the following properties:

\begin{itemize}
\item $\Psi$ is fibered,
\item $\Psi_{\varepsilon}$ is injective and holomorphic on $\hirze[r']$
for every $\varepsilon\in\neigh[k]$.
\end{itemize}
Then 
\begin{align}
\Psi_{\varepsilon}^{0}\left(x,y\right) & =\left(x,~y\sum_{n=0}^{\infty}\psi_{n}y^{n}\right),\label{eq:0-neighborhood_diffeo}
\end{align}
 where, for all $n\in\zp$,
\begin{align*}
\psi_{n}\in\germ{\varepsilon}
\end{align*}
with a common radius of convergence, and $\psi_{0}$ does not vanish
for $\varepsilon=0$. Conversely, any convergent power series $\Psi$
as above defines an analytic family satisfying the above properties
for some $r'>0$ small enough. 
\end{lem}

As a matter of consequence for every $\varepsilon\in\neigh[k]$ and
for any $\left(x,y\right)\in\mathcal{U}_{r}^{0}$ 
\begin{align*}
\Psi_{\varepsilon}\left(x,y\right) & =\left(x,~y\psi_{\varepsilon}\left(y\right)\right)~~~,~\psi_{\varepsilon}\left(0\right)\neq0.
\end{align*}
To preserve globally orbital normal forms~\eqref{eq:orbital_normal_form}
is so demanding that $\psi_{\varepsilon}$ ends up being constant.
Indeed, from 
\begin{align*}
\Psi_{\varepsilon}^{*}\mathcal{X}_{\varepsilon}\left(x,y\right) & =\fonf\left(x,y\right)+y\frac{A_{\varepsilon}\left(x,y\right)}{y\psi_{\varepsilon}'\left(y\right)+\psi_{\varepsilon}\left(y\right)}\pp y=\widetilde{\mathcal{X}}_{\varepsilon}\left(x,y\right),
\end{align*}
where
\begin{align*}
A_{\varepsilon}\left(x,y\right) & :=x\psi_{\varepsilon}\left(y\right)R_{\varepsilon}\left(x,y\psi_{\varepsilon}\left(y\right)\right)-y\psi_{\varepsilon}'\left(y\right)\left(1+\mu x^{k}\right),
\end{align*}
we deduce by setting $x:=0$ that 
\begin{align*}
0 & =A_{\varepsilon}\left(0,y\right)=-y\psi'_{\varepsilon}\left(y\right)
\end{align*}
so that $\psi_{\varepsilon}$ is constant, for otherwise $\widetilde{\mathcal{X}}_{\varepsilon}$
would not be in normal form. In each case we obtain finally 
\begin{align*}
\psi_{\varepsilon}\left(v\right) & =c_{\varepsilon}\in\cc^{\times}
\end{align*}
as expected. The remaining claim is a straightforward consequence
of the study performed in Section~\ref{sec:Formal}.

\subsubsection{\label{subsec:diffeo_Hirzebruch}Proof of Lemma~\ref{lem:diffeo_hirzebruch}}

The expansion~\eqref{eq:0-neighborhood_diffeo} is valid on $\mathcal{U}_{r'}^{0}$
provided $\psi_{n}$ depend holomorphically on $x$. Let us show that
$\psi_{n}$ is constant. Applying the transition mapping $\left(x,y\right)\mapsto\left(\frac{1}{x},y\right)$
we obtain the expression of $\Psi$ in the other chart: 
\begin{align*}
\Psi^{\infty}\left(u,v\right)= & \left(u,~v\sum_{n=0}^{\infty}\psi_{n}\left(\frac{1}{u}\right)v^{n}\right),
\end{align*}
holomorphic in $\left(u,v\right)\in\mathcal{U}_{r'}^{\infty}$. This
implies in particular that each function $u\mapsto\psi_{n}\left(\frac{1}{u}\right)$
must be holomorphic at $0$; the conclusion follows. The converse
statement is straightforward. 

\subsubsection{\label{subsec:saturated_contains_Hirzebruch}Proof of Lemma~\ref{lem:saturated_contains_Hirzebruch}}

We can find $\rho,~r'>0$ such that $\adh{\rho\ww D\times r'\ww D}\subset\mathcal{D}$,
where $\mathcal{D}$ is the domain of $\Psi$. We show that, for some
convenient choice of $r''\leq r'$ every point $\left(x_{*},y_{*}\right)\in\left\{ \left|y\right|<r''\right\} $
can be linked to a point of $\rho\ww D\times r'\ww D$ by a path contained
in a leaf of $\fol{\varepsilon}^{0}$. Only the case $\left|x_{*}\right|>\rho$
is not trivial. Since the singularity at $\left(\infty,0\right)$
is neither a node nor a saddle-node, every small germ of a disk $\left\{ u=u_{*}\right\} $
sufficiently close to $\left\{ u=0\right\} $, which is transverse
to the separatrix $\mathcal{L}$, saturates a full pointed neighborhood
$\neigh^{2}\backslash\left\{ u=0\right\} \subset\mathcal{U}_{r}^{\infty}$
under $\fol{\varepsilon}^{\infty}$. Therefore there exists $r'''>0$
such that $\left\{ 0<\left|u\right|\leq\left|u_{*}\right|,~\left|v\right|<r'''\right\} \subset\mathcal{U}_{\varepsilon}$.
Because $\mathcal{L}$ is invariant by $\fol{\varepsilon}$ and $\mathcal{L}\backslash\left(\left\{ \left|x\right|<\rho\right\} \cup\left\{ \left|u\right|<\left|u_{*}\right|\right\} \right)$
is compact we may reduce $r'''$ to some $r''$ in such a way that
$\rho\sone\times r''\ww D\subset\mathcal{U}_{\varepsilon}$ (flow-box
argument), which settles the proof. 

\section{\label{sec:Temporal}Temporal normal forms}

This section is devoted to proving the temporal part of the Normalization
Theorem and of the Uniqueness Theorem in the case $\tau=0$ (which
particularly implies $\mu_{0}\notin\rr_{\leq0}$). Recall how in Section~\ref{sec:Formal}
we obtained formal normal forms. The time-component $U$ of any unfolding
in orbital normal form~\eqref{eq:orbital_normal_form} 
\begin{align*}
Z & =U\onf
\end{align*}
can be written as
\begin{align*}
\frac{1}{U} & =C+I,
\end{align*}
where
\begin{align*}
I & \in\tx{im}\left(\onf\cdot\right)\\
C & \in\tx{coker}\left(\onf\cdot\right)~~~~,~C\left(0,0\right)=\frac{1}{U\left(0,0\right)}\neq0,
\end{align*}
for a given (arbitrary for now) choice of $\tx{coker}\left(\onf\cdot\right)$,
an algebraic supplementary in $\frml{\varepsilon,x,y}$ to the image
$\tx{im}\left(\onf\cdot\right)$ of the (formal) Lie derivative $\onf\cdot~:~\frml{\varepsilon,x,y}\to\frml{\varepsilon,x,y}$.
According to the discussion following Proposition~\ref{prop:cohomological_conjugacy},
$Z$ is (formally) conjugate to $\frac{1}{C}\onf$. 

We have shown in Lemma~\ref{lem:formal_cohomological} that 
\begin{align*}
\frml{\varepsilon,x,y} & =\tx{im}\left(\onf\cdot\right)~\oplus~\frml{\varepsilon}\left[x\right]_{<k},
\end{align*}
or more precisely that the following sequence of $\frml{\varepsilon}$-linear
operators is exact:
\begin{align}
0\longto\frml{\varepsilon}\longto\frml{\varepsilon,x,y}\overset{\onf\cdot}{\longto}\frml{\varepsilon,x,y}\overset{\widehat{\per}}{\longto}\frml{\varepsilon}\left[x\right]_{\leq k}\longto0 & ,\label{eq:formal_exact_sequence}
\end{align}
where $\widehat{\per}$ maps $G$ to the remainder of the Euclidean
division of its partial function $x\mapsto G\left(x,0\right)$ by
$P$. As a consequence we may take
\begin{align*}
\tx{coker}\left(\onf\cdot\right) & :=\frml{\varepsilon}\left[x\right]_{<k},
\end{align*}
so that $Z$ is formally conjugate to $\frac{1}{\widehat{\per}\left(\frac{1}{U}\right)}\onf$. 
\begin{rem}
\label{rem:formal_cokernel_inverse}The additional fact that 
\begin{align*}
\widehat{\per}\left(\frac{1}{U}\right) & =\frac{1}{\widehat{\per}\left(U\right)}+\OO P
\end{align*}
finally implies that $Z$ is formally conjugate to $u\onf$ where
$u:=\widehat{\per}\left(U\right)$, as in the Formal Normalization
Theorem. This is because one can write (for $u_{0}\left(0\right)\neq0$)
\begin{align*}
\frac{1}{U\left(x,y\right)} & =\frac{1}{u\left(x\right)+\OO{P\left(x\right)}+\OO y}=\frac{1}{u\left(x\right)}\times\frac{1}{1+\OO{P\left(x\right)}+\OO y}\\
 & =\frac{1}{u\left(x\right)}+\OO{P\left(x\right)}+\OO y.
\end{align*}
\end{rem}

The previous argument still works for convergent power series, by
replacing $\frml{\varepsilon,x,y}$ with $\germ{\varepsilon,x,y}$:
if we provide an explicit cokernel in $\germ{\varepsilon,x,y}$ of
$\onf\cdot|_{\germ{\varepsilon,x,y}}$ then we can describe an explicit
family of temporal normal forms. 
\begin{thm}
\label{thm:section_period}Assume $\tau=0$ (which particularly implies
$\mu_{0}\notin\rr_{\leq0}$). Let an orbital normal form $\onf$ be
given. It acts by directional derivative on the linear space $\germ{\varepsilon,x,y}$
in such a way that
\begin{align*}
\germ{\varepsilon,x,y} & =\tx{im}\left(\onf\cdot\right)~\oplus~\germ{\varepsilon}\left[x\right]_{\leq k}~\oplus~\nfsec[k][y].
\end{align*}
(We refer to Section~\ref{subsec:Functional-spaces} for the definition
of the functional spaces.)
\end{thm}

\begin{rem}
\label{rem:section_fixed_parameter}The construction of the cokernel
of $\onf\cdot$ is eventually performed for $\varepsilon$ fixed.
Therefore the theorem can also be specialized in the following way:
for every $\varepsilon\in\neigh[k]$ such that $\mu_{\varepsilon}\notin\rr_{\leq0}$
and every disk $D\supset P_{\varepsilon}^{-1}\left(0\right)$ not
containing any root of $1+\mu_{\varepsilon}x^{k}$, we have the $\cc$-linear
decomposition 
\begin{align*}
\holb[D]\left\{ y\right\}  & =\tx{im}\left(\onf[\varepsilon]\cdot\right)~\oplus~\cc\left[x\right]_{\leq k}~\oplus~xy\pol x_{<k}\left\{ y\right\} .
\end{align*}
If $\mu_{\varepsilon}\in\rr_{\leq0}$ a section of the cokernel is
given by $xP_{\varepsilon}^{\tau}y\pol x_{<k}\left\{ P_{\varepsilon}{}^{\tau}y\right\} $.
\end{rem}

The aim of this section is to prove this theorem but, before doing
so, let us explain how it helps completing the proofs of the Normalization
and Uniqueness Theorems. Every function $U\in\germ{\varepsilon,x,y}^{\times}$
can be written uniquely as
\begin{align*}
U & =\frac{u}{1+uG}
\end{align*}
where $\widehat{\per}\left(G\right)=0$, by simply taking $u:=\widehat{\per}\left(U\right)$
as in Remark~\ref{rem:formal_cokernel_inverse}. Then Theorem~\eqref{thm:section_period}
allows decomposing $G$ uniquely as 
\begin{align*}
G & =Q+I
\end{align*}
with $Q\in\nfsec[k][y]$ and $I\in\onf\cdot\germ{\varepsilon,x,y}$,
so that $Z$ is analytically conjugate to some $\frac{u}{1+uQ}\onf$,
unique up to the action of linear transforms $\left(x,y\right)\mapsto\left(x,cy\right)$
as expected (as follows from Uniqueness Theorem~(2) which has been
proved in the previous section). This yields Uniqueness Theorem~(1).

\subsection{\label{subsec:Proof's-reduction}Reduction of the proof}

We must study the obstructions to solve analytically cohomological
equations of the form
\begin{align*}
\onf\cdot F & =G~~~~~~,~G\in\germ{\varepsilon,x,y}\cap\ker\widehat{\per}.
\end{align*}
First observe that this equation, restricted to the invariant line
$\left\{ y=0\right\} $, is always satisfied by a holomorphic function
$f~:~x\mapsto F\left(x,0\right)$ solving
\begin{align*}
f'\left(x\right) & =\frac{G\left(x,0\right)}{P\left(x\right)}\in\germ{\varepsilon,x}.
\end{align*}
By subtracting $f$ from $F$ and $x\mapsto G\left(x,0\right)$ from
$G$, we may always assume without loss of generality that 
\begin{align*}
G\left(x,0\right) & =F\left(x,0\right)=0,
\end{align*}
\emph{i.e. }$G\in\germ{\varepsilon,x,y}'$ as defined in Section~\ref{subsec:Functional-spaces}.

\bigskip{}

Let 
\begin{align*}
\Delta_{k} & :=\left\{ \varepsilon\in\neigh[k]~:~\#P_{\varepsilon}^{-1}\left(0\right)\leq k\right\} 
\end{align*}
be a germ at $0$ of the discriminant hypersurface of $P_{\varepsilon}$,
so that each open set $\neigh[k]\backslash\Delta_{k}$ consists in
generic values of the parameter for which $P_{\varepsilon}$ has only
simple roots. Proving Theorem~\ref{thm:section_period} will require
to work in the functional spaces
\begin{align*}
\mathcal{H}_{\ell}\left\{ \mathbf{z}\right\}  & :=\bigcup_{\mathcal{D}=\neigh[n]}\holb[\group{\mathcal{E}_{\ell}\times\mathcal{D}}]'~~~~~~,~\mathbf{z}:=\left(z_{1},\ldots,z_{n}\right)
\end{align*}
for some decomposition $\left(\mathcal{E}_{\ell}\right)_{\ell}$ of
$\neigh[k]\backslash\Delta_{k}$ into finitely many (germs of) open
cells as explained in Section~\ref{subsec:squid_sectors}. (We recall
that the definition of the space $\holb[\mathcal{D}]'$ is given in
Section~\ref{subsec:Functional-spaces}.) We choose these spaces
because of the next property.
\begin{lem}
\label{lem:intersection_cellular_space}~
\begin{align*}
\germ{\varepsilon,\mathbf{z}}' & =\bigcap_{\ell}\mathcal{H}_{\ell}\left\{ \mathbf{z}\right\} .
\end{align*}
(By the intersection on the right hand side we of course mean the
functions who have an extension on the unions of the different domains.)
\end{lem}

\begin{proof}
We certainly have
\begin{align*}
\germ{\varepsilon,\mathbf{z}}' & \subset\bigcap_{\ell}\mathcal{H}_{\ell}\left\{ \mathbf{z}\right\} .
\end{align*}
Conversely if $f\in\bigcap_{\ell}\mathcal{H}_{\ell}\left\{ \mathbf{z}\right\} $
then $f$ defines a bounded, holomorphic function on $\left(\neigh[k]\backslash\Delta_{k}\right)\times\neigh[n]$,
which extends holomorphically to $\neigh[k+n]$ according to Riemann's
theorem on removable singularities.
\end{proof}
Working over a fixed cell germ $\mathcal{E}_{\ell}$ is easy as compared
to $\neigh[k]$.
\begin{prop}
\label{prop:cohomog_secto_solutions}\cite{RouTey} Let $\mathcal{E}_{\ell}$
be a parameter cell. There exists $\per[][\ell]$, called the \textbf{period
operator} over $\mathcal{E}_{\ell}$, such that the sequence of $\holb[\mathcal{E}_{\ell}]$-linear
operators is exact:
\begin{align}
0\longto\holb[\mathcal{E}_{\ell}]\longto\mathcal{H}_{\ell}\left\{ x,y\right\} \overset{\onf\cdot}{\longto}\mathcal{H}_{\ell}\left\{ x,y\right\} \overset{\per[][\ell]}{\longto}\prod_{\zsk}\mathcal{H}_{\ell}\left\{ h\right\} \label{eq:exact_sequence_cellular}
\end{align}
where $h$ is a one-dimensional variable (meant to take the values
of a first integral).
\end{prop}

The surjectivity of the period operator $\per[][\ell]$ has not been
established in the cited reference, but it would have followed from
an immediate adaptation of the argument of~\cite[Lemma 3.4]{Tey-SN}.
Here, though, we prove a stronger result by producing an explicit
section to the period operator (Proposition~\ref{prop:section_period_cellular}
to come). The construction of the period operator over $\mathcal{E}_{\ell}$
is explained in Section~\ref{subsec:period_operator} below. It involves
cutting up $\neigh[2]\backslash\left(P_{\varepsilon}^{-1}\left(0\right)\times\left\{ 0\right\} \right)$
into $k$ open (bounded) spiraling sectors and building sectorial
solutions of the cohomological equation. The period operator measures
how much solutions on neighboring sectors disagree on intersections.
Contrary to what would have make things easier 
\begin{align*}
\per[][\ell]\left(\germ{\varepsilon,x,y}'\right) & \neq\bigcap_{p}\prod_{\zsk}\mathcal{H}_{p}\left\{ h\right\} =\prod_{\zsk}\germ{\varepsilon,h}',
\end{align*}
so that $\per[][\ell]$ is neither onto nor into the natural candidate
$\prod_{\zsk}\germ{\varepsilon,h}'$. This situation differs drastically
from the case $\varepsilon=0$, and can be explained. It turns out
that the variable $h$ in the $j^{\tx{th}}$ factor of $\prod_{j\in\zsk}\mathcal{H}_{p}\left\{ h\right\} $
stands for values of the canonical first integral of $\onf$ on the
$j^{\tx{th}}$ sector (see the discussion preceding Definition~\ref{def:period_operator}).
Different sectorial decompositions for fixed $\varepsilon$, corresponding
to different cells $\mathcal{E}_{\ell}$ containing $\varepsilon$,
lead to incommensurable sectorial dynamics: there is no correspondence
between $h$-variables coming from different overlapping cells (see
also Section~\ref{sec:Bernoulli}). Therefore we need to relocate
the obstructions in geometrical space $\left(x,y\right)$, by introducing
a well-chosen section $\persec[][\ell]$ of $\per[][\ell]$. 
\begin{prop}
\label{prop:section_period_cellular} Let $\mathcal{E}_{\ell}$ be
a parameter cell and assume $\tau=0$ (which particularly implies
$\mu_{0}\notin\rr_{\leq0}$). There exists a linear isomorphism
\begin{align*}
\persec[][\ell]~:~\prod_{\zsk}\mathcal{H}_{\ell}\left\{ h\right\}  & \longto x\mathcal{H}_{\ell}\left\{ y\right\} \left[x\right]_{<k}
\end{align*}
such that $\per[][\ell]\circ\persec[][\ell]=\id$. This particularly
means we recover a cellular  cokernel of $\onf\cdot$ as follows:
\begin{align*}
\mathcal{H}_{\ell}\left\{ x,y\right\}  & =\left(\onf\cdot\mathcal{H}_{\ell}\left\{ x,y\right\} \right)~\oplus~x\mathcal{H}_{\ell}\left\{ y\right\} \left[x\right]_{<k}.
\end{align*}
\end{prop}

This proposition is showed later in Section~\ref{subsec:section_period_cellular}
using a refinement of the Cauchy-Heine transform, this time on unbounded
sectors in the $x$-variable. Theorem~\ref{thm:section_period} is
proved once we establish the next gluing property, as done in Section~\ref{subsec:section_period}.
\begin{prop}
\label{prop:section_period} For every parameter cells $\mathcal{E}_{\ell}$
and $\mathcal{E}_{\widetilde{\ell}}$ with non-empty intersection
we have
\begin{align*}
\persec[][\ell]\circ\per[][\ell] & =\persec[][\widetilde{\ell}]\circ\per[][\widetilde{\ell}]
\end{align*}
on $\mathcal{H}_{\ell}\left\{ x,y\right\} \cap\mathcal{H}_{\widetilde{\ell}}\left\{ x,y\right\} $. 
\end{prop}

From Lemma~\ref{lem:intersection_cellular_space} we deduce the identity
\begin{align*}
\nfsec[][y] & =\bigcap_{\ell}x\mathcal{H}_{\ell}\left\{ y\right\} \left[x\right]_{<k},
\end{align*}
hence the proposition actually provides us with a well-defined, surjective
operator
\begin{align}
\mathfrak{K}~:~\germ{\varepsilon,x,y}' & \longto\nfsec[][y]\label{eq:period_section}\\
G & \longmapsto\persec[][\ell]\left(\per[][\ell]\left(G\right)\right),\nonumber 
\end{align}
whose kernel coincides with $\onf\cdot\germ{\varepsilon,x,y}'$, \emph{i.e.
}the sequence of $\germ{\varepsilon}$-linear operators
\begin{align}
0\longto\germ{\varepsilon,x,y}'\overset{\onf\cdot}{\longto}\germ{\varepsilon,x,y}'\overset{\mathfrak{K}}{\longto} & \nfsec[][y]\longto0\label{eq:exact_sequence}
\end{align}
is exact, as required to establish Theorem~\ref{thm:section_period}.

\subsection{\label{subsec:period_operator}Cohomological equation and period
operator}
\begin{thm}
\cite{RouTey}\label{thm:squid_and_solve} For every $\rho>0$ there
exists:

\begin{itemize}
\item a covering of $\neigh[k]\backslash\Delta_{k}$ by finitely many open,
contractible cells $\left(\mathcal{E}_{\ell}\right)_{\ell}$, 
\item for every $\varepsilon\in\mathcal{E}_{\ell}$, a covering of 
\begin{align*}
V_{\varepsilon}: & =\rho\ww D\backslash P_{\varepsilon}^{-1}\left(0\right)
\end{align*}
 into $k$ open, contractible squid sectors 
\begin{align*}
V_{\ell,\varepsilon}^{j} & ,~~~~j\in\zsk,
\end{align*}
\end{itemize}
for which the following properties are satisfied. Recall that the
closure of a subset $A$ of a topological space is written $\adh A$.

\begin{enumerate}
\item Each map $\varepsilon\mapsto\adh{V_{\ell,\varepsilon}^{j}}$ is continuous
for the Hausdorff distance on compact sets and
\begin{align*}
\lim_{\varepsilon\underset{\mathcal{E}_{\ell}}{\to}0}\adh{V_{\ell,\varepsilon}^{j}} & =\adh{V_{0}^{j}}
\end{align*}
 coincides with (the closure of) a usual sector of the limiting saddle-node,
namely
\[
V_{0}^{j}:=\left\{ x~:~0<|x|<\rho,~\arg x\in\left]-\frac{3\pi}{2k}+\eta+j\frac{2\pi}{k}~,~\frac{3\pi}{2k}-\eta+j\frac{2\pi}{k}\right[\right\} 
\]
for some $\eta\in\left]0,\frac{\pi}{2k}\right[$. 
\item We let 
\begin{align*}
V_{\ell}^{j} & :=\bigcup_{\varepsilon\in\mathcal{E}_{\ell}}\left\{ \varepsilon\right\} \times V_{\ell,\varepsilon}^{j}.
\end{align*}
For every $G\in\holb[\mathcal{E}_{\ell}\times\rho\ww D\times{{\neigh}}]'$
there exists a unique family $\left(\sectobase[F][j]{\ell}\right)_{j\in\zsk}$
such that $\sectobase[F][j]{\ell}$ is the unique solution of
\begin{align*}
\onf\cdot F & =G
\end{align*}
in the space $\holb[V_{\ell}^{j}\times{{\neigh}}]'$. Moreover
\begin{align*}
\lim_{\varepsilon\underset{\mathcal{E}_{\ell}}{\to}0}F_{\ell,\varepsilon}^{j} & =F_{0}^{j}
\end{align*}
uniformly on compact sets of $V_{0}^{j}\times\neigh$, where $F_{0}^{j}$
is the canonical sectorial solution of the limiting cohomological
equation~\cite{Tey-EqH}.
\item There exists a solution $F\in\holb[\mathcal{E}_{\ell}\times\rho\ww D\times{{\neigh}}]$
of $\onf\cdot F=G$ if, and only if, for every $\varepsilon\in\mathcal{E}_{\ell}$
and $j\in\zsk$ 
\begin{align*}
F_{\ell,\varepsilon}^{j+1} & =F_{\ell,\varepsilon}^{j}
\end{align*}
on corresponding pairwise intersections of sectors $V_{\ell,\varepsilon}^{j}\times\neigh$. 
\end{enumerate}
\end{thm}

We provide details regarding how squid sectors and parameter cells
are obtained in Section~\ref{subsec:squid_sectors} below. The way
sectorial solutions $\left(F_{\ell}^{j}\right)_{j\in\zsk}$ are built
is explained in~\cite[Section 7]{RouTey}. The third property encodes
all we need to know in order to characterize algebraically the obstructions
to solve analytically cohomological equations. It is, as usual, eventually
a consequence of Riemann's theorem on removable singularities. 
\begin{rem}
\label{rem:corollaries_squid_solve}~

\begin{enumerate}
\item A usual saddle-node sector is divided by rays separated by an angle
slightly larger than $\frac{\pi}{k}$: allowing an extra $\frac{\pi}{2k}$
on each side yields sectors of opening between $\frac{\pi}{k}$ and
$\frac{2\pi}{k}$. However we are in the particular case of a saddle-node
with analytic center manifold, meaning that we need twice less sectors
to describe the singularity structure. Hence the angle between the
dividing rays can be taken as big as $\frac{2\pi}{k}$: allowing an
extra $\frac{\pi}{2k}$ on each side yields an opening between $\frac{2\pi}{k}$
and $\frac{3\pi}{k}$.
\item A corollary to this theorem is the fact that any generic convergent
unfolding is conjugate to its formal normal form over every region
$V_{\ell}^{j}\times\neigh$. In particular each $\onf$ is conjugate
over $V_{\ell,\varepsilon}^{j}\times\neigh$ to $\fonf$ by a fibered
mapping
\begin{align*}
\left(x,y\right) & \longmapsto\left(x,~y\exp N_{\ell,\varepsilon}^{j}\left(x,y\right)\right)
\end{align*}
built upon a sectorial solution of
\begin{align*}
\onf[\varepsilon]\cdot N_{\ell,\varepsilon}^{j} & =-R_{\varepsilon}
\end{align*}
as in Proposition~\ref{prop:cohomological_conjugacy}.
\item A really important property of the construction: it is performed~\cite[Section 7]{RouTey}
for each \emph{fixed} $\varepsilon\in\mathcal{E}_{\ell}$, the holomorphic~/~continuous
dependence on $\varepsilon$ of resulting objects being a by-product.
This greatly simplifies understanding what happens on overlapping
cells. This is also the reason why we omit to include the subscripts
$\ell$ and $\varepsilon$ in the sequel, whenever doing so does not
introduce ambiguity.
\end{enumerate}
\end{rem}

The period operator $\per[][\ell]$ is obtained as follows. Fix $\varepsilon\in\mathcal{E}_{\ell}$
and $\rho>0$ as in the previous theorem. Starting from any $G\in\holf[\rho\ww D\times{{\neigh}}]'$
we can find a unique collection $\left(F^{j}\right)_{j\in\zsk}\in\prod_{j}\holf[V^{j}\times{{\neigh}}]'$
of bounded functions solving the equation $\onf\cdot F=G$ over a
squid sector. On each intersection we have $\onf[~]\cdot F^{j+1}=G=\onf[~]\cdot F^{j}$
so that $F^{j+1}-F^{j}$ is a first integral of $\onf[~]$. Therefore
it factors as 
\begin{align}
F^{j+1}-F^{j} & =T^{j}\circ H^{j}~~~~,~T^{j}\in\germ h'\label{eq:def_period_from_secto_solutions}
\end{align}
where $H^{j}=H_{\ell,\varepsilon}^{j}$ is the \textbf{canonical sectorial
first integral} with connected fibers
\begin{align}
H^{j} & :=\widehat{H}^{j}\exp N^{j},\label{eq:sectorial_first_integral}
\end{align}
obtained from that of the formal normal form 
\begin{align}
\widehat{H}^{j}\left(x,y\right) & :=y\exp\int^{x}-\frac{1+\mu z^{k}}{P\left(z\right)}\dd z\label{eq:model_first_integral}
\end{align}
by composition with the sectorial normalization (Remark~\ref{rem:corollaries_squid_solve}).
We can fix once and for all a determination of each first integral
$\widehat{H}^{j}=\widehat{H}_{\ell}^{j}$ on $V_{\ell}^{j}$ in such
a way that
\begin{align}
\widehat{H}^{j+1} & =\widehat{H}^{j}\exp\nf{2\ii\pi\mu}k\label{eq:formal_transition_map}
\end{align}
in $\sad V{}{j,}$. The linear factor appearing on the right-hand
side is here to accommodate the multivaluedness of $\exp\int^{x}-\frac{1+\mu z^{k}}{P\left(z\right)}\dd z=x^{-\mu}\times\tx{holo}\left(x\right)$
near $\infty$, so that $\widehat{H}^{j+k}=\widehat{H}^{j}$.
\begin{defn}
\label{def:period_operator}Consider a parameter cell $\mathcal{E}_{\ell}$
and $\rho>0$ as in Theorem~\ref{thm:squid_and_solve}. For $G\in\holb[\mathcal{E}_{\ell}\times\rho\ww D\times{{\neigh}}]$
define the \textbf{period} of $G$ with respect to $\onf$ as the
$k$-tuple 
\begin{align*}
\per[][\ell]\left(G\right) & :=\frac{1}{2\ii\pi}\left(T^{j}\right)_{j\in\zsk}\in\prod_{\zsk}\mathcal{H}_{\ell}\left\{ h\right\} 
\end{align*}
where $T_{\varepsilon}^{j}:=T^{j}$ is build as above in~\eqref{eq:def_period_from_secto_solutions}
for $G:=G_{\varepsilon}$ and $\varepsilon\in\mathcal{E}_{\ell}$.
We define $\per[j][\ell]\left(G\right):=\frac{1}{2\ii\pi}T^{j}$ to
be the $j^{\tx{th}}$ component of $\per[][\ell]\left(G\right)$.
\end{defn}

\begin{rem}
Following up on Remark~\ref{rem:corollaries_squid_solve}~(1), it
seems that the period of $R$ must play a special role regarding classification,
since it measures the discrepancy between sectorial orbital conjugacies
to the formal normal form $\fonf$. It is actually the case that the
unfolded Martinet-Ramis modulus is linked to this period through the
relationship
\begin{align*}
\sad{\psi}{\ell}{j,}\left(h\right)=h\exp\left(\frac{2\ii\pi\mu}{k}+\sad{\phi}{\ell}{j,}\left(h\right)\right)=h\exp\left(\frac{2\ii\pi\mu}{k}-\per[j][\ell]\left(R\right)\left(h\right)\right) & .
\end{align*}
A similar formula holds for the temporal modulus, namely $\sad f{\ell}{j,}=\per[j][\ell]\left(\frac{1}{U}-1\right)$.
We refer to~\cite{RouTey} for a more detailed discussion regarding
these integral representations of the modulus of classification.
\end{rem}

We sum up the relevant results needed in the sequel as a corollary
to Theorem~\ref{thm:squid_and_solve}.
\begin{cor}
\label{cor:characterization_solution_cohomog_fixed_epsilon}Pick $\varepsilon\in\neigh[k]\backslash\Delta_{k}$
and $\rho>0$ such that $P_{\varepsilon}^{-1}\left(0\right)\subset\rho\ww D$,
as well as some holomorphic function $G\in\holf[\rho\ww D\times{{\neigh}}]'$.
The following assertions are equivalent.

\begin{enumerate}
\item There exists $F\in\holf[\rho\ww D\times{{\neigh}}]'$ such that $\onf[\varepsilon]\cdot F=G$.
\item There exists $\ell$ with $\varepsilon\in\mathcal{E}_{\ell}$ such
that
\begin{align*}
\per[][\ell]\left(G\right)_{\varepsilon} & =0.
\end{align*}
\item For all $\ell$ with $\varepsilon\in\mathcal{E}_{\ell}$ we have 
\begin{align*}
\per[][\ell]\left(G\right)_{\varepsilon} & =0.
\end{align*}
\end{enumerate}
If moreover $G\in\mathcal{H}_{\ell}\left\{ x,y\right\} $ then 
\begin{align*}
\lim_{\varepsilon\underset{\mathcal{E}_{\ell}}{\to}0} & \per[][\ell]\left(G\right)_{\varepsilon}=\per\left(G_{0}\right)
\end{align*}
 uniformly on $\neigh$, where $\per~:~\germ{x,y}'\to\prod_{\zsk}\germ h'$
is the period operator of the limiting saddle-node~\cite{Tey-EqH}.
\end{cor}

\begin{proof}
For fixed $\varepsilon$ and $\ell$ Theorem~\ref{thm:squid_and_solve}
asserts the equivalence between existence of an analytic solution
$F$ of the cohomological equation $X_{\varepsilon}\cdot F=G$ and
vanishing of the period $\per[][\ell]\left(G\right)_{\varepsilon}$.
But the analyticity of $F$ has nothing to do with the way the underlying
squid sectors are cut, therefore $\per[][\widetilde{\ell}]\left(G\right)_{\varepsilon}=0$
as soon as $\varepsilon\in\mathcal{E}_{\widetilde{\ell}}$.
\end{proof}

\subsection{\label{subsec:squid_sectors}Description of (unbounded) squid sectors
and parameter cells}

To characterize the dynamics, describe the modulus of analytic classification
and more generally build the period operator, we need to work over
$k$ open \emph{squid sectors} in $x$-space covering either $\rho\disc\setminus P_{\varepsilon}^{-1}(0)$
(bounded case) or $\cc\setminus P_{\varepsilon}^{-1}(0)$ (unbounded
case). Since $\left\{ y=0\right\} $ is an analytic center manifold,
each sector in this paper is the union of two consecutive sectors
described originally in~\cite{RouTey}. The cited reference also
guarantees that it is sufficient to limit ourselves to the complement
of the discriminant hypersurface $\Delta_{k}\ni0$ in parameter space.
Although we only reach parameters for which all roots of $P_{\varepsilon}$
are simple, the construction passes without difficulty to the limit
$\varepsilon\to\Delta_{k}$. For $\varepsilon\notin\Delta_{k}$ the
squid sectors are attached to two or three roots. When $\varepsilon\to0$
they converge to the sectors used in the description of the Martinet-Ramis
modulus for convergent saddle-nodes. 

The singular points depend analytically on $\varepsilon\in\neigh[k]\backslash\Delta_{k}$.
To obtain a family of squid sectors suiting our needs, we must ensure
that the sectors vary continuously as $\varepsilon$ does. This is
however not achievable on a full pointed neighborhood of $\Delta_{k}$
in parameter space, for reasons we are about to explain (we particularly
refer to Remark~\ref{rem:modulus_behavior_near_roots}). Even so,
we manage to deal with all values of $\varepsilon$ by covering the
space $\neigh[k]$ with the closure of finitely many contractible
domains $\left(\mathcal{E}_{\ell}\right)_{\ell}$ in $\varepsilon$-space,
which we call \emph{cells}, on which admissible families of squid
sectors exist.

\subsubsection{The dynamics of $\dot{x}=P_{\varepsilon}(x)$}

Let us recall the main features of the vector field $P_{\varepsilon}\pp x$.
When $P_{\varepsilon}$ has distinct roots $x_{\varepsilon}$, each
singular point $x_{\varepsilon}$ has an associated nonzero eigenvalue
$\lambda_{\varepsilon}=P_{\varepsilon}'(x_{\varepsilon})$.
\begin{itemize}
\item The point $x_{\varepsilon}$ is a \textbf{radial node} if $\lambda_{\varepsilon}\in\rr$.
It is attracting (\emph{resp.} repelling) if $\lambda_{\varepsilon}<0$
(resp. $\lambda_{\varepsilon}>0$). 
\item The point $x_{\varepsilon}$ is a \textbf{center} if $\lambda_{\varepsilon}\in i\rr$. 
\item The point $x_{\varepsilon}$ is a \textbf{focus} if $\lambda_{\varepsilon}\notin\rr\cup\ii\rr$.
It is attracting (\emph{resp.} repelling) if $\re{\lambda_{\varepsilon}}<0$
(resp. $\re{\lambda_{\varepsilon}}>0$). 
\end{itemize}
The point $x=\infty$ serves as an organizing center; indeed, the
vector field $P_{\varepsilon}\pp x$ has a pole of order $k-1$ with
$2k$ separatrices at $x=\infty$, alternately attracting and repelling
(see Figure~\ref{fig:foliation_near_infinity}), thus limiting $2k$
saddle sectors at $\infty$. The system is structurally stable in
the neighborhood of $\infty$ for $\varepsilon$ small. These saddle
sectors give a phase portrait resembling $2k$ petals along the boundary
of any (sufficiently large) disk centered at the origin. The relationship
between the magnitude of the parameter and the size of the disk will
be detailed in Section~\ref{subsec:Size-of-sectors}.

\begin{figure}
\hfill{}\subfloat[Neighborhood of $\infty$ for $k=3$]{\includegraphics[width=4.5cm]{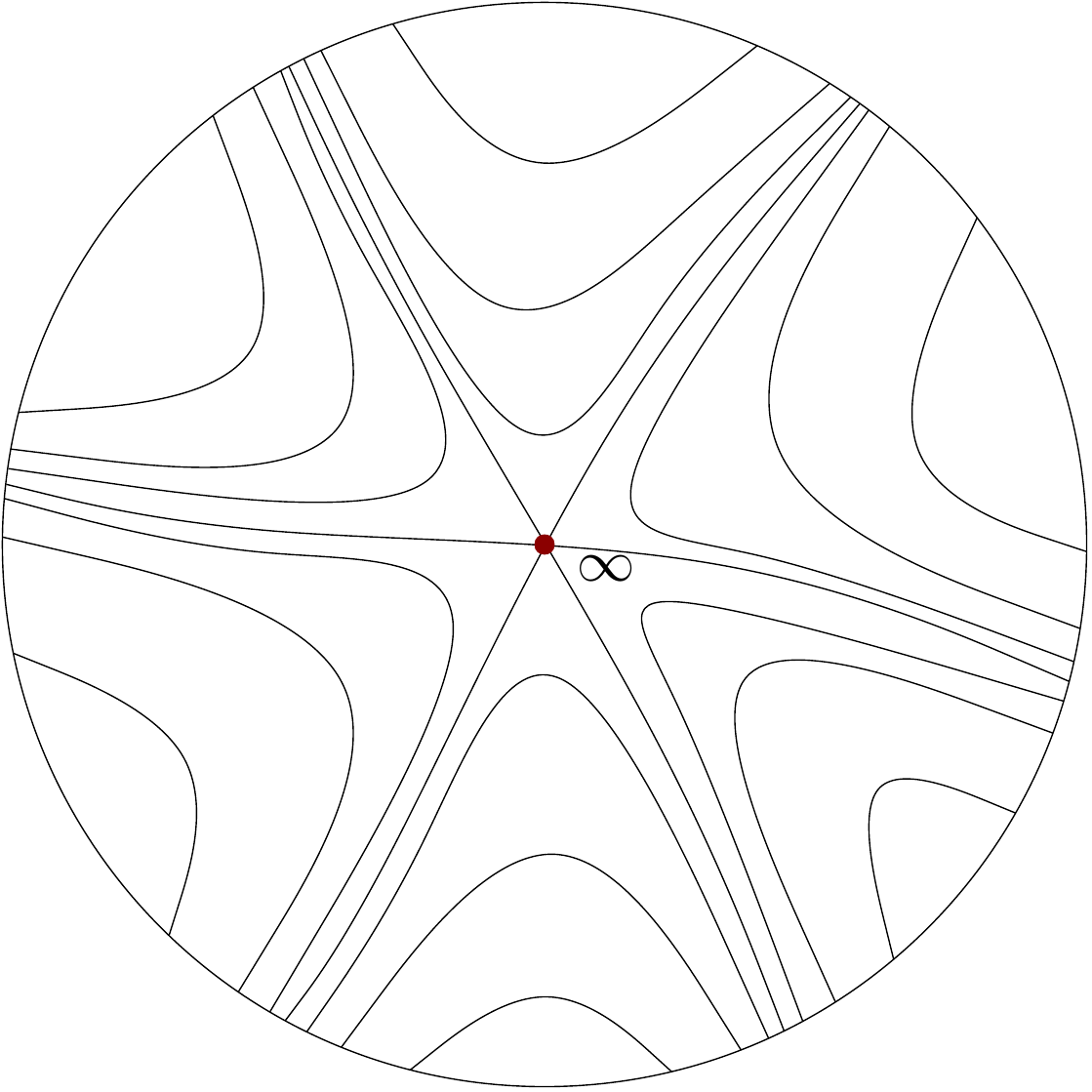}

}\hfill{}\subfloat[The neighborhood of $\rho\protect\disc$ ]{\includegraphics[width=4.5cm]{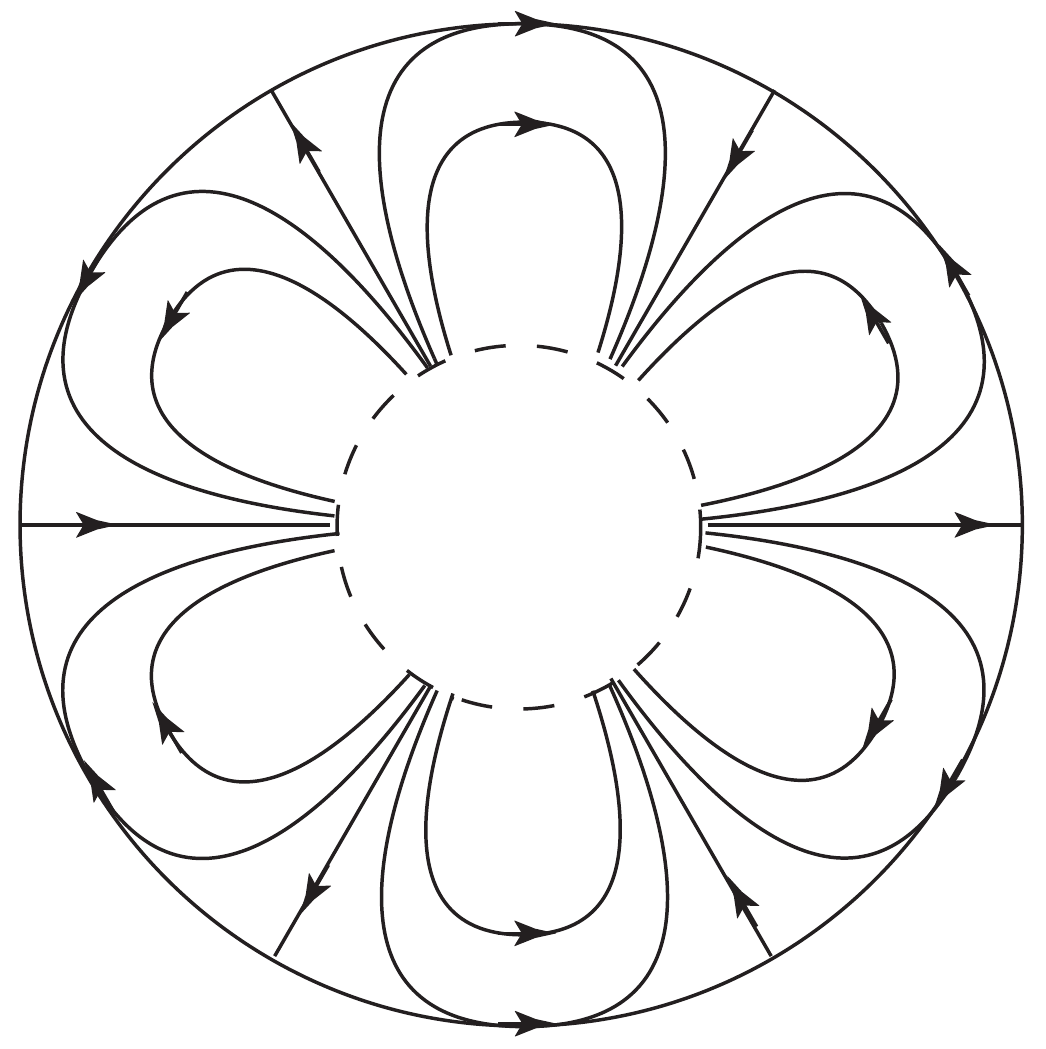}

}\hfill{} \caption{\label{fig:foliation_near_infinity}The separatrices of the pole at
$\infty$ and the petals along the boundary of the disk $\rho\protect\disc$.}
\end{figure}
The dynamics is completely determined by the separatrices of $\infty$.
Because all roots of $P_{\varepsilon}$ are simple, only two types
of behavior occur. 
\begin{itemize}
\item For generic values of $\varepsilon$, following the separatrices from
$\infty$ (either in backward or forward direction) one lands at repelling
($t\to-\infty$) or attracting ($t\to\infty$) singular points $x_{\varepsilon}$
of focus or radial node type. In that case, each singular point is
attached to at least one separatrix and the system is structurally
stable among polynomial systems of degree $k+1$. See Figure~\ref{fig:example_poly_foliation}
for a phase portrait with generic $\varepsilon$.
\item The sets of generic $\varepsilon$ are separated by bifurcation hypersurfaces
of (real) codimension $1$. For these non-generic values of $\varepsilon$
a homoclinic connection occurs between an attracting separatrix and
a repelling separatrix of infinity: there is then a real integral
curve flowing out from infinity in the $x$-plane and flowing back
to infinity in finite time. For these bifurcation sets, the singular
points $P_{\varepsilon}^{-1}\left(0\right)$ can be split into two
(nonempty) subsets $I_{1}$ and $I_{2}$ satisfying 
\begin{equation}
\sum_{x\in I_{m}}\frac{1}{P_{\varepsilon}'(x)}\in\ii\rr,\qquad m=1,~2.\label{eq:bifurcation_locus}
\end{equation}
This can be seen by integrating the $1$-form $\dd t=\frac{\dd x}{P_{\varepsilon}\left(x\right)}$
along a homoclinic orbit, and evaluating residues. When $I_{m}$ is
a singleton, the corresponding singular point is a center. 
\end{itemize}
\begin{figure}
\hfill{}\includegraphics[width=6cm]{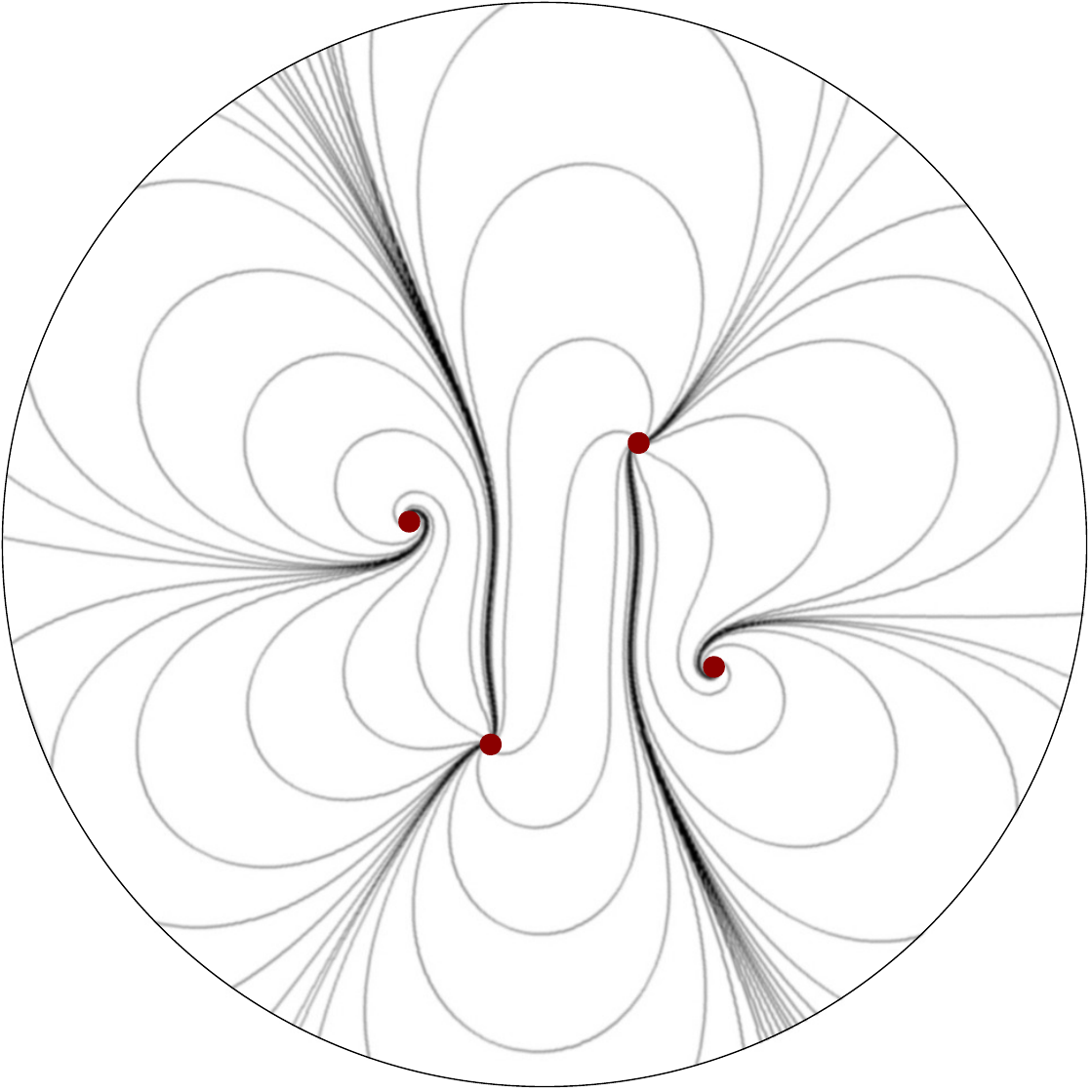}\hfill{}

\caption{\label{fig:example_poly_foliation}An example of a structurally stable
real foliation induced by a complex polynomial vector field of degree
$4$ for $\varepsilon$ in some $K_{\ell}$.}
\end{figure}
The union of the $2k$ separatrices of $\infty$ is called the \textbf{separating
graph} in \cite{DES} (see Figure~\ref{fig:skeleton}(A)). It splits
$\cc$ into $k$ simply connected regions. In each of these regions
we can draw a curve $\gamma_{j}$ connecting the interior of a saddle
sector at $\infty$ to the interior of another saddle sector (see
Figure~\ref{fig:skeleton}(B)). There are exactly $C_{k}=\frac{1}{k+1}\binom{2k}{k}$
ways of pairing two by two the saddle sectors of $\infty$ by non-intersecting
curves, thus providing a topological invariant for the vector field
(we also refer to~\cite{DiaCata}). 

\begin{figure}
\hfill{}\subfloat[\foreignlanguage{english}{The separating graph}]{\includegraphics[width=5.5cm]{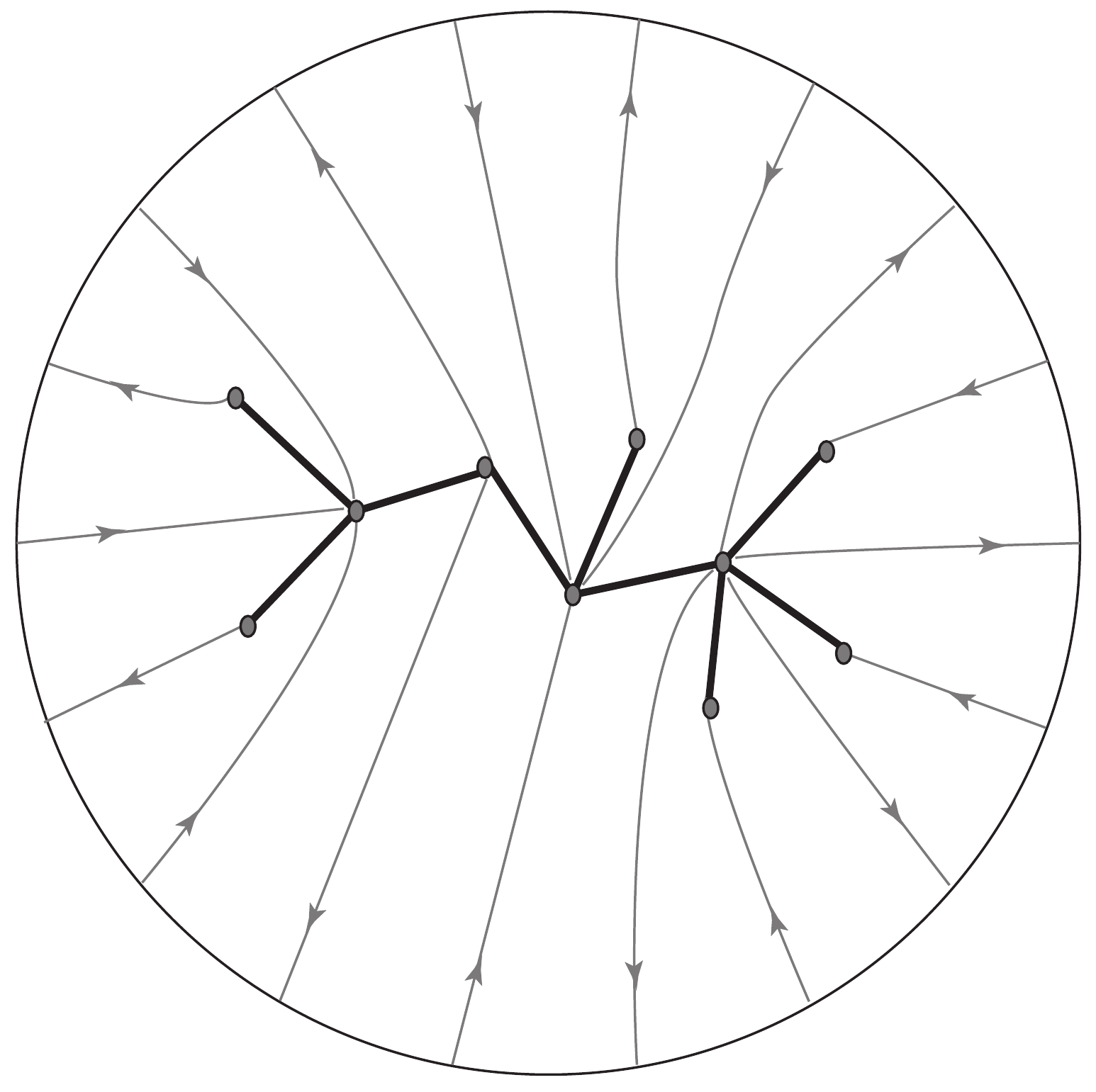}

}\hfill{}\subfloat[\foreignlanguage{english}{The curves $\gamma_{i}$ }]{\includegraphics[width=5.5cm]{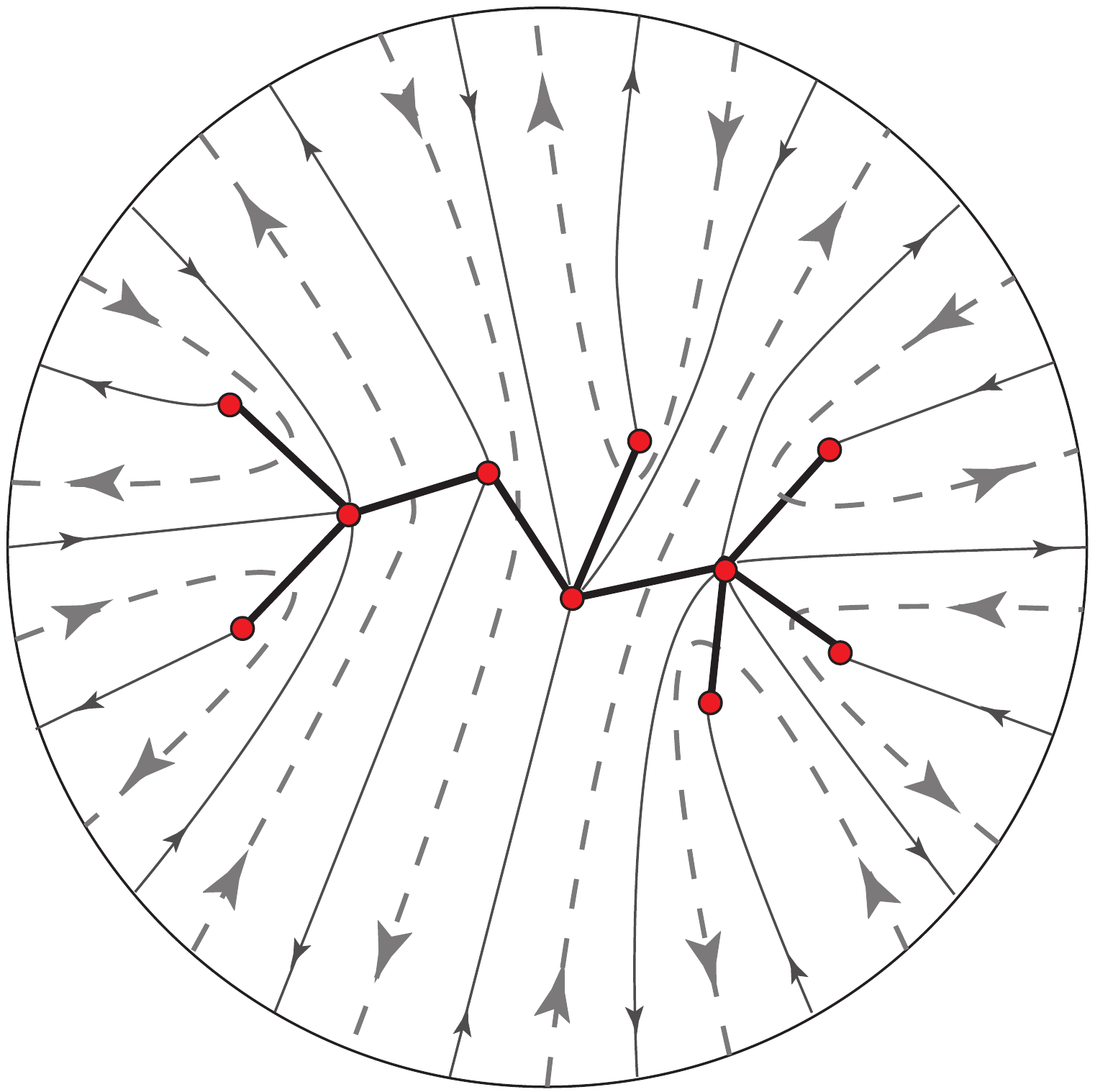}

}\hfill{}

\caption{\label{fig:skeleton}On the left, the separating graph formed by the
separatrices landing at the singular points and flow lines (in bold)
between the singular points. On the right, the curves $\gamma_{i}$
(in dotted lines) used to calculate the $\tau_{i}$.}
\end{figure}

\subsubsection{Rough description of the cells $\mathcal{E}_{\ell}$ in parameter
space}

The non-generic values of $\varepsilon$ form a set of (real) codimension
$1$ which partitions a convenient neighborhood of $0$ in parameter
space (to be described slightly later) into $C_{k}$ open regions
$K_{\ell}$, corresponding to structurally stable vector fields  with
same topological invariant. In each region $K_{\ell}$, the topology
of the phase portrait is completely determined by the topological
way of attaching the $2k$ separatrices to the $k+1$ singular points.
If $x_{\varepsilon}$ is a root of $P_{\varepsilon}$ (depending continuously
on $\varepsilon$) then $\re{P_{\varepsilon}'(x_{\varepsilon})}$
has a constant sign for all $\varepsilon\in K_{\ell}$. Each cell\textbf{
}$\mathcal{E}_{\ell}$ in parameter space will be a small enlargement
of $K_{\ell}$, so that the cells cover the complement of $\Delta_{k}$.

A.~\noun{Douady} and P.~\noun{Sentenac} have also provided a very
clever parametrization of the domains $K_{\ell}$, thus showing that
they are simply connected.
\begin{thm}
\label{thm:DES_cells}\cite{DES} Let $K_{\ell}$ be a maximal domain
corresponding to structurally stable vector fields. Then, there exists
a biholomorphism $\Phi_{\ell}:K_{\ell}\rightarrow\ww H^{k}$, where
$\ww H$ is the upper half-plane. In particular, $K_{\ell}$ is contractible.
The set $\Phi_{\ell}^{-1}\left(\left(\ii\rr_{\geq0}\right)^{n}\right)$,
which we call the \textbf{spine} of $K_{\ell}$, corresponds to polynomial
vector fields with real eigenvalues at each singular point.

The map $\Phi_{\ell}$ is defined as follows: let $\left(\gamma_{\varepsilon}^{j}\right)_{j\in\left\{ 1,\dots,k\right\} }$
be $k$ disjoint loops attached to $\infty$ and pairing the saddle
sectors of $\infty$, without intersecting the separating graph. Then
$\Phi_{\ell}(\varepsilon)=(\tau_{\varepsilon}^{1},\dots,\tau_{\varepsilon}^{k})$,
where 
\begin{align*}
\tau_{\varepsilon}^{j}: & =\int_{\gamma_{\varepsilon}^{j}}\dd t=\int_{\gamma_{\varepsilon}^{j}}\frac{\dd x}{P_{\varepsilon}(x)},
\end{align*}
the orientation of $\gamma_{\varepsilon}^{j}$ being chosen so that
$\im{\tau_{\varepsilon}^{j}}>0$.

Since $\tau_{\varepsilon}^{j}=2\ii\pi\sum_{x\in I}\frac{1}{P_{\varepsilon}'(x)}$,
where $I$ is the set of singular points in a domain bounded by $\gamma_{\varepsilon}^{j}$,
the sum $\tau_{\varepsilon}^{j}$ admits an analytic continuation
outside $K_{\ell}$. In particular, when $\varepsilon$ is a boundary
point of $K_{\ell}$ for which there is a homoclinic loop through
$\infty$, some of the $\tau_{\varepsilon}^{j}$ become real.
\end{thm}

\bigskip{}

The cells have a very useful conic structure, induced by a multiplicative
action of $\rr_{>0}\ni\lambda$ through linear rescaling 
\begin{align}
\left(\varepsilon_{k-1},\dots,\varepsilon_{0},x,t\right) & \longmapsto\left(\lambda^{-(k-2)}\varepsilon_{k-1},\dots,\varepsilon_{1},\lambda\varepsilon_{0},\lambda x,\lambda^{-k}t\right),\label{eq:time_rescalling}
\end{align}
as indeed the differential equation $\dot{x}=P_{\varepsilon}(x)$
is invariant under this action. The cones we use are of the form 
\[
\left\{ \left(\lambda^{2}\varepsilon_{k-1},\dots,\lambda^{k}\varepsilon_{1},\lambda^{k+1}\varepsilon_{0}\right)~:~\lambda\in\left]0,1\right[,~\varepsilon\in K\right\} ,
\]
where $K$ is a relative domain within a sphere-like real hypersurface.
This compact hypersurface takes the form $\left\{ \norm[\varepsilon]{}=\cst\right\} $
with
\begin{equation}
\norm[\varepsilon]{}:=\max\left(|\varepsilon_{k-1}|^{\frac{1}{2}},\dots,|\varepsilon_{1}|^{\frac{1}{k}},|\varepsilon_{0}|^{\frac{1}{k+1}}\right).\label{eq:norm_eps}
\end{equation}

\begin{rem}
The expression~\ref{eq:norm_eps} does not define a norm because
the homogeneity axiom is not satisfied. However, if we take into account
that the $\varepsilon_{j}$ are the symmetric functions in the roots
$\left(x_{0},\dots,x_{k}\right)\in\cc^{k+1}$ of $P_{\varepsilon}$,
it lifts to a norm on $\cc^{k+1}$. Thus $\norm[\varepsilon]{}$ measures
the magnitude of the parameter $\varepsilon$ and the $\norm[\bullet]{}$-balls
form a fundamental basis of neighborhood of $0$. In the following
we consider only these parametric neighborhoods.
\end{rem}

The regions $K_{\ell}$ of structural stability defined above are
cones of this form, and so will be their enlargements to cells $\mathcal{E}_{\ell}$
covering the complement of $\Delta_{k}$. Also, when considering limits
for $\varepsilon\to0$ it will be natural to consider limits for $\lambda\to0$
along orbits of the $\rr_{>0}$-action 
\begin{equation}
\left\{ \left(\lambda^{2}\varepsilon_{k-1},\dots,\lambda^{k}\varepsilon_{1},\lambda^{k+1}\varepsilon_{0}\right))~:~\lambda\in\left]0,1\right[\right\} .\label{eq:curve}
\end{equation}

\subsubsection{\label{subsec:admissible_angles}Saddle- and node-like singular points,
admissible angles}

We want to stress that a singular point $x_{\varepsilon}$ of $\dot{x}=P_{\varepsilon}(x)$
with non-real eigenvalue $\lambda=a+\ii b$ can be both attracting
and repelling depending on how we approach it along logarithmic spirals.
Making sense of this statement entails complexifying the time. Let
us explain how.
\begin{itemize}
\item Consider the linear equation $\dot{x}=\lambda x$. Its solutions are
$x(t)=x_{0}\exp\left(\lambda t\right)$. Now, let us allow complex
values of $t$ along some slanted real line $t=(c+\ii d)T=T\exp\left(\ii\theta\right)$
in $\cc$-space for some fixed $c+\ii d\in\mathbb{S}^{1}$, with $c>0$
(corresponding to $\theta\in\left]-\frac{\pi}{2},\frac{\pi}{2}\right[$)
and $T\in\rr$. Then 
\begin{align*}
x\left(t\left(T\right)\right) & =x_{0}\exp\left(\left(\left(ac-bd\right)+\ii\left(ad+bc\right)\right)T\right),
\end{align*}
and $\lim_{T\to+\infty}x(t(T))=0$ (\emph{resp}. $\lim_{T\to-\infty}x(t(T))=0$)
when $ac-bd<0$ (\emph{resp}. $ac-bd>0$). 
\item Since $b\neq0$, it is always possible to find $c_{1}>0$, $d_{1}$
(\emph{resp}. $c_{2}>0$, $d_{2}$ ) such that $ac_{1}-bd_{1}>0$
(\emph{resp}. $ac_{2}-bd_{2}<0$).
\item Note that approaching the singular point along a line $t=(c+id)T$
in $t$-space is the same as approaching it along a real trajectory
of the rotated equation $\ddd xT=\lambda\exp\left(\ii\theta\right)\times x$.
Such a trajectory is a logarithmic spiral.
\item All these properties hold for the original system too, since the vector
field $P_{\varepsilon}\pp x$ is analytically linearizable near the
singular point (Poincaré's theorem).
\end{itemize}
Locally around each root $x_{\varepsilon}$ the squid sectors will
coincide with domains bounded by asymptotic logarithmic spirals, given
by trajectories of rotated vector fields $\exp\left(\ii\theta_{\varepsilon}\right)P_{\varepsilon}\pp x$.
The angular function $\left(\varepsilon,x\right)\in\mathcal{E}_{\ell}\times\cc\mapsto\theta_{\varepsilon}\left(x\right)\in\left]-\frac{\pi}{4},\frac{\pi}{4}\right[$
will be piecewise constant and zero outside a neighborhood of $\partial K_{\ell}$,
and for $x$ far from the singular points.
\begin{defn}
\label{def:node-saddle_root}Let $\mathcal{E}$ be a domain in the
complement of $\Delta_{k}$. 
\begin{enumerate}
\item An \textbf{admissible angle} on $\mathcal{E}$ is a piecewise constant
function $\theta~:~\mathcal{E}\times\cc\to\left]-\frac{\pi}{4},\frac{\pi}{4}\right[$
such that for any analytic family of roots $\left(x_{\varepsilon}\right)_{\varepsilon\in\mathcal{E}}$
of $\left(P_{\varepsilon}\right)_{\varepsilon\in\mathcal{E}}$, the
function $\varepsilon\in\mathcal{E}\mapsto\re{P_{\varepsilon}'\left(x_{\varepsilon}\right)\exp\left(\ii\theta_{\varepsilon}\left(x_{\varepsilon}\right)\right)}$
keeps a constant sign. In the following we use the notation
\begin{align}
\vartheta & :=\exp\left(\ii\theta\right).\label{eq:vartheta}
\end{align}
\item We say that an analytic family $\left(x_{\varepsilon}\right)_{\varepsilon\in\mathcal{E}}$
of singular points of $\dot{x}=P_{\varepsilon}(x)$ is of \textbf{node
type} on $\mathcal{E}$ if there exists an admissible angle such that
\begin{align*}
\re{P_{\varepsilon}'\left(x_{\varepsilon}\right)\vartheta_{\varepsilon}\left(x_{\varepsilon}\right)} & >0~~~~~\left(\forall\varepsilon\in\mathcal{E}\right)
\end{align*}
and of \textbf{saddle type} on $\mathcal{E}$ if
\begin{align*}
\re{P_{\varepsilon}'\left(x_{\varepsilon}\right)\vartheta_{\varepsilon}\left(x_{\varepsilon}\right)} & <0~~~~~\left(\forall\varepsilon\in\mathcal{E}\right).
\end{align*}
We use the notation $\left(\nod x{\varepsilon}{}\right)_{\varepsilon}$
(\emph{resp.} $\left(\sad x{\varepsilon}{}\right)_{\varepsilon}$)
for a family of roots of node (\emph{resp.} saddle) type on the domain
$\mathcal{E}$.
\end{enumerate}
\end{defn}

\begin{rem}
\label{rem:modulus_behavior_near_roots}

\begin{enumerate}
\item The cells $\mathcal{E}_{\ell}$ in parameter space will be small contractible
enlargements of the cones $K_{\ell}$, on which there exist admissible
angles. Additional constraints will be demanded to these angular functions
in order to guarantee that the cells and sectors meet all technical
requirements.
\item The choice of $\frac{\pi}{4}$ for an upper bound in the size of an
admissible angle $\theta$ is arbitrary as any bound $\alpha\in\left]0,\frac{\pi}{2}\right[$
would do. However the larger $\alpha$, the smaller the bound $\rho$
on $\norm[\varepsilon]{}$. Indeed we approach each singular point
along a trajectory of some vector field $\vartheta_{\varepsilon}\left(x\right)P_{\varepsilon}\left(x\right)$.
When $\theta$ is large and the singular points are not far enough
from $r\sone$, the trajectory follows wide spirals and may escape
$r\ww D$ before landing at the singular point. An ``absolute''
(\emph{i.e.} independent of the bound $\alpha$) necessary condition
for the existence of an admissible angle such that $\left(x_{\varepsilon}\right)_{\varepsilon}$
has node- (\emph{resp. }saddle-) type on a neighborhood of $K_{\ell}$
is that $P_{\varepsilon}'\left(x_{\varepsilon}\right)\notin\rr_{<0}$
(\emph{resp}. $P_{\varepsilon}'\left(x_{\varepsilon}\right)\notin\rr_{>0}$)
for $\varepsilon\in K_{\ell}$. Therefore no admissible angle exists
on a full pointed neighborhood of $\Delta_{k}$. 
\item We can illustrate on the formal normal form why admissible angles
are of paramount importance. In the flow system of $\vartheta\fonf$
for real time 
\begin{align*}
\begin{cases}
\dot{x} & =\vartheta_{\varepsilon}\left(x\right)P_{\varepsilon}\left(x\right)\\
\dot{y} & =\vartheta_{\varepsilon}\left(x\right)y\left(1+\mu_{\varepsilon}x^{k}\right)
\end{cases}
\end{align*}
the variation of the modulus $\phi:=\left|y\right|^{2}=y\overline{y}$
of a solution follows the law 
\begin{align*}
\dot{\phi} & =2\phi\re{\vartheta_{\varepsilon}\left(x\right)\left(1+\mu_{\varepsilon}x^{k}\right)}.
\end{align*}
Close enough to the singularity $\left(x_{\varepsilon},0\right)$
all non-zero solutions therefore accumulate backwards exponentially
fast on $\left(x_{\varepsilon},0\right)$ if $x_{\varepsilon}=\nod x{\varepsilon}{}$
is of node type or, on the contrary, diverge forwards exponentially
fast for a saddle type root $\sad x{\varepsilon}{}$. This behavior
mimics that of a node~/~saddle planar foliation near a point with
real residue $\vartheta_{\varepsilon}\left(x_{\varepsilon}\right)P_{\varepsilon}'\left(x_{\varepsilon}\right)$.
This dynamical dichotomy is the cornerstone of the construction of
the period operator (the modulus of classification) in~\cite{RouTey}.
\end{enumerate}
\end{rem}

\subsubsection{\label{subsec:Size-of-sectors}Size of sectors and of the parameter}

The diameter $\rho$ of the bounded part of the sectors is such that
$|1+\mu x^{k}|>\frac{1}{2}$ when $\left|x\right|<\rho$. Note that
the roots of $P_{\varepsilon}$ all lie within $\sqrt{k}\norm[\varepsilon]{}\adh{\ww D}$.
Indeed it suffices to show that if $\left|x\right|>\sqrt{k}\norm[\varepsilon]{}$,
then $P_{\varepsilon}(x)\neq0$. On the one hand $\left|x^{k+1}\right|>k^{\frac{k+1}{2}}\norm[\varepsilon]{}^{k+1}$.
On the other hand 
\[
\left|\sum_{j=0}^{k-1}\varepsilon_{j}x^{j}\right|\leq\norm[\varepsilon]{}^{k+1}\sum_{j=0}^{k-1}k^{\frac{j}{2}}\leq\norm[\varepsilon]{}^{k+1}k^{\frac{k+1}{2}}.
\]

In fact outside the disk $\sqrt{k}\norm[\varepsilon]{}\adh{\disc}$
the trajectories of $P_{\varepsilon}\pp x$ are petals as depicted
in Figure~\ref{fig:foliation_near_infinity}~(\noun{B}). Set 
\begin{equation}
\rho_{\varepsilon}:=2\sqrt{k}\norm[\varepsilon]{}.\label{eq:rho_eps}
\end{equation}
Then we choose $\varepsilon$ sufficiently small so that $\rho_{\varepsilon}<\frac{\rho}{2}$.
Later in Lemma~\ref{lem:integrable_first-integral} we will further
reduce $\rho$ and $\varepsilon$ so that 
\begin{align}
\left|\mu x^{k}\right|+2\rho\left|P''\left(x\right)\right| & \leq\frac{3}{4}\label{eq:polynomial_infty_estim}
\end{align}
 for $|x|<\rho$.

\subsubsection{The ideal construction of sectors}

Let us now choose a cone $K_{\ell}$ and describe the corresponding
open squid sectors $\left(V_{\ell,\varepsilon}^{j}\right)_{j\in\zsk}$
covering $\rho\disc\setminus P_{\varepsilon}^{-1}(0)$. On a ``large''
neighborhood of the spine of $K_{\ell}$ (to be made precise below),
\emph{i.e.} not too close to the boundary of $K_{\ell}$, they are
limited by real trajectories of $P_{\varepsilon}\pp x$ chosen as
follows (see also Figure~\ref{fig:ideal-decomposition}).

\begin{figure}[H]
\hfill{}\includegraphics[width=8cm]{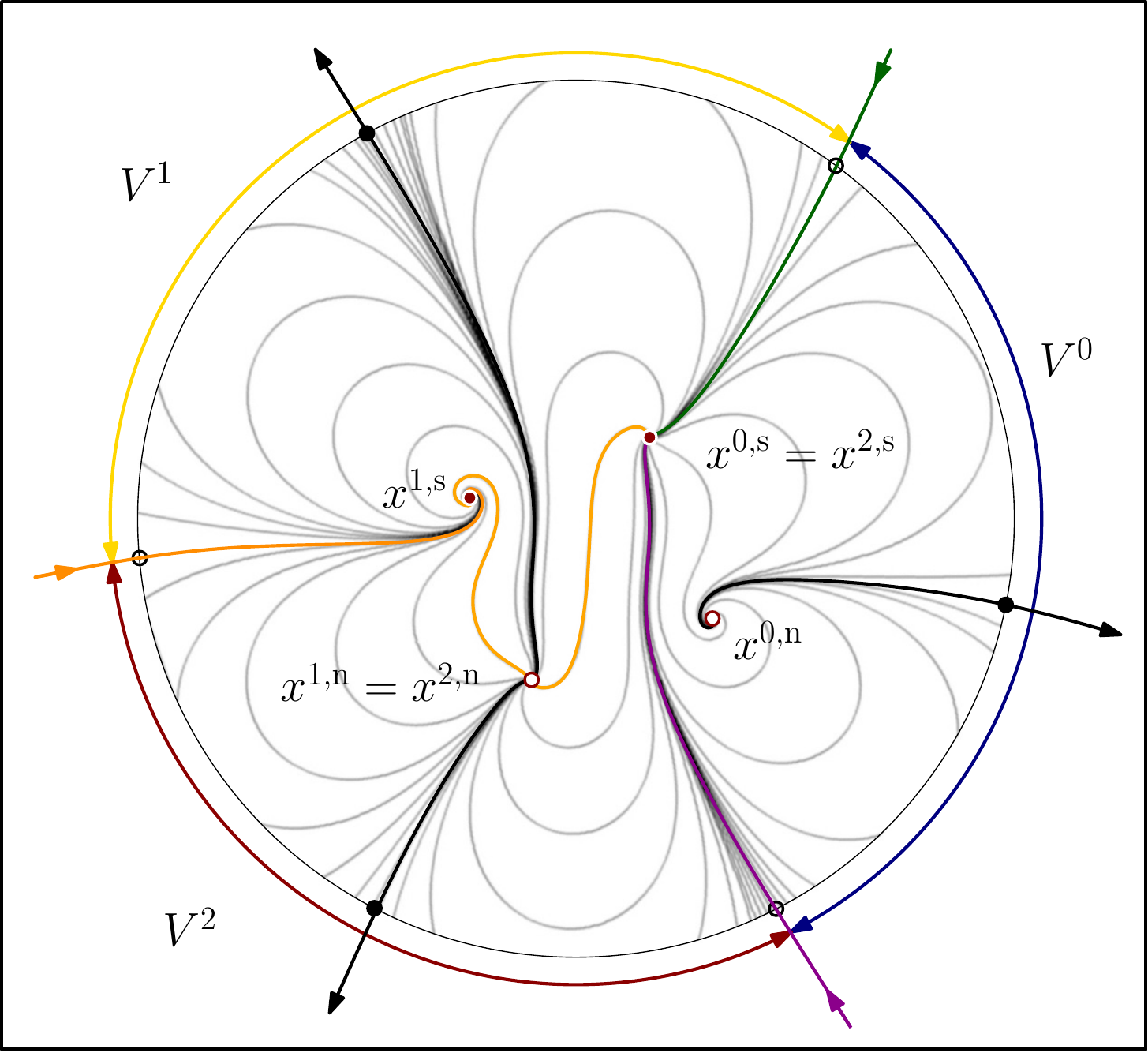}\hfill{}

\caption{\label{fig:ideal-decomposition}Curves involved in the ideal decomposition.
Stable separatrices at $\infty$ in black, unstable ones in green,
orange and purple. }
\end{figure}
\begin{enumerate}
\item The unstable separatrices of $P_{\varepsilon}\pp x$ through $\infty$
split $\rho\sone$ into $k$ arcs. We enlarge slightly these arcs
to an open covering of the circle. Each arc is one piece of the boundary
of a sector $V_{\varepsilon}^{j}$. 
\item Two other pieces of the boundary of $V_{\varepsilon}^{j}$ are given
by the forward trajectories of $P_{\varepsilon}\pp x$ through the
endpoints of the arc, which land in singular points $\sad x{}{j-1,}$
and $\sad x{}{j,}$ (not necessarily distinct) such that $\re{P_{\varepsilon}'\left(\sad x{}{j,}\right)}<0$
(\emph{i.e.} the roots are of saddle type). These trajectories spiral
as soon as $\im{P_{\varepsilon}'\left(\sad x{}{j,}\right)}\neq0$
(which is the generic situation). 
\item Suppose $\sad x{}{j,}\neq\sad x{}{j-1,}$. For a given boundary arc
of $\rho\sone$ there exists one stable separatrix through $\infty$
which cuts it at one point and lands at root $\nod x{}{j,}$ of node
type. This singular point belongs to the boundary of $V_{\varepsilon}^{j}$.
The last two pieces of the boundary are two complete trajectories
of $P_{\varepsilon}\pp x$, one joining $\nod x{}{j,}$ to $\sad x{}{j-1,}$
and the other joining $\nod x{}{j,}$ to $\sad x{}{j,}$. These trajectories
are chosen in such a way that $\left(V_{\varepsilon}^{j}\right)_{j\in\zsk}$
cover $\rho\disc\setminus P_{\varepsilon}^{-1}(0)$. 
\item When $\sad x{}{j,}=\sad x{}{j-1,}$, we introduce two trajectories
between $\sad x{}{j,}$ and $\nod x{}{j,}$, thus introducing a self-intersection
of $V_{\varepsilon}^{j}$. This is motivated by the need of dealing
with ramified functions near $\nod x{}{j,}$. See Figure~\ref{fig:Non-equivalent-decompositions}~(A).
\end{enumerate}
\begin{figure}[h]
\hfill{}\includegraphics[width=8cm]{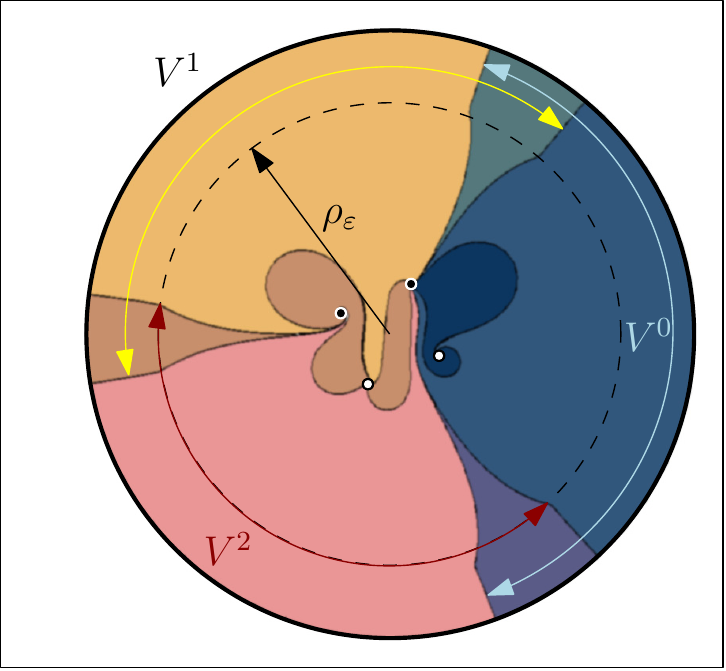}\hfill{}

\caption{\foreignlanguage{english}{\label{fig:Squid_k=00003D3}\foreignlanguage{american}{Decomposition
into bounded, overlapping squid sectors induced by the flow depicted
in Figures~\ref{fig:example_poly_foliation} and~\ref{fig:ideal-decomposition}.}}}
\end{figure}

\subsubsection{The problem with the ideal construction of sectors}

Of course the ideal construction will not always work. It can fail
for the following reasons. For a set $I\subset P_{\varepsilon}^{-1}\left(0\right)$
and $\varepsilon\notin\Delta_{k}$ define
\begin{align}
\nu_{\varepsilon}\left(I\right) & :=\sum_{x\in I}\frac{1}{P_{\varepsilon}'(x)}.\label{eq:subset_residue}
\end{align}

\begin{itemize}
\item The first one is when $\varepsilon$ is not generic: the separatrices
may form a homoclinic loop preventing them to land at singular points.
A homoclinic loop $\gamma$ through $\infty$ partitions the set of
singular points $P_{\varepsilon}^{-1}\left(0\right)$ into $I$ and
$I'$ such that 
\begin{equation}
\re{\nu_{\varepsilon}\left(I\right)}=\re{\nu_{\varepsilon}\left(I'\right)}=0.\label{eq:homoclinic}
\end{equation}
\item When $\varepsilon$ is close to a hypersurface corresponding to a
homoclinic loop, it can also occur that the trajectories through the
endpoints of the arc first exit the disk $\rho\ww D$ before landing
at a singular point. 
\item When $\varepsilon$ crosses a hypersurface corresponding to a homoclinic
loop, then $\re{P_{\varepsilon}'(x_{\varepsilon})}$ can change sign,
thus preventing the above construction to be continuous in $\varepsilon\in\mathcal{E}_{\ell}$. 
\item As $\varepsilon$ approaches $0$ (or, more generally, $\Delta_{k}$)
we would like the sectors to converge to usual sectors associated
to saddle-node singularities.
\end{itemize}

\subsubsection{The remedy in the construction of sectors}

The remedy to all these problems is the same. We want to keep the
above picture all over the cell $\mathcal{E}_{\ell}$ and we want
the cells to cover the complement of $\Delta_{k}$. The boundary of
$K_{\ell}$ is composed of real hypersurfaces corresponding to homoclinic
loop bifurcations. On each such hypersurface we have~\eqref{eq:homoclinic}
for some $I$, while on $K_{\ell}$ the real part of the corresponding
$\nu_{\varepsilon}\left(I\right)$ has a fixed sign and so does $\im{\tau_{j}}$.
But we have seen in Section~\ref{subsec:admissible_angles} that
this is not an obstruction for having the points remaining of node-
or saddle-type: we just need to be sufficiently careful on how we
approach them, by adjusting the spiraling of the sectors. In practice,
this boils down to replacing the piece of a trajectory of $P_{\varepsilon}\pp x$
inside the disk $\rho_{\varepsilon}\disc$ by the piece of a trajectory
of $\exp\left(\ii\theta\right)\times P_{\varepsilon}\pp x$ for some
admissible angle $\theta$ as in Definition~\ref{def:node-saddle_root}
(with some additional specifications).
\begin{prop}
\label{prop:cells_construction}Being given $\delta\in\left]0,\frac{\pi}{4}\right[$
and $\rho>0$, there exists $\eta>0$ such that the following properties
hold.
\begin{enumerate}
\item Let $\mathcal{E}_{\ell}$ be the open set in $\left\{ \norm[\varepsilon]{}<\eta\right\} \backslash\Delta_{k}$
defined by the next conditions:
\begin{itemize}
\item for each homoclinic-loop bifurcation hypersurface on the boundary
of $K_{\ell}$, separating the singular points in two nonempty groups
$I\cup I'$ as in~\eqref{eq:homoclinic}, we have
\begin{align*}
\begin{cases}
\arg\nu_{\varepsilon}\left(I\right)\in\left]-\frac{\pi}{2}-\delta,\frac{\pi}{2}+\delta\right[~~~~~ & \text{if}\quad\re{\nu_{\varepsilon}\left(I\right)}>0\:\:\text{on}\:\:K_{\ell},\\
\arg\nu_{\varepsilon}\left(I\right)\in\left]\frac{\pi}{2}-\delta,\frac{3\pi}{2}+\delta\right[~~~~~ & \text{if}\quad\re{\nu_{\varepsilon}\left(I\right)}<0\:\:\text{on}\:\:K_{\ell},
\end{cases}
\end{align*}
\item for the $\tau_{\varepsilon}^{j}$ defined in Theorem~\ref{thm:DES_cells}
we have
\begin{align*}
\arg\tau_{\varepsilon}^{j}\in\left]-\delta,\pi+\delta\right[~~~~ & \text{for all }j\in\left\{ 1,\dots,k\right\} .
\end{align*}
\end{itemize}
Then $\mathcal{E}_{\ell}$ is a conic contractible neighborhood of
$K_{\ell}$ and $\bigcup_{\ell}\mathcal{E}_{\ell}=\left\{ \norm[\varepsilon]{}<\eta\right\} \backslash\Delta_{k}$.
\item There exists an admissible angle $\theta$ (corresponding to a direction
$\vartheta=\exp\left(\ii\theta\right)$) on $\mathcal{E}_{\ell}$
such that for each homoclinic-loop bifurcation hypersurface on the
boundary of $K_{\ell}$, separating the singular points in two nonempty
groups $I\cup I'$, we have
\begin{equation}
\begin{cases}
\arg\sum_{x\in I}\frac{1}{\vartheta_{\varepsilon}\left(x\right)P_{\varepsilon}'(x)}\in\left]-\frac{\pi}{2}+\delta,\frac{\pi}{2}-\delta\right[~~~~~ & \text{if}\quad\re{\nu_{\varepsilon}\left(I\right)}>0\:\:\text{on}\:\:K_{\ell},\\
\arg\sum_{x\in I}\frac{1}{\vartheta_{\varepsilon}\left(x\right)P_{\varepsilon}'(x)}\in\left]\frac{\pi}{2}+\delta,\frac{3\pi}{2}-\delta\right[~~~~~ & \text{if}\quad\re{\nu_{\varepsilon}\left(I\right)}<0\:\:\text{on}\:\:K_{\ell}.
\end{cases}\label{eq:choice_delta}
\end{equation}
\item Any trajectory of $\vartheta_{\varepsilon}P_{\varepsilon}\pp x$,
starting from a point of $\rho\sone$, does not exit the disk $\rho\ww D$
before landing at a singular point.
\end{enumerate}
\end{prop}

\begin{proof}
~
\begin{enumerate}
\item This is clear.
\item We build the angle $\theta$ (piece-wise constant in $x$) in such
a way that $\vartheta_{\varepsilon}=1$ when $\varepsilon$ is on
the spine of $K_{\ell}$. When we approach a component of $\partial K_{\ell}$
corresponding to a homoclinic loop separating the roots of $P_{\varepsilon}$
as $I\cup I'$, we can rotate the vector field by an angle $\left|\theta\right|\leq2\delta<\frac{\pi}{4}$
so that $\arg\sum_{x\in I}\frac{1}{\vartheta_{\varepsilon}\left(x\right)P_{\varepsilon}'(x)}$
belongs to the given interval. 
\item A more precise quantitative description of the sectors is needed to
show that the magnitude of $\rho$ (in $x$-space) together with the
choice of $\delta$ give constraints on the size $\eta$ of the $\norm[\bullet]{}$-ball
in $\varepsilon$-space, and that taking $\left|\theta\right|$ large
enough is sufficient to secure the conclusion. All this is done in
the time coordinate $t=\int\frac{\dd x}{P_{\varepsilon}(x)}$. We
come back to this below in Section~\ref{subsec:Practical-description}. 
\end{enumerate}
\end{proof}
\begin{defn}
\label{def:cells_and_squids}~
\begin{enumerate}
\item The contractile, conic domain $\mathcal{E}_{\ell}$ given by the previous
proposition is called a \textbf{cell} in parameter space.
\item The $k$ domains in $x$-space built like ideal sectors but bounded
by trajectories of $\vartheta_{\varepsilon}P_{\varepsilon}\pp x$
instead of $P_{\varepsilon}\pp x$ are called \textbf{squid sectors}.
\end{enumerate}
\end{defn}

\subsubsection{\label{subsec:Non-crossing-permutation}Pairing sectors}

\begin{figure}[H]
\hfill{}\includegraphics[width=8cm]{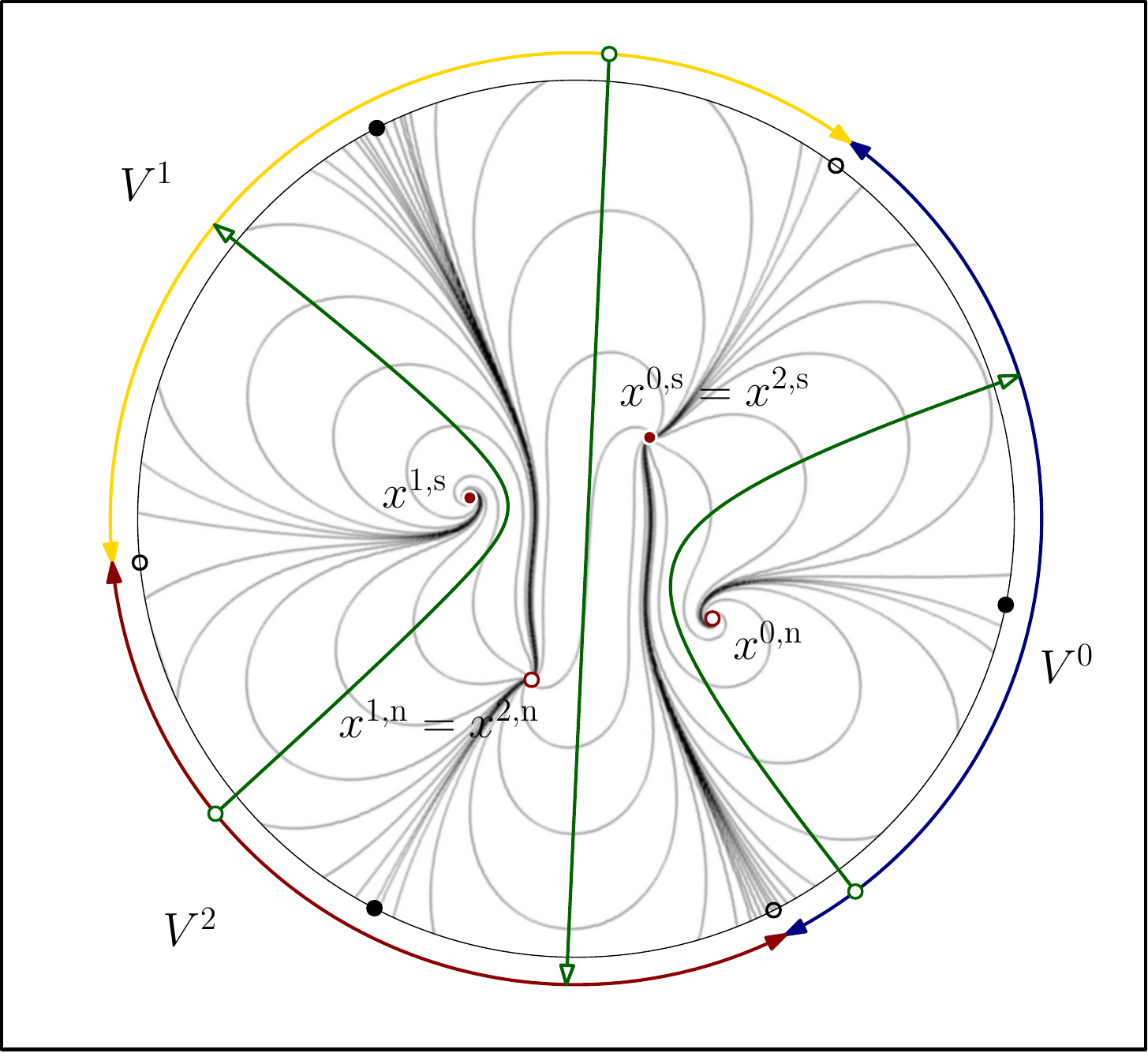}\hfill{}

\caption{\label{fig:squid_permutation} Construction of the non-crossing permutation
$\sigma$; here $\sigma=\left(\protect\begin{array}{ccc}
0 & 1 & 2\protect\\
0 & 2 & 1
\protect\end{array}\right)$.}
\end{figure}
\begin{defn}
Recall that a \emph{non-crossing permutation} $\sigma\in\mathfrak{S}_{k}$
is a permutation such that if $p_{0},\ldots,p_{k-1}$ are circularly
ordered points on a circle, there exist pairwise non-intersecting
curves within the inscribed disc joining $p_{j}$ and $p_{\sigma\left(j\right)}$
for all $j$.

\begin{enumerate}
\item There exists a (non-crossing) permutation $\sigma=\sigma_{\varepsilon}$
on $\left\{ 0,\dots,k-1\right\} $ yielding a pairing of the sector
$V_{\varepsilon}^{j}$ with $V_{\varepsilon}^{\sigma(j)}$ (see Figure~\ref{fig:squid_permutation})
in the following way. If the sector $V_{\varepsilon}^{j}$ shares
its vertices $\sad x{}{j-1,}$ and $\nod x{}{j,}$ with a distinct
sector $V_{\varepsilon}^{j'}$, then we define $\sigma(j):=j'$. Otherwise
we let $\sigma(j)=j$. 
\item The squid sector $V_{\varepsilon}^{j}$ is \textbf{introvert} if $\sigma\left(j\right)=j$,
and \textbf{extrovert} otherwise (see Figure~\ref{fig:extrovert-introvert_squids}).
\end{enumerate}
\end{defn}

The permutation $\sigma$ is a complete topological invariant~\cite{DES,BraDia}
for structurally stable vector field $P\pp x$ (\emph{i.e.} for generic
$\varepsilon$) and any non-crossing permutation can be realized in
this way. In particular $\varepsilon\mapsto\sigma_{\varepsilon}$
is constant on the conic domains $K_{\ell}$.

\begin{figure}[h]
\hfill{}\subfloat[Extrovert squid sector: $\sigma\left(j\right)\protect\neq j$]{\includegraphics[width=6cm]{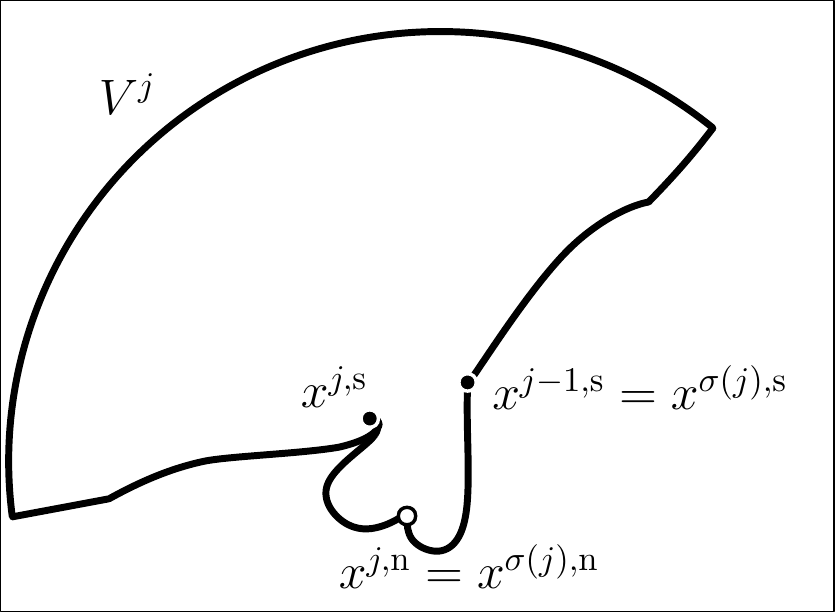}

}\hfill{}\subfloat[Introvert squid sector]{\includegraphics[width=6cm]{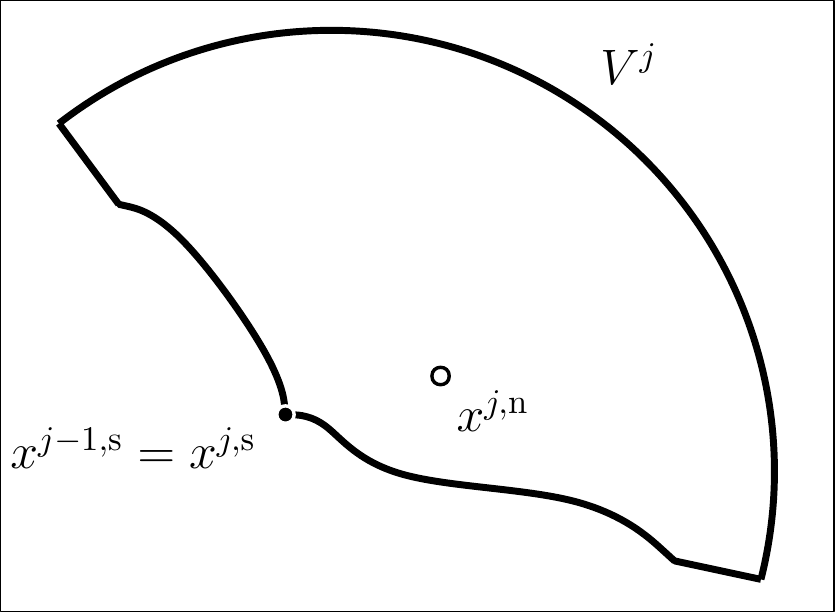}

}\hfill{}

\caption{\label{fig:extrovert-introvert_squids}The two kinds of a bounded
squid sector for $k>1$.}
\end{figure}

\subsubsection{\label{subsec:Practical-description}Practical description and quantitative
estimates}

Here we end the proof of Proposition~\ref{prop:cells_construction}.
As discussed earlier, finding an admissible angular function is equivalent
to finding suitable piecewise affine real curves in the complex time
coordinates. Studying $\dot{x}=P_{\varepsilon}(x)$ for complex values
of the time $t$ is the natural point of view taken by~\cite{DES,BraDia}.
In that setting we could view the whole $x$-line as a single complex
trajectory of the flow of $P_{\varepsilon}\pp x$. Although one might
consequently try to parametrize points in the $x$-variable by values
of the time $t\left(x\right)\in\cc$ this is too simplistic: the time
function is multivalued at $\infty$. Nonetheless, the idea is very
powerful and fruitful if we limit ourselves to simply connected domains
in time space. Let us define the time function by
\begin{align*}
t\left(x\right) & :=\int_{\infty}^{x}\frac{\dd z}{P_{\varepsilon}\left(z\right)}.
\end{align*}
When $\varepsilon$ is generic we obtain 
\[
t\left(x\right)=\sum_{x_{\varepsilon}\in P_{\varepsilon}^{-1}\left(0\right)}\frac{1}{P_{\varepsilon}'(x_{\varepsilon})}\log(x-x_{\varepsilon}).
\]

To understand the image of the complement in $\cc$ of $\rho\ww D$
under this map, it is first necessary to describe the $k$-sheeted
Riemann surface over $\cc$, on which the map $t$ is defined. Since
the integral starts at $x=\infty$ then $t(\infty)=0$. We have to
remember that $\infty$ is a pole of order $k-1$ of the vector field
over which the time function has locally the form $t=-\frac{1}{kx^{k}}$,
which means that it is a ramification point. Moreover, $\infty$ can
be reached in finite time and, for $\varepsilon\neq0$, there are
\emph{periods}, which correspond to the time to go from $\infty$
to $\infty$ along a path circling some singular points. This means
that the time function is multi-valued and there are an infinite number
of ramification points, all corresponding to images of $\infty$.
The distance between two images on one sheet is a period of a loop
around singular points. These periods are all greater than some $C\norm[\varepsilon]{}$
(see Lemma~\ref{lem:time_minoration} below).

The image of the complement in $\cc$ of $\rho\disc$ under the map
$t$ is therefore, for small $\varepsilon$, a union of holes (topological
disks) of approximate radius $\frac{1}{k\rho^{k}}$ in the $k$-sheeted
Riemann surface (Figure~\ref{fig:strips}) over $\cc$, with one
central hole around $0$. The ramifications occur at the images of
$\infty$. Each hole contains an image of $\infty$ by the multivalued
continuation of $t$. A half turn around the central hole corresponds
to an angle of $\frac{\pi}{k}$ on $\partial\disc_{\rho}$ (or to
a saddle sector of $\infty$). Hence one $\tau_{j}$ of Theorem~\ref{thm:DES_cells}
is associated to each half turn, thus pairing the half turns two by
two. Since $\tau_{j}$ is a period in $t$-space, it is a distance
between centers of holes and, on each half sheet, the next hole is
obtained by translating the current hole by $\tau_{j}$. 
\begin{lem}
\label{lem:time_minoration}There exists $C>0$, depending only on
$k$, such that 
\begin{align*}
|\tau_{j}| & >C\norm[\varepsilon]{}^{-k}.
\end{align*}
\end{lem}

\begin{proof}
It suffices to show that there exists $C>0$ such that $|\tau_{j}|>C$
when $\norm[\varepsilon]{}=1$, and then to use the rescaling~\eqref{eq:time_rescalling}.
This is done as follows. Changing the time $t\mapsto t':=\ee^{-\ii\arg\tau_{j}}t$,
then $\tau_{j}'=\ee^{-\ii\arg\tau_{j}}\tau_{j}$ is the time along
a homoclinic loop between two separatrices of $\infty$ for the transformed
vector field. In Section~\ref{subsec:Size-of-sectors}, all roots
have been shown to belong to $\sqrt{k}\adh{\disc}$. The time $\tau_{j}'$
is then larger than twice the minimum time to go from $\infty$ to
$\left\{ \left|x\right|=2k\right\} $, and this minimum is positive
on the compact set $\norm[\varepsilon]{}=1$.
\end{proof}
Let us first describe what happens on the spine of the cell. There,
holes are aligned vertically (the $\tau_{j}$ are pure imaginary)
and each sector (which is an ideal sector) corresponds to a horizontal
strip as in Figure~\ref{fig:strips}. If we want to cover $\rho\disc\setminus\{P_{\varepsilon}^{-1}(0)\}$
then we should cover a little more than a full turn around one hole.
The width of the strip should be a little over $\frac{\tau_{j}}{2}$
on the top side and over $\frac{\tau_{j+1}}{2}$ on the bottom side.
When moving to $t$-space the singular points have been sent to $\infty$,
to the left (\emph{resp}. right) for the singular points of node (\emph{resp}.
saddle) type. In such a picture we see the connected parts of the
intersections of two consecutive sectors that go to the boundary.

The internal intersection parts (that we later call gate parts) can
only be seen by using the periodicity of $t$. There are similar half
strips on the $\sigma(j)$-th sheet, with a hole at a distance $\tau_{j}$
and on the $\sigma(j+1)$-th sheet, with a hole at a distance $\tau_{j+1}$.
Their translations by the corresponding period $\tau_{j}$ and $-\tau_{j+1}$
brings them on the $j$-th sheet where they intersect the initial
strip (Figure~\ref{fig:strips}).

\begin{figure}
\hfill{}\includegraphics[width=0.9\columnwidth]{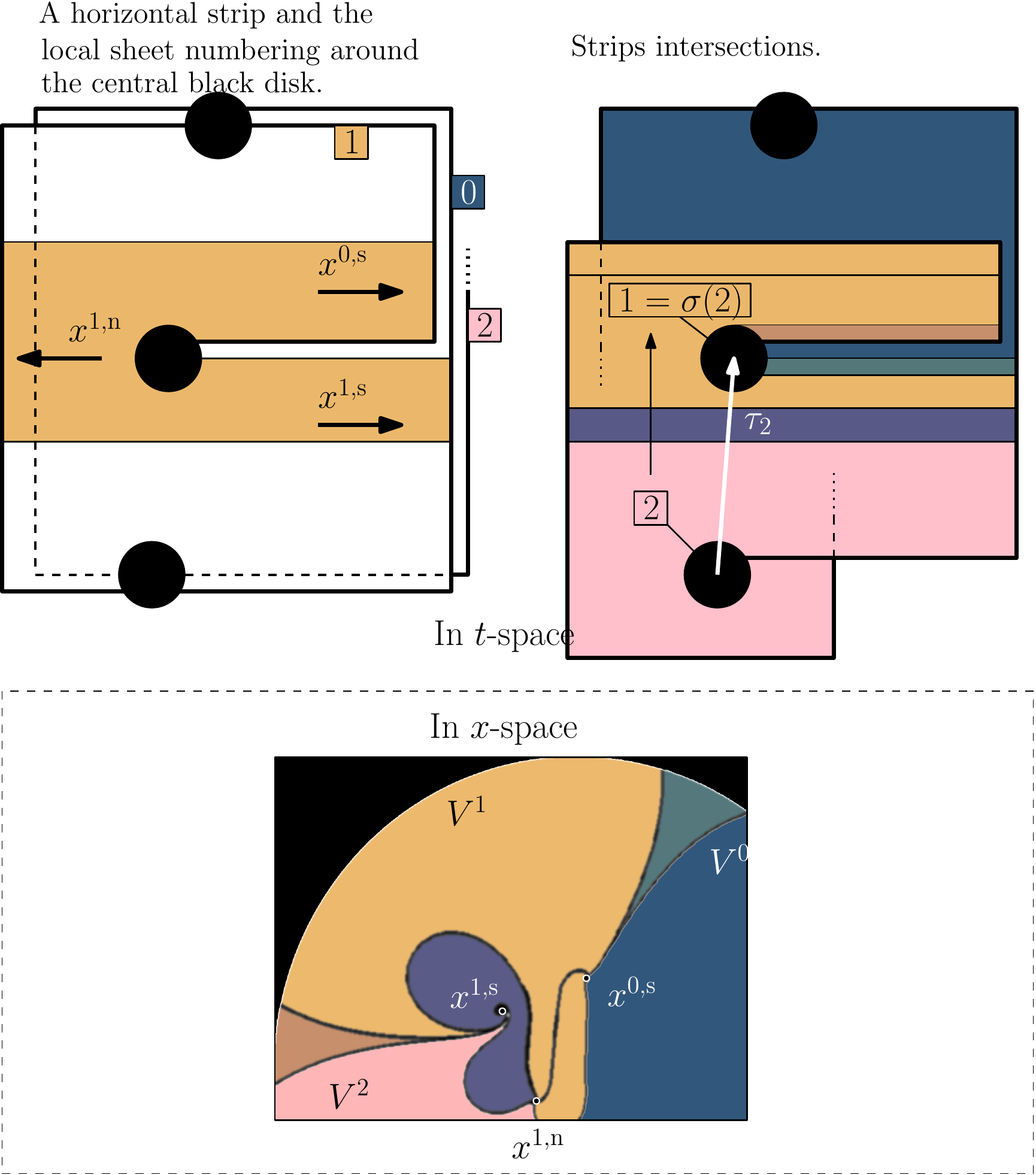}\hfill{}
\caption{\label{fig:strips}The images of these horizontal strips in $t$-space
are sectors $V_{\varepsilon}^{j}$ in $x$-space.}
\end{figure}
If we now move away from the spine of $K_{\ell}$, then two things
happen.
\begin{itemize}
\item On the one hand, the $\tau_{j}$ bend. When they approach the real
line (horizontal direction), then it is no more possible to pass a
horizontal strip because the holes block the way: the remedy is to
slant the strip so that it avoids the hole altogether. 
\item Also, in the $t$-space, each singular point $x_{\varepsilon}$ turns,
since it is located at infinity in the direction of $-\frac{1}{P_{\varepsilon}'(x_{\varepsilon})}$.
An infinite half-strip in the direction $\vartheta=\exp\left(\ii\theta\right)$
can only be sent to a sector with vertex at $x_{\varepsilon}$ if
\begin{equation}
\re{-\frac{\vartheta}{P_{\varepsilon}'(x_{\varepsilon})}}>0,\label{eq:direction_theta}
\end{equation}
(corresponding to the scalar product of $-\frac{1}{P_{\varepsilon}'(x_{\varepsilon})}$
and $\vartheta$ being positive). This forces giving an angle to the
strip in the infinite end of the half-strip approaching a singular
point.
\end{itemize}
The choice of $\delta$ in \eqref{eq:choice_delta} guarantees that
$\tau_{j}$ cannot turn of an angle larger than $\frac{\pi}{2}+\delta$.
The size of the holes is of the order of $\frac{1}{k\rho^{k}}$, which
is very small compared to the $\tau_{j}$ and $\frac{1}{|P_{\varepsilon}'(x_{\varepsilon})|}$
if $\varepsilon$ is sufficiently small.

Now the strip has three infinite ends, a wide one on the left side
attached to a point of node type $\nod x{\varepsilon}{j,}$, and two
thinner ones attached to $\sad x{\varepsilon}{j-1,}$ and $\sad x{\varepsilon}{j,}$.
The slope $\vartheta_{\varepsilon}$ for each infinite end should
be chosen so that \eqref{eq:direction_theta} be satisfied for the
singular point corresponding to that end of the strip.

This is how it is done. In the ideal situation the curves $\gamma_{j}$,
used to pair the saddle sectors (permutation $\sigma$) and to define
the $\tau_{j}$, split the disk into $k+1$ regions, each containing
a singular point. When we are no more in the ideal situation, then
several of the curves $\gamma_{j}$ have disappeared, corresponding
to the fact that some strips are either too thin so as to pass a trajectory
or have disappeared. Then there remains only a few $\gamma_{j}$ dividing
the disk in $m<k+1$ regions. Each of these regions contains some
singular points. In a given region, we have two possibilities: 
\begin{enumerate}
\item either there are several singular points: then they have kept their
saddle or node type and are linked by trajectories that form a tree; 
\item or there is a unique singular point, which is a center or a very widely
spiraling focus. 
\end{enumerate}
For each $\gamma_{j}$ that has disappeared because $\im{\tau_{j}}$
is too small, we bend the strip between the holes while keeping its
width a little more that $\frac{\tau_{j}}{2}$ (resp. $\frac{\tau_{j+1}}{2}$)
(see Figure~\ref{fig:slanted}). This process restores that part
of the strip and forces the bent separatrices to stay inside the disk.

Just before the disappearance of $\gamma_{j}$, each separatrix was
attached to a singular point. If the singular point is close to a
center as in~(2) above, then the bent separatrix will spiral to the
singular point: we may add a little more bending so that it does not
escape the disk before doing so. In~(1) the bent separatrix has no
choice but to cross one of the trajectories of the tree between two
singular points, one of which is the singularity to which it was attached
before. When it does so, we turn to follow a parallel trajectory going
to the singular point then bringing back the strip to the horizontal
direction. We make the same thing for the three infinite ends of each
strip. When doing so, we pay attention to take the same slope at all
infinite ends attached to a given singular point.

\begin{figure}
\centering{}\includegraphics[width=0.9\columnwidth]{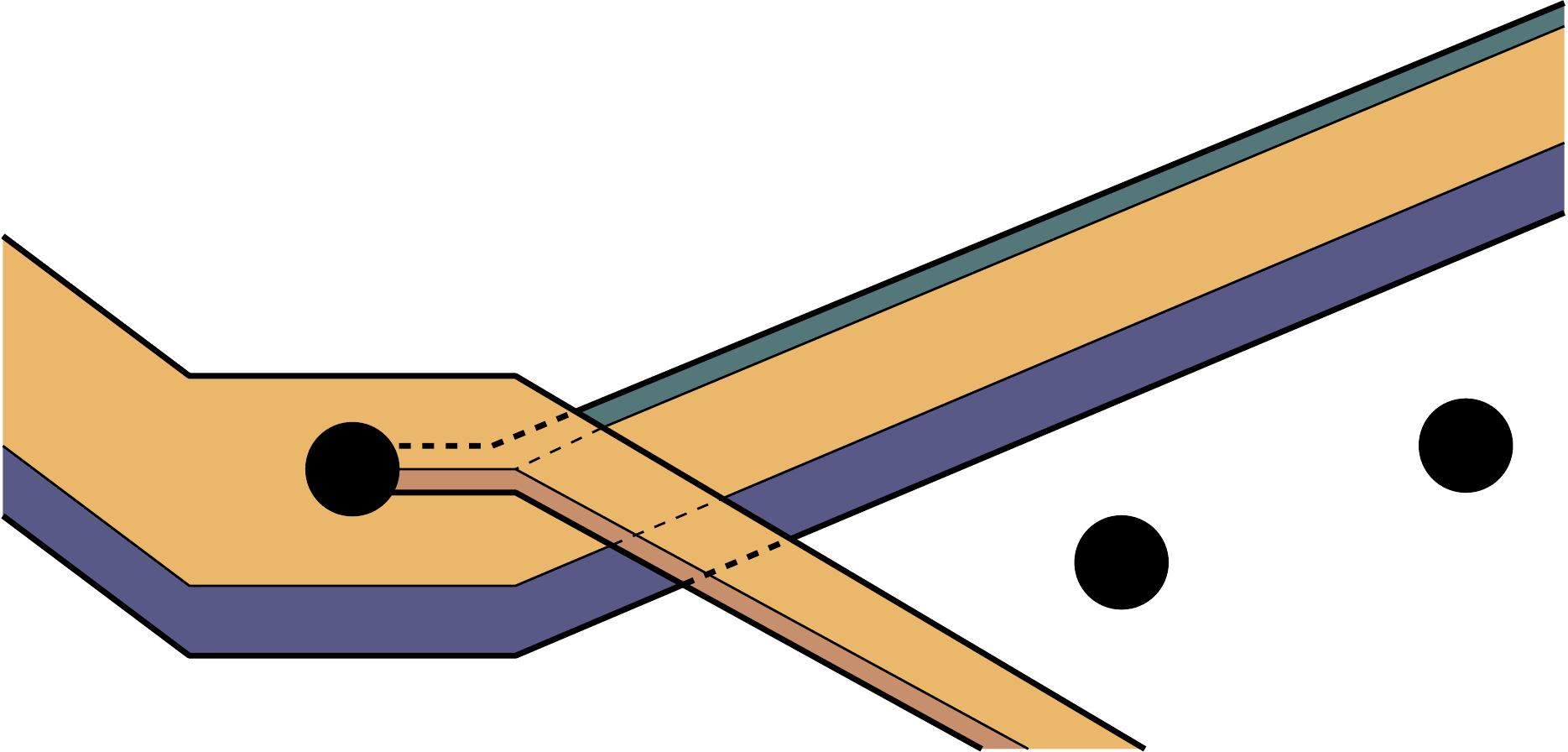}
\caption{\label{fig:slanted}A slanted strip in $t$-space whose image is a
sector $V_{\varepsilon}^{j}$ in $x$-space.}
\end{figure}
\begin{rem}
When $\varepsilon\to0$ along a curve~\eqref{eq:curve} then $P_{\varepsilon}'(x_{\varepsilon})\to0$
and the half-strips are replaced by half-planes. More generally when
$\varepsilon$ tends to a point of $\Delta_{k}$, some half-strips
are replaced by half-planes.
\end{rem}

\subsubsection{Large (unbounded) squid sectors}

When $\mu_{0}\notin\rr_{\leq0}$, we will also need a covering of
the whole of $\cc$ by $k$ sectors. For that purpose, we append to
the sectors $V_{\varepsilon}^{j}$ an infinite part obtained in the
following way: if $x_{1}$ and $x_{2}$ are the endpoints of the boundary
arc of $V_{\varepsilon}^{j}$ along $\rho\disc$, then we follow geometric
spirals $x_{m}\exp\left((1+\ii\nu)\rr_{\geq0}\right)$ for $m\in\left\{ 1,2\right\} $
and some $\nu$ such that 
\[
\re{\mu_{0}}>\nu\im{\mu_{0}}.
\]
If we come back to the representation of the sector in $t$-space,
this amounts to appending some spiraling sector inside the holes (a
neighborhood of $\infty$ in $x$-space is covered by a sector of
opening $2k\pi$ in $t$-space). 

\begin{figure}
\hfill{}\includegraphics[width=0.9\columnwidth]{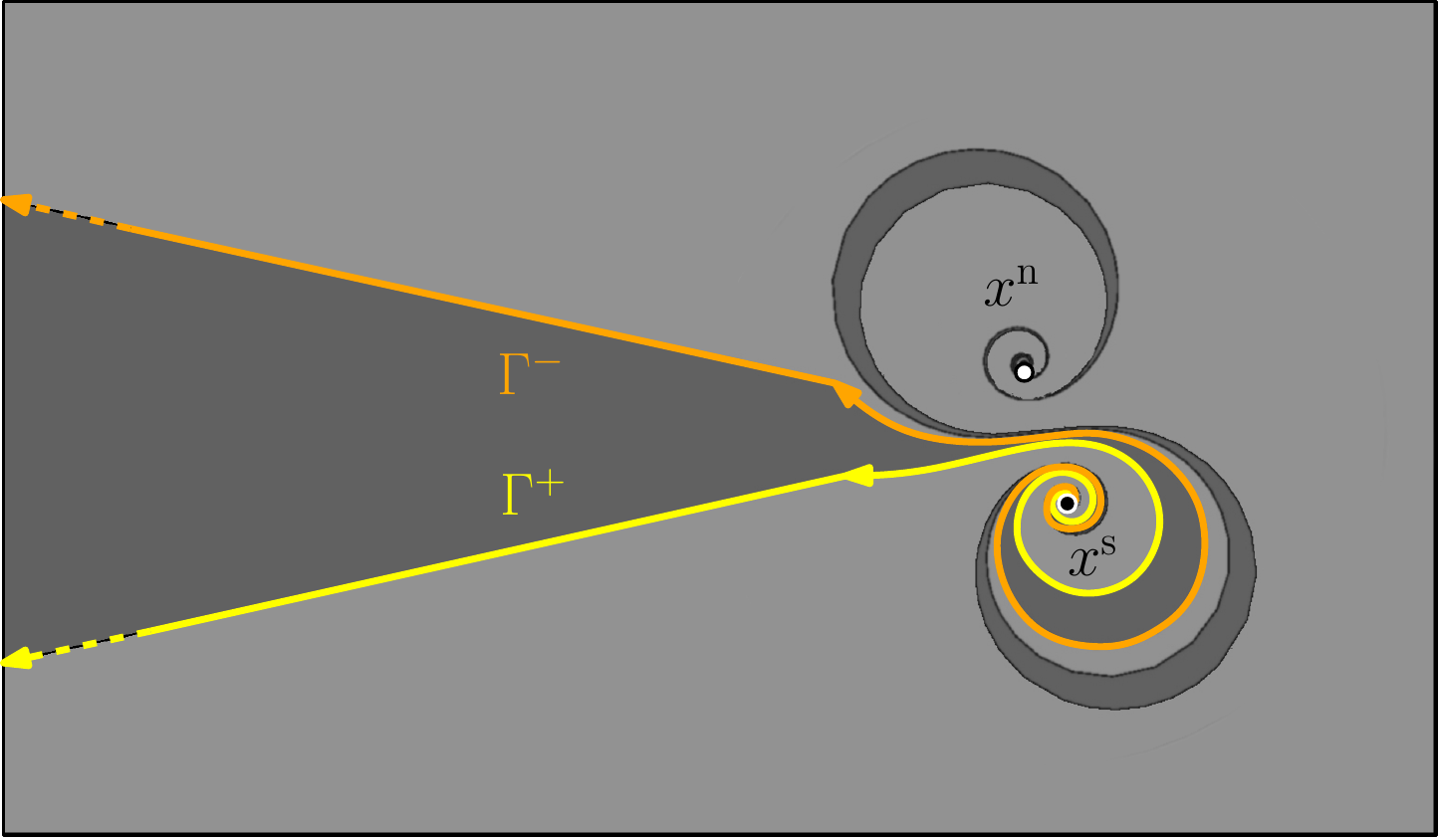}\hfill{}

\caption{\label{fig:unbounded_squid_sectors}Unbounded squid sector for $k=1$
and $\protect\re{\mu}>0$. When $\protect\re{\mu}\protect\leq0$ and
$\mu\protect\notin\protect\rr$ the shaded area bends to form a geometric
spiral near $\infty$. See also Figure~\ref{fig:unbounded_3} for
the case $k>1$.}

\end{figure}
We still denote by $V_{\varepsilon}^{j}$ the resulting unbounded
sectors, since the context will never be ambiguous. 

\subsubsection{Intersections of squid sectors}
\begin{defn}
We let $\Gamma^{j,+}$ (\emph{resp.} $\Gamma^{j-1,-}$) be the part
of the boundary of the unbounded sector $V_{\varepsilon}^{j}$ joining
$\sad x{}{j,}$ (\emph{resp.} $\sad x{}{j-1,}$) to $\infty$ with
this orientation. The intersection of two squid sectors $V_{\varepsilon}^{j}$
and $V_{\varepsilon}^{j'}$ is made of up to three parts in general,
and up to four parts when $k=2$ (see Figure~\ref{fig:Non-equivalent-decompositions}).

\begin{itemize}
\item If $j'=j+1$ (resp. $j'=j-1$) a (connected) \textbf{saddle part}
$\sad V{\varepsilon}{j,}$ (resp. $\sad V{\varepsilon}{j-1,}$), bounded
by the two curves $\Gamma^{j\pm}$ (resp. $\Gamma^{j-1,\pm}$) to
the common point $\sad x{}{j,}$ (resp. $\sad x{}{j-1,}$) of saddle
type. When $k=1$, the saddle-part corresponds to a self-intersection. 
\item If $j'=\sigma\left(j\right)$ a \textbf{gate part} $\gat V{\varepsilon}{j,}$
included in $\rho_{\varepsilon}\ww D$ and adherent to the two singular
points $\sad x{}{j,}$ and $\nod x{}{j,}$. When $j=\sigma(j)$, the
gate part of an introvert sector corresponds to a self-intersection. 
\item If $j'=\sigma\left(j\right)$ and $j=\sigma\left(j'\right)$ for $j\neq j'$,
a second gate part $\gat V{\varepsilon}{j',}$ adherent to the singular
points $\sad x{}{j-1,}=\sad x{}{j',}$ and $\nod x{}{j,}$ (Figure~\ref{fig:Non-equivalent-decompositions}~(B)). 
\end{itemize}
\end{defn}

\begin{figure}[h]
\hfill{}\subfloat[$\varepsilon=\protect\ee^{\protect\ii\pi}$]{\includegraphics[width=6cm]{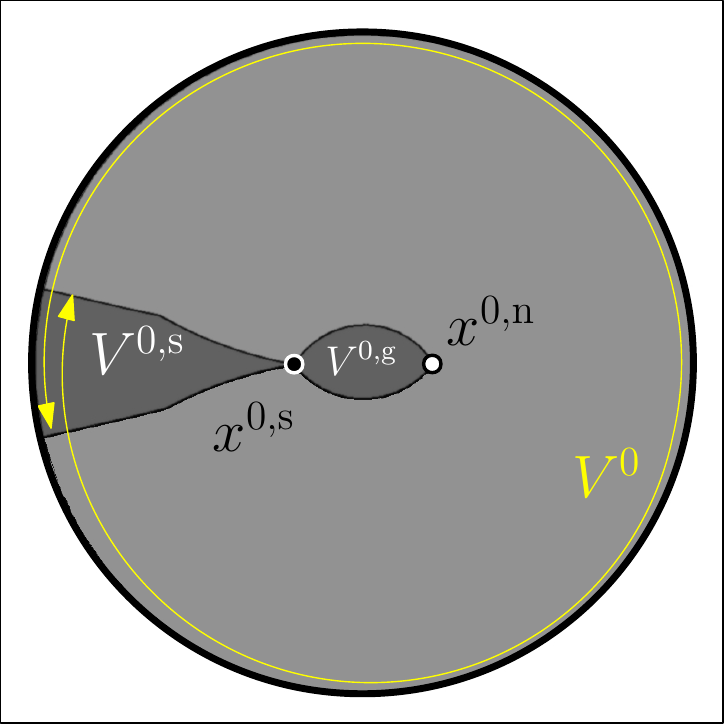}

}\hfill{}\subfloat[$\varepsilon=\protect\ee^{\protect\ii\protect\nf{15\pi}8}$]{\includegraphics[width=6cm]{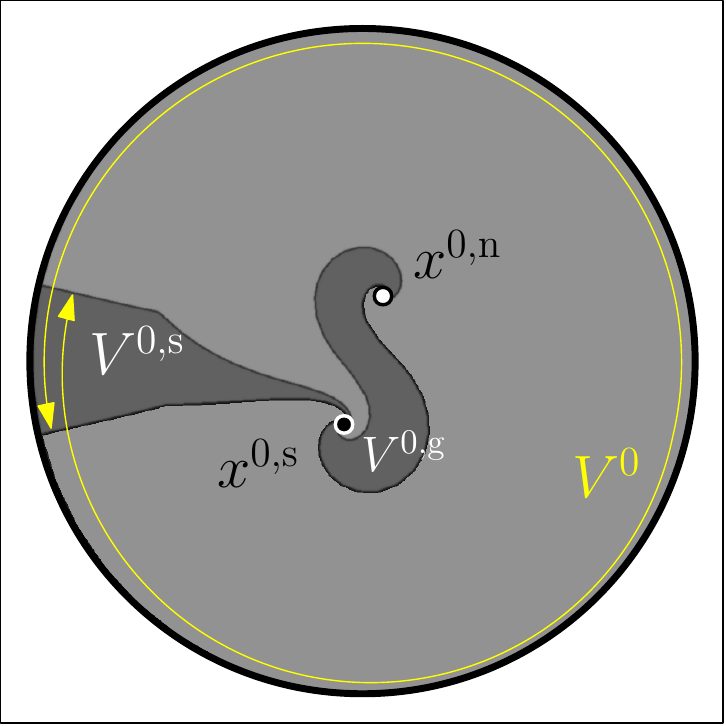}

}\hfill{}

\hfill{}\subfloat[$\varepsilon=\protect\ee^{\protect\ii\protect\nf{17\pi}8}$]{\includegraphics[width=6cm]{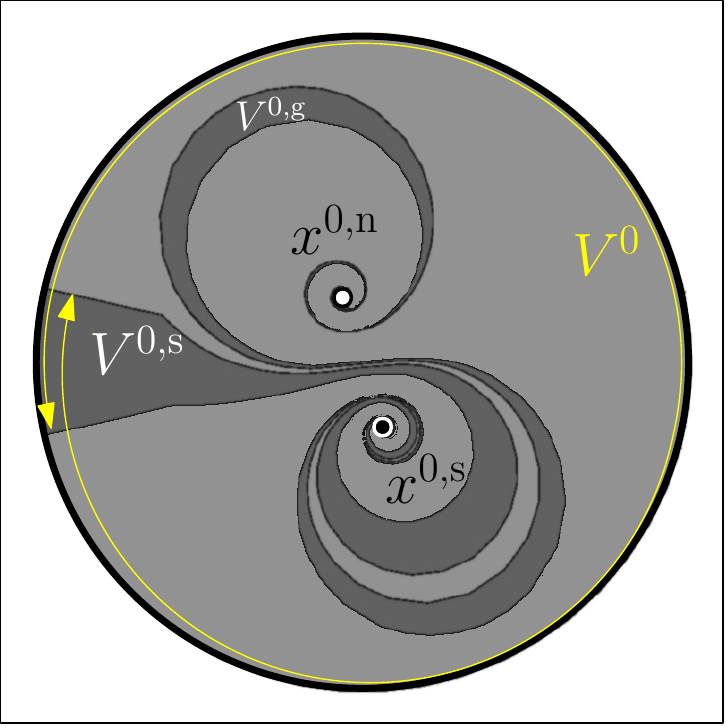}

}\hfill{}

\caption{\label{fig:Squid_k=00003D1}Squid sectors for different values of
$\varepsilon$ when $k=1$.}
\end{figure}

\subsubsection{Non-equivalent decompositions}

For the same value of the parameter $\varepsilon$ in the intersection
of two cells (or a cell's self-intersection), the disc $\rho\ww D$
is split in non-equivalent ways into bounded squid sectors (see Figures~\ref{fig:parametre_space_k=00003D1}
and~\ref{fig:Non-equivalent-decompositions}). By ``non-equivalent''
we mean that at least one boundary of a squid sector is attached to
another root of $P_{\varepsilon}$ when passing from one cell to the
other.

\begin{figure}[h]
\hfill{}\includegraphics[width=13cm]{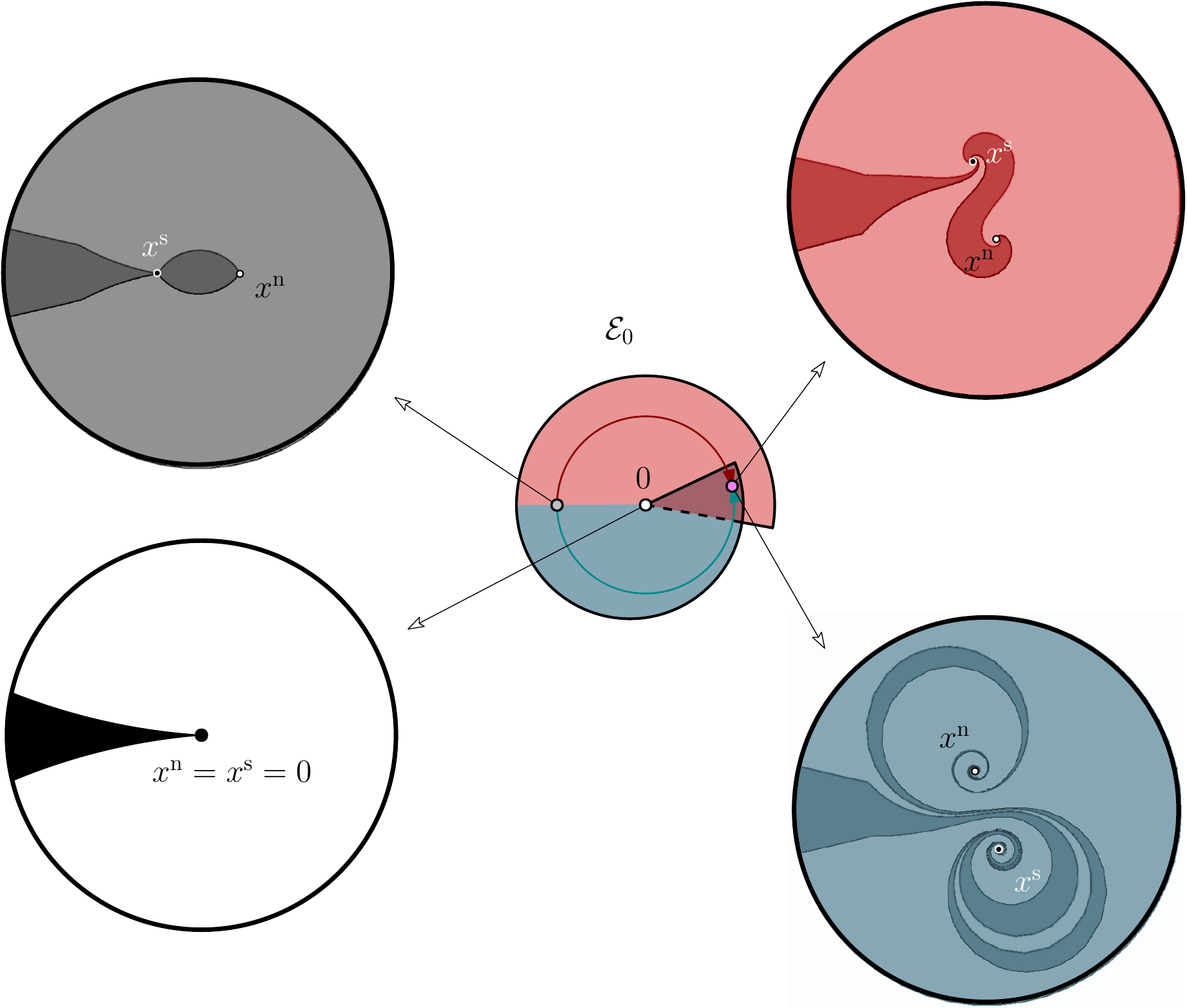}\hfill{}

\caption{\label{fig:parametre_space_k=00003D1}The single (self-overlapping)
cell $\mathcal{E}$ with diverse configurations when $k=1$. Non-equivalent
decompositions are shown on the right. In each picture the location
of the node-like singularity $\protect\nod x{}{}$ is given by the
analytic continuation of the principal determination of $\sqrt{-\varepsilon}$.}
\end{figure}
\begin{figure}[h]
\hfill{}\subfloat[With two introvert squid sectors: $\sigma=\left(\protect\begin{array}{cc}
0 & 1\protect\\
0 & 1
\protect\end{array}\right)$]{\includegraphics[width=6cm]{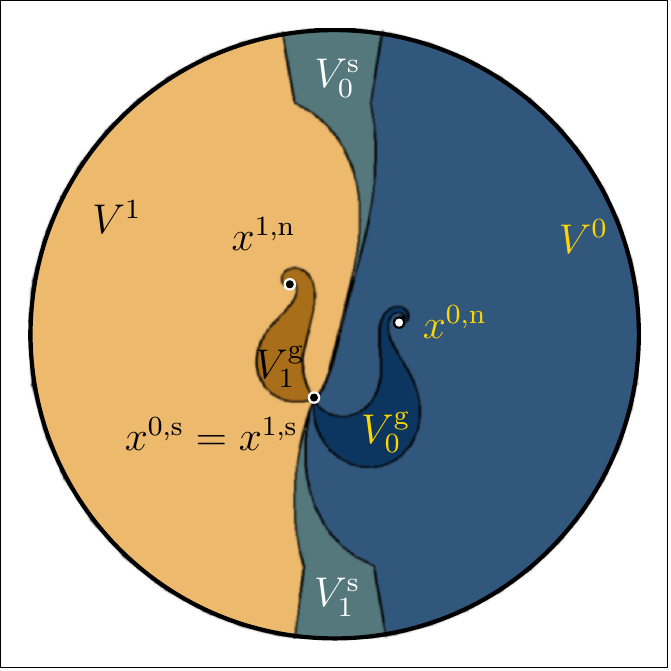}

}\hfill{}\subfloat[With two extrovert squid sectors: $\sigma=\left(\protect\begin{array}{cc}
0 & 1\protect\\
1 & 0
\protect\end{array}\right)$]{\includegraphics[width=6cm]{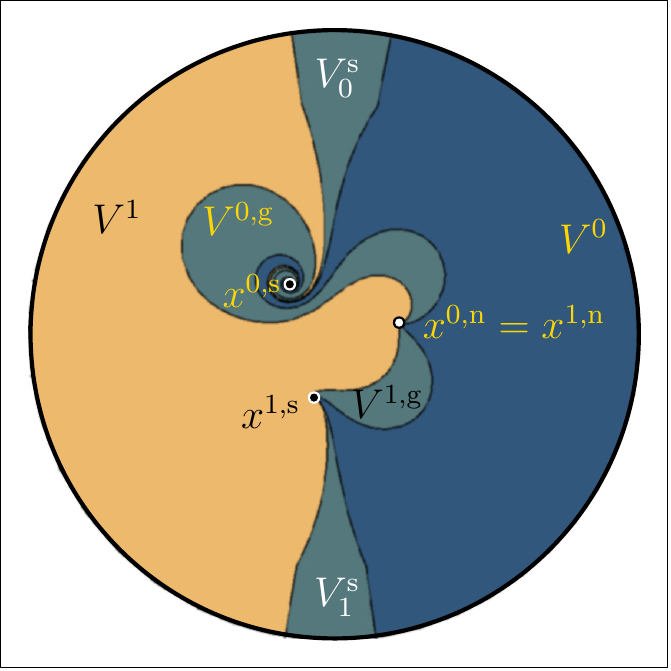}

}\hfill{}

\caption{\label{fig:Non-equivalent-decompositions}Non-equivalent decompositions
for same $\varepsilon$ when $k=2$.}
\end{figure}

\subsubsection{Some useful estimates}

We shape the squid sectors in this way because in doing so we gain
control on the convergence and on the magnitude of integrals involved
in the Cauchy-Heine transform appearing in the next section, in the
wake of Remark~\ref{rem:modulus_behavior_near_roots}. In the following
lemma we use the boundary $\Gamma^{j,\pm}$ of saddle-parts of unbounded
sectors as depicted in Figure~\ref{fig:unbounded_squid_sectors}.
\begin{lem}
\label{lem:integrable_first-integral}Assume $\tau=0$ (which particularly
implies $\mu_{0}\notin\rr_{\leq0}$). One can take $\rho$ and $\mathcal{E}_{\ell}$
sufficiently small so that the following properties hold. 
\begin{enumerate}
\item For all $r>0$ the model first integral~\eqref{eq:model_first_integral}
is bounded on $\sad V{\varepsilon}{j,}\times r\ww D$, more precisely
there exists $C>0$ such that 
\begin{align*}
\left(\forall\varepsilon\in\mathcal{E}_{\ell}\right)~~~~\sup_{\sad V{\varepsilon}{j,}\times r\ww D}\left|\widehat{H}^{j}\right| & \leq rC.
\end{align*}
\item Also $\widehat{H}^{j}$ is $\frac{\dd z}{z-x}$-absolutely integrable
over any component $\Gamma=\Gamma^{j,\pm}$ of the boundary of saddle
part intersections (given the outgoing orientation): for all $x\in V_{\varepsilon}^{j}\backslash\Gamma$
and $y\in\cc$ we have 
\begin{align*}
\int_{\Gamma}\widehat{H}^{j}\left(z,y\right)\frac{\dd z}{z-x} & =:yI^{j}\left(x\right)\in\cc.
\end{align*}
\item There exist a constant $C>0$ such that for all $\varepsilon\in\mathcal{E}_{\ell}$
and all $x\in V_{\varepsilon}^{j}\backslash\Gamma$ 
\begin{align*}
\left|I^{j}\left(x\right)\right| & \leq\frac{C}{\left|z_{*}-x_{*}\right|},
\end{align*}
where $z_{*}=\Gamma\cap\rho\sone$ and $x_{*}$ is likewise the intersection
of $\rho\sone$ and the curve passing through $x$ built in the same
way as $\Gamma$. 
\end{enumerate}
\end{lem}

\begin{proof}
Because $\widehat{H}^{j}$ is linear in $y$ we may only consider
the case $y:=1$. Let 
\[
\widehat{h}~:~x\mapsto\widehat{H}^{j}\left(x,1\right)
\]
be the corresponding partial function. The proof is done in two steps,
corresponding to the two different components ``inner'' (inside
$\rho\disc$) and ``outer'' ($|x|\geq\rho$). We parametrize $\Gamma$
by a piecewise analytic curve $z~:~\ww R\to\cc$ detailed below, such
that (with the obvious abuse of notations) 
\begin{align*}
\begin{cases}
z\left(-\infty\right) & =\infty\\
z\left(0\right) & =z_{\pm}^{j}\in\rho\sone\\
z\left(\infty\right) & =\sad x{}{j,}
\end{cases} & .
\end{align*}
In what follows, $C>0$ indicates a real constant (independent on
$\varepsilon$) whose value varies according to the place where it
appears.
\begin{enumerate}
\item We invoke again the variational argument presented in Remark~\eqref{rem:modulus_behavior_near_roots}.
Over $\left]0,\infty\right[$ we follow the flow of $\vartheta P\pp x$
and we can indeed estimate the modulus 
\begin{align*}
\phi\left(t\right) & :=\left|\widehat{h}\left(z\left(t\right)\right)\right|,
\end{align*}
as $\widehat{h}$ is solution of 
\begin{align*}
\frac{\dd{\widehat{h}}}{\widehat{h}} & =-\left(1+\mu z^{k}\right)\frac{\dd z}{P_{\varepsilon}},
\end{align*}
so that 
\begin{align*}
\frac{\dot{\phi}}{\phi}\left(t\right) & =-\re{\vartheta\left(1+\mu z^{k}\right)}.
\end{align*}
Since $\frac{1}{2\left|\mu_{0}\right|}>\rho^{k}$, and taking the
hypothesis $\left|\arg\vartheta\right|<\frac{\pi}{4}$ into account
we obtain 
\begin{align*}
\frac{\dot{\phi}}{\phi} & \leq-C<0
\end{align*}
and 
\begin{align}
\left|\widehat{h}\left(z\left(t\right)\right)\right| & \leq\left|\widehat{h}\left(z\left(t_{\varepsilon}\right)\right)\right|.\label{eq:estim_saddle_inner}
\end{align}
Over $\left]-\infty,0\right[$ we follow the flow of 
\begin{align*}
\dot{z} & =-\left(1+\ii\nu\right)z,
\end{align*}
above which the modulus of $\widehat{h}$ is governed by 
\begin{align*}
\frac{\dot{\phi}}{\phi} & =\re{\frac{\left(1+\mu z^{k}\right)\left(1+\ii\nu\right)z}{P_{\varepsilon}\left(z\right)}}\\
 & =\re{\frac{\left(1+\mu z^{k}\right)\left(1+\ii\nu\right)}{z^{k}}\times\frac{z^{k+1}}{P_{\varepsilon}\left(z\right)}}\\
 & \geq C\re{\mu+\ii\mu\nu}=:\alpha>0
\end{align*}
for $|z|$ sufficiently large since $\nu$ is chosen in such a way
that $\re{\mu+\ii\nu\mu}>0$ for all $\varepsilon\in\mathcal{E}_{\ell}$.
To conclude the proof we only have to remark that $\sup_{\left|z\right|=\rho}\left|\widehat{h}\left(z\right)\right|\geq\left|\widehat{h}\left(z\left(0\right)\right)\right|$
converges uniformly towards $\sup_{\left|z\right|=\rho}\left|\widehat{H}_{0}^{j}\left(z,1\right)\right|<\infty$
as $\varepsilon\to0$. 
\item and (3) We use the following trick. We work with the integral 
\[
J^{j}(x)=\int_{-\infty}^{\infty}\widehat{h}(z(t))\frac{z'(t)dt}{z(t)-x(t)},
\]
where $t\mapsto x\left(t\right)$ is defined similarly as $t\mapsto z\left(t\right)$
except for the fact that it passes through $x$, uniquely defining
$x_{*}=x\left(0\right)$. To conclude we will need to bound away from
$0$ (uniformly in $\varepsilon$) the quantity $\left|\frac{z\left(t\right)-x}{z\left(t\right)-x\left(t\right)}\right|$.
But this is clear from the pictures because if $\tx{dist}\left(x,\Gamma\right)$
is realized for $z=z\left(t\right)$ then $x\simeq x\left(t\right)$.

Now, to study $J^{j}(x)$ we repeat the above argument but with the
function 
\begin{align*}
\phi_{\sharp}\left(t\right) & :=\left|\frac{\widehat{h}\left(z\left(t\right)\right)A_{\sharp}\left(z\left(t\right)\right)}{z\left(t\right)-x\left(t\right)}\right|,~~~~\sharp\in\left\{ 0,\infty\right\} ,
\end{align*}
where $A_{0}\left(z\right):=\vartheta P_{\varepsilon}\left(z\right)$
and $A_{\infty}\left(z\right):=-(1+i\nu)z$. The variations of $\phi_{\sharp}$
are governed by 
\begin{align*}
\frac{\dot{\phi}_{\sharp}}{\phi_{\sharp}}\left(t\right) & =\re{-\frac{1+\mu z^{k}}{P_{\varepsilon}\left(z\left(t\right)\right)}A_{\sharp}\left(z\left(t\right)\right)+A_{\sharp}'\left(z\left(t\right)\right)-\frac{A_{\sharp}\left(z\left(t\right)-A_{\sharp}\left(x\left(t\right)\right)\right)}{z\left(t\right)-x\left(t\right)}}
\end{align*}
for $t$ in the corresponding interval so that $\dot{z}=A_{\sharp}\left(z\right)$
and $\dot{x}=A_{\sharp}\left(x\right)$. In the case $\sharp=\infty$,
the sum of the last two terms vanishes and then 
\begin{align*}
\frac{\dot{\phi_{\infty}}}{\phi_{\infty}} & \geq C>0
\end{align*}
for large $z$ (hence $t$ close to $-\infty$) from the choice of
$\nu$. Let us now deal with the case $\sharp=0$. We have chosen
$\rho>\rho_{\varepsilon}$ so that 
\begin{align*}
\sup_{\left|z\right|<\rho}\left(\left|\mu z^{k}\right|+2\rho\left|P''\left(z\right)\right|\right) & \leq\frac{3}{4}.
\end{align*}
Because for all $x,~z\in\rho\ww D$ 
\begin{align*}
\left|P\left(x\right)-P\left(z\right)-\left(x-z\right)P'\left(z\right)\right| & \leq\left|x-z\right|^{2}\sup_{\rho\ww D}\left|P''\right|
\end{align*}
we obtain 
\begin{align*}
\frac{\dot{\phi_{0}}}{\phi_{0}} & \leq-C<0
\end{align*}
and 
\begin{align*}
\left|\phi_{0}\left(t\right)\right| & \leq\left|\phi_{0}\left(0\right)\right|\exp\left(-Ct\right)
\end{align*}
for $t\geq0$. 

Therefore the integral 
\begin{align*}
\int_{z\left(0\right)}^{z\left(t\right)}\widehat{h}\left(z\right)\frac{\dd z}{z-x}= & \vartheta\int_{0}^{t}\widehat{h}\left(z\left(t\right)\right)\frac{P_{\varepsilon}\left(z\left(t\right)\right)}{z\left(t\right)-x}\dd t
\end{align*}
is absolutely convergent as $t\to\infty$ and 
\begin{align*}
\left|\int_{z\left(0\right)}^{z\left(\infty\right)}\widehat{h}\left(z\right)\frac{\dd z}{z-x}\right| & \leq C\left|\phi_{0}\left(0\right)\right|.
\end{align*}
But $C\left|\phi_{0}\left(0\right)\right|\leq\frac{C}{\left|x\left(0\right)-z\left(0\right)\right|}$
as expected. 
\end{enumerate}
\end{proof}

\subsection{\label{subsec:section_period_cellular}Cellular section of the period:
proof of Proposition~\ref{prop:section_period_cellular}}

The cellular section $\persec[][\ell]$ of the period operator is
obtained from a variation on the method introduced in Section~\ref{subsec:Normalizing_Savelev}
to normalize the glued abstract manifold by solving a linear Cousin
problem. It is an unfolding of the technique used in~\cite{SchaTey}
for $\varepsilon=0$. The initial data is a $k$-tuple
\begin{align*}
T=\left(T^{j}\right)_{j} & \in\prod_{\zsk}\mathcal{H}_{\ell}\left\{ h\right\} 
\end{align*}
and we seek $Q\in x\mathcal{H}_{\ell}\left\{ y\right\} \left[x\right]_{<k}$,
that is 
\begin{align*}
Q\left(x,y\right) & =x\sum_{n>0}Q_{n}\left(x\right)y^{n}
\end{align*}
for some polynomial $Q_{n}\in\holb[\mathcal{E}_{\ell}]\left[x\right]_{<k}$
in $x$ of degree less than $k$, such that 
\begin{align*}
\per[][\ell]\left(Q\right) & =T.
\end{align*}
We then define the section as 
\begin{align*}
\mathfrak{S}_{\ell}\left(T\right) & :=Q.
\end{align*}

The construction goes along the following steps. They are performed
for fixed $\varepsilon$ in a fixed $\mathcal{E}_{\ell}$, with explicit
control on the parametric regularity. Hence we omit mentioning explicitly
the dependence on $\varepsilon$ and $\ell$. For $r>0$ define
\begin{align*}
\sect[r] & :=\left\{ \left(x,y\right)\in V^{j}\times\cc~:~\left|y\right|<r\right\} .
\end{align*}
We define in a similar fashion the fibered intersections $\sect[r][j,\sharp]$
for $\sharp\in\left\{ \tx s,\tx g\right\} $. 
\begin{itemize}
\item Build sectorial, bounded functions $F^{j}$ on $\sect[r][j]$ such
that
\begin{align}
F^{j+1}-F^{j} & =2\ii\pi T^{j}\circ H^{j}\label{eq:period_cousin}
\end{align}
on $\sad{\mathcal{V}}r{j,}$, where $H^{j}$ is the $j^{\tx{th}}$
canonical sectorial first integral of $\onf$, as in~\eqref{eq:sectorial_first_integral}.
This is done again by a Cauchy-Heine transform (Section~\ref{subsec:Cauchy-Heine-transform}).
\item Because of the functional equation~\eqref{eq:period_cousin} the
identity $\onf\cdot F^{j+1}=\onf\cdot F^{j}$ holds and allows to
patch together a holomorphic function $Q:=\onf\cdot F^{j}$ on a whole
$\cc\times\neigh$ which, by construction, satisfies
\begin{align*}
\per[j]\left(Q\right)\circ H^{j} & =F^{j+1}-F^{j}\\
 & =T^{j}\circ H^{j}
\end{align*}
(Section~\ref{subsec:Regularity-of-section}).
\item Growth control near $x=\infty$ and a final normalization allows concluding
that $Q\in x\germ y\left[x\right]_{<k}$ (Section~\ref{subsec:Growth-control-of-section}).
\end{itemize}

\subsubsection{\label{subsec:Cauchy-Heine-transform}Cauchy-Heine transform}
\begin{defn}
\label{def:etha_adapted}In the following we fix a collection $N=\left(N^{j}\right)_{j}\in\prod_{j\in\zsk}\holb[\group{{\sect[r]}}]'$,
which is a $k$-tuple of functions with an expansion
\begin{align*}
N^{j}\left(x,y\right) & =\sum_{n>0}N^{j,n}\left(x\right)y^{n}
\end{align*}
uniformly absolutely convergent on every $\sect[r']$ for all $0\leq r'<r$,
whose norm is given by 
\begin{align*}
\norm[N]{} & :=\max_{j}\sup_{\sect[r][j]}\left|N^{j}\right|.
\end{align*}

\begin{enumerate}
\item We define the $j^{\tx{th}}$ \textbf{sectorial first integral associated
to} $N$ as the holomorphic function 
\begin{align*}
H_{N}^{j}~:~\sect[r] & \longto\cc\\
\left(x,y\right) & \longmapsto\widehat{H}^{j}\left(x,y\right)\exp N^{j}\left(x,y\right),
\end{align*}
where $\widehat{H}^{j}$ is the sectorial canonical model first integral~\eqref{eq:model_first_integral}
continued over unbounded squid sectors. 
\item For a given $\eta>0$ we say that $N$ is \textbf{$\eta$-adapted}
if $H_{N}^{j}\left(\sad{\mathcal{V}}r{j,}\right)\subset\eta\ww D$.
\end{enumerate}
\end{defn}

Of course we prove in due time (Corollary~\ref{cor:secto_normalization_estimate})
that $N:=N_{\varepsilon}$, defined as the collection of sectorial
solutions of the normalizing equation $\onf[\varepsilon]\cdot N_{\varepsilon}^{j}=-R_{\varepsilon}$,
satisfies the hypothesis of the definition and that $\sup\left|H_{N}^{j}\left(\sad{\mathcal{V}}r{j,}\right)\right|\to0$
as $r\to0$ (uniformly in $\varepsilon\in\mathcal{E}_{\ell}$), mainly
because it is already the case for the model first integral (Lemma~\ref{lem:integrable_first-integral}~(1)).
Therefore, for given $\eta>0$, it will always be possible to find
$r$ (independently on $\varepsilon$) such that $N$ is $\eta$-adapted,
allowing us to use the next result, genuinely the key point in building
the cellular section of the period.

\begin{figure}
\hfill{}\includegraphics[width=0.9\columnwidth]{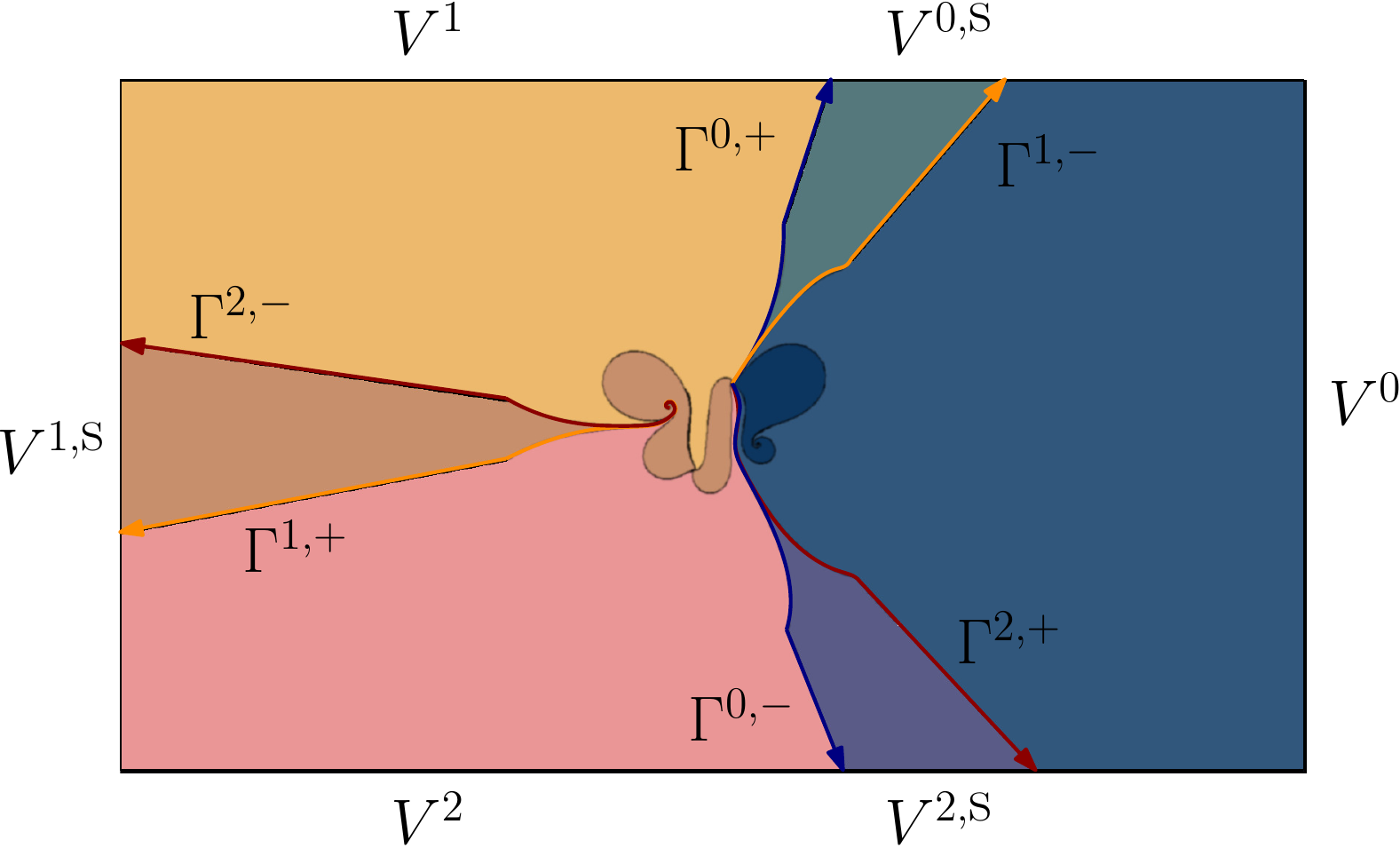}\hfill{}

\caption{\label{fig:unbounded_3}Unbounded squid sectors and paths of integration
($k>1$).}

\end{figure}
\begin{prop}
\label{prop:period_cousin}Assume $\tau=0$ (which particularly implies
$\mu_{0}\notin\rr_{\leq0}$). Let $\mathcal{E}_{\ell}$ be a fixed
cell as in Section~\ref{subsec:squid_sectors}. For every $T\in\prod_{\zsk}\holb[\eta\ww D]'$
holomorphic on a disc of radius $\eta>0$, for every $\eta$-adapted
collection $N$, the $k$-tuple of functions 
\begin{align*}
\mathfrak{F}=\mathfrak{F}\left(T,N\right) & :=\left(F^{j}\right)_{j}\in\prod_{j\in\zsk}\holb[\group{{\sect[r]}}]'
\end{align*}
defined by
\begin{align}
F^{j}\left(x,y\right) & :=\sum_{p\neq j+1}\int_{\Gamma^{p,-}}\frac{T^{p-1}\left(H_{N}^{p-1}\left(z,y\right)\right)}{z-x}\dd z+\int_{\Gamma^{j,+}}\frac{T^{j}\left(H_{N}^{j}\left(z,y\right)\right)}{z-x}\dd z\label{eq:cauchy-heine_period}
\end{align}
fulfills the next conclusions. The paths of integration $\Gamma^{j,\pm}$
bounds the unbounded squid sectors in the following way~: The boundary
of the saddle part $\sad V{}{j,}$ of (unbounded) squid sectors is
$\Gamma^{j,+}\cup\Gamma^{j+1,-}$, as in Figures~\ref{fig:unbounded_squid_sectors}
and~\ref{fig:unbounded_3}, and we set
\begin{align*}
\norm[T']{} & :=\max_{j}\sup_{\eta\ww D}\left|\ddd{T^{j}}h\right|.
\end{align*}

\begin{enumerate}
\item For every $\left(x,y\right)\in\sad{\mathcal{V}}r{j,}$ 
\begin{align}
F^{j+1}\left(x,y\right)-F^{j}\left(x,y\right) & =2\ii\pi T^{j}\left(H_{N}^{j}\left(x,y\right)\right)\label{eq:cousin_relation}
\end{align}
while for every $\left(x,y\right)\in\gat{\mathcal{V}}r{\sigma\left(j\right),}$
\begin{align*}
F^{j}\left(x,y\right) & =F^{\sigma\left(j\right)}\left(x,y\right).
\end{align*}
 (When $k=1$ we refer to~(3) of the following remark for a fuller
explanation.)
\item $F^{j}\in\holb[{\sect}]'$. 
\item There exists $K>0$ independent on $T$, $N$, $r$ and $\varepsilon$
such that the following estimates hold. 

\begin{enumerate}
\item 
\begin{eqnarray*}
\norm[\mathfrak{F}]{} & \leq & rK\norm[T']{}\exp\norm[N]{}.
\end{eqnarray*}
\item 
\begin{eqnarray*}
\norm[y\ppp{\mathfrak{F}}y]{} & \leq & rK\norm[T']{}\norm[1+y\ppp Ny]{}\exp\norm[N]{}.
\end{eqnarray*}
\item 
\begin{eqnarray*}
\norm[x\ppp{\mathfrak{F}}x]{} & \leq & rK\norm[T']{}\norm[1+x\ppp Nx]{}\exp\norm[N]{}.
\end{eqnarray*}
\end{enumerate}
\end{enumerate}
\end{prop}

\begin{rem}
~
\begin{enumerate}
\item The absolute convergence of the integrals involved in~\eqref{eq:cauchy-heine_period}
is established in the course of the proof, mainly thanks to the estimates
given by Lemma~\ref{lem:integrable_first-integral}. Notice also
that for fixed $\varepsilon$ and $y$ the mapping $x\mapsto F^{j}\left(x,y\right)$
is holomorphic on $V^{j}$ since the squid sector does not contain
any of the curves $\Gamma^{p,-}$ except for $p=j+1$.
\item The integral expression~\eqref{eq:cauchy-heine_period} and Item~(3)
above clearly show that $\mathfrak{F}$, as a function of $\varepsilon\in\mathcal{E}_{\ell}$,
has the same regularity as $T$.
\item In the case $k=1$ the expression~\eqref{eq:cauchy-heine_period}
yields $F\left(x,y\right)=\int_{\Gamma^{+}}\left(\cdots\right)\dd z$,
which can be analytically continued in the $x$-variable on the self-overlapping
squid sector (Figure~\ref{fig:unbounded_squid_sectors}). As $x$
reaches $\Gamma^{-}$ ``from below'' the analytic continuation coincides
with $\int_{\Gamma^{-}}\left(\cdots\right)\dd z$, because the difference
of determination is given by
\begin{align*}
D\left(x,y\right) & :=\int_{\Gamma^{+}-\Gamma^{-}}\frac{T\left(H_{N}\left(z,y\right)\right)}{z-x}\dd z,
\end{align*}
and Cauchy's formula asserts that $D\left(x,y\right)=0$ whenever
$x$ is outside the saddle-part $\sad V{}{}$ enclosed by $\Gamma^{+}\cup\Gamma^{-}$.
On the contrary if $x\in\sad V{}{}$ then $D\left(x,y\right)=2\ii\pi T\left(H_{N}\left(x,y\right)\right)$,
which is the way to understand~\eqref{eq:cousin_relation}. 
\end{enumerate}
\end{rem}

\begin{proof}
This proposition follows the general lines of \cite[Theorem 2.5]{SchaTey}
for $\varepsilon=0$. A simpler instance of the strategy can be found
in Lemma~\ref{lem:Savelev-Cousin}. Except when necessary we drop
every sub- and super-scripts.

\begin{enumerate}
\item This is nothing but Cauchy residue formula. We indeed compute (omitting
to include the integrand for the sake of readability) 
\begin{align*}
F^{j+1}\left(x,y\right)-F^{j}\left(x,y\right) & =\int_{\Gamma^{j+1,+}}-\int_{\Gamma^{j,+}}+\sum_{p\neq j+2}\int_{\Gamma^{p,-}}-\sum_{p\neq j+1}\int_{\Gamma^{p,-}}\\
 & =\left(\int_{\Gamma^{j+1,-}}-\int_{\Gamma^{j,+}}\right)-\left(\int_{\Gamma^{j+2,-}}-\int_{\Gamma^{j+1,+}}\right).
\end{align*}
The candidate singularity in the common integrand $\frac{T^{p}\left(H_{N}^{p}\left(z,y\right)\right)}{z-x}$
in $\int_{\Gamma^{p+1,-}}-\int_{\Gamma^{p,+}}$ is $z=x$. This happens
only when $x\in\sad V{}{p,}$. By hypothesis $x\in\sad V{}{j,}$ hence~\eqref{eq:cousin_relation}
holds.\\
Actually one needs to use a growing family of compact loops within
$\sad V{}{j,}$ converging toward $\partial\sad V{}{j,}$, then to
apply Cauchy formula to each one of them and take the limit. The only
possible choice for the connected component of $\cc\backslash\left(\Gamma^{j+1,-}\cup\Gamma^{j,+}\right)$
for which this construction works is $\sad V{}{j,}$, since in that
sector we can establish tame estimates for the growth of the integrand
(see~(3) below), and we can also establish untamed estimates outsides
a neighborhood of $\adh{\sad V{}{j,}}$.
\item Taking for granted that the integrand defining $F\left(x,y\right)$
for $\left(x,y\right)\in\adh{{\sect[r][~]}}$ is bounded from above
by a real-analytic, integrable function on $\partial\sad V{}{}$,
the analyticity of $F$ on $\sect[r][~]$ is clear from the definition~\eqref{eq:cauchy-heine_period}.
Integration paths used to evaluate $F$ can be slightly deformed outwards
without changing the value of the integral, which shows that $F$
can be analytically continued to any point $\left(x,y\right)$ with
$x\in\partial V\backslash P_{\varepsilon}^{-1}\left(0\right)$ and
$\left|y\right|\leq r$. Concluding that $F$ extends as a continuous
function on $\adh{{\sect[r][~]}}\backslash P_{\varepsilon}^{-1}\left(0\right)$
is again a consequence of~\eqref{eq:cauchy-heine_period} for $y$
is an extraneous parameter. Dominated convergence of $F\left(x,y\right)$,
continuity on $\adh{\mathcal{V}_{r}}\cap P_{\varepsilon}^{-1}\left(0\right)$
and boundedness of $F$ are established in~(3).
\item We begin with proving~(a). Since, for $p\in\zsk$,
\begin{align*}
\left|T^{p}\left(h\right)\right| & \leq\left|h\right|\norm[T']{},
\end{align*}
we deduce
\begin{align*}
\left|\frac{T^{p}\left(H\left(z,y\right)\right)}{z-x}\right| & \leq\frac{\left|\widehat{H}\left(z,y\right)\right|}{\left|z-x\right|}\norm[T']{}\exp\norm[N]{}.
\end{align*}
We then invoke the estimates derived for the model family in Lemma~\ref{lem:integrable_first-integral},
showing dominated convergence for $F\left(x,y\right)$. In order to
bound $F$ it is sufficient to consider only the problem of bounding
$F$ near a single $\Gamma:=\Gamma^{j,+}$. A uniform bound $K$ for
the rightmost sum of integrals simply requires bounding uniformly
$\frac{1}{\left|z_{*}-x_{*}\right|}$ where $z_{*},~x_{*}\in\rho\sone$.
Of course no uniform bound in $x$ exists when $x$ tends to $\Gamma$
(\emph{i.e.} $x_{*}$ tends to $z_{*}$). To remedy this problem we
bisect $\sad V{}{}$ with a curve $\widehat{\Gamma}$ parallel to
$\Gamma$ and passing through the middle of the arc $\rho\sone\cap\sad V{}{}$.
When $x$ is taken in the component of $V^{j}\backslash\widehat{\Gamma}$
not accumulating on $\Gamma$ the value of $\frac{1}{\left|z_{*}-x_{*}\right|}$
is uniformly bounded. When $x$ is taken in the other part we use
the functional relation~\eqref{eq:cousin_relation}: in that configuration
$x$ is understood as an element of $V^{j+1}$ far from $\Gamma^{j+1,-}$
and we are back to the situation we just solved.

A little bit more detailed analysis allows proving that $x\mapsto F\left(x,y\right)$
is Cauchy\footnote{A function $f$ from a metric space $E$ to another one $F$ is Cauchy
at $a$ if for all $\varepsilon>0$ there exists $\delta>0$ such
that $x,~y\in\tx B\left(a,\delta\right)$ implies $d\left(f\left(x\right),f\left(y\right)\right)<\varepsilon$.} near $\sad x{}{j,}$, so that $F$ extends continuously to $\left\{ \sad x{}{j,}\right\} \times r\ww D$.
Items (b) and~(c) are obtained much in the same way, the details
are straightforward adaptations of~(a).
\end{enumerate}
\end{proof}

\subsubsection{\label{subsec:Regularity-of-section}Holomorphy of $Q_{\varepsilon}$}

Now all functions $\onf[\varepsilon]\cdot\mathfrak{F}^{j}$ patch
on intersecting squid sectors to define 
\begin{align*}
Q & \in\holf[\left(\cc\backslash P_{\varepsilon}^{-1}\left(0\right)\right)\times r\ww D].
\end{align*}
If we show that $Q$ is bounded near each disk $\left\{ \sad x{}{j,}\right\} \times r\ww D$
then Riemann's theorem on removable singularities guarantees the holomorphic
extension of $Q$ to $\cc\times\ww D$. But
\begin{align}
\left|Q\left(x,y\right)\right| & \leq\left|P_{\varepsilon}\left(x\right)\right|\norm[\ppp Fx]{}+\left(1+\left|\mu\right|\left|x\right|^{k}+\left|R\left(x,y\right)\right|\right)\norm[y\ppp Fy]{}\label{eq:estim_secto_section}
\end{align}
so that taking Proposition~\ref{prop:period_cousin}~(3) into account
brings the conclusion.

\subsubsection{\label{subsec:Growth-control-of-section}Growth control of $Q_{\varepsilon}$
near $x=\infty$}

In Section~\ref{subsec:Parametric-normalization} we prove that the
$k$-tuple of sectorial solutions $N$ of the cohomological equation
of normalization $\onf\cdot N^{j}=-R$ satisfy the conditions $\norm[x\ppp Nx]{}\le\frac{1}{3}$
and $\norm[y\ppp Ny]{}\le\frac{1}{3}$ if $r$ is chosen small enough
(Corollary~\ref{cor:secto_normalization_estimate}). 
\begin{lem}
For every fixed $y\in r\ww D$ the entire function $x\mapsto Q\left(x,y\right)$
is actually a polynomial of degree at most $k$, and
\begin{align}
Q\left(x,y\right) & =\sum_{n>0}q_{n}\left(x\right)y^{n}~~~,~~q_{n}\in\pol x_{\leq k}~\label{eq:section_polynomial_expansion}
\end{align}
on $\cc\times r\ww D$. 
\end{lem}

\begin{proof}
Since $x\mapsto R\left(x,y\right)$ is a polynomial of degree at most
$k$, there exists a constant $C>0$ such that $1+\left|\mu\right|\left|x\right|^{k}+\left|R\left(x,y\right)\right|\leq C\left|x\right|^{k}$
for every $\left|x\right|\geq\rho$. The bound~\eqref{eq:estim_secto_section}
on $x\mapsto Q\left(x,y\right)$ also holds near $\infty$ so that
\begin{align*}
\left|Q\left(x,y\right)\right| & \leq\left|\frac{P_{\varepsilon}\left(x\right)}{x}\right|\norm[x\ppp{F^{j}}x]{}+C\left|x\right|^{k}\norm[y\ppp Fy]{}.
\end{align*}
From Proposition~\ref{prop:period_cousin}(3)(b,c) and the control
on $\norm[x\ppp{N^{j}}x]{}$, $\norm[y\ppp{N^{j}}y]{}$ we infer $\norm[x\ppp{F^{j}}x]{},~\norm[y\ppp{F^{j}}y]{}<+\infty$,
from which we deduce $\frac{P_{\varepsilon}\left(x\right)}{x}\norm[x\ppp{F^{j}}x]{}=\OO{x^{k}}$
and finally $Q\left(x,y\right)=\OO{x^{k}}$ as well.
\end{proof}
To complete the proof of Proposition~\ref{prop:section_period_cellular}
we need to modify $Q$ so that $Q\left(0,y\right)=0$. In order not
to change the period of $Q$ we can only subtract from $Q$ a function
of the form $\onf[\varepsilon]\cdot F$ with $F$ holomorphic. This
is done by setting 
\begin{align*}
F\left(y\right) & :=\int_{0}^{y}\frac{Q\left(0,v\right)}{v}\dd v,
\end{align*}
so that $Q-\onf[\varepsilon]\cdot F$ vanishes on $\left\{ x=0\right\} $
while still admitting an expansion of the form~\eqref{eq:section_polynomial_expansion}.

\subsection{\label{subsec:section_period}Stitching cellular sections together:
proof of Proposition~\ref{prop:section_period}}

Fix $\neigh[k]\backslash\Delta_{k}$ and $\rho>0$ not larger than
what is allowed in Lemma~\ref{lem:integrable_first-integral}, and
take $G\in\holb[\rho\ww D\times{{\neigh}}]'$. We prove now that for
any fixed $\varepsilon\in\mathcal{E}_{\ell}$, at most one $Q\in x\pol x_{<k}\left\{ y\right\} '$
exists such that $G-Q\in\tx{im}\left(\onf[\varepsilon]\cdot\right)$,
that is $\per\left(G\right)=\per\left(Q\right)$. This amounts to
showing that $\tx{im}\left(\onf[\varepsilon]\cdot\right)\cap x\pol x_{<k}\left\{ y\right\} '=\left\{ 0\right\} $
for all fixed $\varepsilon\in\neigh[k]\backslash\Delta_{k}$. 

Let $G\in\tx{im}\left(\onf[\varepsilon]\cdot\right)\cap x\pol x_{<k}\left\{ y\right\} '$
and write 
\begin{align*}
G\left(x,y\right) & =\onf[\varepsilon]\cdot\sum_{n\geq d}F_{n}\left(x\right)y^{n}=\sum_{n\geq d}G_{n}\left(x\right)y^{n}\in\holf[r\ww D\times{{\neigh}}]~,~d\in\nn~;
\end{align*}
we claim that $G_{d}=0$, which is sufficient to establish the result.
It turns out that for its part of least degree in $y$ the cohomological
equation only depends on its formal normal form:
\begin{align*}
\fonf[\varepsilon]\cdot\left(y^{d}F_{d}\left(x\right)\right) & =y^{d}G_{d}\left(x\right).
\end{align*}
Such a relation holds if and only if the period of $y^{d}G_{d}$ along
the formal normal form vanishes: $\widehat{\per}\left(y^{d}G_{d}\right)=0$.
Therefore we need to prove that 
\begin{align*}
\widehat{\per}~:~x\pol x_{<k}y^{d} & \longto\cc^{k}h^{d}\\
y^{d}G_{d} & \longmapsto\widehat{\per}\left(y^{d}G_{d}\right)
\end{align*}
 is injective if $\varepsilon$ is small enough. As recalled in Corollary~\ref{cor:characterization_solution_cohomog_fixed_epsilon}
we know that for every $a\in\nn$
\begin{align*}
\lim_{\underset{\mathcal{E}_{\ell}}{\varepsilon\longto0}}\widehat{\per}\left(x^{a}y^{b}\right) & =\widehat{\per}_{0}\left(x^{a}y^{b}\right),
\end{align*}
where $\widehat{\per}_{0}$ is the period of the model saddle-node
$\fonf[0]$. The auxiliary result \cite[Proposition 2]{Tey-ExSN}
states precisely that $\widehat{\per}_{0}$ is invertible, and therefore
so is $\widehat{\per}$ for small $\varepsilon$ as expected.

\section{\label{sec:Analytic}Orbital Realization Theorem}

In this section we address the inverse problem for the classification
of unfoldings performed in~\cite{RouTey}, in the special case of
convergent unfoldings of formal invariant $\mu$ with 
\begin{align*}
\mu_{0}\notin\rr_{\leq0}
\end{align*}
and $\tau=0$. The residual cases $\mu_{0}\leq0$ or $\tau>0$ are
dealt with in Section~\ref{sec:tau}. Also notice that we only carry
this study for the orbital part, the case of the temporal realization
is explained in~\cite{TeySurvey} when $k=1$. Generalizing this
approach for $k>1$ using the tools introduced in Section~\ref{sec:Temporal}
should not be difficult.

\bigskip{}

We summarize in Section~\ref{subsec:Invariants} how the invariants
of classification are built. They unfold Martinet-Ramis's invariants~\cite{MaRa-SN}
for the limiting saddle-node, obtained as transition maps between
sectorial spaces of leaves. Yet the construction can only be carried
out analytically on a given parametric cell $\mathcal{E}_{\ell}$,
yielding a cellular invariant $\mathfrak{m}_{\ell}\in\prod_{\zsk}\mathcal{H}_{\ell}\left\{ h\right\} $
(see Section~\ref{subsec:Proof's-reduction} for the definition of
the functional spaces $\mathcal{H}_{\ell}$ and Section~\ref{subsec:Invariants}
for the definition of $m_{\ell}$). The orbital modulus $\mathfrak{m}\left(X\right)$
of an unfolding $X$ consists in the whole collection $\left(\mathfrak{m}_{\ell}\right)_{\ell}$.
\begin{defn}
\label{def:realizable}We say that $\left(\mu,\mathfrak{m}\right)\in\germ{\varepsilon}\times\prod_{\ell}\mathcal{H}_{\ell}\left\{ h\right\} ^{k}$
is\textbf{ realizable} if there exists a generic convergent unfolding
$X$ with formal orbital class $\mu$ and orbital modulus $\mathfrak{m}=\mathfrak{m}\left(X\right)$.
\end{defn}

In Section~\ref{subsec:Parametric-normalization} we prove the next
result.
\begin{thm}
\label{thm:cellular_realization}Assume $\tau=0$ (which particularly
implies $\mu_{0}\notin\rr_{\leq0}$). Fix a germ at $0\in\neigh[k+1]$
of a cell $\mathcal{E}_{\ell}$. Given $\mathfrak{m}_{\ell}\in\prod_{\zsk}\mathcal{H}_{\ell}\left(h\right)$
and $\mu$ with $\mu_{0}\notin\rr_{\leq0}$, there exists a unique
$R_{\ell}\in x\mathcal{H}_{\ell}\left\{ y\right\} \left[x\right]_{<k}$
such that
\begin{align*}
\onf[\ell,\varepsilon] & :=\fonf+yR_{\ell,\varepsilon}\pp y
\end{align*}
has $\mathfrak{m}_{\ell}$ for transition maps in sectorial space
of leaves (\emph{i.e.} for modulus).
\end{thm}

The fact that this ``analytical synthesis'' gives unique forms of
the same kind as those given by Loray's ``geometric'' construction
bolsters the naturalness of the normal forms presented here. Indeed
the next corollary provides an indirect solution of the inverse problem.
\begin{cor}
\label{cor:realizable_iff_glue}A couple $\left(\mu,\mathfrak{m}\right)$
with $\mu_{0}\notin\rr_{\leq0}$ is realizable if and only if $R_{\ell,\varepsilon}=R_{\widetilde{\ell},\varepsilon}$
for all $\varepsilon\in\mathcal{E}_{\ell}\cap\mathcal{E}_{\widetilde{\ell}}$
and all $\left(\ell,\widetilde{\ell}\right)$. 
\end{cor}

\begin{proof}
The equality $R_{\ell}=R_{\widetilde{\ell}}$ on $\mathcal{E}_{\ell}\cap\mathcal{E}_{\widetilde{\ell}}$
defines a bounded, holomorphic function $R$ in the parameter $\varepsilon\in\neigh[k+1]\backslash\Delta_{k}$,
which extends holomorphically to a whole neighborhood $\neigh[k+1]$
by Riemann's theorem on removable singularities. The corresponding
unfolding $\onf$ has modulus $\mathfrak{m}\left(\onf\right)=\left(\mathfrak{m}_{\ell}\right)_{\ell}$
by construction.

Conversely, the Normalization Theorem tells us that we can as well
assume that the vector field is in normal form $\onf$~\eqref{eq:orbital_normal_form},
without changing the orbital modulus $\mathfrak{m}=\mathfrak{m}\left(\onf\right)$.
Moreover, the normalization can be performed by tangent-to-identity
mappings in the $y$-variable. According to Theorem~\ref{thm:cellular_realization},
$R_{\ell}$ is uniquely determined by the component $\mathfrak{m}_{\ell}$
of $\mathfrak{m}$, hence $R=R_{\ell}$ on $\mathcal{E}_{\ell}$.
\end{proof}
Somehow this characterization is not satisfying since it involves
the auxiliary unfolding $\onf[\ell]$. In Section~\ref{subsec:Compatibility-condition}
we present an intrinsic characterization of realizable $\left(\mu,\mathfrak{m}\right)$
as a \emph{compatibility condition} imposed on the different dynamics
induced by each pair $\left(\mu,\mathfrak{m}_{\ell}\right)$ on the
sectorial space of leaves (Definition~\ref{def:compatibility_condition}).
Roughly speaking the condition requires that the abstract holonomy
groups be conjugate over cells overlaps. In case of an actual unfolding
$X$ (\emph{i.e.} realizable $\left(\mu,\mathfrak{m}\right)$) these
groups represent in the space of leaves the actual weak holonomy group
induced by $X$ in $\left(x,y\right)$-space. 

\subsection{\label{subsec:Invariants}Classification moduli}

Starting from a generic convergent unfolding $X$ of codimension $k$
in prepared form~\eqref{eq:prepared_form_orbital} with given orbital
formal invariant $\mu$ (with no restriction on $\mu_{0}$), we can
build the following $k$-tuple of periods (Definition~\ref{def:period_operator})
on a germ of a cellular decomposition $\left(\mathcal{E}_{\ell}\right)_{1\leq\ell\leq C_{k}}$,
called the \textbf{orbital modulus }of $X$:
\begin{align}
\mathfrak{m}\left(X\right) & :=\left(\mathfrak{m}_{\ell}\left(X\right)\right)_{1\leq\ell\leq C_{k}},\nonumber \\
\mathfrak{m}_{\ell}\left(X\right) & :=\left(\sad{\phi}{\ell}{j,}\right)_{j\in\zsk},\nonumber \\
\sad{\phi}{\ell}{j,} & :=2\ii\pi\per[j][\ell]\left(-R\right)\in\mathcal{H}_{\ell}\left\{ h\right\} .\label{eq:invariant_as_period}
\end{align}
We state the main result of~\cite{RouTey} in the specific context
of convergent unfoldings.
\begin{defn}
\label{def:moduli_equivalence}~

\begin{enumerate}
\item Fix a germ of a cell $\mathcal{E}_{\ell}$. For $c\in\germ{\varepsilon}^{\times}$,
$\theta\in\zsk$ and $f=\left(f^{j}\right)_{j\in\zsk}\in\mathcal{H}_{\ell}\left\{ h\right\} ^{k}$
define
\begin{align*}
\left(c,\theta\right)^{*}f~:~\left(\varepsilon,h\right) & \longmapsto\left(f_{\varepsilon}^{j+\theta}\left(c_{\varepsilon}h\right)\right)
\end{align*}
and extend component-wise this action to tuples. 
\item We say that two collections $\mathfrak{m},~\widetilde{\mathfrak{m}}\in\prod_{\ell}\mathcal{H}_{\ell}\left\{ h\right\} ^{k}$
are \textbf{equivalent} if there exists $c\in\germ{\varepsilon}^{\times}$
and $\theta\in\zsk$ such that 
\begin{align}
\left(c,\theta\right)^{*}\mathfrak{m} & =\mathfrak{\widetilde{m}}.\label{eq:moduli_equivalence}
\end{align}
\end{enumerate}
\end{defn}

\begin{rem}
The presentation of Definition~\ref{def:moduli_equivalence} is equivalent
to that of~\cite{MaRa-SN} for $\varepsilon=0$. The transition functions
there are simply given by $\psi^{j,s}(h)=h\exp\left(\frac{2i\pi\mu}{k}+\phi^{j,s}\right)$.
This fact will be explained in more details in Section~\ref{subsec:Compatibility-condition}.
\end{rem}

\begin{thm}
\label{thm:classification}\cite{RouTey} Two generic, prepared convergent
unfoldings $X$ and $\widetilde{X}$, in the same formal orbital class
$\mu$ with respective orbital moduli $\mathfrak{m}\left(X\right)$
and $\mathfrak{m}\left(\widetilde{X}\right)$, are equivalent by some
local analytic diffeomorphism if and only if their respective orbital
moduli $\mathfrak{m}\left(X\right)$ and $\mathfrak{m}\left(\widetilde{X}\right)$
are equivalent. Moreover $X$ is locally equivalent to its formal
normal form $\fonf$ if and only if $\mathfrak{m}\left(X\right)=0$.
\end{thm}

The pair $\left(c,\theta\right)$ involved in the equivalence between
moduli has a geometrical interpretation. First set $\lambda:=\exp2\ii\pi\nf{\theta}k$
and apply the diagonal mapping
\begin{align*}
\left(\varepsilon_{0},~\ldots,~\varepsilon_{k-1},~x\right)\longmapsto & \left(\varepsilon_{0}\lambda^{-1},~\ldots,~\varepsilon_{j}\lambda^{j-1},~\ldots,~\varepsilon_{k-1}\lambda^{k-2},~x\lambda\right)
\end{align*}
to $X$ so that the moduli of the new unfolding, still written $X$,
differs from the original by a shift in the indices $j$ of offset
$\theta$, as explained in Section~\ref{subsec:Preparation}. According
to Corollary~\ref{cor:formal_symmetries} we may as well restrict
our study now to fibered conjugacies $\Psi$ between $X$ and $\widetilde{X}$
fixing $\left\{ y=0\right\} $. Under these assumptions we have 
\begin{align*}
\Psi~:~\left(\varepsilon,x,y\right)\mapsto\left(\varepsilon,x,y\left(c+\oo 1\right)\right) & .
\end{align*}
This very fact explains why $c$ is independent on the cell $\mathcal{E}_{\ell}$
in the equivalence relation~\eqref{eq:moduli_equivalence}.

\subsection{\label{subsec:Parametric-normalization}Parametric normalization:
proof of Theorem~\ref{thm:cellular_realization}}

In this section we solve the inverse problem on a given parametric
cell $\mathcal{E}_{\ell}$ when $\mu_{0}$ is not in $\rr_{\leq0}$.
Given any collection 
\begin{align*}
\mathfrak{m}_{\ell} & :=\left(\sad{\phi}{}{j,}\right)_{j}\in\prod_{\zsk}\mathcal{H}_{\ell}\left\{ h\right\} 
\end{align*}
we can fix $\eta>0$ such that every $\sad{\phi}{}{j,}$ belongs to
$\holb[\mathcal{E}_{\ell}\times\eta\ww D]'$. The strategy is to synthesize
a $k$-tuple of sectorial functions $\left(H^{j}\right)_{j}$ whose
transition maps over saddle parts are determined by $\mathfrak{m}_{\ell}$
as in~\eqref{eq:invar_transition_maps} below, then to recognize
that they actually are sectorial first-integrals of a holomorphic
vector field $X_{\varepsilon}$ in normal form.

\bigskip{}

We repeat the recipe of Theorem~\ref{thm:fibred_gluing_parametric}
in order to solve the nonlinear equation
\begin{align}
H^{j+1} & =H^{j}\exp\left(\nf{2\ii\pi\mu}k+\sad{\phi}{}{j,}\circ H^{j}\right),\label{eq:invar_transition_maps}
\end{align}
by successively solving the linear Cousin problem of Proposition~\ref{prop:period_cousin}
in the way we explain now. For $\varepsilon:=0$ this is precisely
the technique of~\cite{SchaTey}.

We want to find a solution of $N=\frac{1}{2\pi i}\mathfrak{F}\left(\mathfrak{m}_{\ell},N\right)$
with $N\in\holb[V_{\ell}^{j}\times r\ww D]'$, where $\mathfrak{F}$
is given in~\eqref{eq:cauchy-heine_period}, and we build one through
an iterative process. We start from 
\begin{align*}
N_{0} & :=\left(0\right)_{j}
\end{align*}
and build 
\begin{align*}
N_{n+1} & :=\frac{1}{2\ii\pi}\mathfrak{F}\left(\mathfrak{m}_{\ell},N_{n}\right)
\end{align*}
given by Proposition~\ref{prop:period_cousin}. The fact that each
sequence $\left(N_{n}^{j}\right)_{n}$ converges uniformly to some
$N^{j}\in\holb[V_{\ell}^{j}\times r\ww D]'$ for some $r>0$ follows
in every other respect the argument presented in the proof of Theorem~\ref{thm:fibred_gluing_parametric},
thus we shall not repeat it here.

So far we have built a $k$-tuple of bounded, holomorphic functions
$N=\left(N^{j}\right)_{j}$ satisfying the next properties.
\begin{cor}
\label{cor:secto_normalization_estimate}Assume $\tau=0$ (which particularly
implies $\mu_{0}\notin\rr_{\leq0}$). Let 
\begin{align*}
H^{j} & :=\widehat{H}^{j}\exp N^{j}
\end{align*}
be the canonical first-integral associated with $N^{j}$.

\begin{enumerate}
\item $\left(H^{j}\right)_{j}$ is a solution of~\eqref{eq:invar_transition_maps}.
\item Up to decrease $r>0$ we can assume that:

\begin{enumerate}
\item $N$ is $\eta$-adapted (as in Definition~\ref{def:etha_adapted}),
more precisely:
\begin{align*}
\norm[H^{j}]{} & \leq rC
\end{align*}
for some constant $C>0$,
\item $\left|x\ppp{N^{j}}x\right|\leq\frac{1}{3}$ and $\left|y\ppp{N^{j}}y\right|\leq\frac{1}{3}$
on $V^{j}\times r\ww D$.
\end{enumerate}
\end{enumerate}
\end{cor}

\begin{proof}
~

\begin{enumerate}
\item Because $H^{j}=\widehat{H}^{j}\exp N^{j}$ and $\widehat{H}^{j+1}=\widehat{H}^{j}\exp\nf{2\ii\pi\mu}k$
(see~\eqref{eq:formal_transition_map}) we have
\begin{align*}
\frac{H^{j+1}}{H^{j}} & =\exp\left(2\ii\pi\nf{\mu}k+N^{j+1}-N^{j}\right).
\end{align*}
Because $\left(N^{j}\right)_{j}$ is obtained as the fixed-point of
the Cauchy-Heine operator
\begin{align*}
\left(N^{j}\right)_{j} & \longmapsto\frac{1}{2\ii\pi}\mathfrak{F}\left(\left(\sad{\phi}{}{j,}\right)_{j},\left(N^{j}\right)_{j}\right),
\end{align*}
according to Proposition~\ref{prop:period_cousin}~(1) the identity
$N^{j+1}-N^{j}=\sad{\phi}{}{j,}\circ H^{j}$ holds, which validates
the claim.
\item ~

\begin{enumerate}
\item This is clear thanks to Proposition~\ref{prop:period_cousin}.
\item Up to decrease slightly $\eta$ we can assume that the derivative
of each component of $\mathfrak{m}_{\ell}$ is bounded on $\eta\ww D$.
From the construction of $N^{j}$ and Proposition~\ref{prop:period_cousin}~(3)
we have
\begin{align*}
\frac{\norm[y\frac{\partial N_{n+1}}{\partial y}]{}}{1+\norm[y\frac{\partial N_{n}}{\partial y}]{}} & \leq\frac{rK}{2\ii\pi}\norm[\mathfrak{m}']{}\exp\norm[N_{n}]{}\leq\frac{1}{4},
\end{align*}
if $r$ is taken small enough. The conclusion follows by taking the
limit $n\to\infty$. The argument for $x\ppp Nx$ is identical.
\end{enumerate}
\end{enumerate}
\end{proof}
Now define
\begin{align*}
X^{j} & :=\fonf+yR^{j}\pp y
\end{align*}
with
\begin{align}
R^{j} & :=-\frac{P\ppp Nx+y\left(1+\mu x^{k}\right)\ppp Ny}{1+y\ppp{N^{j}}y}.\label{eq:R_secto_realization}
\end{align}

\begin{lem}
\label{lem:synthesis_vector_field}~

\begin{enumerate}
\item $X^{j}\cdot N^{j}=-R^{j}$ or, equivalently, $X^{j}\cdot H^{j}=0$.
\item $R^{j+1}=R^{j}$ on $\sad{\mathcal{V}}{}{j,}$.
\end{enumerate}
\end{lem}

\begin{proof}
This is formally the same proof as for $\varepsilon=0$: we refer
to~\cite{SchaTey} for details.

\begin{enumerate}
\item follows from elementary calculations. 
\item is equivalent to showing $X^{j}\cdot H^{j+1}=0$. But this condition
is met because of~(1) and the fact that $H^{j+1}$ is a function
of $H^{j}$, as per~\eqref{eq:invar_transition_maps}.
\end{enumerate}
\end{proof}
The lemma indicates that all pieces of $\left(R^{j}\right)_{j}$ glue
together into a holomorphic function $R$. From~\eqref{eq:R_secto_realization}
and the estimates on the derivatives of $N^{j}$ obtained in Corollary~\ref{cor:secto_normalization_estimate}
we conclude that $R$ is bounded near the roots of $P_{\varepsilon}$
(hence Riemann's theorem on removable singularities applies). The
argument of Section~\ref{subsec:Growth-control-of-section} can now
be invoked identically with $Q:=R$ to obtain
\begin{align*}
R\left(x,y\right) & =\sum_{n>0}r_{n}\left(x\right)y^{n}
\end{align*}
for some polynomials $r_{n}$ in $x$ of degree at most $k$. We can
simplify $R$ further by applying to $\fonf+Ry\pp y$ the change of
coordinates
\begin{align*}
\left(x,y\right) & \longmapsto\left(x,~y\exp N\left(y\right)\right)~~,~N\in y\germ y
\end{align*}
 where
\begin{align*}
N' & =-\frac{R\left(0,y\right)}{y\left(1+R\left(0,y\right)\right)}.
\end{align*}
The new vector field $\fonf+\widetilde{R}y\pp y$ satisfies $\widetilde{R}\in x\mathcal{H}_{\ell}\left\{ y\right\} \left[x\right]_{<k}$,
as sought.
\begin{rem}
Notice that Lemma~\ref{lem:synthesis_vector_field} asserts $\left(x,y\right)\mapsto\left(x,N_{\ell}^{j}\left(x,y\right)\right)$
is a fibered normalization of $\onf$ over squid sectors.
\end{rem}

\subsection{\label{subsec:Compatibility-condition}Compatibility condition}

Here we impose no restriction on $\mu_{0}$.

\subsubsection{Node-leaf coordinates}

To each squid sector $V_{\ell}^{j}$ we attach a unique natural coordinate
$h$ which parametrizes the space of leaves $\Omega_{\ell}^{j}$ over
that sector: this coordinate corresponds to values taken by the canonical
first-integral $H_{\ell}^{j}$ (with connected fibers) as defined
in Corollary~\ref{cor:secto_normalization_estimate}. Moreover,
\begin{align*}
H_{\ell}^{j}\left(V_{\ell}^{j}\times\neigh\right) & =\cc.
\end{align*}
This comes from the fact that the sector's shape adheres to the point
in a node-like configuration, forcing the model first integral $\widehat{H}_{\ell}^{j}$
to be surjective: a complete proof of the above statement can be found
in~\cite{RouTey}. This space of leaves is customarily compactified
as the Riemann sphere $\Omega_{\ell}^{j}$ by adding the point $\infty$
corresponding to the ``vertical separatrices'' $\left\{ x=\nod x{}{j,}\right\} $
of the node-type singularity.

Because we deal with convergent unfoldings, this coordinate is completely
determined by the space of leaves of the singular point $\nod x{}{j,}$
of node type attached to $V^{j}$, with two distinguished leaves corresponding
to $0$ (along $\left\{ y=0\right\} $) and $\infty$ (along $\left\{ x=\nod x{}{j,}\right\} $).
In particular, it remains the same when we change the point(s) of
saddle type $\sad x{}{j,}$ and $\sad x{}{\sigma\left(j\right),}$
attached to a sector $V^{j}$ but leave the point of node type $\nod x{}{j,}$
unchanged, while passing from one cell to another.

Let us prove briefly the result on which the compatibility condition
is built. We recall that $\rho_{\varepsilon}$ is the radius of a
disk containing all roots of $P_{\varepsilon}$, as defined by~\eqref{eq:rho_eps}.
\begin{lem}
\label{lem:first-int_multiplier} For every $x_{*}\in V_{\ell}^{j}\backslash\rho_{\varepsilon}\ww D$
the partial mapping
\begin{align*}
h_{\ell}^{j}~:~y & \longmapsto H_{\ell}^{j}\left(x_{*},y\right)
\end{align*}
is a local diffeomorphism near $0$ whose multiplier at $0$ does
not depend on $\ell$. In particular for any $\widetilde{\ell}$ such
that $\mathcal{E}_{\ell}\cap\mathcal{E}_{\widetilde{\ell}}\neq\emptyset$,
the diffeomorphism
\begin{align*}
\delta & :=h_{\widetilde{\ell}}^{j}\circ\left(h_{\ell}^{j}\right)^{\circ-1}
\end{align*}
is tangent-to-identity. Moreover there exists $\eta_{1},\eta_{2},r>0$
such that for all $\varepsilon\in\adh{\mathcal{E}_{\ell}\cap\mathcal{E}_{\widetilde{\ell}}}$
\[
\eta_{1}\disc\subset\delta_{\varepsilon}\left(r\disc\right)\subset\eta_{2}\disc
\]
and $\delta_{\varepsilon}$ is injective on $r\ww D$. 

In the sequel we write this map $\delta_{\widetilde{\ell}\leftarrow\ell}$.
\end{lem}

\begin{proof}
According to Corollary~\ref{cor:secto_normalization_estimate} we
have 
\begin{align*}
H_{\ell}^{j}\left(x_{*},y\right) & =y\widehat{H}_{\ell}^{j}\left(x_{*},1\right)+\oo y.
\end{align*}
Since $x_{*}$ lies outside the disk containing the roots of $P$
the value of $\widehat{H}_{\ell}^{j}\left(x_{*},1\right)$, as fixed
by the determination chosen in~\eqref{eq:model_first_integral},
does not depend on $\ell$ (but it does on $j$). The existence of
$\eta_{1},\eta_{2},r>0$ satisfying the expected properties is a consequence
of~\cite[Corollary 8.8]{RouTey} and Lemma~\ref{lem:integrable_first-integral}~(1).
\end{proof}
\begin{defn}
\label{def:node_coordinate}For a choice of $x_{*}^{j}\in V^{j}\backslash\rho_{\varepsilon}\ww D$
we call $h_{\ell}^{j}$ the \textbf{node-leaf coordinate} of the unfolding
$X_{\ell}$ above $x_{*}$ in the sector $V_{\ell}^{j}$ and relative
to the cell $\mathcal{E}_{\ell}$.
\end{defn}

\subsubsection{Necklace dynamics}

\begin{figure}
\hfill{}\includegraphics[width=0.9\columnwidth]{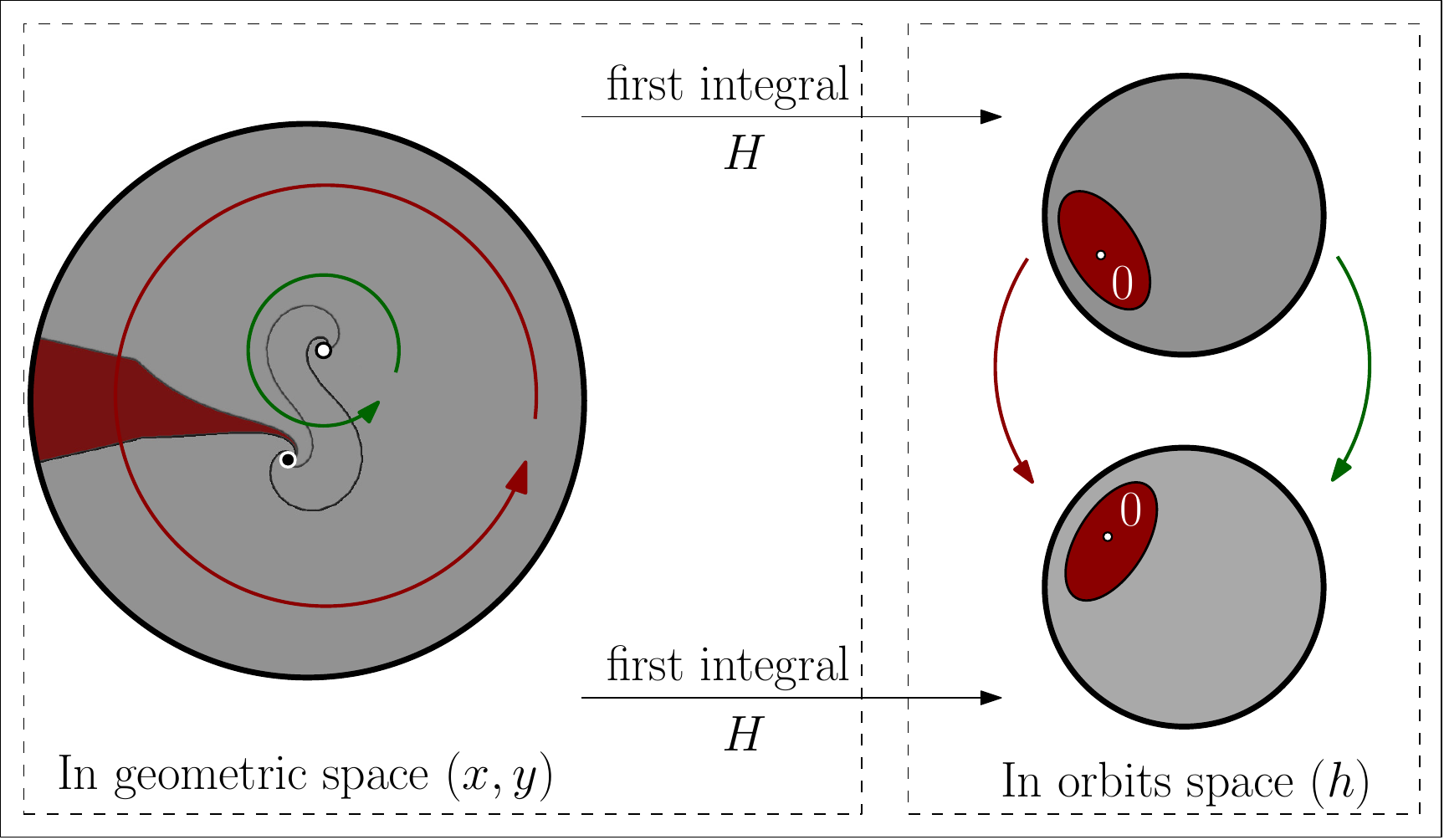}\hfill{}

\caption{Passing from geometric to orbits space \emph{via }the sectorial first
integral $H^{j}$. Colored arrows show how the change of determination
in $H^{j}$ takes place as a mapping between two sectorial spaces
of leaves: the necklace dynamics.}
\end{figure}
Here we work in a fixed germ of a cell $\mathcal{E}_{\ell}$ for fixed
$\varepsilon\in\mathcal{E}_{\ell}$; we drop the $\ell$ and $\varepsilon$
indices whenever not confusing. According to the constructions performed
in~\cite{RouTey}, and hinted at by Theorem~\ref{thm:cellular_realization},
the orbital modulus $\left(\mu,\mathfrak{m}\left(X\right)\right)$
of a convergent unfolding encodes the way the different node-leaf
coordinates glue above the intersection of squid sectors:
\begin{align*}
\begin{cases}
H^{j+1}=H^{j}\exp\left(\nf{2\ii\pi\mu}k+\sad{\phi}{}{j,}\circ H^{j}\right) & ~~~\mbox{above }\sad V{}{j,},\\
H^{\sigma\left(j\right)}=L_{\nu^{j}}\circ H^{j} & ~~~\mbox{above }\gat V{}{j,},
\end{cases}
\end{align*}
 where 
\begin{align*}
L_{c} & ~:~h\longmapsto ch~~~~,~c\neq0,
\end{align*}
and $\nu^{j}=\nu_{\ell}^{j}\in\cc^{\times}$ relates to the dynamical
invariants $\mu$ and the residues $\left(\frac{1}{P_{\varepsilon}'(x^{m})}\right)_{m}$
at the roots $\left(x^{m}\right)_{0\leq m\leq k}$: indeed, the ramification
at the linear level of the first integral at a singular point, given
by $\exp\left(-2\ii\pi\frac{1+\mu_{\varepsilon}\left(x^{m}\right)^{k}}{P_{\varepsilon}'(x^{m})}\right)$,
is equal to the product of all ramifications when crossing sectors
while turning around the point, \emph{i.e. }to the product of one
factor $\exp\nf{2\ii\pi\mu}k$ for each crossed sector $\sad V{}{j,}$
and one factor $\nu_{j}$ for each crossed sector $\gat V{}{j,}$.
It is therefore rather natural to consider the germs of diffeomorphisms
in node-leaf coordinate
\begin{align}
\sad{\psi}{\ell}{j,} & ~:~h\longmapsto h\exp\left(\nf{2\ii\pi\mu}k+\sad{\phi}{\ell}{j,}\left(h\right)\right),\label{eq:necklace_diffeo}\\
\gat{\psi}{\ell}{j,} & ~:~h\longmapsto\nu_{\ell}^{j}h,\nonumber 
\end{align}
where $\left(\mathfrak{m}_{\ell}\right)_{\ell}=\mathfrak{m}\left(X\right)$
and $\mathfrak{m}_{\ell}=\left(\sad{\phi}{\ell}{j,}\right)_{j}$.
Obviously one can do the same construction starting from any tuple
$\mathfrak{m}\in\prod_{\ell}\mathcal{H}_{\ell}\left\{ h\right\} ^{k}$.
\begin{rem}
For some value of the parameter $\varepsilon$ in a given cell $\mathcal{E}_{\ell}$,
the saddle mappings $\sad{\psi}{}{}$ are entirely determined by $\mu$
and $\mathfrak{m}$, while the gate mappings $\gat{\psi}{}{}$ are
entirely determined by $\mu$.
\end{rem}

The dynamics induced by these germs is of interest to us only if it
encodes the underlying dynamics of the unfolding (weak holonomy group).
A necessary condition is that the latter group does not depend on
$\ell$, \emph{i.e.} on the peculiar way of slicing the space into
sectors which is imposed by our construction. Therefore we only want
to consider the ``abstract'' holonomy representation of $\pi_{1}\left(\rho\ww D\backslash P_{\varepsilon}^{-1}\left(0\right),x_{*}\right)$
in the space of leaves. Let us describe this representation (see Figure~\ref{fig:necklace dynamics}
for an example).

\begin{figure}
\hfill{}\includegraphics[width=6cm]{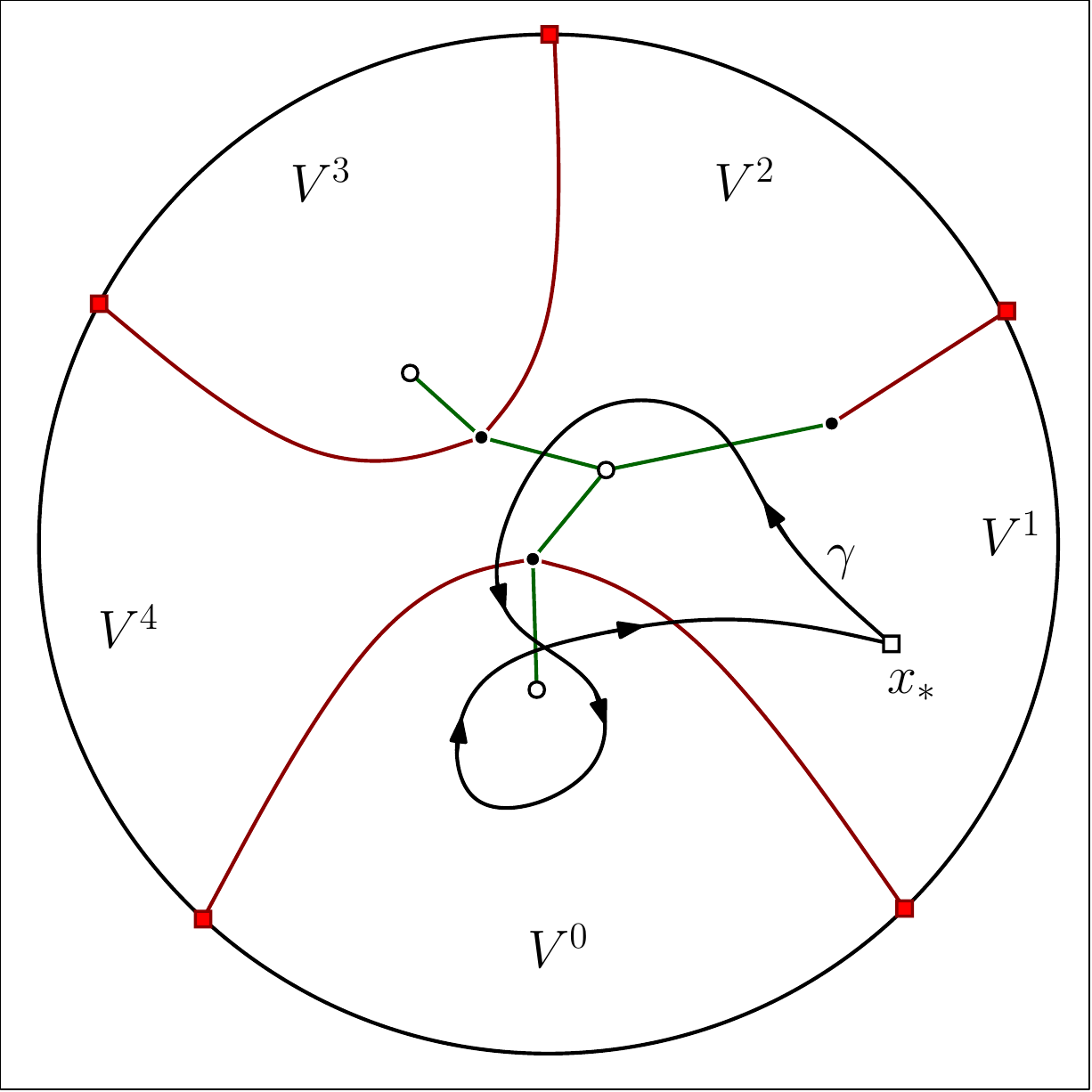}\hfill{}\includegraphics[width=6cm]{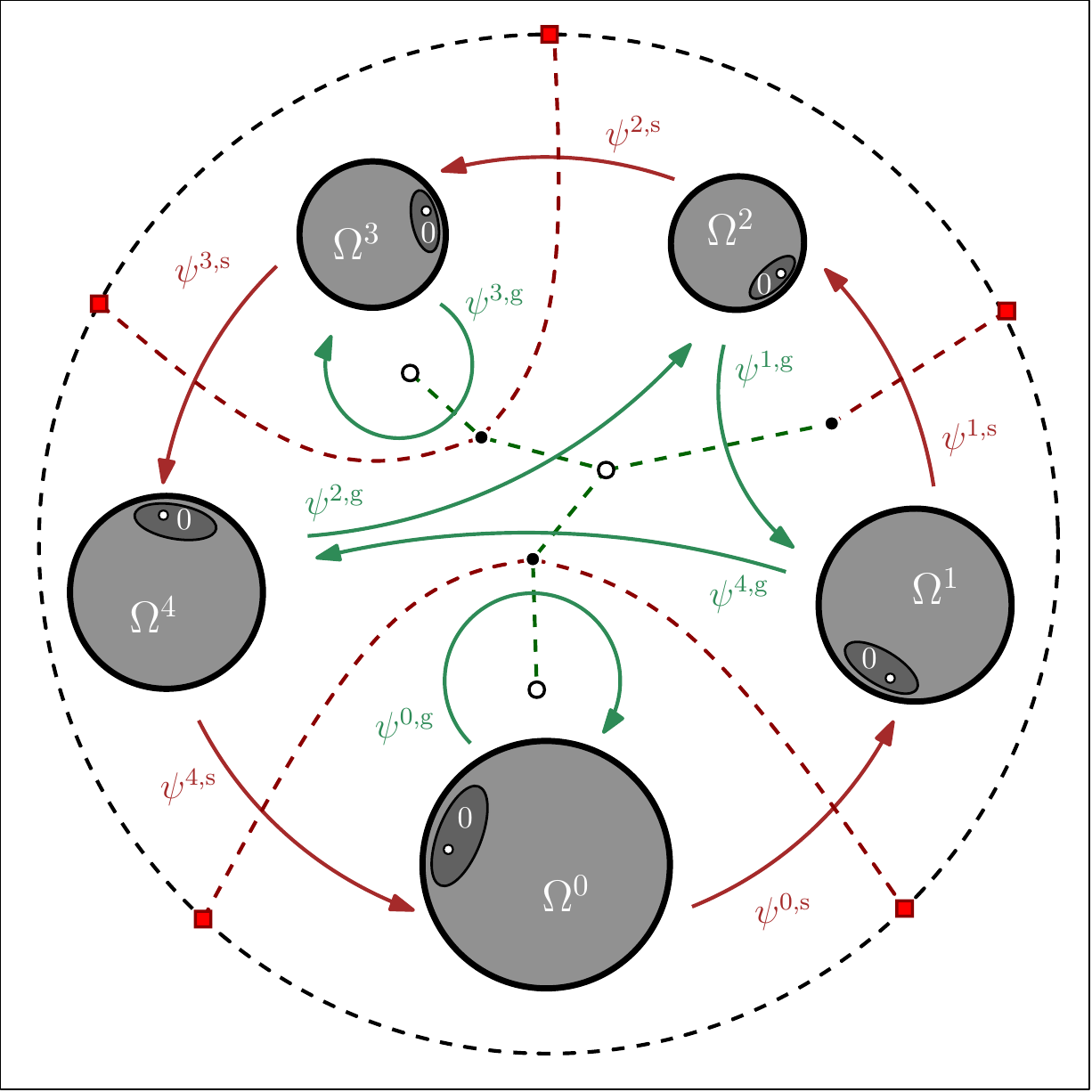}\hfill{}

\caption{\label{fig:necklace dynamics}Schematics of the necklace dynamics
and of the corresponding sectorial decomposition for $k=5$ and $\sigma=\left(\protect\begin{array}{ccccc}
0 & 1 & 2 & 3 & 4\protect\\
0 & 4 & 1 & 3 & 2
\protect\end{array}\right)$. The loop $\gamma\in\pi_{1}\left(\rho\protect\ww D\backslash P_{\varepsilon}^{-1}\left(0\right),x_{*}\right)$
corresponds to the word $\mathfrak{w}\left(\gamma\right)=\protect\tx{s_{0}^{+}g_{0}^{+}g_{0}^{+}s_{4}^{+}g_{2}^{-}g_{1}^{-}}$
in necklace dynamics.}
\end{figure}
\begin{defn}
\label{def:necklace_dynamics}We fix a base-sector $V^{j_{*}}$ and
a base-point $x_{*}\in V^{j_{*}}\backslash\rho_{\varepsilon}\ww D$,
as well as some $\mathfrak{m}_{\ell}=\left(\sad{\phi}{\ell}{j,}\right)_{j}\in\mathcal{H}_{\ell}\left\{ h\right\} ^{k}$.

\begin{enumerate}
\item To any loop $\gamma\in\pi_{1}\left(\rho\ww D\backslash P_{\varepsilon}^{-1}\left(0\right),x_{*}\right)$
we associate the multiplicative word $\mathfrak{w}_{\ell}\left(\gamma\right)$
in the $4k$ letters $\left\{ \tx s_{j}^{\pm},~\tx g_{j}^{\pm}~:~j\in\zsk\right\} $
obtained by keeping track of bounded squid sectors boundaries crossed
successively when traveling along $\gamma$. The superscript $+$
(\emph{resp.} $-$) is given to $\tx s_{j}$ according to whether
one crosses the saddle boundary from $V^{j}$ to $V^{j+1}$ (\emph{resp}.
from $V^{j+1}$ to $V^{j}$), ``in the same direction'' as $\sad{\psi}{}{j,}$
(\emph{resp}. $\left(\sad{\psi}{}{j,}\right)^{\circ-1}$) applies.
For $\tx g_{j}$ we take the same convention for gate transitions
$\gat{\psi}{}{j,}$ and postulate the algebraic relations $\left(\tx s_{j}^{\pm}\right)^{-1}=\tx s_{j}^{\mp}$,
$\left(\tx g_{j}^{\pm}\right)^{-1}=\tx g_{j}^{\mp}$. 
\item To any word $\mathfrak{w}=\prod_{n}\omega_{j_{n}}^{\pm}$ we associate
the germ
\begin{align*}
\psi_{\ell}\left[\mathfrak{w}\right]~:~h & \longmapsto\bigcirc_{n}\left(\psi_{\ell}^{j_{n},\omega}\right)^{\circ\pm1}.
\end{align*}
For instance
\begin{align*}
\psi\left[\tx{s_{0}^{+}g_{0}^{+}g_{0}^{+}s_{4}^{+}g_{2}^{-}g_{1}^{-}}\right] & =\sad{\psi}{}{0,}\circ\left(\gat{\psi}{}{0,}\right)^{\circ2}\circ\sad{\psi}{}{4,}\circ\left(\gat{\psi}{}{2,}\right)^{\circ-1}\circ\left(\gat{\psi}{}{1,}\right)^{\circ-1}.
\end{align*}
\item We write
\begin{align*}
\mathcal{W}_{\ell} & :=\mathfrak{w}_{\ell}\left(\pi_{1}\left(\rho\ww D\backslash P_{\varepsilon}^{-1}\left(0\right),x_{*}\right)\right)
\end{align*}
 the image group of \textbf{admissible words}, that is all the words
corresponding to all the encodings~(1) of a loop with given base-point
$x_{*}$ in a disk of given radius $\rho$ punctured with the roots
of $P_{\varepsilon}$.
\item Let $\mathfrak{m}=\left(\mathfrak{m}_{\ell}\right)_{\ell}\in\prod_{\ell}\mathcal{H}_{\ell}\left\{ h\right\} $.
The collection of image groups $\mathcal{G}\left(\mathfrak{m}\right)=\left(\mathcal{G}_{\ell}\right)_{\ell}$
of germs of a biholomorphism fixing $0$ given by 
\begin{align*}
\mathcal{G}_{\ell} & :=\psi_{\ell}\left[\mathcal{W}_{\ell}\right],
\end{align*}
is called the \textbf{necklace dynamics} associated to $\left(\mu,\mathfrak{m}\right)$
based at the sector $V^{j_{*}}$.
\end{enumerate}
\end{defn}

\begin{rem}
\label{rem:necklace_dynamics}~

\begin{enumerate}
\item To keep notations light we write $\psi_{\ell}\left[\gamma\right]$
instead of $\psi_{\ell}\left[\mathfrak{w}_{\ell}\left(\gamma\right)\right]$
for $\gamma\in\pi_{1}\left(\rho\ww D\backslash P^{-1}\left(0\right),x_{*}\right)$.
The context will never be ambiguous.
\item Obviously the morphisms $\mathfrak{w}_{\ell}$ and $\mathfrak{w}_{\widetilde{\ell}}$
are distinct. The change of cell in $\mathcal{E}_{\ell}\cap\mathcal{E}_{\widetilde{\ell}}$
can be translated algebraically as a group isomorphism $\mathcal{W}_{\ell}\longto\mathcal{W}_{\widetilde{\ell}}$.
For instance when $k=1$ the isomorphism acts on generators as
\begin{align*}
\begin{cases}
\tx g^{+} & \longmapsto\tx g^{-}\tx s^{+}\\
\tx s^{+}\tx g^{-} & \longmapsto\tx g^{+}
\end{cases}
\end{align*}
with notations of Figure~\ref{fig:compare_cell_holnomy}.
\end{enumerate}
\end{rem}

\begin{figure}
\hfill{}\includegraphics[width=10cm]{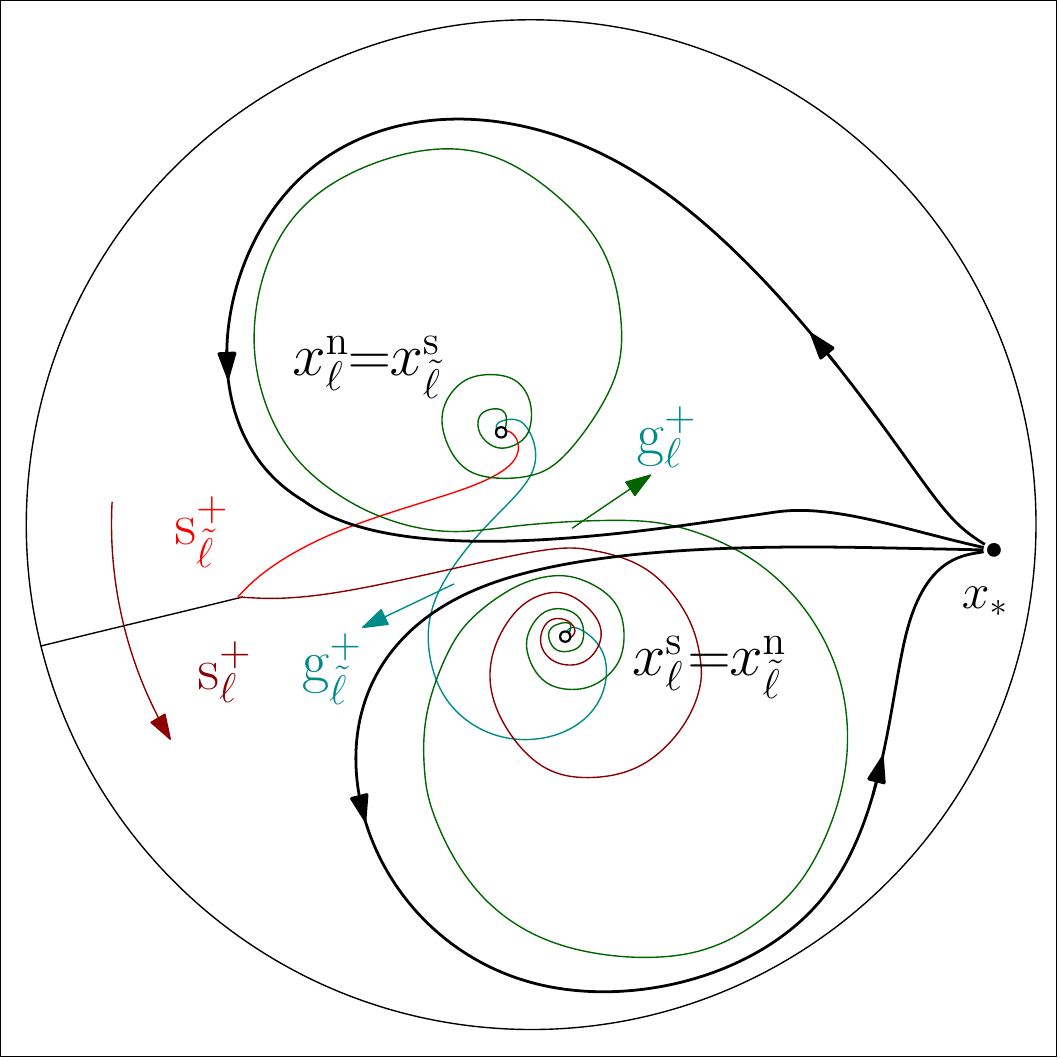}\hfill{}

\caption{\label{fig:compare_cell_holnomy}}
The generators of the two holonomies on the self intersection of the
unique cell $\mathcal{E}$.
\end{figure}
\begin{rem}
~

\begin{enumerate}
\item The groups $\mathcal{W}_{\ell}$ and $\mathcal{G}_{\ell}$ do not
depend on the particular choice of the base-point $x_{*}\in V^{j_{*}}$,
but do depend on the base-sector $V^{j_{*}}$.
\item Changing the base-sector from $V^{j_{*}}$ to another sector $V^{j}$
induces an inner conjugacy between respective necklace dynamics.
\end{enumerate}
\end{rem}

\subsubsection{Compatibility condition}
\begin{defn}
\label{def:compatibility_condition}Let $\mathfrak{m}\in\prod_{\ell}\mathcal{H}_{\ell}\left\{ h\right\} ^{k}$
and $\mu\in\germ{\varepsilon}$. We say that $\left(\mu,\mathfrak{m}\right)$
satisfies the \textbf{compatibility condition} if the different necklace
dynamics (\emph{i.e.} abstract holonomy pseudogroups) combined to
form $\mathcal{G}\left(\mathfrak{m}\right)$ are conjugate, in the
sense that there exists $x_{*}\in\rho\ww D\backslash\rho_{\varepsilon}\ww D$
in a fixed base sector $V^{j_{*}}$ such that for every $\ell,~\widetilde{\ell}$
and any connected component $C$ of $\mathcal{E}_{\ell}\cap\mathcal{E}_{\widetilde{\ell}}\neq\emptyset$
there exists a (perhaps small) subdomain $\Lambda\subset C$ such
that for all $\varepsilon\in\Lambda$ there exists $\delta_{\widetilde{\ell}\leftarrow\ell,\varepsilon}\in\diff[\cc,0]$
satisfying:

\begin{itemize}
\item $\delta_{\widetilde{\ell}\leftarrow\ell,\varepsilon}'\left(0\right)=1$,
\item for all $\gamma\in\pi_{1}\left(\rho\ww D\backslash P^{-1}\left(0\right),x_{*}\right)$,
\begin{align}
\delta_{\widetilde{\ell}\leftarrow\ell,\varepsilon}^{*}\psi_{\ell,\varepsilon}\left[\gamma\right] & =\psi_{\widetilde{\ell},\varepsilon}\left[\gamma\right]\label{eq:compatibility}
\end{align}
where $\delta^{*}\psi=\delta^{-1}\circ\psi\circ\delta$ is the usual
conjugacy for diffeomorphisms.
\end{itemize}
\end{defn}

\begin{figure}
\centering{}\includegraphics[width=6cm]{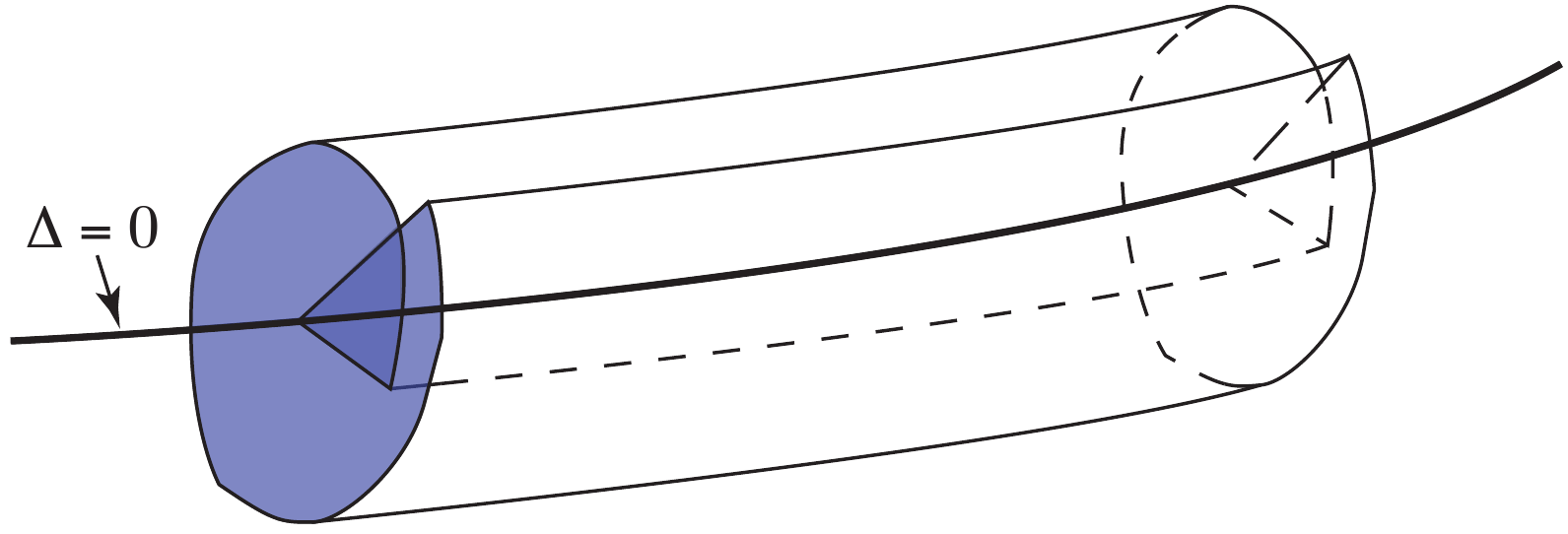} \caption{ \label{fig:auto_inter}A cell $\mathcal{E}_{\ell}$ having self-intersection
around a regular part of $\Delta_{k}$.}
\end{figure}
\begin{rem}
\label{rem:compatibility_self-intersection}Notice that the compatibility
condition also applies when $\widetilde{\ell}=\ell$, \emph{i.e.}
$\mathcal{E}_{\ell}$ is a self-intersecting cell with self-intersection
$\mathcal{E}_{\ell}^{\cap}$ around a regular part of $\Delta_{k}$
as in Figure~\ref{fig:auto_inter}, with the obvious adaptations.
To avoid confusion we denote by $\overline{\varepsilon}$ and $\widetilde{\varepsilon}$
the ``distinct points'' corresponding to the same parameter $\varepsilon\in\mathcal{E}_{\ell}^{\cap}$
seen from two different overlapping parts of the cell. More generally
we decorate objects with corresponding signs, like $\overline{\psi}$
or $\widetilde{\psi}$ in order to really stand for $\psi_{\ell,\overline{\varepsilon}}$
and $\psi_{\ell,\widetilde{\varepsilon}}$ respectively.
\end{rem}

\begin{lem}
\label{lem:realizable_thus_compatible}If $\left(\mu,\mathfrak{m}\right)$
is realizable then the compatibility condition holds.
\end{lem}

\begin{proof}
Fix some point $x_{*}\in V^{0}\backslash\rho_{\varepsilon}\ww D$
and take $\delta_{\widetilde{\ell}\leftarrow\ell}:=h_{\widetilde{\ell}}^{0}\circ\left(h_{\ell}^{0}\right)^{\circ-1}$
on $\mathcal{E}_{\ell}\cap\mathcal{E}_{\widetilde{\ell}}$ as in Lemma~\ref{lem:first-int_multiplier}. 
\end{proof}
\bigskip{}

\begin{rem}
~

\begin{enumerate}
\item Although we do not impose that the mappings $\delta_{\widetilde{\ell}\leftarrow\ell}$
exist on the connected component $C$ of $\mathcal{E}_{\ell}\cap\mathcal{E}_{\widetilde{\ell}}$,
nor depend analytically on $\varepsilon\in\Lambda\subset C$, it will
be true retrospectively and the dynamical conjugacies $\delta_{\widetilde{\ell}\leftarrow\ell}$
are always of the form described in Lemma~\ref{lem:first-int_multiplier}.
In particular the collection $\left(\delta_{\widetilde{\ell}\leftarrow\ell}\right)_{\ell,\widetilde{\ell}}$
is a cocycle: 
\begin{align*}
\delta_{\ell_{2}\leftarrow\ell_{1}}\circ\delta_{\ell_{1}\leftarrow\ell_{0}} & =\delta_{\ell_{2}\leftarrow\ell_{0}}
\end{align*}
whenever all three mappings are simultaneously defined.
\item The compatibility condition could be weakened further. The existence
of $\delta_{\widetilde{\ell}\leftarrow\ell}$ as above is only needed
for $\varepsilon$ belonging to a set $\Lambda$ of full analytic
Zariski closure, \emph{i.e.} such that if a holomorphic function $f$
on $C$ satisfies $f|_{\Lambda}=0$ then $f=0$. The cornerstone of
the proof of the Realization Theorem consists indeed in applying Corollary~\ref{cor:realizable_iff_glue}:
it suffices to check whether the identity $R_{\ell}-R_{\widetilde{\ell}}=0$
holds on every connected component $C$ of $\mathcal{E}_{\ell}\cap\mathcal{E}_{\widetilde{\ell}}$.
\end{enumerate}
\end{rem}

\subsection{\label{subsec:realization_proof}Normal forms stitching: proof of
Orbital Realization Theorem when $\mu_{0}\protect\notin\protect\rr_{\protect\leq0}$
and $\tau=0$. }

Thanks to Lemma~\ref{lem:realizable_thus_compatible}, only the converse
direction of the Realization Theorem still requires a full proof at
this stage. Assume then that the compatibility condition holds. Let
us fix a base point $x_{*}$ in a base sector $V^{j_{*}}$ and pick
$\varepsilon\in\Lambda\subset C\subset\mathcal{E}_{\ell}\cap\mathcal{E}_{\widetilde{\ell}}$
as in Definition~\ref{def:compatibility_condition}. Recalling Lemma~\ref{lem:first-int_multiplier},
the tangent-to-identity mapping
\begin{align}
\Psi~:~\left(x_{*},y\right) & \longmapsto\left(x_{*},\left(h_{\widetilde{\ell}}^{0}\right)^{\circ-1}\circ\delta_{\widetilde{\ell}\leftarrow\ell}\circ h_{\ell}^{0}\right)\label{eq:psi_compat}
\end{align}
conjugates the weak holonomy pseudogroups given by the representation
\begin{align*}
\holo{\ell}~:~\pi_{1}\left(\rho\ww D\backslash P^{-1}\left(0\right),x_{*}\right) & \longto\diff[\left\{ x=x_{*}\right\} ,0].
\end{align*}

Let us formulate a direct consequence of the main results of~\cite{DES}
(see~\cite{BraDia}) in a manner adapted to our setting.
\begin{lem}
\label{lem:gate_is_holomorphic}The map $\varepsilon\in\mathcal{E}_{\ell}\mapsto\left(\nu_{\varepsilon}^{j}\right)_{j\in\zsk}$
is holomorphic and locally injective. In particular there exists a
subdomain $\Lambda'\subset\Lambda$ such that for all $\varepsilon\in\Lambda'$,
every singular point of $X_{\ell}$ and $X_{\widetilde{\ell}}$ is
hyperbolic. 
\end{lem}

Using an extension of the Mattei-Moussu construction for hyperbolic
singularities (see below) we can analytically continue $\Psi$ (defined
in~\eqref{eq:psi_compat}) on a whole neighborhood of $\left\{ y=0\right\} $
as a fibered equivalence between $\onf[\ell]$ and $\onf[\widetilde{\ell}]$.
The argument developed in Section~\ref{subsec:Uniqueness_orbital}
(to prove uniqueness of the normal form) is performed for fixed $\varepsilon$,
therefore there exists
\begin{align*}
c & \in\cc^{\times}
\end{align*}
such that
\begin{align*}
R_{\ell,\varepsilon}\left(x,cy\right) & =R_{\widetilde{\ell},\varepsilon}\left(x,y\right).
\end{align*}
But the conjugacy $\Psi$ is tangent to the identity in the $y$-variable
thus $c=1$. Therefore $R_{\ell,\varepsilon}=R_{\widetilde{\ell},\varepsilon}$
on $\Lambda$, thus on $C$ by analytic continuation. Since this argument
can be carried out for any connected component $C$ of any cellular
intersection, Corollary~\ref{cor:realizable_iff_glue} yields the
conclusion. 
\begin{rem}
In fact $\Psi$ itself must be the identity, therefore 
\begin{align*}
\delta_{\widetilde{\ell}\leftarrow\ell} & =h_{\widetilde{\ell}}^{0}\circ\left(h_{\ell}^{0}\right)^{\circ-1}
\end{align*}
as in Lemma~\ref{lem:first-int_multiplier}.
\end{rem}

\bigskip{}

There only remains a single gap in the above argument, namely that
of extending $\Psi$ near each hyperbolic singularity. Let $\fol{\ell}$
be the foliation induced by $\onf[\ell]$ and take a germ $\Sigma\subset\left\{ x=x_{*}\right\} $
of a transverse disk at $\left(x_{*},0\right)$ in such a way that
$\Psi$ is holomorphic and injective on $\Sigma$. The union of the
saturation $\sat[\fol{\ell}]{\Sigma}$ and the vertical separatrices
$P^{-1}\left(0\right)$ is a full neighborhood of $\left\{ y=0\right\} $
since no singular point of $\fol{\ell}$ is a node. Therefore $\Psi$
can be extended as a fibered, injective mapping by the usual path-lifting
technique except along the separatrices $P^{-1}\left(0\right)$. Up
to divide $\onf[\ell]$ and $\onf[\widetilde{\ell}]$ by a local holomorphic
unit near each singularity, we can assume that the hypothesis of Lemma~\ref{lem:mattei-moussu}
are met. This completes the proof of the Realization Theorem when
$\mu_{0}\notin\rr_{\leq0}$.

\section{\label{sec:tau}General case $\tau>0$}

In this section we fix $\tau\in\nn$ such that 
\begin{align*}
\mu_{0}+\tau\left(k+1\right) & \notin\rr_{\leq0}.
\end{align*}

\subsection{End of proof of (orbital) Normalization, Uniqueness and Realization
Theorems}

We explain now how to reduce the case $\tau>0$ to the case $\tau=0$
already dealt with. We exploit the observation that formally $\nfsec$
is the pullback of $\nfsec[][y]$ by the mapping
\begin{align}
T~:~\left(\varepsilon,x,y\right) & \longmapsto\left(\varepsilon,x,P_{\varepsilon}^{\tau}\left(x\right)y\right).\label{eq:tau_blow-up}
\end{align}
Albeit not invertible along the lines $\left\{ P_{\varepsilon}\left(x\right)=0\right\} $
(its image is not a neighborhood of $\left\{ y=0\right\} $), the
mapping $T$ transforms the model unfolding
\begin{align}
\fonf\left(x,y\right) & =P_{\varepsilon}\left(x\right)\pp x+y\left(1+\mu_{\varepsilon}x^{k}\right)\pp y\label{eq:fnf}
\end{align}
into
\begin{align*}
\widehat{\tx Y} & :=T^{*}\fonf=P_{\varepsilon}\pp x+(1+\tau P_{\varepsilon}'+\mu_{\varepsilon}x^{k})y\pp y.
\end{align*}
Observe that
\begin{align*}
\tau P'\left(x\right)+\mu x^{k} & \sim_{\infty}\left(\tau\left(k+1\right)+\mu\right)x^{k},
\end{align*}
so that involving $P^{\tau}$ in this way shifts the formal invariant
by $\tau\left(k+1\right)$. Apart from the fact that $\widehat{Y}$
is not in prepared form~\eqref{eq:prepared_form_orbital}, all the
theory developed before for the realization theorem applies in this
case too. Let us be more specific. The key property we used intensively
was to be able to perform most arguments for fixed $\varepsilon$.
This was proved sufficient because automorphisms of prepared forms
fixing the $x$-variable must also fix the canonical parameter $\varepsilon$.
\begin{lem}
\label{lem:preparation_automorphisms}~

\begin{enumerate}
\item The group of (fibered) symmetries 
\begin{align*}
(\varepsilon,x,y)\mapsto(\eta\left(\varepsilon\right),X\left(\varepsilon,x\right),Y\left(\varepsilon,x,y\right))
\end{align*}
of (the unfolding of) vector fields defined by~\eqref{eq:fnf}, is
isomorphic to $\zsk\times\cc^{\times}$ through the linear representation
\begin{align}
\zeta_{0}~:~\zsk\times\cc^{\times} & \longto\GL{k+2}{\cc}\label{eq:ind_rotation}\\
\left(\theta,c\right) & \longmapsto\left((\varepsilon_{0},\ldots,\varepsilon_{k-1},x,y)\mapsto\left(\alpha\varepsilon_{0},\varepsilon_{1},\dots,\varepsilon_{k-1}\alpha^{-(k-2)},~\alpha x,~cy\right)\right)
\end{align}
where $\alpha=\exp\nf{2\ii\pi\theta}k$. 
\item This statement continues to hold in the more general case of an unfolding
\begin{equation}
P_{\varepsilon}\left(x\right)\frac{\partial}{\partial x}+(1+Q_{\varepsilon}(x))y\frac{\partial}{\partial y},\label{eq:fnf_Q}
\end{equation}
where $Q_{\varepsilon}\in\pol x_{\leq k}$ is a polynomial in $x$
of degree at most $k$ and $Q_{\varepsilon}(0)=0$, save for the fact
that the representation $\zeta_{\tau}~:~\zsk\times\cc^{\times}\to\diff[\cc^{k+2},0]$
has no reason to be linear. 
\item In particular, any symmetry tangent to the identity is the identity. 
\end{enumerate}
\end{lem}

\begin{proof}
(1) is shown in~\cite{RouTey}. For (2), there exists a diffeomorphism
$\Psi$ of the form $(\varepsilon,x,y)\mapsto(\eta,X,Y)$ transforming
a general formal normal form~\eqref{eq:fnf_Q} to the standard formal
normal form~\eqref{eq:fnf}. Then any symmetry of a general formal
normal form is given by $\Psi^{-1}\circ\zeta_{0}\left(\theta,c\right)\circ\Psi$
for some $\left(\theta,c\right)\in\zsk\times\cc^{\times}$. (3) follows.
\end{proof}
\begin{rem}
\label{rem:remark_sectors}~

\begin{enumerate}
\item In view of Lemma~\ref{lem:preparation_automorphisms}, we could have
replaced~\eqref{eq:fnf} by some other~\eqref{eq:fnf_Q} in all
our constructions regarding realization. In such a form, the parameters
are again canonical, as long as we consider changes of coordinates
tangent to the identity. 
\item The structure of sectors, and also the decomposition in cells $\mathcal{E}$,
are determined from $P_{\varepsilon}$ alone in~\eqref{eq:fnf}:
only the size of the neighborhoods of the origin in $x$-space and
in parameter space might need to be slightly adjusted when passing
from the coordinates $\left(x,y\right)$ to the coordinates $\left(x,P^{\tau}\left(x\right)y\right)$.
Hence, instead of considering \eqref{eq:fnf}, we could have taken
a normal form \eqref{eq:fnf_Q} with the same sectors $V_{\ell}^{j}$
and same cells $\mathcal{E}_{\ell}$. 
\end{enumerate}
\end{rem}

The rest of our argument relies on the next transport result.
\begin{lem}
\label{lem:modulus_unchanged}~

\begin{enumerate}
\item $\left(\mu,\mathfrak{m}\right)$ satisfies the compatibility condition
if and only if $\left(\mu+\tau\left(k+1\right),\mathfrak{m}\right)$
does.
\item Take $\onf$ in orbital normal form~\eqref{eq:orbital_normal_form}
with $\tau:=0$. Consider the corresponding unfolding
\begin{align*}
\mathcal{Y} & :=T^{*}\onf=P_{\varepsilon}\pp x+y\left(1+\tau P'_{\varepsilon}+\mu x^{k}+R\left(x,P^{\tau}y\right)\right)\pp y,
\end{align*}
for $T$ as in~\eqref{eq:tau_blow-up}. Then $\onf$ and $\mathcal{Y}$
have same orbital invariant $\mathfrak{m}\left(\onf\right)=\mathfrak{m}\left(\mathcal{Y}\right)$. 
\end{enumerate}
\end{lem}

We postpone the proof till Section~\ref{subsec:modulus_unchanged}.
In the meantime we finish establishing the main theorems.

\subsubsection{End of proof of Orbital Realization Theorem}

Let $\left(\mu,\mathfrak{m}\right)$ satisfy the compatibility condition
and let us prove it is realizable as the orbital modulus of some convergent
unfolding. Normalization and Realization Theorems so far hold when
$\tau=0$ (in particular $\mu_{0}\notin\rr_{\leq0}$): in that case
$\mathfrak{m}$ is the modulus of an unfolding in normal form 
\begin{equation}
P_{\varepsilon}\left(x\right)\frac{\partial}{\partial x}+y\left(1+\mu(\varepsilon)x^{k}+y\sum_{j=1}^{k}x^{j}R_{j}(y)\right)\frac{\partial}{\partial y}.\label{eq:realisation_mu}
\end{equation}

To consider the case $\tau>0$, we need to use the following remark:
the whole proof for $\tau=0$ would have worked \emph{verbatim} with
the formal part and parameters given in some alternate form~\eqref{eq:fnf_Q}.
This would have produced a realization of the form 
\begin{equation}
P_{\varepsilon}(x)\frac{\partial}{\partial x}+y\left(1+Q_{\varepsilon}(x)+y\sum_{i=j}^{k}x^{j}R_{j}(y)\right)\frac{\partial}{\partial y},\label{eq:realisation_Q}
\end{equation}
with new canonical parameters. Let $\tau$ be a positive integer such
that $\mu_{0}+\tau(k+1)>0$ and consider the new formal normal form
\[
\widehat{\tx Y}\left(x,y\right)=P_{\varepsilon}(x)\frac{\partial}{\partial x}+(1+\tau P_{\varepsilon}'(x)+\mu(\varepsilon)x^{k})y\frac{\partial}{\partial y}
\]
corresponding to $Q_{\varepsilon}:=\tau P_{\varepsilon}'+\mu(\varepsilon)x^{k}$
in~\eqref{eq:fnf_Q}, with formal invariant 
\begin{align*}
\widehat{\mu} & :=\mu+\tau(k+1).
\end{align*}
But according to Lemma~\ref{lem:modulus_unchanged}:
\begin{enumerate}
\item $\left(\widehat{\mu},\mathfrak{m}\right)$ is compatible,
\item it is realized in the form~\eqref{eq:realisation_Q},
\item the change $\left(x,y\right)\mapsto\left(x,P_{\varepsilon}^{-\tau}(x)y\right)$
transforms~\eqref{eq:realisation_Q} back into an unfolding 
\[
P_{\varepsilon}(x)\frac{\partial}{\partial x}+y\left(1+\mu(\varepsilon)x^{k}+\sum_{j=1}^{k}x^{j}R_{j}(P_{\varepsilon}^{\tau}(x)y)\right)\frac{\partial}{\partial y},
\]
\item the latter unfolding is holomorphic on a whole neighborhood of $\neigh[k+2]$,
and is therefore a realization of $\left(\mu,\mathfrak{m}\right)$.
\end{enumerate}

\subsubsection{End of proof of Normalization Theorem}

The proof we just finished shows that any realizable $\left(\mu,\mathfrak{m}\right)$
can be realized in normal form.

\subsubsection{End of proof of Uniqueness Theorem}

Each vector field $\onf[\varepsilon]$ of the unfolding in normal
form~\eqref{eq:orbital_normal_form} is holomorphic on a domain
\begin{align*}
D\left(r\right) & :=\bigcup_{\varepsilon\in\neigh[k]}\left\{ \left(\varepsilon,x,y\right)~:~\left|x\right|<\rho,~\left|P_{\varepsilon}^{\tau}\left(x\right)y\right|<r\right\} .
\end{align*}
Let $E$ be a neighborhood of $0$ in $\cc^{k+2}$ and $\Psi~:~E\to\neigh[k+2]$
be a local conjugacy between normal forms $\onf$ and $\widetilde{\mathcal{X}}$,
which can be assumed fibered thanks to Corollary~\ref{cor:formal_symmetries}~(2).
We can use the Uniqueness Theorem in the coordinates $\left(x,P^{\tau}\left(x\right)y\right)$
(given by the Uniqueness Theorem for $\mu_{0}\notin\rr_{\leq0}$,
already proved) at the cost of showing that $T^{*}\Psi=T\circ\Psi\circ T^{\circ-1}$
is holomorphic and injective on some small neighborhood of $\left(0,0\right)$
uniformly in $\varepsilon$. This is not trivial since the image of
$E\cap\left\{ \varepsilon=\cst\right\} $ by $T$ can never be such
a uniform neighborhood of $\left(0,0\right)$ if $E$ is bounded in
the $y$-variable. But $T\left(D\left(r\right)\cap\left\{ \varepsilon=\cst\right\} \right)$
is, so we wish to extend $\Psi$ to some $D\left(r'\right)\subset D\left(r\right)$.
The usual path-lifting technique in the foliation $\fol{\varepsilon}$
induced by $\onf[\varepsilon]$ allows to extend $\Psi_{\varepsilon}$
on 

\begin{align*}
\mathcal{U}_{\varepsilon} & :=\left\{ \varepsilon\right\} \times\sat[\fol{\varepsilon}]E\subset D\left(r\right).
\end{align*}
Using the special form of the normal form $\mathcal{X}_{\varepsilon}$
we conclude the proof of the Uniqueness Theorem.
\begin{lem}
\label{lem:saturated_contains_separatrices}Assume that $\rho>0$
is small enough so that $\left|\mu_{\varepsilon}x^{k}+\tau P'_{\varepsilon}\left(x\right)\right|<\frac{1}{4}$
for all $x\in\rho\disc$ and all $\norm[\varepsilon]{}$ small enough.
There exists $r\geq r'>0$ such that for $\mathcal{U}_{\varepsilon}=\mathcal{U}_{\varepsilon}\left(r\right)$
defined as above one has $D\left(r'\right)\subset\bigcup_{\varepsilon\in\neigh[k]}\mathcal{U}_{\varepsilon}\subset D\left(r\right)$.
\end{lem}

\begin{proof}
For a solution of the flow system
\begin{align*}
\begin{cases}
\dot{x} & =-P_{\varepsilon}\left(x\right)\\
\dot{y} & =-y\left(1+\mu_{\varepsilon}x^{k}+R_{\varepsilon}\left(x,y\right)\right)
\end{cases}
\end{align*}
with $t\in\rr$ and initial value $\left(x_{*},y_{*}\right)$, the
modulus of $\phi\left(t\right):=\left|P_{\varepsilon}^{\tau}\left(x\left(t\right)\right)y\left(t\right)\right|$
satisfies
\begin{align*}
\dot{\phi} & =-\phi\re{1+\mu_{\varepsilon}x^{k}+R_{\varepsilon}+\tau P_{\varepsilon}'}.
\end{align*}
Since $R_{\varepsilon}\left(x,0\right)=0$ we can choose $r$ so small
that $\left|\mu_{\varepsilon}x^{k}+R_{\varepsilon}+\tau P_{\varepsilon}'\left(x\right)\right|<\frac{1}{2}$
for all $\left(\varepsilon,x,y\right)\in D\left(r\right)$, and $\dot{\phi}<-\nf{\phi}2$.
Hence starting at $\left(x_{*},y_{*}\right)$ with $\left|P_{\varepsilon}^{\tau}\left(x_{*}\right)y_{*}\right|<r$
and $\left|x_{*}\right|<\rho$, the trajectory for positive $t$ never
escapes $D\left(r\right)$. But $t\mapsto\left|y\left(t\right)\right|$
is also exponentially decreasing, therefore we eventually reach a
point within $E$. 

Again, this is the ideal situation, because it may happen that $x\left(t\right)$
exits $\adh{\rho\disc}$. If $\left|x\left(t_{0}\right)\right|=\rho$
then we modify the trajectory $x$ by solving $\dot{x}=\pm\ii P_{\varepsilon}\left(x\right)$
from $t_{0}$ on, the sign being chosen so that $\pm\ii P_{\varepsilon}\left(x\left(t_{0}\right)\right)$
points inside $\rho\sone$, until we reach a point $x\left(t_{1}\right)$
through which the solution of $\dot{x}=-P_{\varepsilon}\left(x\right)$
stays in $\rho\disc$ in positive time (\emph{i.e.} accumulate on
an attractive singularity). While for $t\in\left[t_{0},t_{1}\right]$
we cannot control the sign of $\dot{\phi}=\pm\phi\im{\mu_{\varepsilon}x^{k}+R_{\varepsilon}+\tau P_{\varepsilon}'}$,
resulting in a probable increase in $\phi$, the total amount by which
$\frac{\phi}{r}$ increases is bounded uniformly in $\left(x_{*},y_{*}\right)$
and $\varepsilon$. Therefore there exists a radius $r\geq r'>0$
for which, if $\left(x_{*},y_{*}\right)\in D\left(r'\right)$, the
modified trajectory $t\geq0\mapsto\left(x\left(t\right),y\left(t\right)\right)$
does not escape from $D\left(r\right)$ and thus eventually enters
$E$.
\end{proof}

\subsubsection{\label{subsec:modulus_unchanged}Proof of Lemma~\ref{lem:modulus_unchanged}}

First, as noted in Remark~\ref{rem:remark_sectors}, we can choose
the same sectors in $x$ and same cells in the parameter $\varepsilon$,
possibly after adjusting their diameter. Also, we have chosen to take
the linear parts of the $\sad{\psi}{\ell}{j,}$ of the form $\exp\nf{2\ii\pi\mu}k$.
This choice is arbitrary. What is needed is that the product of these
linear parts be equal to $\exp2\ii\pi\mu$. Because $(k+1)\tau\in\zz$,
so that $\exp2\ii\pi\mu=\exp\left(2\pi\ii\left(\mu+\left(k+1\right)\tau\right)\right)$,
we are perfectly entitled to take the same linear parts for $\mathfrak{m}\left(\onf\right)$
and $\mathfrak{m}\left(\mathcal{Y}\right)$.

The Camacho-Sad index $\widetilde{\lambda}^{j}$ (\emph{resp.} $\lambda^{j}$)
of the singular point $\left(z,0\right)\in P^{-1}\left(0\right)\times\left\{ 0\right\} $
in $\mathcal{Y}_{\varepsilon}$ (\emph{resp.} $\mathcal{X}_{\varepsilon}$),
relatively to the invariant line $\left\{ y=0\right\} $, is given
by 
\[
\widetilde{\lambda}^{j}=\frac{P_{\varepsilon}'(z)}{1+\tau P_{\varepsilon}'(z)+\mu_{\varepsilon}z^{k}},\qquad\lambda_{j}=\frac{P_{\varepsilon}'(z)}{1+\mu_{\varepsilon}z^{k}}.
\]
Hence, $\frac{1}{\widetilde{\lambda}^{j}}=\frac{1}{\lambda^{j}}+\tau$,
yielding $\exp\nf{2\ii\pi}{\widetilde{\lambda}^{j}}=\exp\nf{2\ii\pi}{\lambda^{j}}$.
This means that the gate transition maps are the same for both dynamical
necklaces induced by $\left(\mu,\mathfrak{m}\right)$ and by $(\widehat{\mu},\mathfrak{m})$.
Thus, the holonomies involved in the compatibility condition are the
same provided~(2) holds. In particular, this means that $\left(\widehat{\mu},\mathfrak{m}\right)$
satisfies the compatibility condition, proving~(1).

Show now that $\mathfrak{m}$ is the analytic part of the modulus
of $\mathcal{Y}$. It suffices to consider a fixed $\varepsilon\in\mathcal{E}_{\ell}$
and a corresponding saddle part $\sad V{}{j,}$. Recall how a normalizing
map between $\mathcal{Y}$ and its formal model, as in Remark~\ref{rem:corollaries_squid_solve},
defines the canonical sectorial first integral 
\begin{align*}
\widetilde{H}(x,y) & =y\widetilde{E}(x)\exp\widetilde{N}^{j}\left(x,y\right),
\end{align*}
where $\widetilde{E}(x)=\prod_{j=0}^{k}(x-x^{j})^{-\nf 1{\widetilde{\lambda}_{j}}}$
is the multiplier in the model first integral of $\mathcal{Y}_{\varepsilon}$.
Let $\sad{\psi}{}{j,}~:~h\mapsto h\exp\left(\nf{2\ii\pi\mu}k+\mathfrak{m}\left(h\right)\right)$
be the Martinet-Ramis invariant as in Section~\ref{subsec:Invariants},
that is 
\[
\widetilde{H}^{j+1}=\sad{\widetilde{\psi}}{}{j,}\circ\widetilde{H}^{j}.
\]
Let us now move to $\mathcal{X}$. It is clear that a normalizing
map over $\mathcal{V}^{j}$ transforming $\mathcal{X}_{\varepsilon}$
into its normal form is given by 
\begin{align*}
\left(x,y\right) & \longmapsto\left(x,y\exp N^{j}\left(x,y\right)\right)\\
N_{j}\left(x,y\right) & =\widetilde{N}_{j}(x,P_{\varepsilon}^{\tau}(x)y).
\end{align*}
Moreover, the domain of this map is of the form $V^{j}\times\{|P_{\varepsilon}^{\tau}(x)y|<r\}$.
Since 
\begin{align*}
\prod_{j=0}^{k}(x-x^{j})^{-\frac{1}{\lambda_{j}}} & =\widetilde{E}(x)P^{\tau}(x)
\end{align*}
the canonical first integral of $\onf$ has the form 
\begin{align*}
H^{j}(x,y) & =E(x)y\exp N_{j}(x,y)=\widetilde{E}(x)\left(P^{\tau}(x)y\right)\exp\widetilde{N}^{j}\left(x,P^{\tau}\left(x\right)y\right).
\end{align*}
It follows at once that 
\[
H^{j+1}=\sad{\widetilde{\psi}}{}{j,}\circ H^{j},
\]
yielding the conclusion $\sad{\psi}{}{j,}=\sad{\widetilde{\psi}}{}{j,}$
as expected.

\subsection{Section of the period operator: end of proof of the Normalization
Theorem}

Let $\onf$ be a generic unfolding in orbital normal form~\eqref{eq:orbital_normal_form},
understood as a derivation. Theorem~\ref{thm:squid_and_solve} holds
regardless of the value of $\mu_{0}$ or $\tau$. The study performed
in Section~\ref{sec:Temporal} to establish Theorem~\ref{thm:section_period}
can be repeated here but for the fact that the canonical section of
the period operator needs to be adapted. The mapping defined in~\eqref{eq:period_section}
becomes 
\begin{align*}
\mathfrak{K}~:~\germ{\varepsilon,x,y}' & \longto\nfsec[][P^{\tau}y]\\
G & \longmapsto\persec[][\ell]\left(\per[][\ell]\left(G\right)\right)
\end{align*}
whose kernel coincides with $\onf\cdot\germ{\varepsilon,x,y}'$, \emph{i.e.
}the sequence of $\germ{\varepsilon}$-linear operators
\begin{align*}
0\longto\germ{\varepsilon,x,y}'\overset{\onf\cdot}{\longto}\germ{\varepsilon,x,y}'\overset{\mathfrak{K}}{\longto} & \nfsec[][P^{\tau}y]\longto0
\end{align*}
is exact. Up to this modification the temporal part of Realization
Theorem is established.

\bigskip{}

The most obvious reason why one must adapt the target space of the
section operator is computational. Proposition~\ref{prop:period_analytic_k=00003D1}
below recalls the formula for the period of the formal model $\fonf$
for $k=1$. For $xy^{m}\in\nfsec[][y]$, $m\in\nn$, it may happen
that $\widehat{\per}\left(xy^{m}\right)$ vanishes, exactly when $m\mu\in\zz_{\leq0}$.
This situation cannot happen if $\mu_{0}\notin\rr_{\leq0}$, of course.
Pre-composing $xy^{m}$ by $P^{\tau}\left(x\right)y$ yields
\begin{align*}
\widehat{\per}\left(xP^{m\tau}\left(x\right)y^{m}\right) & =\widehat{\per}\left(x^{m\tau\left(k+1\right)+1}y^{m}\right)+\OO{\varepsilon},
\end{align*}
and by hypothesis $m\left(\mu_{0}+\left(k+1\right)\tau\right)\notin\zz_{\leq0}$.
As already noticed, the presence of $P^{\tau}$ acts as a shift by
$\left(k+1\right)\tau$ on powers of $x$. Here it guarantees that
$\mathfrak{S}_{\ell}$ remains invertible. Notice that the map $\mathfrak{S}_{\ell}$
needs to undertake the same modification as in~\eqref{eq:tau_blow-up};
compare~\eqref{eq:cauchy-heine_period}. We will not go into further
details.

\subsection{Alternate normal forms}

The normal forms we propose in the Normalization Theorem are not strictly
speaking a generalization of~\cite{Loray,SchaTey}, which is what
we expected to accomplish in the first place and which we propose
as a conjecture.
\begin{conjecture}
\label{conj:alternate_normal_forms} Fix $k\in\nn$, a germ of holomorphic
function $\mu\in\germ{\varepsilon}$, and $\widehat{\tau}\in\zp$
such that $\mu_{0}+\widehat{\tau}\notin\rr_{\leq0}$. Any generic
convergent unfolding of a germ of saddle-node holomorphic vector field
with the formal invariant $\mu$ is orbitally conjugate to an unfolding
of the form
\begin{align*}
\fonf+y\widehat{R}\pp y~~~~~~~,~~~~\widehat{R}\in x\pol x_{<k}\left\{ x^{\widehat{\tau}}y\right\} .
\end{align*}
Such a form is unique up to conjugacy by linear maps $\left(\varepsilon,x,y\right)\mapsto\left(\varepsilon,x,c_{\varepsilon}y\right)$,
$c\in\germ{\varepsilon}^{\times}$.
\end{conjecture}

(A similar conjecture can be stated for the temporal part.) This conjecture
is very likely to be true as we almost managed to ascertain both the
geometric normalization and the cellular realization in that form.
In both questions we encountered difficulties of a technical nature,
which can surely be overcome by bringing in tedious estimates. 

\section{\label{sec:Bernoulli}Bernoulli unfoldings}

The primary aim of this section is to establish that the compatibility
condition is not trivially satisfied by proving the Parametrically
Analytic Orbital Moduli Theorem. The most difficult direction is~(1)$\Rightarrow$(2).
The whole proof is geared toward using rigidity results of Abelian
finitely generated pseudogroups $G<\diff[\cc,0]$. Let us briefly
explain how Abelian pseudogroups come into consideration here. Elements
$\psi_{\ell}\left[\gamma\right]$ and $\psi_{\widetilde{\ell}}\left[\gamma\right]$
in overlapping cellular necklace dynamics are conjugate by the transition
mapping $\delta_{\ell\leftarrow\ell}$ coming from the compatibility
condition. The parametric holomorphy of $\mathfrak{m}$ forces the
equality $\psi_{\ell}\left[\Gamma\right]=\psi_{\widetilde{\ell}}\left[\Gamma\right]$
for well-chosen loops $\Gamma$, from which stems the commutativity
relation 
\begin{align*}
\psi_{\ell}\left[\Gamma\right]\circ\delta_{\widetilde{\ell}\leftarrow\ell} & =\delta_{\widetilde{\ell}\leftarrow\ell}\circ\psi_{\ell}\left[\Gamma\right].
\end{align*}
Such pseudogroups are completely understood and form now a classical
topic of complex dynamical systems, we refer for instance to~\cite{CerMou,LoRigid}.
``Bernoulli diffeomorphisms'' (defined below) play a central role
in this theory as archetypal examples of solvable and Abelian pseudogroups.

\subsection{Bernoulli diffeomorphism}
\begin{defn}
We say that $\psi\in\diff[\cc,0]$ is a \textbf{Bernoulli} \textbf{diffeomorphism
of index $d\in\nn$} if there exist $\alpha,~\beta\in\cc$ with $\alpha\neq0$
such that 
\begin{align*}
\psi\left(h\right) & =\frac{\alpha h}{\left(1+\beta h^{d}\right)^{\nf 1d}}=:\ber[d]{\alpha}{\beta}\left(h\right).
\end{align*}
We define $\bber$ the set of all such algebraic functions, regardless
of the special values of $\alpha$ and $\beta$ (these are in particular
germs of analytic diffeomorphisms at the origin). Of course when $d\neq\widetilde{d}$
the intersection $\bber\cap\bber[\widetilde{d}]$ coincides with the
group $\tx{GL}_{1}\left(\cc\right)$.
\end{defn}

Let us quickly state without proof the next basic property.
\begin{lem}
\label{lem:Bernoulli_properties}The set $\bber$ is a group equipped
with a semi-direct law. More precisely
\begin{align*}
\ber[d]{\alpha}{\beta}\circ\ber[d]{\widetilde{\alpha}}{\widetilde{\beta}} & =\ber[d]{\alpha\widetilde{\alpha}}{\beta\widetilde{\alpha}^{d}+\widetilde{\beta}}.
\end{align*}
\end{lem}

The definition of Bernoulli diffeomorphisms is motivated by the following
computation.
\begin{lem}
\label{lem:bernoulli_unfold_with_bernoulli_modulus}The necklace dynamics
of an unfolding of Bernoulli vector field $\onf=\fonf+y^{d+1}r\left(x\right)\pp y$
consists in Bernoulli diffeomorphisms of index $d$. Moreover
\begin{align*}
\mathfrak{m}\left(\onf\right) & =-\frac{1}{d}\log\left(1+2\ii\pi d\widehat{\per}\left(y^{d}r\right)\right).
\end{align*}
\end{lem}

\begin{proof}
As in~\cite[Section 3.3]{Tey-ExSN} one tries and finds an expression
for the sectorial first integrals $H^{j}$ in the form 
\begin{align*}
H^{j}\left(x,y\right) & =\frac{\widehat{H}^{j}\left(x,y\right)}{\left(1-df^{j}\left(x\right)y^{d}\right)^{\nf 1d}}.
\end{align*}
Because 
\begin{align*}
\onf\cdot H^{j} & =\frac{\widehat{H}^{j}}{\left(1-df^{j}\left(x\right)y^{d}\right)^{\nf 1d+1}}\left(\left(1-df^{j}\left(x\right)y^{d}\right)y^{d}r\left(x\right)+\onf\cdot\left(f^{j}\left(x\right)y^{d}\right)\right)\\
 & =\frac{\widehat{H}^{j}}{\left(1-df^{j}\left(x\right)y^{d}\right)^{\nf 1d+1}}\left(y^{d}r\left(x\right)+\widehat{X}\cdot\left(f^{j}\left(x\right)y^{d}\right)\right),
\end{align*}
then $H^{j}$ is a first integral for $\onf$ if and only if 
\begin{align*}
\fonf\cdot\left(y^{d}f^{j}\left(x\right)\right) & =-y^{d}r\left(x\right).
\end{align*}
This equation admits a formal solution (Lemma~\ref{lem:formal_cohomological})
because $\fonf$ is linear in the $y$-variable, and the $f^{j}\left(x\right)y^{d}$
are the sectorial solutions of this equation (Theorem~\ref{thm:squid_and_solve}).
In fact $\left(x,y\right)\mapsto\left(x,\frac{y}{\left(1-dy^{d}f^{j}\left(x\right)\right)^{1/d}}\right)$
is the canonical sectorial normalization of $\mathcal{X}$.

First notice that by definition of the period operator for the formal
model (Definition~\ref{def:period_operator}) we have for all $\left(x,y\right)\in\sad V{}{j,}\times\cc$:
\begin{align*}
y^{d}f^{j+1}\left(x\right)-y^{d}f^{j}\left(x\right) & =-\widehat{\per}^{j}\left(y^{d}r\right)\left(\widehat{H}^{j}\left(x,y\right)\right).
\end{align*}
From the special form of $H^{j}$ we deduce
\begin{align*}
\frac{H^{j+1}}{H^{j}\exp\nf{2\ii\pi\mu}k} & =\left(1-dy^{d}\frac{f^{j+1}-f^{j}}{1-df^{j}y^{d}}\right)^{-\nf 1d}=\exp\left(-\frac{1}{d}\log\left(1+d2\ii\pi\frac{\widehat{\per}^{j}\left(y^{d}r\right)\left(\widehat{H}^{j}\right)}{\left(\widehat{H}^{j}\right)^{d}}\times\left(H^{j}\right)^{d}\right)\right).
\end{align*}
Because $\widehat{H}$ is linear in the $y$-variable we know that
$\widehat{\per}\left(y^{d}r\right)\left(h\right)=\alpha h^{d}$ for
some complex coefficients $\alpha=\left(\alpha_{\varepsilon}^{j}\right)_{j\in\zsk}$.
Hence $\frac{\widehat{\per}\left(y^{d}r\right)\left(\widehat{H}\right)}{\widehat{H}^{d}}\times H^{d}=\widehat{\per}\left(y^{d}r\right)\left(H\right)$.
The rest follows from~\eqref{eq:invar_transition_maps}.
\end{proof}

\subsection{Holomorphic modulus: proof of the Parametrically Analytic Orbital
Moduli Theorem}

The direction (2)$\Rightarrow$(1) is a consequence of Lemma~\ref{lem:bernoulli_unfold_with_bernoulli_modulus}
above and of Proposition~\ref{prop:period_analytic_k=00003D1} below
stating that the model period operator $\widehat{\per}\left(y^{d}r\right)$
is analytic in the parameter when $k=1$ and $d\mu\in\zz$.

Conversely let us suppose that $\left(\mu,\mathfrak{m}\right)$ is
realizable and that $\mathfrak{m}_{\ell}=\phi|_{\mathcal{E}_{\ell}\times\neigh}$
for some holomorphic $k$-tuple 
\begin{align*}
\phi & =\left(\phi^{j}\right)_{j}\in h\germ{\varepsilon,h}^{k}.
\end{align*}
If $\phi=0$ then $\mathfrak{m}=\mathfrak{m}\left(\fonf\right)$ (Theorem~\ref{thm:classification}),
so we can as well assume that $\phi\neq0$. We first establish that
$k=1$ by contraposition, and then present the case $k=1$. That case
can be found originally in~\cite[Proposition 6]{TeySurvey} for $\mu=0$.
We generalize here the result to arbitrary $\mu$.

\bigskip{}

Recall that for $c\in\cc^{\times}$ we write 
\begin{align*}
L_{c}~:~ & h\longmapsto ch.
\end{align*}

\subsubsection{Reduction to the case $k=1$}

Assume then that $k>1$ and prove $\phi=\left(\phi^{j}\right)_{j}=0$.
For each $j\in\zsk$ there exists a cell $\mathcal{E}_{\ell}$ for
which $\sad x{}{j,}$ is attached to only one saddle sector $\sad V{}{j,}$.
Let $\nod x{}{j,}$ be the node point attached to $\sad x{}{j,}$
in the boundary of $\sad V{}{j,}$. The cell $\mathcal{E}_{\ell}$
self-intersects around a regular part of $\Delta_{k}$ in such a way
that the nature of the points $\sad x{}{j,}$ and $\nod x{}{j,}$
is exchanged when seen from one part or the other of the intersection.
With the conventions discussed in Remark~\ref{rem:compatibility_self-intersection},
by this we mean
\begin{align*}
\begin{cases}
\sad{\overline{x}}{}{j,}= & \nod{\widetilde{x}}{}{j,}\\
\nod{\overline{x}}{}{j,}= & \sad{\widetilde{x}}{}{j,}
\end{cases} & .
\end{align*}
We refer to Figures~\ref{fig:auto_inter} and~\ref{fig:analytic_invar_k>1}.

Fix a base-point and base-sector $x_{*}\in V^{j}\backslash\rho_{\varepsilon}\ww{D~}$
and take $\gamma^{-}$, $\gamma^{+}$ two loops based at $x_{*}$
of index $1$ around respectively $\nod x{\overline{\varepsilon}}{}$
and $\sad x{\overline{\varepsilon}}{}$, and index $0$ with respect
to the other roots of $P$ as in Figure~\ref{fig:analytic_invar_k>1}.
Let $\Gamma:=\gamma^{+}\gamma^{-}$ be a loop encircling only $\left\{ \nod x{}{},\sad x{}{}\right\} $.
The compatibility condition ensures the existence of a tangent-to-identity
map 
\begin{align*}
\delta & :=\delta_{\ell\leftarrow\ell}
\end{align*}
which conjugates the respective necklace dynamics based at $x_{*}$.
In particular
\begin{align}
\delta^{*}\overline{\psi}\left[\gamma^{\pm}\right] & =\widetilde{\psi}\left[\gamma^{\pm}\right],\label{eq:dynamics_conjugacy_bernoulli}\\
\delta^{*}\overline{\psi}\left[\Gamma\right] & =\widetilde{\psi}\left[\Gamma\right].\nonumber 
\end{align}

\begin{figure}
\hfill{}\subfloat[In the parameter $\overline{\varepsilon}$]{\includegraphics[width=6cm]{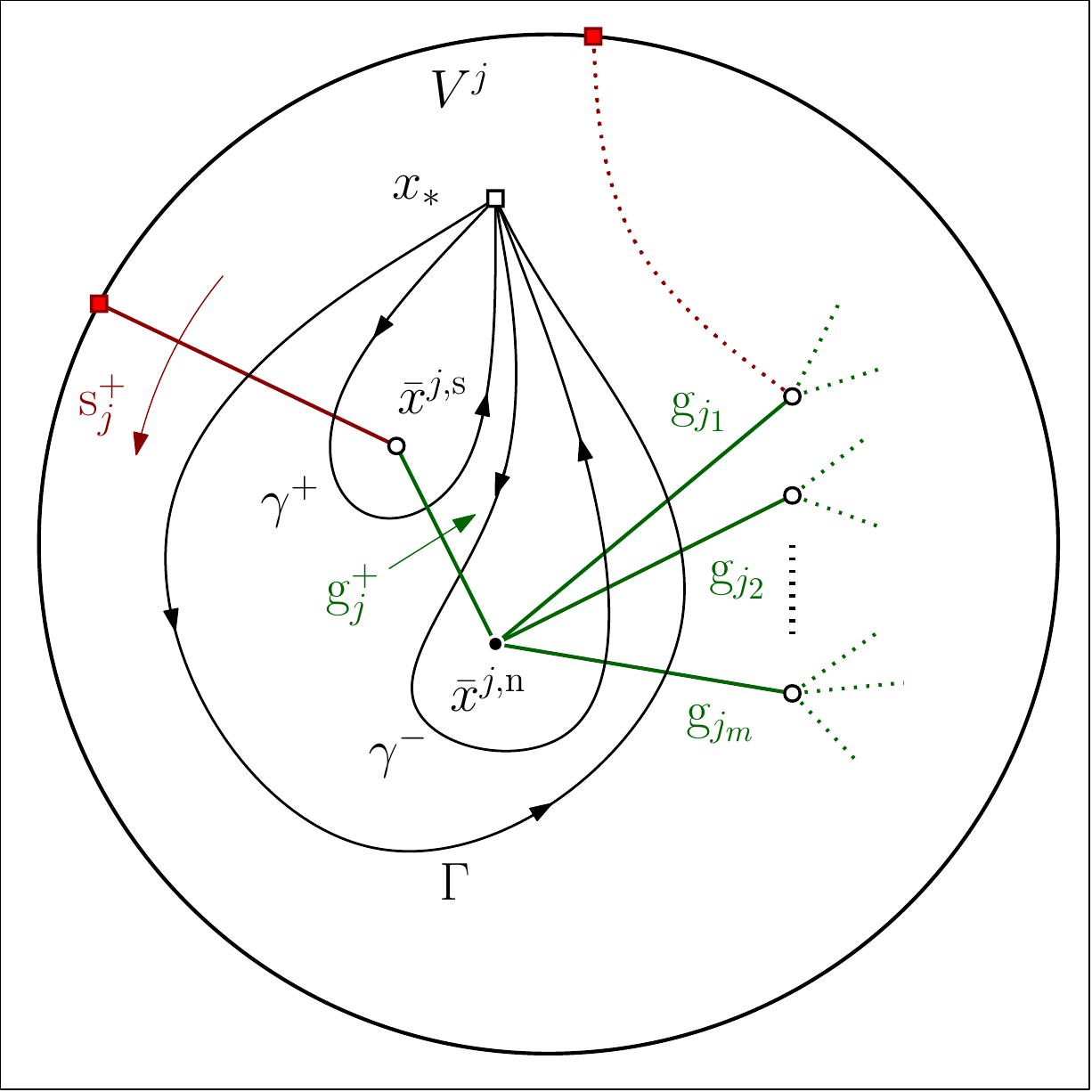}

}\hfill{}\subfloat[In the parameter $\widetilde{\varepsilon}$]{\includegraphics[width=6cm]{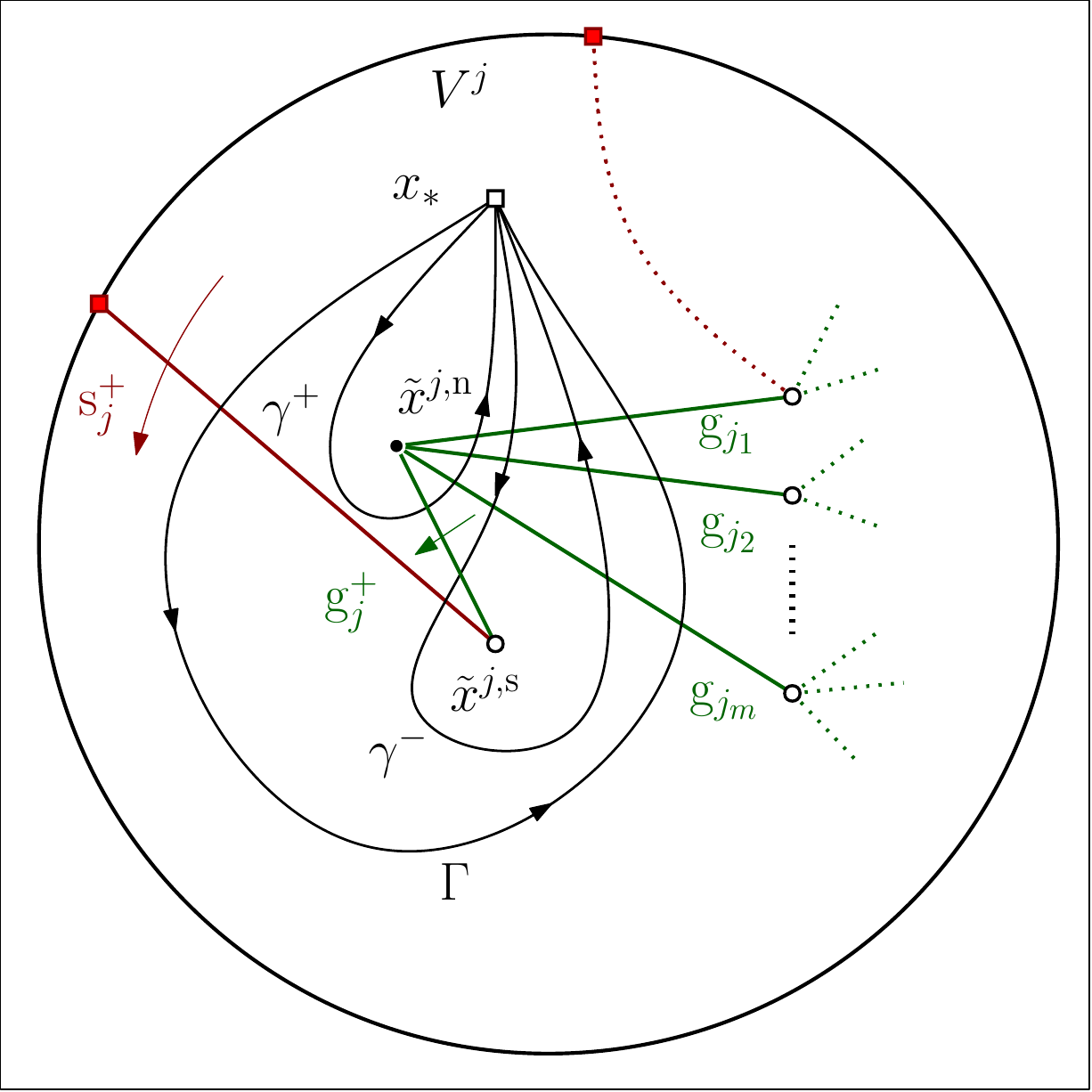}

}\hfill{}

\caption{\label{fig:analytic_invar_k>1}The construction involved in Lemma~\ref{lem:analytic_invar_thus_commute}.}
\end{figure}
\begin{lem}
\label{lem:analytic_invar_thus_commute}We follow the notations of
Figure~\ref{fig:analytic_invar_k>1}. Let $k>1$, and let $m\geq1$
be the number of singular points different from $\sad x{}{j,}$ and
$\nod x{}{j,}$. Each passage of a gate by $\Gamma$ in the figure
yields a linear gate map $L_{\overline{\nu}^{j_{p}}}$ (\emph{resp}.
$L_{\widetilde{\nu}^{j_{p}}}$) for some $\overline{\nu}^{j_{p}}\in\cc^{\times}$
(\emph{resp.} $\widetilde{\nu}^{j_{p}}\in\cc^{\times}$) and $1\leq p\leq m$.
We also set $\gat{\psi}{}{j,}=:L_{\nu^{j}}$.

\begin{enumerate}
\item The equality $\overline{\psi}\left[\Gamma\right]=\widetilde{\psi}\left[\Gamma\right]$
holds, defining a germ $\Delta\in\diff[\cc,0]$. 
\item ~
\begin{align*}
\Delta'\left(0\right)\exp\frac{-2\ii\pi\mu}{k} & =\prod_{p=1}^{m}\overline{\nu}^{j_{p}}=\prod_{p=1}^{m}\widetilde{\nu}^{j_{p}}.
\end{align*}
\item ~ 
\begin{align*}
\delta\circ\Delta & =\Delta\circ\delta.
\end{align*}
\end{enumerate}
\end{lem}

\begin{proof}
Observe that 
\begin{align}
\widetilde{\psi}\left[\gamma^{+}\right] & =L_{\nf{\widetilde{c}}{\widetilde{\nu}^{j}}}~~~~~,~~\widetilde{c}:=\prod_{p=1}^{m}\widetilde{\nu}^{j_{p}}.\label{eq:linear_modulus_analytic}
\end{align}
The linear part is invariant by conjugacy so that $\widetilde{c}=\overline{\nu}^{j}\widetilde{\nu}^{j}\exp\nf{2\ii\pi\mu}k$.
Similarly, considering $\gamma^{-}$ yields $\overline{c}=\overline{\nu}^{j}\widetilde{\nu}^{j}\exp\nf{2\ii\pi\mu}k$.
Hence
\begin{align*}
\overline{c} & =\prod_{p=1}^{m}\overline{\nu}^{j_{p}}=\overline{\nu}^{j}\widetilde{\nu}^{j}\exp\nf{2\ii\pi\mu}k=\widetilde{c}.
\end{align*}
Since $\widetilde{\psi}\left[\Gamma\right]=L_{\widetilde{c}}\circ\sad{\psi}{}{j,}$
and $\overline{\psi}\left[\Gamma\right]=L_{\overline{c}}\circ\sad{\psi}{}{j,}$
the result follows.
\end{proof}
Recall that the map
\begin{align*}
\varepsilon\in\mathcal{E}_{\ell} & \longmapsto\left(\nu_{\varepsilon}^{j}\right)_{j\in\zsk}
\end{align*}
is locally injective (Lemma~\ref{lem:gate_is_holomorphic}). In particular
$\Delta'\left(0\right)$ is not constant and therefore must take non-rational
values on a small subdomain $\Lambda\subset\mathcal{E}_{\ell}^{\cap}$.
It follows that for $\varepsilon\in\Lambda$ the Abelian group $\left\langle \delta,\Delta\right\rangle <\diff[\cc,0]$
is non-resonant and therefore formally linearizable~\cite{LoRigid}.
Hence $\delta=\id$. 
\begin{lem}
\label{lem:analytic_invariant_trivial}If $\delta=\id$ then $\phi^{j}=0$.
\end{lem}

\begin{proof}
According to~\eqref{eq:dynamics_conjugacy_bernoulli}, $\delta$
conjugates $\psi_{\overline{\varepsilon}}\left[\gamma^{+}\right]=L_{\overline{\nu}^{j}}\circ\sad{\psi}{}{j,}$
to $\psi_{\widetilde{\varepsilon}}\left[\gamma^{+}\right]$, but the
latter is linear thanks to~\eqref{eq:linear_modulus_analytic}, therefore
$\sad{\psi}{}{j,}$ itself is linear. It can only mean that $\sad{\phi}{}{j,}=0=\phi^{j}$
using~\eqref{eq:necklace_diffeo}, the equality holding on the whole
cell $\mathcal{E}_{\ell}$ by analytic continuation. 
\end{proof}
Since $j$ is arbitrary we just established
\begin{align*}
\left(k>1\right) & \Longrightarrow\left(\phi=0\right).
\end{align*}

\subsubsection{The case $k=1$: end of the proof of the Parametrically Analytic
Orbital Moduli Theorem}

Since $k=1$ we drop the index $j=0$. We work in the self-intersection
$\mathcal{E}^{\cap}$ of the single parametric cell, and use the notations
and constructions involved just above. In particular Figure~\ref{fig:analytic_invar_k>1}
remains the same except for the fact that there are no gate passage
$j_{1},\ldots,j_{m}$ on the right-hand side of the pictures. 

Recall that we consider a system with $\mathfrak{m}\neq0$. Lemma~\ref{lem:analytic_invariant_trivial}
forbids $\delta=\id$, thus $\Delta=\sad{\psi}{}{}$ is non-linear
($\Delta$ was introduced in Lemma~\ref{lem:analytic_invar_thus_commute}).
Then $\left\langle \delta,\Delta\right\rangle $ is an Abelian group.
Consequently there exists~\cite{MaRa-SN} a formal tangent-to-identity
change $\widehat{\varphi}$ in the variable $h$, unique $d\in\nn$,
$\lambda\in\cc$ and $t\in\cc\backslash\left\{ 0\right\} $ such that,
writing $\widehat{f}:=\widehat{\varphi}^{*}f$ for all $f\in\diff[\cc,0]$,
\begin{align*}
\widehat{\delta} & =\flow{Z\left(d,\lambda\right)}1{}\\
\widehat{\Delta} & =\alpha\flow{Z\left(d,\lambda\right)}t{}~,~\alpha\in\cc^{\times}\\
Z\left(d,\lambda\right) & =\frac{h^{d+1}}{1+\lambda h^{d}}\pp h.
\end{align*}
Commutativity forces the relation 
\begin{align*}
\alpha^{d} & =1.
\end{align*}
Since $\alpha=\exp2\ii\pi\mu$ this gives $d\mu\in\zz$ as expected.
Observe that for all $s\in\cc$
\begin{align*}
\flow{Z\left(d,0\right)}s{} & \in\bber[d],
\end{align*}
therefore we aim at showing $\lambda=0$. This is ultimately done
by applying the next lemma.
\begin{lem}
\cite[Assertions 1.1 to 1.4]{CerMou}\label{lem:formal_abelian_diffeos}
In the following $\xi$ is a formal diffeomorphism in the variable
$h$ at $0$.

\begin{enumerate}
\item Let $Z,~\widetilde{Z}$ be formal vector fields in the variable $h$
at $0$ belonging to $h^{d+1}\frml h^{\times}\pp h$. If $\xi^{*}\flow Z1{}=\flow{\widetilde{Z}}1{}$
then $\xi^{*}Z=\widetilde{Z}$ (the converse is trivial).
\item Assume that $\xi^{*}Z\left(d,\lambda\right)=aZ\left(d,\lambda\right)$
with $a\neq1$. Then $\lambda=0$ and $\xi\in\bber$ (in particular
$\xi$ is analytic). 
\end{enumerate}
\end{lem}

Let us show now that $\lambda=0$ and $\widehat{\varphi}\in\bber$
itself, forcing $\Delta=\sad{\psi}{}{}\in\bber$ by application of
Lemma~\ref{lem:Bernoulli_properties}. The key is to exploit the
fact~\eqref{eq:dynamics_conjugacy_bernoulli}, which can be rewritten
as: 
\begin{align}
\delta^{*}\overline{\psi}\left[\tx g^{+}\tx s^{+}\right] & =\widetilde{\psi}\left[\tx g^{-}\right]=L_{1/\widetilde{\nu}}.\label{new_number}
\end{align}
Indeed, referring to Definition~\ref{def:necklace_dynamics} for
the definition of the letters $\tx g^{\pm},~\tx s^{\pm}$ and their
image by $\psi\left[\bullet\right]$, and looking at Figure~\ref{fig:analytic_invar_k>1},
we compare the holonomy maps around the upper singular point. On the
left, the singular point is of saddle type and the holonomy map is
the composition of $\overline{\psi}\left[s^{+}\right]$ (crossing
the saddle sector in the direction of the arrow) with $\overline{\psi}\left[\tx g^{+}\right]$
(crossing the gate sector in the direction of the arrow). On the right,
the same singular point is of node type. Turning around, it comes
to crossing the gate sector in the inverse direction of the arrow.
Hence its holonomy map is $\widetilde{\psi}\left[\tx g^{-}\right]$.
The last equality in \eqref{new_number} follows from the fact that
$\widetilde{\psi}\left[\tx g^{+}\right]=L_{\widetilde{\nu}}$. Note
that \eqref{new_number} means that $\delta$ linearizes $\overline{\psi}\left[\tx g^{+}\tx s^{+}\right]$.

Of course the multipliers at the fixed point in \eqref{new_number}
must be the same. On the left, this multiplier is simply that of $\overline{\psi}\left[\tx g^{+}\tx s^{+}\right]$,
since conjugacy by $\delta$ preserves the multiplier. On the one
hand the multiplier at the fixed point of $\overline{\psi}\left[\tx s^{+}\right]$
is $\exp2\ii\pi\mu$, according to \eqref{eq:necklace_diffeo} for
$k:=1$, as indeed $\overline{\psi}\left[\tx g^{+}\right]=L_{\overline{\nu}}$.
On the other hand $\widetilde{\psi}\left[\tx g^{-}\right]=L_{1/\widetilde{\nu}}$
so that

\begin{align*}
\widetilde{\nu~}\overline{\nu}\exp2\ii\pi\mu & =1.
\end{align*}
We also have $\overline{\psi}\left[\tx s^{+}\right]=\Delta$, since
it is the holonomy obtained by turning counterclockwise around the
two singular points. Hence, replacing in \eqref{new_number} yields
$L_{\overline{\nu}}\circ\Delta\circ\delta=\delta\circ L_{1/\widetilde{\nu}}$.
Composing both sides on the left with $L_{\widetilde{\nu}}$ and taking
$\widehat{\varphi}^{*}$ on both sides yields 
\begin{align*}
\widehat{L}_{1/\widetilde{v}}^{*}\widehat{\delta} & =\widehat{L}_{\widetilde{\nu}}\circ\widehat{\delta}\circ\widehat{L}_{1/\widetilde{\nu}}=\widehat{L}_{\overline{\nu}\widetilde{\nu}}\circ\widehat{\Delta}\circ\widehat{\delta}.
\end{align*}
 For the sake of simplicity we only deal with the case $\mu\in\zz$,
the general case can be adapted by taking into account that $\widehat{L}_{\overline{\nu}\widetilde{\nu}}^{\circ d}=\id$.
Under the current hypothesis $\widehat{L}_{\overline{\nu}\widetilde{\nu}}=\id$,
so that $\widehat{L}_{1/\widetilde{\nu}}$ is a formal conjugacy between
$\widehat{\delta}=\flow{Z\left(d,\lambda\right)}1{}$ and $\widehat{\Delta}\circ\widehat{\delta}=\flow{Z\left(d,\lambda\right)}{1+t}{}=\flow{\left(1+t\right)Z\left(d,\lambda\right)}1{}$
for some $t=t_{\varepsilon}\in\cc^{\times}$. According to Lemma~\ref{lem:formal_abelian_diffeos}
with $\xi:=\widehat{L}_{1/\widetilde{\nu}}$ and $a:=1+t\neq1$, we
have $\lambda=0$ and $\widehat{L}_{1/\widetilde{\nu}}\in\bber$. 

\bigskip{}

So far $\widehat{\varphi}$ is a formal linearization of $\widehat{L}_{1/\widetilde{\nu}}$
which is tangent-to-identity. For values $\varepsilon$ of the parameter
corresponding to $\widetilde{\nu}\notin\rr$ (say $\im{\widetilde{\nu}}>0$)
the fix-point $0$ of $\widehat{L}_{1/\widetilde{\nu}}$ is hyperbolic:
the map $\widehat{\varphi}$ is locally holomorphic at $0$, unique
and therefore given by 
\begin{align*}
\widehat{\varphi} & :=\lim_{n\to\infty}L_{-n/\widetilde{\nu}}\circ\widehat{L}_{1/\widetilde{\nu}}^{\circ n}
\end{align*}
uniformly on a neighborhood of $0$. Lemma~\ref{lem:Bernoulli_properties}
implies that for every $n\in\nn$ we have
\begin{align*}
L_{-n/\widetilde{\nu}}\circ\widehat{L}_{1/\widetilde{\nu}}^{\circ n} & \in\bber[d],
\end{align*}
therefore $\widehat{\varphi}\in\bber$ as requested, since the group
$\bber$ is closed for the topology of local uniform convergence.
This completes the proof of the Parametrically Analytic Orbital Moduli
Theorem.

\section{\label{sec:Computations}A few words about computations}

All the discussion regarding the actual (symbolic or numeric) computations
of normal forms and moduli of saddle-nodes, as presented in~\cite[Section 4]{SchaTey}
for saddle-nodes, can be repeated \emph{verbatim }in the case of convergent
unfoldings: we will not reproduce it here. We nonetheless present
in Section~\ref{subsec:first-order} a consequence of one particular
result, thus unfolding the main result of~\cite{Tey-ExSN}, which
leads us to try and compute the period associated to the formal orbital
normal form $\fonf$ in Section~\ref{subsec:Model_period_computations}. 

\subsection{\label{subsec:first-order}Computation of the dominant term of the
orbital invariant}

\begin{figure}
\hfill{}\includegraphics[width=5cm]{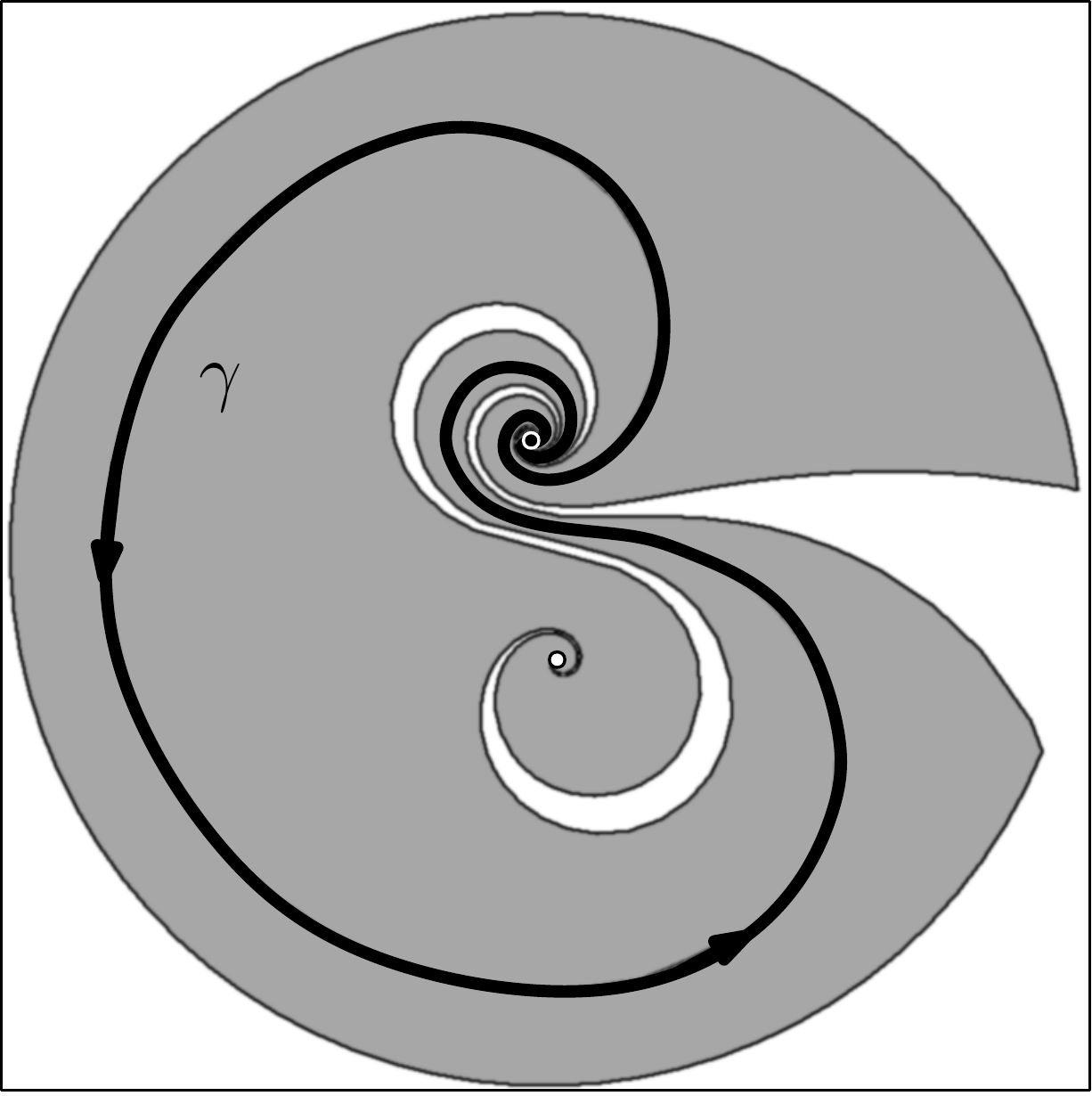}\hfill{}

\caption{\label{fig:asymptotic_cycle}The (asymptotic) path of integration
used to compute the period, which is a cycle when $k=1$.}

\end{figure}
The next lemma holds for a fixed value of $\varepsilon\in\mathcal{E}_{\ell}$.
\begin{lem}
\emph{(See \cite[Proposition 4.1]{SchaTey})} Let $r_{n}\in x\pol x_{<k}$
be the coefficients of 
\begin{align}
R\left(x,y\right) & =\sum_{n>0}r_{n}\left(x\right)\left(P^{\tau}\left(x\right)y\right)^{n}\label{eq:normal_form_expression}
\end{align}
 in the normal form $\onf$. Let $c_{\ell}^{j,p}\left(n,m\right)\in\cc$
be the coefficients of the period 
\begin{align*}
\per[j][\ell]\left(x^{n}y^{m}\right)\left(h\right) & =\sum_{p>0}c_{\ell}^{j,p}\left(n,m\right)h^{p}
\end{align*}
relative to $\onf$. Then we have the following properties.
\begin{lyxlist}{00.00.0000}
\item [{\textbf{Triangularity:}}] $c_{\ell}^{j,p}\left(n,m\right)=0$,
if $p<m$ and 
\begin{align*}
c_{\ell}^{j,m}\left(m,m\right)h^{m} & =2\ii\pi\widehat{\per[][\ell]}^{j}\left(x^{n}y^{m}\right)\left(h\right)
\end{align*}
 is independent of $R$.
\item [{\textbf{Algebraicity:}}] For $p>m,$ the coefficient $c_{\ell}^{j,p}\left(n,m\right)$
depends polynomially in the $k\left(p-m\right)$ variables given by
the coefficients of $r_{1},\ldots,r_{p-m}$ and vanishes when $R=0$.
\end{lyxlist}
\end{lem}

\begin{proof}
It is exactly the proof done in~\cite[Proposition 4.1]{SchaTey}
since exchanging $x^{k+1}$ for $P_{\varepsilon}\left(x\right)$ does
not modify anything in the actual computation. We give some brief
elements of the proof. 

Let us drop all indexes and let $x\mapsto y\left(x,h\right)$ be the
sectorial solution of the differential equation induced by the vector
field $\onf$ with initial value $H\left(x_{*},y\left(x_{*},h\right)\right)=h$
(here $x_{*}$ is fixed once and for all in $\sad V{}{}$). Computing
$\per\left(x^{n}y^{m}\right)\left(h\right)$ requires to compute the
integral $\int_{\gamma}x^{n}y\left(x,h\right)^{m}\frac{\dd x}{P\left(x\right)}$
for an asymptotic path $\gamma\subset\cc\times\left\{ 0\right\} $
joining the two nodes in the closure of the union of consecutive squid
sectors (see Figure~\ref{fig:asymptotic_cycle}). This integral is
absolutely convergent because $m>0$ and $\gamma$ spirals in the
right manner (see also Lemma~\ref{lem:integrable_first-integral}).
Since
\begin{align*}
H\left(x,y\right) & =\widehat{H}\left(x,y\right)\exp N\left(x,y\right),
\end{align*}
with $\widehat{H}$ linear in the $y$-variable, we necessary have
\begin{align}
y\left(x,h\right) & =\widehat{y}\left(x,h\right)+h\OO h\label{eq:sectorial_solution_expansion}
\end{align}
where $\widehat{y}\left(x,h\right)=\frac{h}{\widehat{H}\left(x,1\right)}$
is the solution corresponding to the formal model $\fonf$. This gives
the triangularity. The algebraicity property stems from the fact that
the computation can be performed formally in the $y$-variable. The
sought property is true for the expansion~\eqref{eq:sectorial_solution_expansion}
(by studying the inverse of the normalizing mapping) because it is
true for solutions of cohomological equations $\mathcal{X}\cdot N=-R$.
\end{proof}
We extract from this statement useful consequences. 
\begin{prop}
\label{prop:first_order_invariant}~

\begin{enumerate}
\item The quantity 
\begin{align*}
\inf\left\{ n~:~r_{n}\neq0\right\}  & =\inf\left\{ n~:~\left(\exists j\right)~\phi_{n}^{j}\neq0\right\} \\
 & =:d\in\nbar
\end{align*}
does not depend on the cell. 
\item The valuation $d$ is infinite if and only if the unfolding is analytically
conjugate to its formal normal form.
\item If $d<\infty$ the dominant term of the invariant is given by the
period of the formal model 
\begin{align*}
2\ii\pi\widehat{\per}\left(r_{d}y^{d}\right) & =\left(h\mapsto\phi_{d}^{j}h^{d}\right)_{j\in\zsk}.
\end{align*}
\end{enumerate}
\end{prop}

\begin{rem}
The value of $d$ does not depend on the cell but may differ from
the value obtained on the boundary $\Delta_{k}$. Yet the analytic
continuation principle ensures that
\begin{align*}
\min_{\varepsilon\in\Delta_{k}}\inf\left\{ n~:~r_{n,\varepsilon}\neq0\right\}  & \geq d
\end{align*}
 because $R$ is analytic.
\end{rem}

From this proposition we deduce a final normalization ensuring uniqueness.
\begin{cor}
Assume the generic convergent unfolding $X$ is not analytically conjugate
to its formal normal form $\fonf$ defined in~\eqref{eq:formal_orbital_normal_form}.
There exists a unique $\left(\kappa,j,d\right)\in\zp\times\left\{ 1,2,\ldots,k\right\} \times\nn$
such that $X$ is conjugate to the normal form $\onf[~]=\fonf+Ry\pp y$
as in~\eqref{eq:normal_form_expression} where:
\begin{align*}
r_{\varepsilon,d}\left(x\right) & =\varepsilon^{\kappa}x^{j}+\oo{x^{j}}\\
r_{\varepsilon,n} & =0~~~~\mbox{if }n<d.
\end{align*}
\end{cor}

Notice that in the case $\kappa>0$ this normal form may fail to deliver
meaningful information at the limit $\varepsilon\to0$. Take the extreme
case $R_{\varepsilon}\left(x,y\right)=\varepsilon^{\kappa}x^{j}y^{d}$
with $\kappa>0$: for every $\varepsilon\neq0$ the vector field $\onf$
is not equivalent to the model $\fonf[\varepsilon]$ but $\onf[0]$
is.

\subsection{\label{subsec:Model_period_computations}Formula  for the period
of formal models}

Unfortunately only the case $k=1$ seems tractable enough to obtain
closed-form expressions involving the Gamma function. For the case
$k=2$ one could derive a closed-form formula additionally using generalized
hypergeometric functions, which is already stretching a bit far what
a ``closed-form'' is. There is no evidence that similar calculations
can be performed for $k>2$. 
\begin{prop}
\label{prop:period_analytic_k=00003D1}\cite[Proposition 8]{TeySurvey}
Here $k=1$. Let us introduce the double covering $\varepsilon=-s^{2}$
in the parameter space. Then for $m\in\nn$ and $n\in\zp$: 
\begin{align*}
\widehat{\per}_{s}\left(x^{n}y^{m}\right)\left(h\right) & =h^{m}\times\frac{\left(-m\right)^{n+m\mu}}{\Gama{n+m\mu}}\times t_{s,n,m}\times T_{s,m}\\
t_{s,n,m} & :=\frac{1}{2^{n}}\sum_{p+q=n}\binom{n}{p}\prod_{j=0}^{p-1}\left(1-s\left(\mu+\frac{2j}{m}\right)\right)\prod_{j=0}^{q-1}\left(1+s\left(\mu+\frac{2j}{m}\right)\right)\\
T_{s,m} & :=\frac{\left(-\frac{2s}{m}\right)^{m\mu}}{1+s\mu}\times\frac{\Gama{-\frac{m}{2s}+\frac{m\mu}{2}}}{\Gama{-\frac{m}{2s}-\frac{m\mu}{2}}}.
\end{align*}
This period is holomorphic and bounded in the parameter $s$ on the
sector 
\begin{align*}
S & :=\left\{ 0<\left|s\right|<\frac{1}{2\left|\mu_{0}\right|},~\frac{\pi}{4}<\arg s<\frac{7\pi}{4}\right\} 
\end{align*}
 and extends continuously at $0$ by 
\begin{align*}
\widehat{\per}_{0}\left(x^{n}y^{m}\right)\left(h\right) & =h^{m}\times\frac{\left(-m\right)^{n+m\mu_{0}}}{\Gama{n+m\mu_{0}}}.
\end{align*}
For given $s$ small enough, the period is zero if and only if $n+m\mu_{\varepsilon}\in\zz_{\leq0}$.
The period is an even function of $s$ (\emph{i.e.} holomorphic in
the parameter $\varepsilon$) if and only if $m\mu\in\zz$. In that
case $\mu$ is a rational constant.
\end{prop}

The result is shown by using the Pochhammer contour integral formula
for the Beta function. Indeed an affine change of coordinates sends
$\left(x-s\right)^{\alpha}\left(x+s\right)^{\beta}$ to a multiple
of $\left(1-z\right)^{\alpha}z^{\beta}$. The final expression comes
from diverse classical properties of the Gamma function. The eventual
lack of evenness of the period comes from the term $T_{s,m}$. If
$m\mu$ is not an integer then $T_{s,m}$ is multivalued and has an
accumulation of zeros and poles as $s\to0$ outside the sector $S$.
Only the coincidence of these two infinite sets when $m\mu\in\zz$
allows the period to be holomorphic through lucky root~/~pole cancellations.

\bigskip{}

Since $T_{s,m}$ is independent on $n$, any nonzero period $\widehat{\per}\left(y^{m}g\right)$
of a germ $g\in\germ{\varepsilon,x}$ is holomorphic in $\varepsilon$
if and only if $m\mu\in\zz$. From Lemma~\ref{lem:bernoulli_unfold_with_bernoulli_modulus},
Theorem~\ref{thm:section_period} and the Parametrically Analytic
Orbital Moduli Theorem we can generalize this observation.
\begin{cor}
\label{prop:period_analytic}Let $G\in\germ{\varepsilon,x,y}$ with
$G\left(\bullet,0\right)=\OO P$. Let us assume that the period $\widehat{\per}\left(G\right)$
is nonzero. Then, $\widehat{\per}\left(G\right)$ is holomorphic in
the parameter if and only if all three conditions hold:

\begin{itemize}
\item $k=1$ ,
\item there exists $d\in\nn$ such that $d\mu\in\zz$,
\item there exist two germs $F\in\germ{\varepsilon,x,y}$ and $Q\in\nfsec[1][P^{d\tau}y^{d}]\backslash\left\{ 0\right\} $
such that 
\begin{align*}
G & =Q+\fonf\cdot F.
\end{align*}
\end{itemize}
\end{cor}

The fact that the period is never a holomorphic function of the parameter
if $k>1$ is probably a sign that a ``simple'' formula for $\widehat{\per}\left(x^{n}y^{m}\right)$
does not exist.

\bibliographystyle{plain}
\bibliography{bibliography/unfoldings,bibliography/foliations,bibliography/saddle-node}

\end{document}